\documentclass{article}

\usepackage[final,nonatbib]{neurips_2022}

\usepackage[utf8]{inputenc} % allow utf-8 input
\usepackage[T1]{fontenc}    % use 8-bit T1 fonts
\usepackage[pagebackref=true]{hyperref}  % hyperlinks
\renewcommand*{\backrefalt}[4]{%
    \ifcase #1 \footnotesize{(Not cited.)}%
    \or        \footnotesize{(Cited on page~#2)}%
    \else      \footnotesize{(Cited on pages~#2)}%
    \fi}
\usepackage{url}            % simple URL typesetting
\usepackage{booktabs}       % professional-quality tables
\usepackage{amsfonts}       % blackboard math symbols
\usepackage{nicefrac}       % compact symbols for 1/2, etc.
\usepackage{microtype}      % microtypography

\usepackage{natbib}
\usepackage{sidecap}

\usepackage{amssymb, amsmath, amsthm, latexsym}
\usepackage{url}
\usepackage{algorithm}
\usepackage{algorithmic}
\usepackage{tabularx}
\usepackage{paralist}
\usepackage{mathtools}

\usepackage{bbm} %for indicators 1
\usepackage{wrapfig}
\usepackage{makecell}
\usepackage{multirow}
\usepackage{booktabs}

\usepackage{nicefrac}       % nice fractions

\usepackage[flushleft]{threeparttable} % http://ctan.org/pkg/threeparttable

\usepackage{caption}
\usepackage{multirow}
\usepackage{colortbl}
\definecolor{bgcolor}{rgb}{0.8,1,1}
\definecolor{bgcolor2}{rgb}{0.8,1,0.8}
\definecolor{niceblue}{rgb}{0.0,0.19,0.56}

\usepackage{hyperref}
\hypersetup{colorlinks,linkcolor={blue},citecolor={niceblue},urlcolor={blue}}

\usepackage[dvipsnames]{xcolor}
\usepackage{pifont}
\newcommand{\cmark}{{\color{PineGreen}\ding{51}}}%
\newcommand{\xmark}{{\color{BrickRed}\ding{55}}}%

\usepackage{tikz-cd} %for a diagram

\usepackage{subfigure}

\newcommand{\R}{\mathbb{R}}
\newcommand{\eqdef}{\stackrel{\text{def}}{=}}

\def\<#1,#2>{\left\langle #1,#2\right\rangle}

\newcolumntype{Y}{>{\centering\arraybackslash}X}

\newtheorem{lemma}{Lemma}[section]
\newtheorem{theorem}{Theorem}[section]

\newtheorem{assumption}{Assumption}[section]
\newtheorem{corollary}{Corollary}[section]

\theoremstyle{plain}

\usepackage{xspace}

\newcommand\figscale{0.25}
\newcommand\evofigscale{0.21}

\newcommand{\algname}[1]{{\sf  #1}\xspace}

\newcommand{\circledOne}{\text{\ding{172}}}
\newcommand{\circledTwo}{\text{\ding{173}}}
\newcommand{\circledThree}{\text{\ding{174}}}
\newcommand{\circledFour}{\text{\ding{175}}}
\newcommand{\circledFive}{\text{\ding{176}}}
\newcommand{\circledSix}{\text{\ding{177}}}
\newcommand{\circledSeven}{\text{\ding{178}}}
\newcommand{\circledEight}{\text{\ding{179}}}

\usepackage[colorinlistoftodos,bordercolor=orange,backgroundcolor=orange!20,linecolor=orange,textsize=scriptsize]{todonotes}

\newcommand{\cD}{{\cal D}}

\newcommand{\cO}{{\cal O}}
\newcommand{\cP}{{\cal P}}

\newcommand{\cX}{{\cal X}}

\newcommand{\EE}{\mathbb{E}}
\newcommand{\PP}{\mathbb{P}}

\newcommand{\tx}{\widetilde{x}}
\newcommand{\tX}{\widetilde{X}}

\newcommand{\tF}{\widetilde{F}}

\def\Bxi{\boldsymbol{\xi}}

\def\clip{\texttt{clip}}
\def\avg{\texttt{avg}}
\def\gap{\texttt{Gap}}

\usepackage{hyperref}
\graphicspath{{plots/}}

\usepackage{makecell}

\usepackage{accents}
\newlength{\dhatheight}

\usepackage{pgfplotstable} % loads pgfplots which loads tikz
\usetikzlibrary{automata, positioning, arrows, shapes, fit, calc, intersections}
\usepgfplotslibrary{statistics}
\include{figure}

\newcommand{\pd}[1]{{\color{black}#1}}
\newcommand{\ggi}[1]{{\color{black}#1}}

\title{Clipped Stochastic Methods \\ for Variational Inequalities with Heavy-Tailed Noise}

\author{%
  Eduard Gorbunov\thanks{Equal contribution.}~~\thanks{Corresponding author: \texttt{eduard.gorbunov@mbzuai.ac.ae}}\\
  MIPT, Russia\\
  Mila \& UdeM, Canada\\
  MBZUAI, UAE\\
  \And
  Marina Danilova$^\ast$\\
  MIPT, Russia\\
  \And
  David Dobre$^\ast$\\
  Mila \& UdeM, Canada\\
  \And
  Pavel Dvurechensky\\
  WIAS, Germany\\
  \And
  Alexander Gasnikov\\
  MIPT, Russia\\
  HSE University, Russia\\
  IITP RAS, Russia\\
  \And
  Gauthier Gidel \\
  Mila \& UdeM, Canada\\
  Canada CIFAR AI Chair \\
}

\begin{document}

\maketitle

\begin{abstract}
    Stochastic first-order methods such as Stochastic Extragradient (\algname{SEG}) or Stochastic Gradient Descent-Ascent (\algname{SGDA}) for solving smooth minimax problems and, more generally, variational inequality problems \eqref{eq:main_problem} have been gaining a lot of attention in recent years due to the growing popularity of adversarial formulations in machine learning. However, while high-probability convergence bounds are known to reflect the actual behavior of stochastic methods more accurately, most convergence results are provided in expectation. Moreover, the only known high-probability complexity results have been derived under restrictive sub-Gaussian (light-tailed) noise and bounded domain assumption~\citep{juditsky2011solving}. In this work, we prove the first high-probability complexity results with logarithmic dependence on the confidence level for stochastic methods for solving monotone and structured non-monotone \ref{eq:main_problem}s with non-sub-Gaussian (heavy-tailed) noise and unbounded domains. In the monotone case, our results match the best-known ones in the light-tails case \citep{juditsky2011solving}, and are novel for structured non-monotone problems such as negative comonotone, quasi-strongly monotone, and/or star-cocoercive ones. We achieve these results by studying \algname{SEG} and \algname{SGDA} with clipping. In addition, we numerically validate that the gradient noise of many practical GAN formulations is heavy-tailed and show that clipping improves the performance of \algname{SEG}/\algname{SGDA}.
\end{abstract}

\section{Introduction}

Recently, game formulations have been receiving a lot of interest from the optimization and machine learning communities. In such problems, different models/players competitively minimize their loss functions, e.g., see adversarial example games~\citep{bose2020adversarial}, hierarchical reinforcement learning~\citep{wayne2014hierarchical,vezhnevets2017feudal}, and  generative adversarial networks (GANs)~\citep{goodfellow2014generative}. Very often, such problems are studied through the lens of solving a \emph{variational inequality problem} (VIP) \citep{harker1990finite, ryu2021large, gidel2019variational}. In the unconstrained case, for an operator\footnote{For example, when we deal with a differentiable game/minimax problem, $F$ can be chosen as the concatenation of the gradients of the players' objective functions, e.g., see \citep{gidel2019variational} for the details.} $F:\R^d \to \R^d$ VIP can be written as follows:
\begin{equation}
    \text{find } x^* \in \mathcal{X}^* \qquad \text{where} \qquad \mathcal{X}^* := \{ x^* \in \R^d \;\; \text{such that}\;\; F(x^*) = 0\}. \label{eq:main_problem} \tag{VIP}
\end{equation}
In machine learning applications, operator $F$ usually has an expectation form, i.e., $F(x) = \EE_{\xi}[F_\xi(x)]$, where $x$ corresponds to the parameters of the model, $\xi$ is a sample from some (possibly unknown) distribution $\cD$, and $F_{\xi}(x)$ is the operator corresponding to the sample $\xi$.

Such problems are typically solved via first-order stochastic methods such as Stochastic Extragradient (\algname{SEG}), also known as \algname{Mirror-Prox} algorithm \citep{juditsky2011solving}, or Stochastic Gradient Descent-Ascent (\algname{SGDA}) \citep{dem1972numerical, nemirovski2009robust} due to their practical efficiency. However, despite the significant attention to these methods and their modifications, their convergence is usually analyzed in expectation only, e.g., see \citep{gidel2019variational, hsieh2019convergence, hsieh2020explore, mishchenko2020revisiting, loizou2021stochastic}. In contrast, while high-probability analysis more accurately reflects the behavior of the stochastic methods \citep{gorbunov2020stochastic}, a little is known about it in the context of solving \ref{eq:main_problem}. To the best of our knowledge, there is only one work addressing this question for monotone variational inequalities \citep{juditsky2011solving}. However, \citep{juditsky2011solving} derive their results under the assumption that the problem has ``light-tailed'' (sub-Gaussian) noise and bounded domain, which is restrictive even for simple classes of minimization problems \citep{zhang2020adaptive}. This leads us to the following open question.
\begin{gather*}
    \textbf{{\color{blue} Q1}: }\textit{Is it possible to achieve the same high-probability results as in \citep{juditsky2011solving}}\\
    \textit{without assuming that the noise is sub-Gaussian and the domain is bounded?}
\end{gather*}
\looseness=-1
Next, in the context of GANs' training, empirical investigation~\citep{jelassi2022adam} and practical use~\citep{gulrajani2017improved,miyato2018spectral,tran2019self,brock2019large,Sauer2021ARXIV} indicate the practical superiority of \algname{Adam}-based methods, e.g., alternating \algname{SGDA} with stochastic estimators from \algname{Adam} \citep{kingma2014adam}, over classical methods like \algname{SEG} or (alternating) \algname{SGDA}. This interesting phenomenon has no rigorous theoretical explanation yet. In contrast, there exists a partial understanding of why \algname{Adam}-like methods perform well in different tasks such as training attention models \citep{zhang2020adaptive}. In particular, \citet{zhang2020adaptive} empirically observe that stochastic gradient noise is heavy-tailed in several NLP tasks and theoretically shows that vanilla \algname{SGD} \citep{robbins1951stochastic} can diverge in such cases and its version with gradient clipping (\algname{clipped-SGD}) \citep{pascanu2013difficulty} converges. Moreover, the state-of-the-art high-probability convergence results for heavy-tailed stochastic minimization problems are also obtained for the methods with gradient clipping \citep{nazin2019algorithms, gorbunov2020stochastic, gorbunov2021near, cutkosky2021high}. Since \algname{Adam} can be seen as a version of adaptive \algname{clipped-SGD} with momentum \citep{zhang2020adaptive}, these advances establish the connection between good performance of \algname{Adam}, heavy-tailed gradient noise, and gradient clipping for \emph{minimization problems}. Motivated by these recent advances in understanding the superiority of \algname{Adam} \pd{for minimization problems}, we formulate the following research question.
\begin{gather*}
    \textbf{{\color{blue} Q2}: }\textit{In the training of popular GANs, does the gradient noise have heavy-tailed distribution}\\
    \textit{and does gradient clipping improve the performance of the classical \algname{SEG}/\algname{SGDA}?}
\end{gather*}
In this paper, we give positive answers to {\color{blue}\bf Q1} and {\color{blue}\bf Q2}. In particular, we derive high-probability results for \algname{clipped-SEG} and \algname{clipped-SDGA} for monotone and structured non-monotone \ref{eq:main_problem}s with non-sub-Gaussian noise, validate that the gradient noise in several GANs formulations is indeed heavy-tailed, and show that gradient clipping does significantly improve the results of \algname{SEG}/\algname{SGDA} in these tasks. That is, our work closes a noticeable gap in theory of stochastic methods for solving \ref{eq:main_problem}s.

\subsection{Technical Preliminaries}\label{sec:techicalities}

Before we summarize our main contributions, we introduce some notations and technical assumptions.

\looseness=-1
\textbf{Notation.}
Throughout the text $\langle x, y \rangle$ is the standard inner-product, $\|x\| = \sqrt{\langle x, x \rangle}$ denotes $\ell_2$-norm, $B_{r}(x) = \{u\in \R^d\mid \|u - x\| \leq r\}$. $\EE[X]$ and $\EE_{\xi}[X]$ are full and conditional expectations w.r.t.\ the randomness coming from $\xi$ of random variable $X$. $\PP\{E\}$ denotes the probability of event $E$. $R_0 = \|x^0 - x^*\|$ denotes the distance between the starting point $x^0$ of a method and the solution $x^*$ of \ref{eq:main_problem}.\footnote{If not specified, we assume that $x^*$ is the projection of $x^0$ to the solution set of \ref{eq:main_problem}.}

\textbf{High-probability convergence for \ref{eq:main_problem}.} For deterministic \ref{eq:main_problem}s there exist several convergence metrics $\cP(x)$. These metrics include restricted gap-function \citep{nesterov2007dual} $\gap_R(x) := \max_{y\in B_{R}(x^*)}\langle F(y), x - y \rangle$, (averaged) squared norm of the operator $\|F(x)\|^2$, and squared distance to the solution $\|x - x^*\|^2$. Depending on the assumptions on the problem, one or another criterion is preferable. For example, for monotone and strongly monotone problems $\gap_R(x)$ and $\|x-x^*\|^2$ are valid metrics of convergence, while in the non-monotone case, one typically has to use $\|F(x)\|^2$.

In the stochastic case, one needs either to upper bound $\EE[\cP(x)]$ or to derive a bound on $\cP(x)$ that holds with some probability. In this work, we focus on the second type of bounds. That is, for a given confidence level $\beta \in (0,1]$, we aim at deriving bounds on $\cP(x)$ that hold with probability at least $1 - \beta$, where $x$ is produced by \algname{clipped-SEG}/\algname{clipped-SGDA}. However, to achieve this goal one has to introduce some assumptions on the stochastic noise such as \emph{bounded variance} assumption, which is standard in the stochastic optimization literature \citep{ghadimi2012optimal, ghadimi2013stochastic, juditsky2011first, nemirovski2009robust}. We rely on a weaker version of this assumption.
\begin{assumption}[Bounded Variance]\label{as:UBV}
    There exists a bounded set $Q \subseteq \R^d$ and $\sigma \geq 0$ such that
    \begin{equation}
        \EE\left[\|F_\xi(x) - F(x)\|^2\right] \leq \sigma^2,\quad \forall\; x \in Q. \label{eq:UBV}
    \end{equation}
\end{assumption}
In contrast to most of the existing works assuming \eqref{eq:UBV} on the whole space/domain, we need \eqref{eq:UBV} to hold only on some \emph{bounded} set $Q$. More precisely, in the analysis, we rely on \eqref{eq:UBV} only on some ball $B_r(x^*)$ around the solution $x^*$ and radius $r \sim R_0$. Although we consider an \emph{unconstrained} \ref{eq:main_problem}, we manage to show that the iterates of \algname{clipped-SEG}/\algname{clipped-SGDA} stay inside this ball with high probability. Therefore, for achieving our purposes, it is sufficient to make all the assumptions on the problem only in some ball around the solution.

We also notice that the majority of existing works providing high-probability analysis with logarithmic dependence\footnote{Using Markov inequality, one can easily derive high-probability bounds with non-desirable polynomial dependence on $\nicefrac{1}{\beta}$, e.g., see the discussion in \citep{davis2021low, gorbunov2021near}.} on $\nicefrac{1}{\beta}$ rely on the so-called \emph{light tails} assumption: $\EE[\exp(\nicefrac{\|F_\xi(x) - F(x)\|^2}{\sigma^2})] \leq \exp(1)$ meaning that the noise has a sub-Gaussian distribution. This assumption always implies \eqref{eq:UBV} but not vice versa. However, even for minimization problems, there are only few works that do not rely on the light tails assumption \citep{nazin2019algorithms, davis2021low, gorbunov2020stochastic, gorbunov2021near, cutkosky2021high}. In the context of solving \ref{eq:main_problem}s, the existing high-probability guarantees in \citep{juditsky2011solving} rely on the light tails assumption.

\textbf{Optimization properties.} We also need to introduce several assumptions about the operator $F$. We start with the standard Lipschitzness. As we write above, all assumptions are introduced only on some bounded set $Q$, which will be specified later.

\begin{assumption}[Lipschitzness]\label{as:lipschitzness}
    Operator $F:\R^d \to \R^d$ is $L$-Lipschitz on $Q\subseteq \R^d$, i.e.,
    \begin{gather}
        \|F(x) - F(y)\| \leq L \|x - y\|, \quad \forall\; x,y\in Q. \label{eq:Lipschitzness} \tag{Lip}
    \end{gather}
\end{assumption}

Next, we need to introduce some assumptions\footnote{Each complexity result, which we derive, relies only on one or two of these assumptions simultaneously.} on the monotonicity of $F$, since approximating local first-order optimal solutions is intractable in the general non-monotone case \citep{daskalakis2021complexity, diakonikolas2021efficient}. We start with the standard monotonicity assumption and its relaxations.

\begin{assumption}[Monotonicity]\label{as:monotonicity}
    Operator $F:\R^d \to \R^d$ is monotone on $Q\subseteq \R^d$, i.e.,
    \begin{gather}
        \langle F(x) - F(y), x - y \rangle \geq 0, \quad \forall\; x,y\in Q. \label{eq:monotonicity} \tag{Mon}
    \vspace{-2mm}
    \end{gather}
\end{assumption}
\begin{assumption}[Star-Negative Comonotonicity]\label{as:negative_comon}
    Operator $F:\R^d \to \R^d$ is $\rho$-star-negatively comonotone on $Q\subseteq \R^d$ for some $\rho \in [0, +\infty)$, i.e., for any $x^*$ we have
    \begin{gather}
        \langle F(x), x - x^* \rangle \geq -\rho\|F(x)\|^2, \quad \forall\; x\in Q. \label{eq:negative_comon} \tag{SNC}
    \end{gather}
    \ggi{When $\rho = 0$, the operator $F$ is called star-monotone (\emph{SM}) on $Q$.}
\end{assumption}
\eqref{eq:negative_comon} is also known as weak Minty condition \citep{diakonikolas2021efficient}, a relaxation of negative comonotonicity \citep{bauschke2021generalized}. Another name of star-monotonicity (SM) is variational stability condition \citep{hsieh2020explore}.
The following assumption is a relaxation of strong monotonicity.
\begin{assumption}[Quasi-Strong Monotonicity]\label{as:str_monotonicity}
    Operator $F:\R^d \to \R^d$ is $\mu$-quasi strongly monotone on $Q\subseteq \R^d$ for some $\mu \geq 0$, i.e.,
    \begin{gather}
        \langle F(x), x - x^* \rangle \geq \mu\|x - x^*\|^2, \quad \forall\, x\in Q, \label{eq:str_monotonicity} 
        \quad \text{where $x^*$ is the unique solution of \eqref{eq:main_problem}.}
        \tag{QSM}
    \end{gather}
\end{assumption}
Under this name, the above assumption is introduced by \citep{loizou2021stochastic}, but it is also known as strong coherent \cite{song2020optimistic} and strong stability \citep{mertikopoulos2019learning} conditions. Moreover, unlike strong monotonicity that always implies monotonicity, \eqref{eq:str_monotonicity} can hold for non-monotone operators \citep[Appendix A.6]{loizou2021stochastic}. However, in contrast to \eqref{eq:monotonicity}, \eqref{eq:str_monotonicity} allows to achieve linear convergence\footnote{Linear convergence can be also achieved under different relatively weak assumptions such as sufficient bilinearity \citep{abernethy2019last,loizou2020stochastic} or error bound condition \citep{hsieh2020explore}.} of deterministic methods for solving \ref{eq:main_problem}.

Finally, we consider a relaxation of standard cocoercivity: $\ell\langle F(x) - F(y), x - y \rangle \geq \|F(x) - F(y)\|^2$.
\begin{assumption}[Star-Cocoercivity]\label{as:star_cocoercivity}
    Operator $F:\R^d \to \R^d$ is $\ell$-star-cocoercive on $Q\subseteq \R^d$ for some $\ell > 0$, i.e.,
    \begin{gather}
        \|F(x)\|^2 \leq \ell\langle F(x), x - x^* \rangle, \quad \forall\, x \in Q, \label{eq:star_cocoercivity} 
        \quad \text{where $x^*$ is the projection of $x$ on $\mathcal{X}^*$.}
        \tag{SC}
    \end{gather}
\end{assumption}

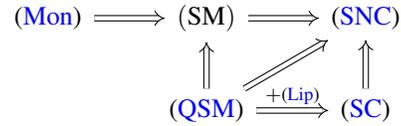
\begin{wrapfigure}{o}{0.4\textwidth}
\vspace{-6mm}
    \centering
    \begin{tikzcd} 
        \eqref{eq:monotonicity} \arrow[Rightarrow, r]& (\text{SM}) \arrow[Rightarrow, r] & \eqref{eq:negative_comon}\\
        & \eqref{eq:str_monotonicity} \arrow[Rightarrow, r, "+\eqref{eq:Lipschitzness}"] \arrow[Rightarrow, u] \arrow[Rightarrow, ru] & \eqref{eq:star_cocoercivity} \arrow[Rightarrow, u] 
    \end{tikzcd}
    \caption{\small
    \ggi{Relation between the assumptions on the structured non-monotonicity of the problem.}}
    \label{fig:assumption_relation}
    \vspace{-1cm}
\end{wrapfigure}
This assumption is introduced by \citep{loizou2021stochastic}. One can construct an operator $F$ being star-cocoercive, but not cocoercive \citep{gorbunov2022extragradient}. Moreover, although cocoercivity implies monotonicity and Lipschitzness, there exist operators satisfying \eqref{eq:star_cocoercivity}, but neither \eqref{eq:monotonicity} nor \eqref{eq:Lipschitzness} \citep[\S A.6]{loizou2021stochastic}. We summarize the relations between the assumptions in Fig.~\ref{fig:assumption_relation}.

\subsection{Our Contributions}\label{sec:contributions}
\begin{table*}[t]
    \centering
    \scriptsize
    \vspace{-1cm}
    \caption{\scriptsize Summary of known and new high-probability complexity results for solving variational inequalities. Column ``Setup'' indicates the additional assumptions in addition to Assumption~\ref{as:UBV}. All assumptions are made only on some ball around the solution with radius $\sim R_0$ (unless the opposite is indicated). By the complexity we mean the number of \pd{stochastic} oracle calls needed for a method to guarantee that $\PP\{\text{Metric} \leq \varepsilon\} \geq 1 - \beta$ for some $\varepsilon> 0$, $\beta \in (0,1]$ and ``Metric'' is taken from the corresponding column. For simplicity, we omit numerical and logarithmic factors in the complexity bounds. Column ``HT?'' indicates whether the result is derived in the heavy-tailed case (without assuming that the noise is sub-Gaussian) and column ``UD?'' shows whether the analysis works on unbounded domains. Notation: $\tx^K_{\text{avg}} = \frac{1}{K+1}\sum_{k=0}^K \tx^k$ (for \ref{eq:clipped_SEG}), $x^K_{\text{avg}} = \frac{1}{K+1}\sum_{k=0}^K x^k$ (for \ref{eq:clipped_SGDA}), $L$ = Lipschitz constant; $D$ = diameter of the domain (used in \citep{juditsky2011solving}); $\gap_{D}(x) = \max_{y \in \cX}\langle F(y), x - y\rangle$, where $\cX$ is a bounded domain with diameter $D$ where the problem is defined (used in \citep{juditsky2011solving}); $\sigma^2$ = bound on the variance (in the results from \citep{juditsky2011solving} $\sigma^2$ is a sub-Gaussian variance); $R$ = any upper bound on $\|x^0 - x^*\|$; $\mu$ = quasi-strong monotonicity parameter; $\ell$ = star-cocoercivity parameter.}
    \label{tab:comparison_of_rates}
    \begin{threeparttable}
        \begin{tabular}{|c|c c c c c c|}
        \hline
        Setup & Method & Citation & Metric & Complexity & HT? & UD?\\
        \hline\hline
        \multirow{2.5}{*}{\eqref{eq:monotonicity}+\eqref{eq:Lipschitzness}} & \algname{Mirror-Prox} & \citep{juditsky2011solving}\tnote{\color{blue}(1)} & $\gap_D(\tx^K_{\text{avg}})$ & $\max\left\{\frac{LD^2}{\varepsilon}, \frac{\sigma^2 D^2}{\varepsilon^2}\right\}$ & \xmark & \xmark \\
        & \ref{eq:clipped_SEG} & Thm.~\ref{thm:main_result_gap_SEG} \& Cor.~\ref{cor:main_result_gap_SEG} & $\gap_R(\tx^K_{\text{avg}})$ & $\max\left\{\frac{LR^2}{\varepsilon}, \frac{\sigma^2 R^2}{\varepsilon^2}\right\}$  & \cmark & \cmark\\
        \hline
        \eqref{eq:negative_comon}+\eqref{eq:Lipschitzness} & \ref{eq:clipped_SEG} & Thm.~\ref{thm:main_result_avg_sq_norm_SEG} \& Cor.~\ref{cor:main_result_avg_sq_norm_SEG} \tnote{{\color{blue} (2)}} & $\frac{1}{K+1}\sum\limits_{k=0}^K \|F(x^k)\|^2$ & $L^2\max\left\{\frac{R^2}{\varepsilon}, \frac{\sigma^2 R^2}{\varepsilon^2}\right\}$ & \cmark & \cmark\\
        \hline
        \eqref{eq:str_monotonicity}+\eqref{eq:Lipschitzness} & \ref{eq:clipped_SEG} & Thm.~\ref{thm:main_result_str_mon_SEG} \& Cor.~\ref{cor:main_result_SEG_str_mon} & $\|x^K - x^*\|^2$ & $\max\left\{\frac{L}{\mu}, \frac{\sigma^2}{\mu\varepsilon}\right\}$ & \cmark & \cmark\\
        \hline\hline
        \eqref{eq:monotonicity}+\eqref{eq:star_cocoercivity} & \ref{eq:clipped_SGDA} & Thm.~\ref{thm:main_result_SGDA} \& Cor.~\ref{cor:main_result_SGDA} & $\gap_R(x^K_{\text{avg}})$ & $\max\left\{\frac{\ell R^2}{\varepsilon}, \frac{\sigma^2 R^2}{\varepsilon^2}\right\}$ & \cmark & \cmark\\
        \hline
        \eqref{eq:star_cocoercivity} & \ref{eq:clipped_SGDA} & Thm.~\ref{thm:main_result_SGDA_sq_norm} \& Cor.~\ref{cor:main_result_SGDA_sq_norm} & $\frac{1}{K+1}\sum\limits_{k=0}^K \|F(x^k)\|^2$ & $\ell^2\max\left\{\frac{R^2}{\varepsilon}, \frac{\sigma^2 R^2}{\varepsilon^2}\right\}$ & \cmark & \cmark\\
        \hline
        \eqref{eq:str_monotonicity}+\eqref{eq:star_cocoercivity} & \ref{eq:clipped_SGDA} & Thm.~\ref{thm:main_result_SGDA_str_mon} \& Cor.~\ref{cor:main_result_SGDA_str_mon} & $\|x^K - x^*\|^2$ & $\max\left\{\frac{\ell}{\mu}, \frac{\sigma^2}{\mu\varepsilon}\right\}$ & \cmark & \cmark\\
        \hline
    \end{tabular}
    \begin{tablenotes}
        {\scriptsize \item [{\color{blue}(1)}] Monotonicity and Lipschitzness of $F$ are assumed on the whole domain.
        
        \item [{\color{blue}(2)}] The results holds for any $0 \leq \rho \leq \nicefrac{1}{(640 L A)}$, where $A = \ln \frac{8(K+1)}{\beta}$, if parameters of the method are set properly. Moreover, batchsizes should be large enough (see Thm.~\ref{thm:SEG_meta_theorem} and \ref{thm:main_result_avg_sq_norm_SEG} for the details).
        }
    \end{tablenotes}
    \end{threeparttable}
    \vspace{-3mm}
\end{table*}

\textbf{$\diamond$ New high-probability results for \ref{eq:main_problem}s with heavy-tailed noise.} In our work, we circumvent the limitations of the existing high-probability analysis of stochastic methods for solving \ref{eq:main_problem}s \citep{juditsky2011solving} and derive the first high-probability results for the methods that solve monotone \ref{eq:main_problem}s with heavy-tailed noise in the unconstrained case. The key algorithmic ingredient helping us to achieve these results is a proper modification of \algname{SEG} and \algname{SGDA} based on the \emph{gradient clipping} \pd{and leading to our } \ref{eq:clipped_SEG} and \ref{eq:clipped_SGDA}. Moreover, we derive several high-probability results for \ref{eq:clipped_SEG}/\ref{eq:clipped_SGDA} applied to solve structured non-monotone \ref{eq:main_problem}s. To the best of our knowledge, these results do not have analogs even under the light tails assumption. We summarize the derived complexity results in Tbl.~\ref{tab:comparison_of_rates}.

\textbf{$\diamond$ Tight analysis.} The derived complexities satisfy a desirable for high-probability results property: they have logarithmic dependence on $\nicefrac{1}{\beta}$, where $\beta$ is a confidence level. Next, up to the logarithmic factors, our results match known lower bounds\footnote{These lower bounds are derived for the convergence in expectation. Deriving tight lower bounds for the convergence with high-probability for solving stochastic \ref{eq:main_problem}s is an open question.} in monotone and strongly monotone cases \citep{beznosikov2020distributed}. Moreover, we recover and even improve the results from \citep{juditsky2011solving}, which are obtained in the light-tailed case, since our bounds do not depend on the diameter of the domain, see Tbl.~\ref{tab:comparison_of_rates} for the details.

\looseness=-1
\textbf{$\diamond$ Weak assumptions.} One of the key features of our theoretical results is that it relies on assumptions restricted \emph{to a ball around the solution}. We achieve this via showing that, with high probability, \algname{clipped-SEG}/\algname{clipped-SGDA} do not leave a ball (with a radius proportional to $R_0$) around the solution. In contrast, the existing works on stochastic methods for solving \ref{eq:main_problem}s usually make assumptions such as boundedness of the variance and Lipschitzness on the whole domain of the considered problem. Since many practical tasks are naturally unconstrained, such assumptions become too unrealistic since, e.g., stochastic gradients and their variance in the training of neural networks with more than $2$ layers grow polynomially fast when $x$ goes to infinity. However, for a large class of problems including the ones with polynomially growing operators, boundedness of the variance and Lipschitzness hold on \emph{any compact set}. That is, our analysis covers a broad class of problems.

\textbf{$\diamond$ Numerical experiments.} We empirically observe that heavy-tailed gradient noise arises in the training of several practical GANs formulations including StyleGAN2 and WGAN-GP. Moreover, our experiments show that gradient clipping significantly improves the convergence of \algname{SEG}/\algname{SGDA} on such tasks. These results shed a light on why \algname{Adam}-based methods are good at training GANs. Our codes are publicly available: \url{https://github.com/busycalibrating/clipped-stochastic-methods}.

\subsection{Closely Related Work}\label{sec:related_work}

In this section, we discuss the most closely related works. Further discussion is deferred to \S~\ref{app:extra_related_work}.

\textbf{High-probability convergence.} To the best of our knowledge, the only work deriving high-probability convergence in the context of solving \ref{eq:main_problem}s is \citep{juditsky2011solving}. In particular, \citet{juditsky2011solving} consider monotone and Lipschitz VIP defined\footnote{In this case, the goal is to find $x^* \in \cX$ such that inequality $\langle F(x^*), x - x^* \rangle \geq 0$ holds for all $x \in \cX$.} on a convex compact set $\cX$ with the diameter $D := \max_{x,y\in\cX}\|x-y\|$, and assume that the noise in $F_{\xi}(x)$ is light-tailed. In this setting, \citet{juditsky2011solving} propose a stochastic version of the celebrated Extragradient method (\algname{EG}) \citep{korpelevich1976extragradient} with non-Euclidean proximal setup -- \algname{Mirror-Prox}. Moreover, \citet{juditsky2011solving} derive that after $K = \cO(\max\{\frac{LD^2}{\varepsilon}, \frac{\sigma^2 D^2}{\varepsilon^2} \ln^2\frac{1}{\beta}\})$ \pd{stochastic oracle calls}  for some $\varepsilon > 0$, $\beta \in (0,1]$, the averaged extrapolated iterate $\tx_{\text{avg}}^K = \frac{1}{K+1}\sum_{k=0}^K \tx^k$ (see also \ref{eq:clipped_SEG}) satisfies $\gap_D(\tx_{\text{avg}}^K) \leq \varepsilon$ with probability at least $1 - \beta$. Although this result has a desirable logarithmic dependence on $\nicefrac{1}{\beta}$ and optimal dependence on $\varepsilon$ \citep{beznosikov2020distributed}, it is derived only for (i) light-tailed case and (ii) bounded domains. Our results do not have such limitations.

\section{Clipped Stochastic Extragradient}
In this section, we consider a version of \algname{SEG} with gradient clipping. That is, we apply clipping operator $\clip(y, \lambda) := \min\{1, \nicefrac{\lambda}{\|y\|}\} y$, which is defined for any $y \in \R^d$ (when $y = 0$ we set $\clip(0, \lambda) := 0$) and any clipping level $\lambda > 0$, to the mini-batched stochastic estimators used both at the extrapolation and update steps of \algname{SEG}. This results in the following iterative algorithm:
\begin{gather}
    x^{k+1} = x^k - \gamma_2 \tF_{\Bxi_2^k}(\tx^k),\quad \text{where}\quad \tx^k = x^k - \gamma_1 \tF_{\Bxi_1^k}(x^k), \tag{\algname{clipped-SEG}}\label{eq:clipped_SEG}\\
    \tF_{\Bxi_1^k}(x^k) = \ggi{\clip\left(\frac{1}{m_{1,k}}\sum\limits_{i=1}^{m_{1,k}} F_{\xi_{1}^{i,k}}(x^k), \lambda_{1,k}\right),\quad} \tF_{\Bxi_2^k}(\tx^k) = \clip\left(\frac{1}{m_{2,k}}\sum\limits_{i=1}^{m_{2,k}} F_{\xi_{2}^{i,k}}(\tx^k), \lambda_{2,k}\right), \notag
    % \\F_{\Bxi_1^k}(x^k) = \frac{1}{m_{1,k}}\sum\limits_{i=1}^{m_{1,k}} F_{\xi_{1}^{i,k}}(x^k),\quad F_{\Bxi_2^k}(\tx^k) = \frac{1}{m_{2,k}}\sum\limits_{i=1}^{m_{2,k}} F_{\xi_{2}^{i,k}}(\tx^k), \label{eq:batched_estimators_SEG}
\end{gather}
where $\{\xi_1^{i,k}\}_{i=1}^{m_{1,k}}, \{\xi_2^{i,k}\}_{i=1}^{m_{2,k}}$ are independent samples from the distribution $\cD$. In the considered algorithm, clipping bounds the effect of heavy-tailedness of the gradient noise, but also creates a bias that one has to properly control in the analysis. Moreover, we allow different stepsizes $\gamma_1, \gamma_2$ \citep{hsieh2020explore, diakonikolas2021efficient, gorbunov2022stochastic}, different batchsizes $m_{1,k}, m_{2,k}$, and different clipping levels $\lambda_{1,k}, \lambda_{2,k}$. In particular, taking $\gamma_2 < \gamma_1$ is crucial for our analysis to handle \ref{eq:main_problem} satisfying star-negative comonotonicity \eqref{eq:negative_comon} with $\rho > 0$.
Our convergence results for \ref{eq:clipped_SEG} are summarized in the following theorem. \pd{For simplicity, we omit here the numerical constants, which are explicitly given in \S~\ref{app:clipped_SEG_proofs}.}
\begin{theorem}\label{thm:SEG_meta_theorem}
    Consider \ref{eq:clipped_SEG} run for $K \geq 0$ iterations. Let $R \geq R_0 = \|x^0 - x^*\|$.\newline
    \textbf{Case 1.}  Let Assump.~\ref{as:UBV}, \ref{as:lipschitzness}, \ref{as:monotonicity} hold for $Q = B_{4R}(x^*)$, where $\gamma_1 = \gamma_2 = \gamma$ with $0 < \gamma = \cO(\nicefrac{1}{(L A)})$, $\lambda_{1,k} = \lambda_{2,k} \equiv \lambda = \Theta(\nicefrac{R}{(\gamma A)})$, $m_{1,k} = m_{2,k} \equiv m = \Omega(\max\{1, \nicefrac{(K+1) \gamma^2\sigma^2 A}{R^2}\})$, where $A = \ln\frac{6(K+1)}{\beta}$, $\beta \in (0,1]$ are such that $A \geq 1$.\newline
    \textbf{Case 2.}  Let Assump.~\ref{as:UBV}, \ref{as:lipschitzness}, \ref{as:negative_comon} hold for $Q = B_{3R}(x^*)$, where $\gamma_2 + 2\rho \leq \gamma_1 \leq \cO(\nicefrac{1}{(L A)})$, $\lambda_{1,k} \equiv \lambda_1 = \Theta(\nicefrac{R}{(\gamma_1 A)})$, $\lambda_{2,k} \equiv \lambda_2 = \Theta(\nicefrac{R}{(\gamma_2 A)})$, $m_{1,k} \equiv m_1 = \Omega(\max\{1, \nicefrac{\max\{\gamma_1\gamma_2(K+1), \sqrt{\gamma_1^3\gamma_2(K+1)}\}\sigma^2 A}{R^2}\})$, $m_{2,k} \equiv m_2 = \Omega(\max\{1, \nicefrac{(K+1) \gamma_2^2\sigma^2 A}{R^2}\})$, where $A = \ln\frac{8(K+1)}{\beta}$, $\beta \in (0,1]$ are such that $A \geq 1$.\newline
    \textbf{Case 3.}  Let Assump.~\ref{as:UBV}, \ref{as:lipschitzness}, \ref{as:str_monotonicity} hold for $Q = B_{3R}(x^*)$, where $\gamma_1 = \gamma_2 = \gamma$ with $0 < \gamma \leq \cO(\nicefrac{1}{(L A)})$, $\lambda_{1,k} = \lambda_{2,k} = \lambda_k = \Theta(\nicefrac{\exp(-\gamma\mu(1+\nicefrac{k}{2}))R}{(\gamma A)})$, $m_{1,k} = m_{2,k} = m_k = \Omega(\max\{1, \nicefrac{(K+1) \gamma^2\sigma^2 A}{(\exp(-\gamma\mu k)R^2)}\})$, where $A = \ln\frac{6(K+1)}{\beta}$ , $\beta \in (0,1]$ are such that $A \geq 1$.
    
    Then, to guarantee $\gap_{R}(\tx_{\text{avg}}^K) \leq \varepsilon$ in \textbf{Case 1} with $\tx_{\text{avg}}^K = \frac{1}{K+1}\sum_{k=0}^K \tx^k$, $\frac{1}{K+1}\sum_{k=0}^{K} \|F(x^k)\|^2 \leq L\varepsilon$ in \textbf{Case 2}, $\|x^K - x^*\|^2 \leq \varepsilon$ in \textbf{Case 3}, with probability $\geq 1-\beta$ \ref{eq:clipped_SEG} requires %\ggi{the oracle complexity of  \ref{eq:clipped_SEG} is,}
    \begin{eqnarray}
        \text{\textbf{Case 1} and \textbf{2}}: &  \widetilde{\cO}\left(\max\left\{\frac{LR^2}{\varepsilon}, \frac{\sigma^2R^2}{\varepsilon^2}\right\}\right) 
        \qquad  \text{and} \qquad \text{\textbf{Case 3}:}& \widetilde{\cO}\left(\max\left\{\tfrac{L}{\mu}, \tfrac{\sigma^2}{\mu\varepsilon}\right\}\right)\,
        \label{eq:SEG_complexity_Gap}
    \end{eqnarray}
    oracle calls. The above guarantees hold in  two different regimes: large step-sizes $\gamma \approx \nicefrac{1}{LA}$ (requiring large batch-sizes), and small step-sizes, allowing small batch-sizes $m=\cO(1)$.
\end{theorem}
\begin{proof}[Proof sketch in \textbf{Case 1}]
    Modifying the analysis of \algname{EG}, we first derive that for all $t \geq 0$ we have $\gap_R(\tx_{\text{avg}}^t)$ and $\|x^t - x^*\|^2$ are not greater than $\max_{u\in B_R(x^*)}\{\|x^0 - u\|^2 + 2\gamma \sum_{l = 0}^{t-1} \langle x^l - u - \gamma F(\tx^l), \theta_l \rangle + \gamma^2 \sum_{l=0}^{t-1}(\|\theta_l\|^2 + 2\|\omega_l\|^2)\}$ if $x^l, \tx^l$ lie in $B_{4R}(x^*)$ for all $l = 0,1,\ldots,t-1$, where $\theta_l =  F(\tx^l) - \tF_{\Bxi_2^l}(\tx^l)$ and $\omega_l = F(x^l) - \tF_{\Bxi_1^l}(x^l)$. Next, using this recursion and the induction argument, we show that with high probability $x^t, \tx^t \in B_{4R}(x^*)$ for all $t = 0,1, \ldots, K+1$. This gives us an upper bound for $\gap_R(\tx_{\text{avg}}^t)$. After that, we upper bound $\max_{u\in B_R(x^*)}\{2\gamma \sum_{l = 0}^{t-1} \langle x^l - u - \gamma F(\tx^l), \theta_l \rangle + \gamma^2 \sum_{l=0}^{t-1}(\|\theta_l\|^2 + 2\|\omega_l\|^2)\}$ by $5R^2$ with high-probability. We achieve this via the proper choice of the clipping level $\lambda = \Theta(\nicefrac{R}{(\gamma A)})$ implying that $\|F(\tx^l)\| \leq \nicefrac{\lambda}{2}$ and  $\|F(x^l)\| \leq \nicefrac{\lambda}{2}$ with high probability for all $l = 0,1,\ldots, t-1$. This clipping politics helps to properly bound the bias and the variance of the clipped estimators using Lem.~\ref{lem:bias_variance}. After that, it remains to apply the Bernstein inequality for the martingale differences (Lem.~\ref{lem:Bernstein_ineq}). See the detailed proof in \S~\ref{app:clipped_SEG_proofs}.
\end{proof}

In addition to the discussion given in the introduction (see \S~\ref{sec:contributions}, \ref{sec:related_work} and Tbl.~\ref{tab:comparison_of_rates}), we provide here several important details about the derived results. First of all, we notice that up to the logarithmic factors depending on $\beta$ our high-probability convergence results recover the state-of-the-art in-expectation ones for \algname{SEG} in the monotone \citep{beznosikov2020distributed}, star-negative comonotone \citep{diakonikolas2021efficient}, and quasi-strongly monotone \citep{gorbunov2022stochastic} cases. Moreover, as we show in Corollaries~\ref{cor:main_result_gap_SEG} and \ref{cor:main_result_SEG_str_mon}, to achieve these results in monotone and quasi-strongly monotone cases, it is sufficient to choose constant batchsize $m = \cO(1)$ and small enough stepsize $\gamma$. 
\pd{In contrast, when the operator is star-negatively comonotone, we do rely on the usage of large stepsize $\gamma_1$ and large batchsize $m_1 = \cO(K)$ for the extrapolation step to obtain \eqref{eq:SEG_complexity_Gap}.}
However, known in-expectation results from \citep{diakonikolas2021efficient, lee2021fast} also rely on large $\cO(K)$ batchsizes in this case. We leave the investigation of this limitation to the future work.

\section{Clipped Stochastic Gradient Descent-Ascent}
In this section, we extend the approach described above to the analysis of \algname{SGDA} with clipping. That is, we consider the following algorithm:
\begin{gather}
    x^{k+1} = x^k - \gamma \tF_{\Bxi^k}(x^k),\;\; \text{where}\;\; \tF_{\Bxi^k}(x^k) = \ggi{\clip\left(\frac{1}{m_k}\sum\limits_{i=1}^{m_k} F_{\xi^{i,k}}(x^k), \lambda_k\right)} \tag{\algname{clipped-SGDA}}\label{eq:clipped_SGDA}
    % \\
    % F_{\Bxi^k}(x^k) = \frac{1}{m_k}\sum\limits_{i=1}^{m_k} F_{\xi^{i,k}}(x^k), \label{eq:batched_estimator_SGDA}
\end{gather}
where $\{\xi^{i,k}\}_{i=1}^{m}$ are independent samples from the distribution $\cD$. In the above method, clipping plays a similar role as in \ref{eq:clipped_SEG}. Our convergence results for \ref{eq:clipped_SGDA} are summarized below. The general idea of the proof is similar to the one for \ref{eq:clipped_SEG}, see the details in Appendix~\ref{app:clipped_SGDA_proofs}\pd{, where we also explicitly give the constants that are omitted here for simplicity}.
\begin{theorem}\label{thm:SGDA_meta_theorem}
    Consider \ref{eq:clipped_SGDA} run for $K \geq 0$ iterations. Let $R \geq R_0 = \|x^0 - x^*\|$.\newline
    \textbf{Case 1.}  Let Assump.~\ref{as:UBV}, \ref{as:monotonicity}, \ref{as:star_cocoercivity} hold for $Q = B_{3R}(x^*)$, where $0 < \gamma = \cO(\nicefrac{1}{(\ell A)})$, $\lambda_{k} \equiv \lambda = \Theta(\nicefrac{R}{(\gamma A)})$, $m_{k} \equiv m = \Omega(\max\{1, \nicefrac{(K+1) \gamma^2\sigma^2 A}{R^2}\})$, where $\beta \in (0,1]$ and $A = \ln\frac{6(K+1)}{\beta} \geq 1$.\newline
    \textbf{Case 2.}  Let Assump.~\ref{as:UBV}, \ref{as:star_cocoercivity} hold for $Q = B_{2R}(x^*)$, where $0 < \gamma = \cO(\nicefrac{1}{(\ell A)})$, $\lambda_{k} \equiv \lambda = \Theta(\nicefrac{R}{(\gamma A)})$, $m_{k} \equiv m = \Omega(\max\{1, \nicefrac{(K+1) \gamma^2\sigma^2 A}{R^2}\})$, where $\beta \in (0,1]$ and $A = \ln\frac{4(K+1)}{\beta} \geq 1$.\newline
    \textbf{Case 3.}  Let Assump.~\ref{as:UBV}, \ref{as:str_monotonicity}, \ref{as:star_cocoercivity} hold for $Q = B_{2R}(x^*)$, where $0 < \gamma = \cO(\nicefrac{1}{(\ell A)})$, $\lambda_{k} = \Theta(\tfrac{e^{-\gamma\mu(1+\nicefrac{k}{2})}R}{(\gamma A)})$, $m_k = \Omega(\max\{1, \tfrac{(K+1) \gamma^2\sigma^2 A}{e^{\gamma\mu k}R^2}\})$, where $\beta \in (0,1]$ and $A = \ln\frac{4(K+1)}{\beta}\geq 1$.
    
    Then, to guarantee $\gap_{R}(\tx_{\text{avg}}^K) \leq \varepsilon$ in \textbf{Case 1} with $\tx_{\text{avg}}^K = \frac{1}{K+1}\sum_{k=0}^K \tx^k$, $\frac{1}{K+1}\sum_{k=0}^{K} \|F(x^k)\|^2 \leq \ell\varepsilon$ in \textbf{Case 2}, $\|x^K - x^*\|^2 \leq \varepsilon$ in \textbf{Case 3}, with probability $\geq 1-\beta$, \ref{eq:clipped_SGDA} requires %\ggi{the oracle complexity of \ref{eq:clipped_SGDA} is,}
    \begin{eqnarray}
        \text{\textbf{Case 1} and \textbf{2}:} & \widetilde{\cO}\left(\max\left\{\frac{\ell R^2}{\varepsilon}, \frac{\sigma^2R^2}{\varepsilon^2}\right\}\right)
        \qquad \text{and} \qquad 
        \text{\textbf{Case 3}:} & \widetilde{\cO}\left(\max\left\{\tfrac{\ell}{\mu}, \tfrac{\sigma^2}{\mu\varepsilon}\right\}\right) \,
        \label{eq:SGDA_complexity_Gap}
    \end{eqnarray}
    oracle calls. The above guarantees hold in  two different regimes: large step-sizes $\gamma \approx \nicefrac{1}{\ell A}$ (requiring large batch-sizes), and small step-sizes, allowing small batch-sizes $m=\cO(1)$.
\end{theorem}

As for the \ref{eq:clipped_SEG}, up to the logarithmic factors depending on $\beta$, our high-probability convergence results from the theorem above recover the state-of-the-art in-expectation ones for \algname{SGDA} under monotonicity and star-cocoercivity \citep{beznosikov2022stochastic}, star-cocoercivity \citep{gorbunov2022stochastic}\footnote{Although \citet{gorbunov2022stochastic} do not consider \algname{SGDA} explicitly, it does fit their framework implying that \algname{SGDA} achieves $\frac{1}{K+1}\sum_{k=0}^K\EE[\|F(x^k)\|^2] \leq \varepsilon$ after $K = \cO(\max\{\frac{\ell R^2}{\varepsilon}, \frac{\sigma^2 R^2}{\varepsilon^2}\})$ iterations with $m = 1$.}, quasi-strong monotonicity and star-cocoercivity \citep{loizou2021stochastic} assumptions. Moreover, in \emph{all the cases}, one can achieve the results from \eqref{eq:SGDA_complexity_Gap} using constant batchsize $m = \cO(1)$ and small enough stepsize $\gamma$ (see Corollaries~\ref{cor:main_result_SGDA}, \ref{cor:main_result_SGDA_sq_norm}, \ref{cor:main_result_SGDA_str_mon} for the details). Finally, to the best of our knowledge, we derive the first high-probability convergence results for \algname{SGDA}-type methods.

\section{Experiments}\label{sec:experiments}

To validate our theoretical results, we conduct experiments on heavy-tailed min-max problems to demonstrate the importance of clipping when using non-adaptive methods such as \algname{SGDA} or \algname{SEG}.
We train a Wasserstein GAN with gradient penalty \citep{gulrajani2017improved} on CIFAR-10 \citep{krizhevsky2009cifar10} using \algname{SGDA}, \algname{clipped-SGDA}, and \algname{clipped-SEG}, and show the evolution of the gradient noise histograms during training.
We demonstrate that gradient clipping also stabilizes the training of more sophisticated GANs when using \algname{SGD} by training a StyleGAN2 model \citep{karras2020stylegan2} on FFHQ \citep{karras2019stylegan1}, downsampled to $128\times128$.
Although the generated sample quality is not competitive with a model trained with \algname{Adam},
we find that StyleGAN2 fails to generate anything meaningful when trained with regular \algname{SGDA} whereas clipped methods learn meaningful features.
We do not claim to outperform state-of-the-art adaptive methods (such as \algname{Adam}); our focus is to validate our theoretical results and to demonstrate that clipping improves \algname{SGDA} in this context.

\begin{figure}
    \vspace{-1cm}
    \hspace{-20pt}
    \begin{tabular}{cc}
        \begin{tabular}{c}
            \subfigure[\scriptsize{\algname{SGDA}} ($67.4$)]{
                \includegraphics[width=.19\textwidth]{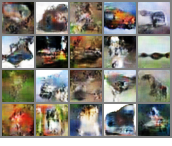}
                \label{fig:samples_sgda}
            } 
            \subfigure[\scriptsize{\algname{clipped-SGDA}} ($19.7$)]{
                \includegraphics[width=.19\textwidth]{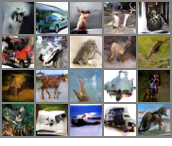}
                \label{fig:samples_clip_sgda}
            } 
            \subfigure[\scriptsize{\algname{clipped-SEG}} ($25.3$)]{
                \includegraphics[width=.19\textwidth]{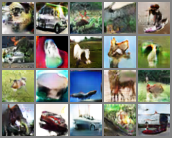}
                \label{fig:samples_clip_seg}
            } \\
            \subfigure[Generator]{\includegraphics[width=0.25\textwidth]{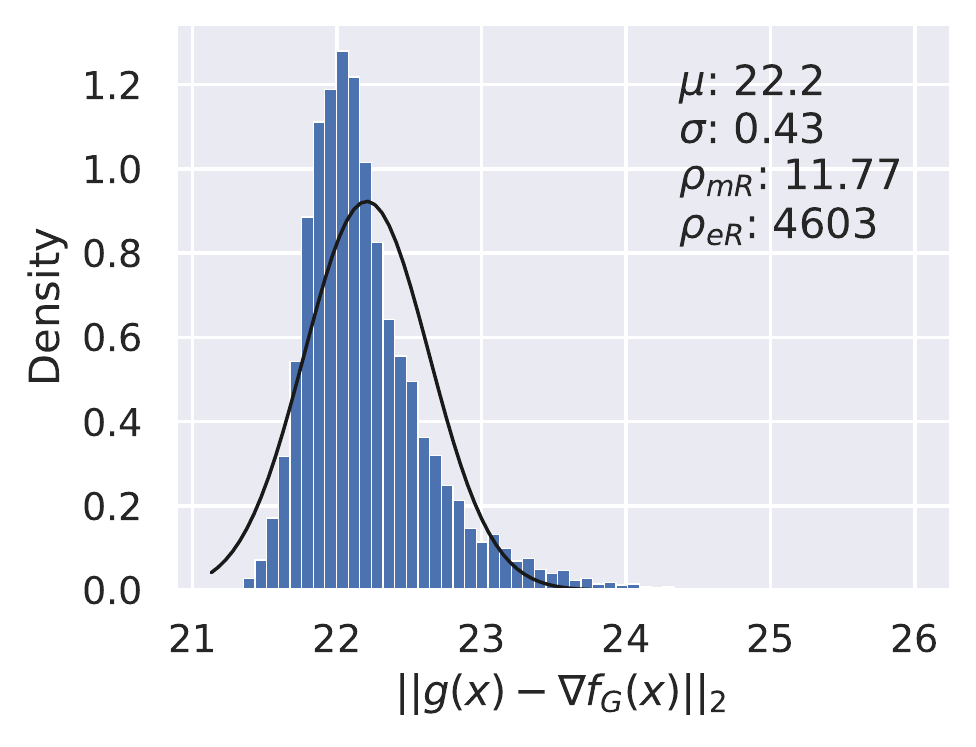}}
            \subfigure[Discriminator]{\includegraphics[width=0.25\textwidth]{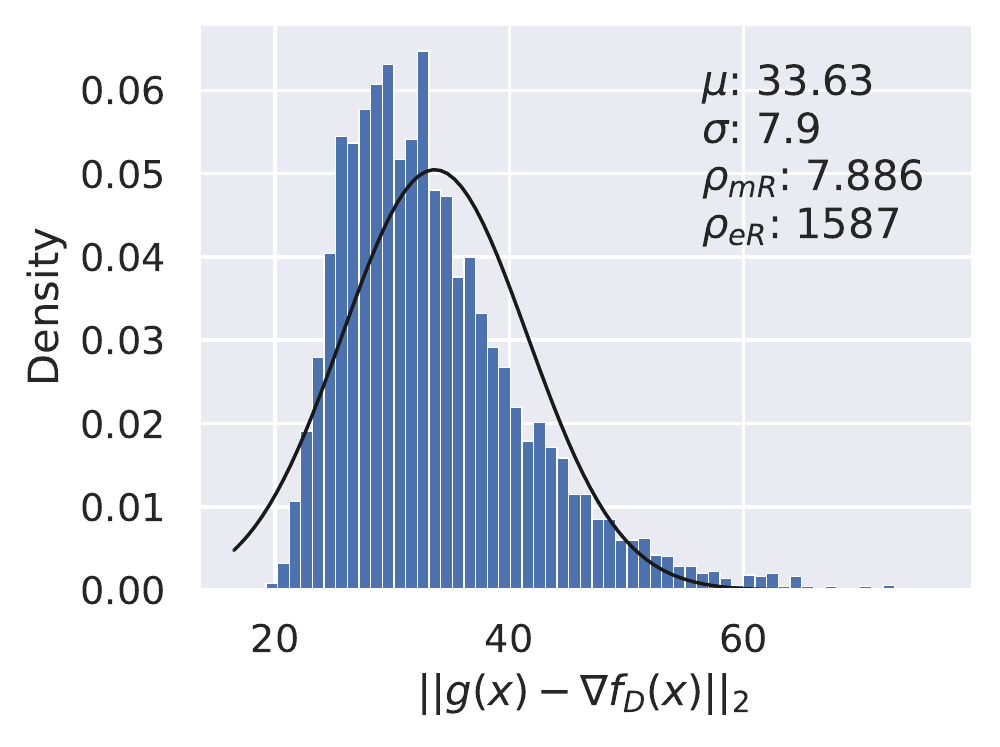}} \\
        \end{tabular}  &

        % \hspace{-17pt}
        \raisebox{-6\normalbaselineskip}[0pt][0pt]{
            \subfigure[Hyperparameters]{\includegraphics[width=0.35\textwidth]{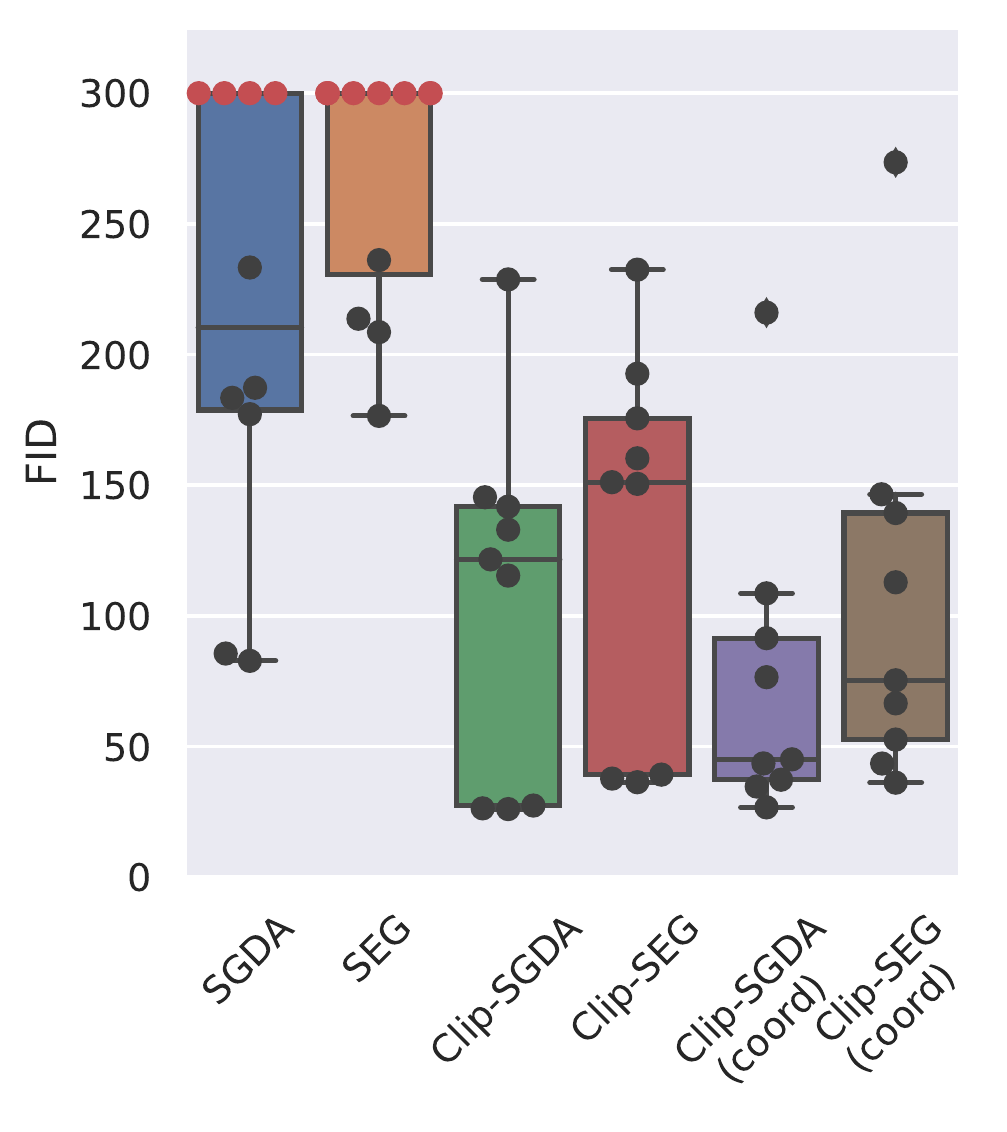}}
        }\\

    \end{tabular}
    \caption{\small
    (a, b, c) Random samples generated from the best CIFAR models trained (100k steps) by the specified optimizers, with their corresponding FIDs. 
    (d, e) Gradient noise histograms for WGAN-GP upon initialization; a Gaussian was fit using maximum likelihood estimation to showcase the tail of the distribution.
    $p_{mR}$ is a value estimating the ``heavy tailed-ness'' of the distribution.
    (f) The best (lowest) FID scores obtained within the first 35k (out of 100k) training iterations for the same WGAN-GP model trained on CIFAR-10 when sweeping over different hyperparameters.
    Red points indicate that the run \textit{diverged}, meaning that the loss becomes \texttt{NaN}.
    Note that we clip the FID for \textit{diverged} runs to 300 to not skew the boxplot.
    Boxes are the quartiles of data.}
    \label{fig:wgangp}
    \vspace{-6mm}
\end{figure}

\textbf{WGAN-GP.} In this section, we focus on the ResNet architecture proposed in \cite{gulrajani2017improved}, and we adapt our code from a publicly available WGAN-GP implementation.\footnote{\url{https://github.com/w86763777/pytorch-gan-collections}}
We first compute the gradient noise distribution and validate if it is heavy-tailed. 
Taking the fixed randomly initialized weights, we iterate through 1000 steps (without parameter updates) to compute the noise norm for each minibatch (sample with replacement as in normal GAN training).
We show the distributions for the generator and discriminator in Fig.~\ref{fig:wgangp} and also compute $p_{mR} = F_{1.5}(X)$ and $p_{eR} = F_{3}(X)$, where $F_\lambda(X) = P(Q_3 + \lambda (Q_3 - Q_1) < X)$, $X$ is the gradient noise distribution, and $Q_i$ is the $i^{th}$ quartile.
This is a metric introduced by \cite{jordanova2017measuringheavy} to measure how heavy-tailed a distribution is based on the distribution's quantiles, where $p_{mR}$ and $p_{eR}$ quantify ``mild'' and ``extreme'' (right side) heavy tails respectively.
A normal distribution should have $p_{mR\mathcal{N}} \approx 0.0035$ and $p_{eR\mathcal{N}}\approx 1.2\times10^{-6}$.
In all tests, we compute the ratios $\rho_{mR} = p_{mR} / p_{mR\mathcal{N}}$ and $\rho_{eR} = p_{eR} / p_{eR\mathcal{N}}$ and empirically find that we at least have mild heavy tails, and sometimes extremely heavy tails. 

We train the ResNet generator on CIFAR-10 with \algname{SGDA}/\algname{SEG}, and \algname{clipped-SGDA}/\algname{SEG}.
We use the default architectures and training parameters specified in \cite{gulrajani2017improved} ($\lambda_{GP}=10$, $n_{dis}=5$, learning rate decayed linearly to 0 over 100k steps), with the exception of doubling the feature map of the generator.
The clipped methods are implemented as standard gradient norm clipping, applied after computing the gradient penalty term for the critic. In addition to norm clipping, we also test coordinate-wise gradient clipping, which is more common in practice \citep{goodfellow2016deep}. 
For all methods, we tune the learning rates and clipping threshold where applicable. We do not use momentum following a standard practice for GAN training~\citep{gidel2019negative}.

We find that clipped methods outperform regular \algname{SGDA} and \algname{SEG}.
In addition to helping prevent exploding gradients, clipped methods achieve a better \textit{Fr\'echet inception distance} (FID) score \citep{heusel2017fid}; the best FID obtained for clipped methods is $19.65$ in comparison to $67.37$ for regular \algname{SGDA}, both trained for 100k steps.
A summary of FIDs obtained during hyperparameter tuning is shown Fig.~\ref{fig:wgangp}, where the best FID score obtained in the first 35k iterations is drawn for each hyperparameter configuration and optimization method.
At a high level, we do a log-space sweep over  $[2\mathrm{e-}5, 0.2]$ for the learning rates, $[10^{-1},10]$ for the norm-clip parameter, and $[10^{-3},10^{-1}]$ for the coordinate clip parameter (with some exceptions) -- please refer to \S~\ref{app:extra_exps} for further details.
We also show the evolution of the gradient noise histograms during training for \algname{SGD} and \algname{clipped-SGDA} in Fig.~\ref{fig:wgangp_evo}.
Note that for regular \algname{SGD}, the noise distribution remains heavy tailed and does not appear to change much throughout training.
In contrast, the noise histograms for \algname{clipped-SGDA} (in particular the generator) seem to develop lighter tails during training.

\textbf{StyleGAN2.} We extend our experiments to StyleGAN2, but limit our scope to \algname{clipped-SGDA} with coordinate clipping as coordinate \algname{clipped-SGDA} generally performs the best, and StyleGAN2 is expensive to train.
We train on FFHQ downsampled to $128\times128$ pixels, and use the recommended StyleGAN2 hyperparameter configuration for this resolution (batch~size~$=32$, $\gamma=0.1024$, map~depth~$=2$, channel~multiplier~$=16384$), see further experimental details in \S~\ref{app:extra_exps}. 
We obtain the gradient noise histograms at initialization and for the best trained \algname{clipped-SGDA} from the hyperparameter sweep and once again observe heavy tails (especially in the discriminator, see Fig.~\ref{fig:stylegan2}).
We find that \textit{all} regular \algname{SGDA} runs fail to learn anything meaningful, with FID scores fluctuating around $320$ and only able to generate noise.
In contrast, while there is a clear gap in quality when compared to what StyleGAN2 is capable of, a model trained with \algname{clipped-SGDA} with appropriately set hyperparameters is able to produce images that distinctly resemble faces (see Fig.~\ref{fig:stylegan2}).

\begin{figure}
    \hspace{-35pt}  
    \begin{tabular}{ccccc}
        \begin{tabular}{l}
            % Generator noise
            \raisebox{3.8\normalbaselineskip}[0pt][0pt]{
                \rotatebox[origin=c]{90}{Generator}
            }\\
            % Discriminator noise
            \raisebox{-2.4\normalbaselineskip}[0pt][0pt]{
                \rotatebox[origin=c]{90}{Discriminator}
            }\\
        \end{tabular} &
        
        \hspace{-25pt}
        \begin{tabular}{c}
            \includegraphics[width=\figscale\textwidth]{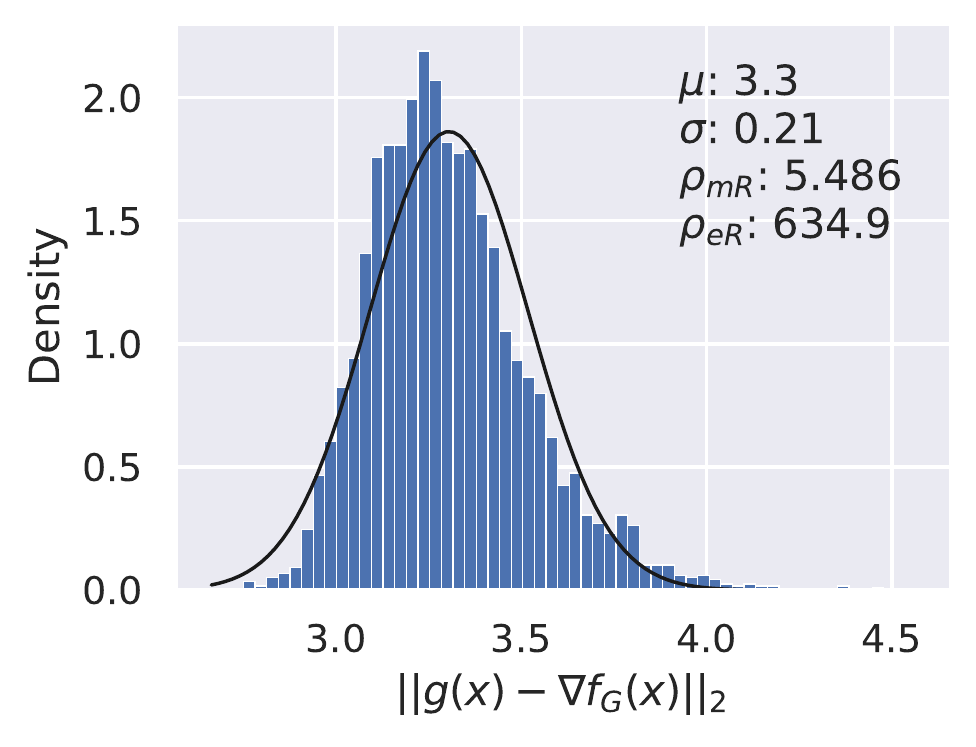} \\
            \includegraphics[width=\figscale\textwidth]{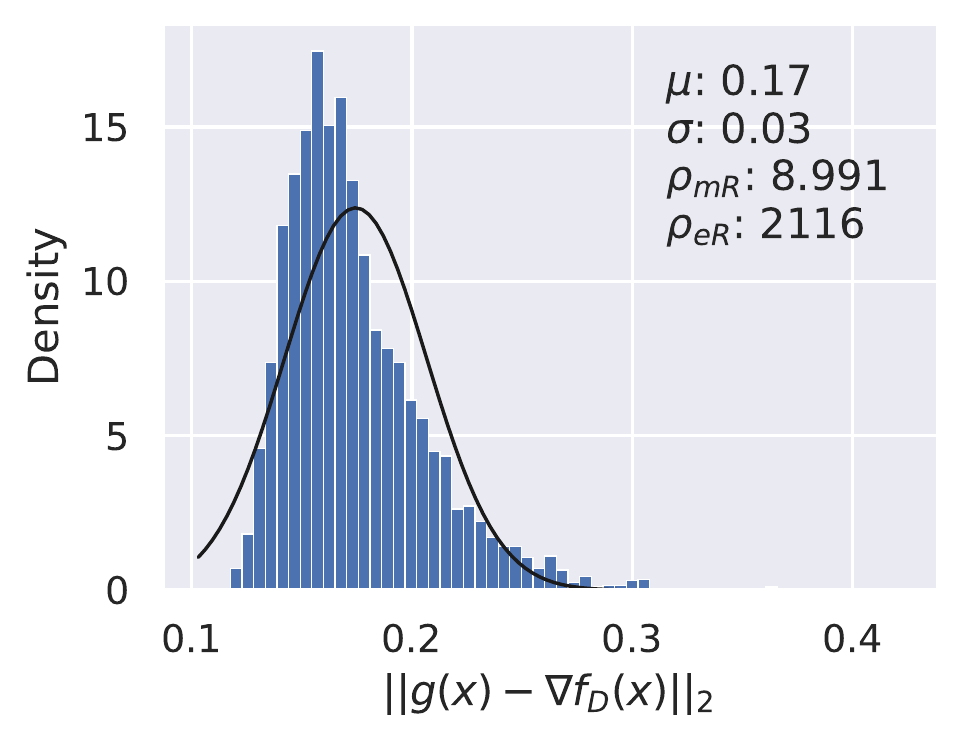} \\
        \end{tabular} &
        
        \hspace{-25pt}
        \begin{tabular}{c}
            \includegraphics[width=\figscale\textwidth]{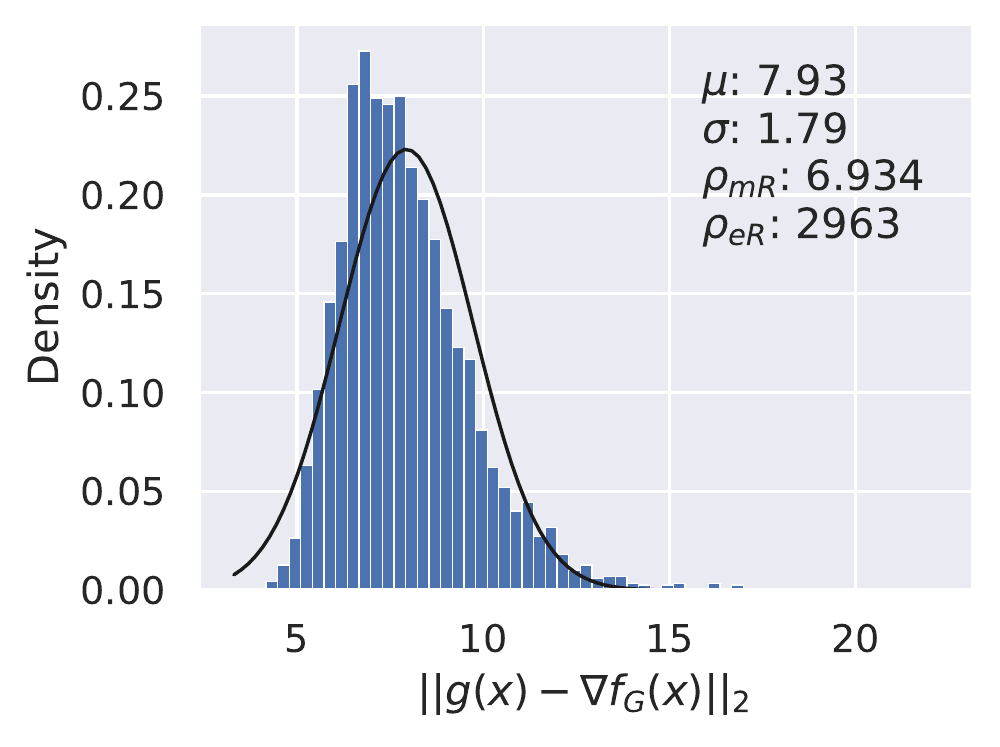} \\
            \includegraphics[width=\figscale\textwidth]{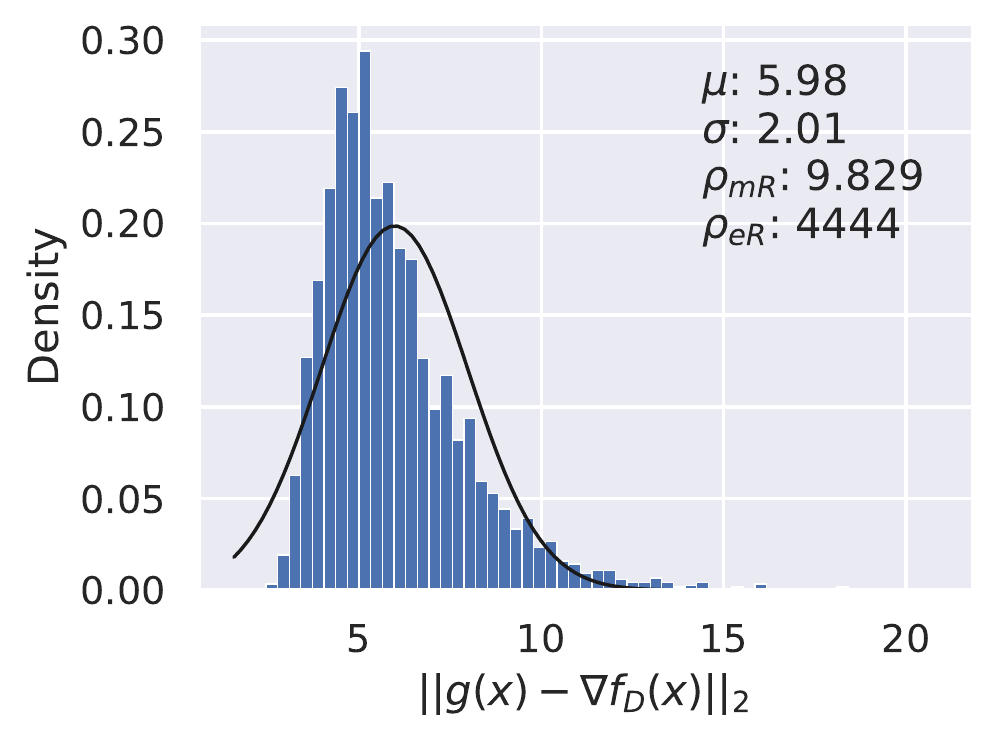} \\
        \end{tabular} &
        
        \hspace{-18pt}
        \raisebox{-5.6\normalbaselineskip}[0pt][0pt]{
            \includegraphics[width=0.235\textwidth]{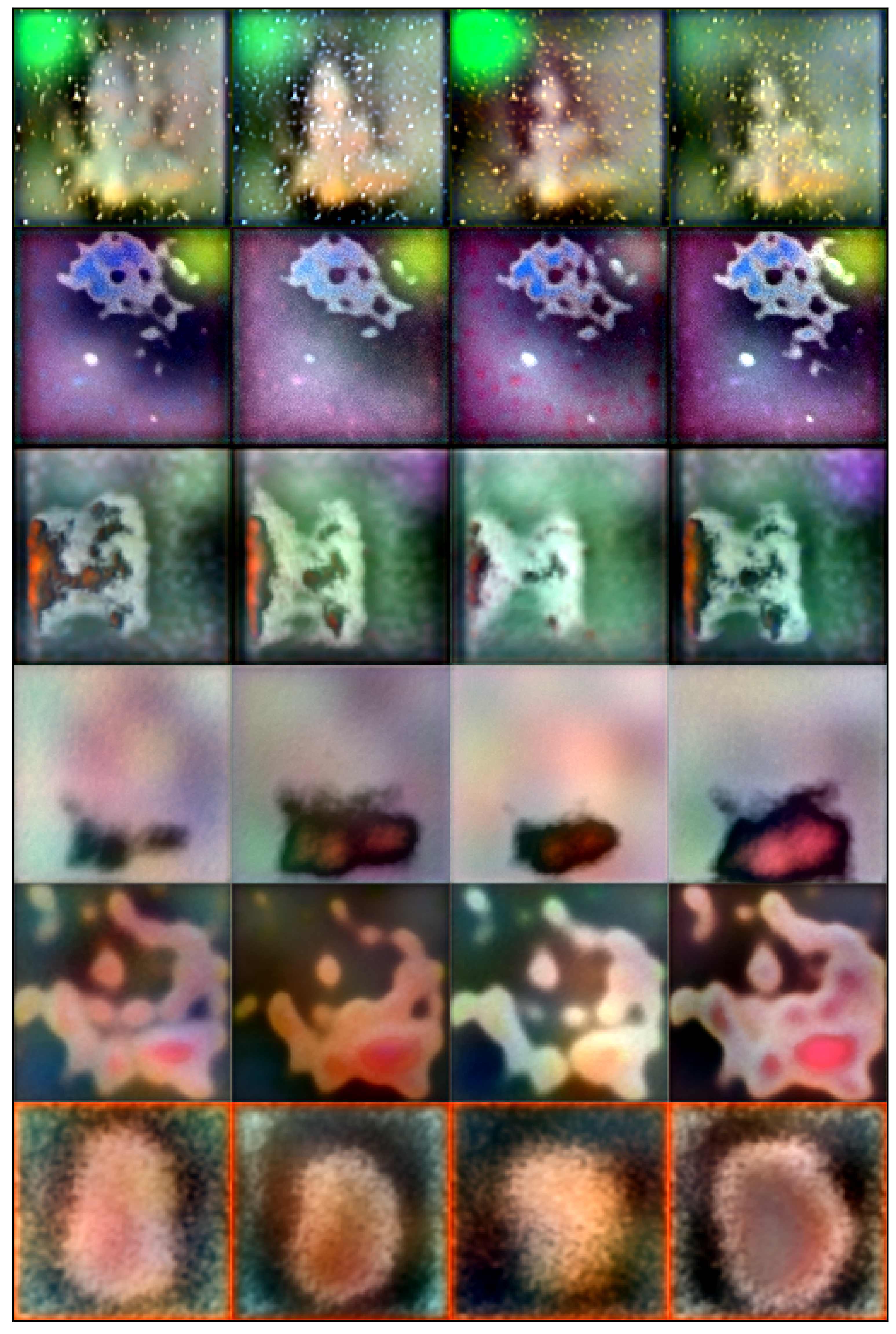}
        } & 
        
        \hspace{-14pt}
        \raisebox{-5.6\normalbaselineskip}[0pt][0pt]{
            \includegraphics[width=0.235\textwidth]{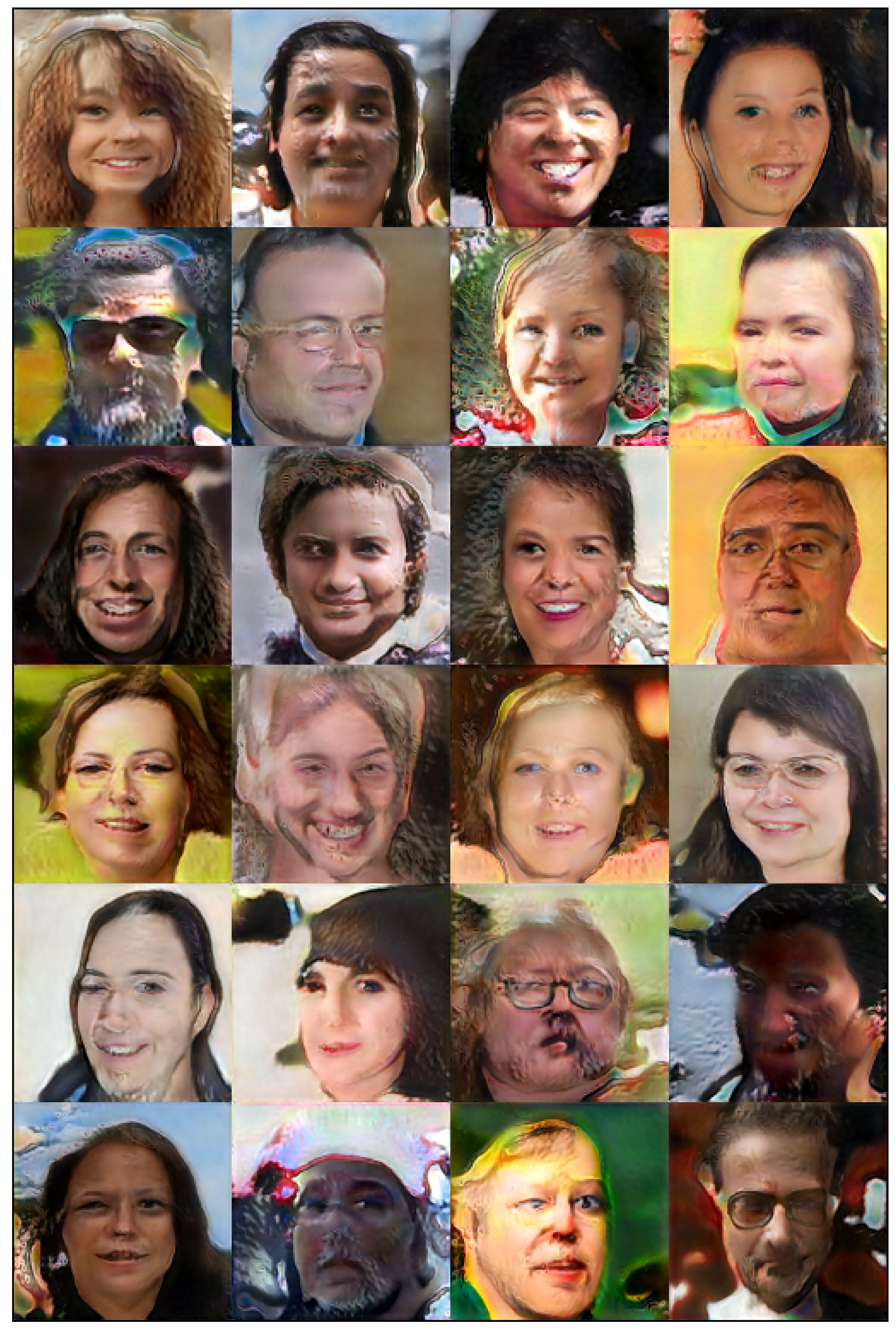}
        } \\
        
        & 
        (a) Initialization & 
        \hspace{-18pt}(b) \small{\algname{clipped-SGDA}} & 
        \hspace{-18pt}(c) \small{\algname{SGDA}} & 
        \hspace{-18pt}(d) \small{\algname{clipped-SGDA}} \\
    \end{tabular}

    \caption{\small
    (a) Gradient noise histograms for StyleGAN2 at random initialization. 
        (b) Gradient noise histograms for StyleGAN2 trained with \algname{clipped-SGDA}.
        (c) Random samples generated from several models trained with \algname{SGDA} with different learning rates (FID $>300$). 
        Each row corresponds to a different trained model, and all of our attempts to train StyleGAN2 with \algname{SGDA} produced similar results.
        (d) Random samples generated from the best \algname{clipped-SGDA} trained model (FID~$=72.68$).}
    \label{fig:stylegan2}
    \vspace{-4mm}
\end{figure}

\begin{figure}
    % \hspace{-25pt}  
    \centering
    \begin{tabular}{lcccc}
        % Generator noise
        \raisebox{3.5\normalbaselineskip}[0pt][0pt]{
            \rotatebox[origin=c]{90}{Generator}
        }\hspace{-7pt} & 
        \hspace{-7pt}
        \includegraphics[width=\evofigscale\textwidth]{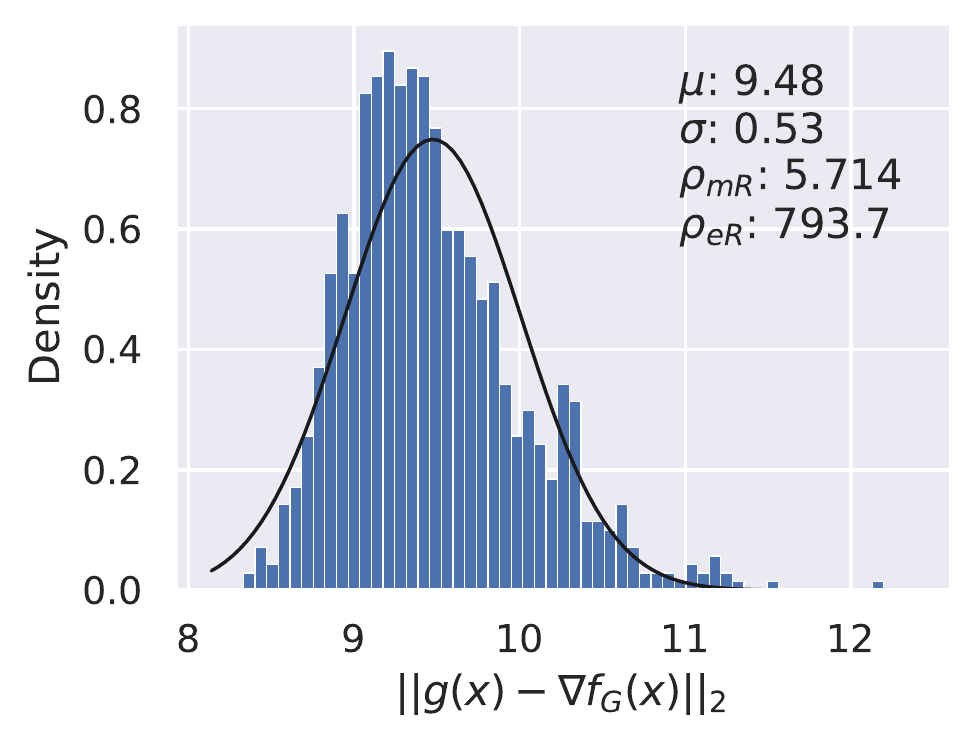} &   
        \hspace{-10pt}
        \includegraphics[width=\evofigscale\textwidth]{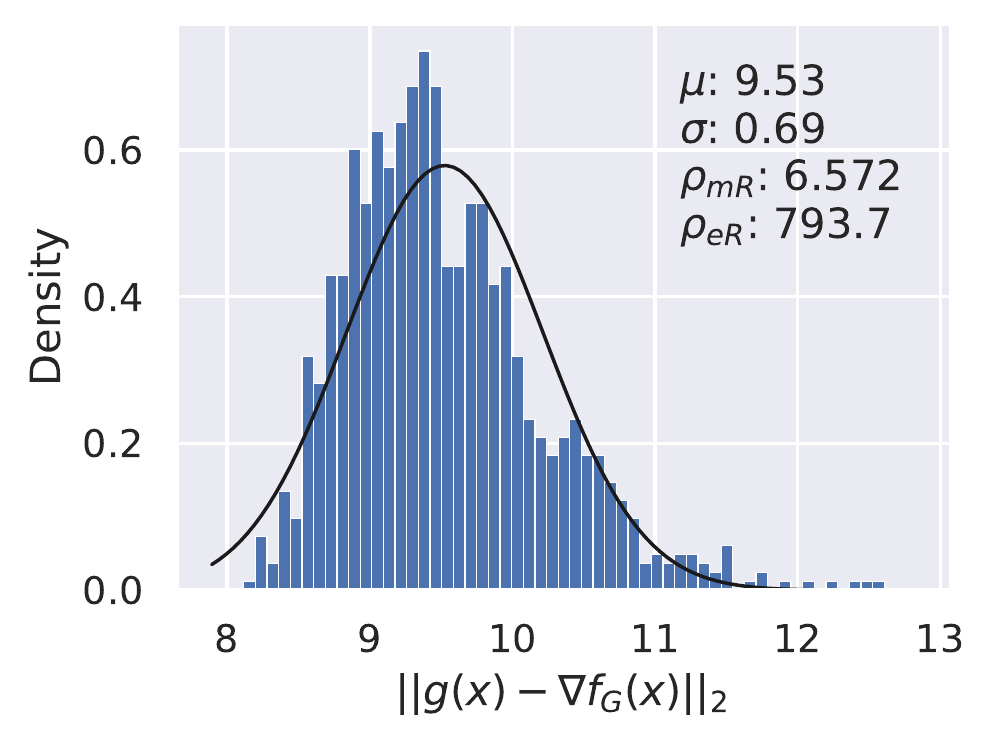} &
        \hspace{-10pt}
        \includegraphics[width=\evofigscale\textwidth]{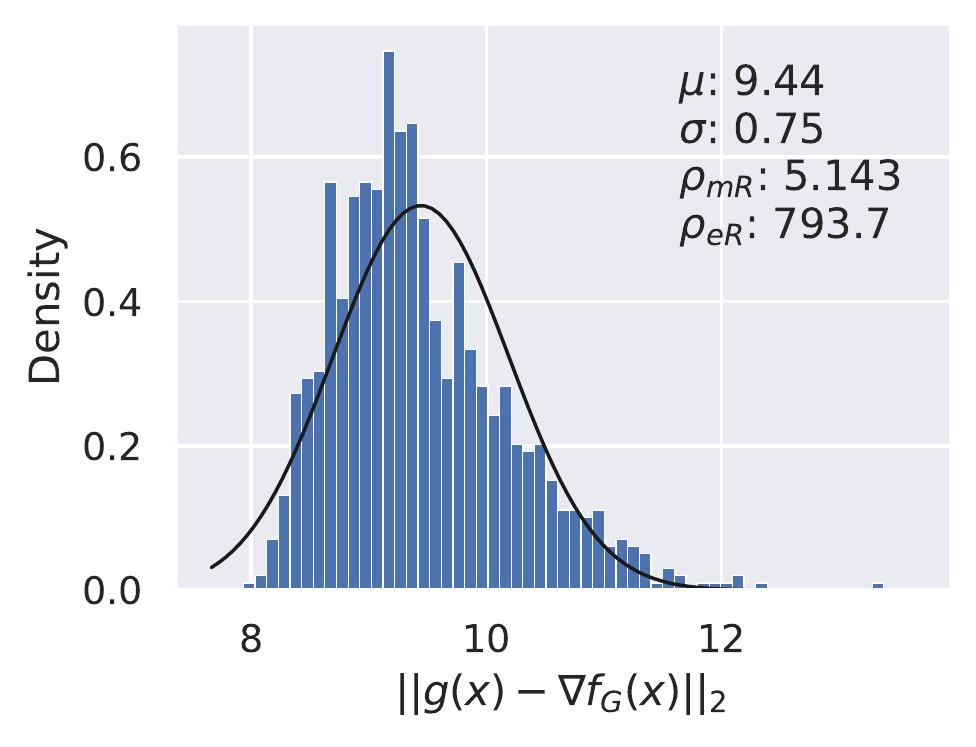} &   
        \hspace{-10pt}
        \includegraphics[width=\evofigscale\textwidth]{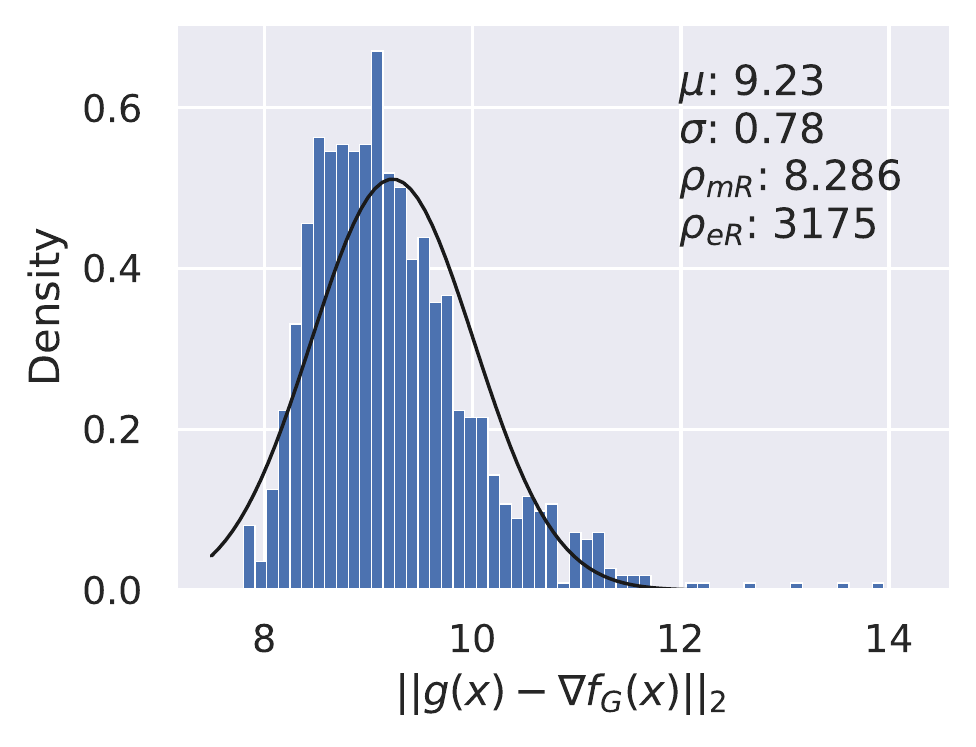} \\
    
        % Discriminator noise
        \raisebox{3.5\normalbaselineskip}[0pt][0pt]{
            \rotatebox[origin=c]{90}{Discriminator}
        }\hspace{-7pt} & 
        \hspace{-7pt}
        \includegraphics[width=\evofigscale\textwidth]{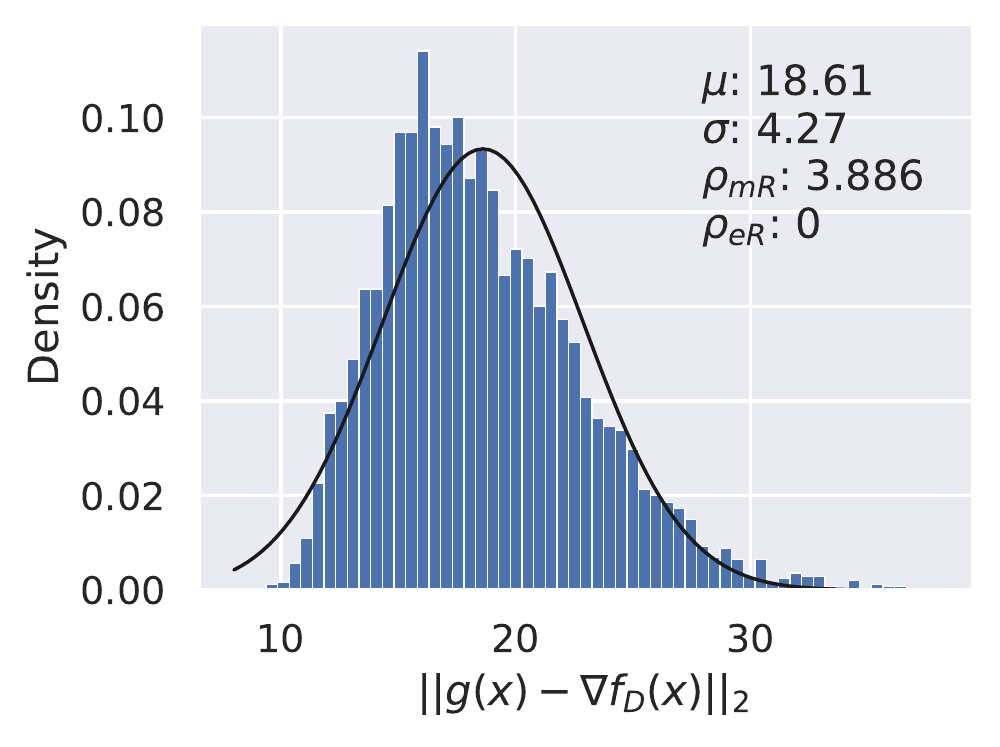} &   
        \hspace{-10pt}
        \includegraphics[width=\evofigscale\textwidth]{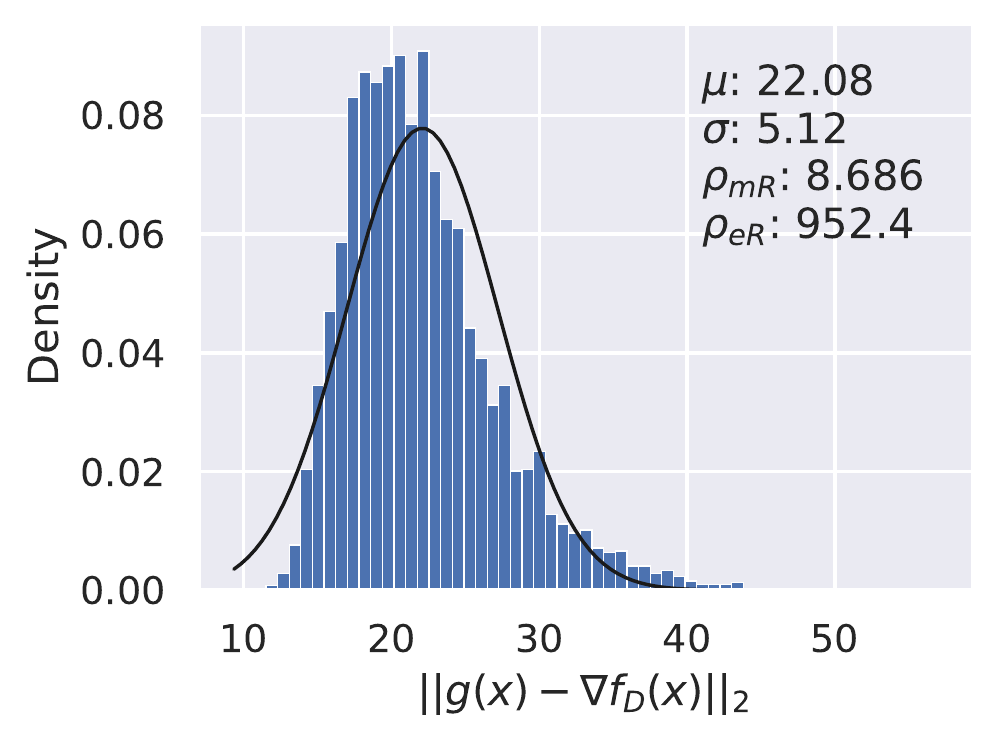} &
        \hspace{-10pt}
        \includegraphics[width=\evofigscale\textwidth]{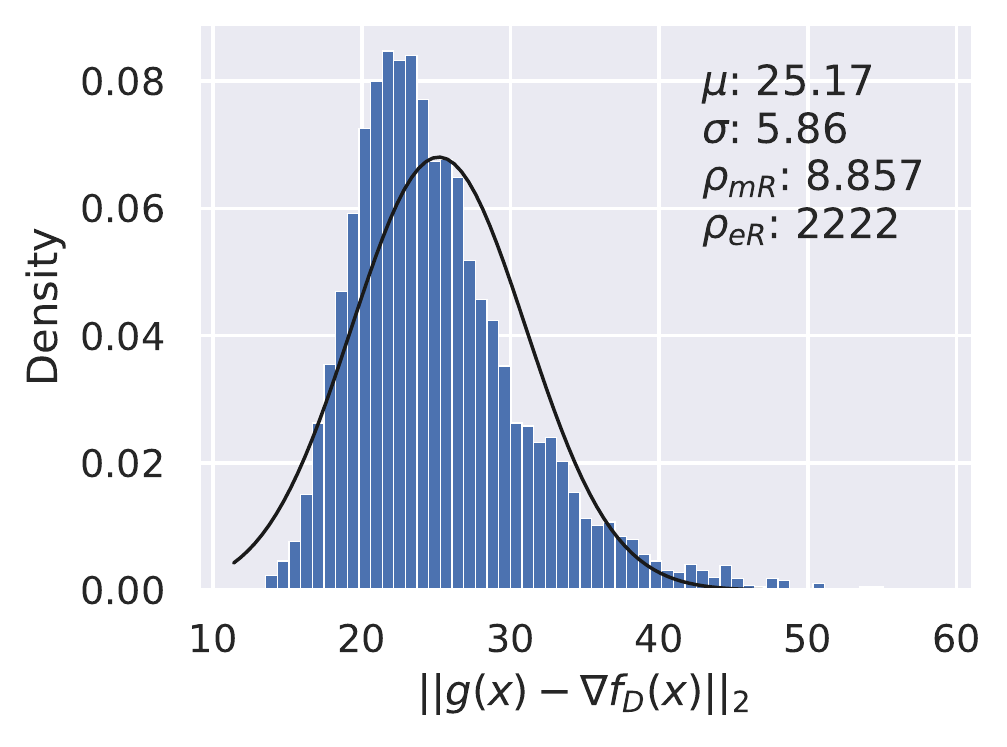} &   
        \hspace{-10pt}
        \includegraphics[width=\evofigscale\textwidth]{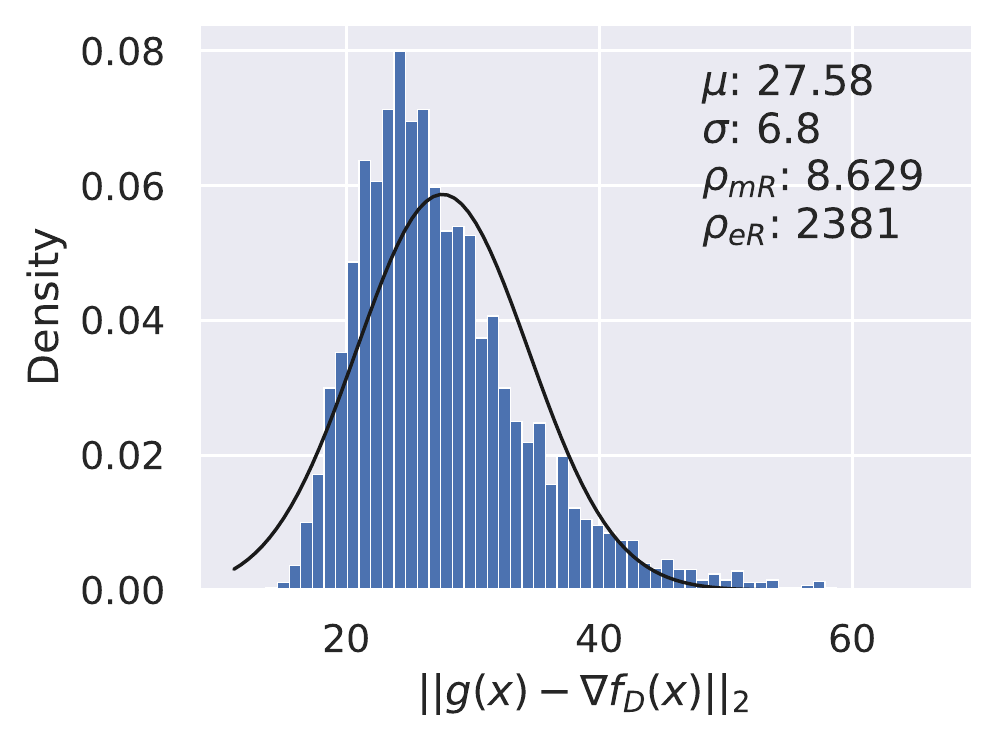} \\

        \cmidrule{2-5}\vspace{-3.5mm}\\
        
        \raisebox{3.5\normalbaselineskip}[0pt][0pt]{
            \rotatebox[origin=c]{90}{Generator}
        }\hspace{-7pt} & 
        \hspace{-7pt}
        \includegraphics[width=\evofigscale\textwidth]{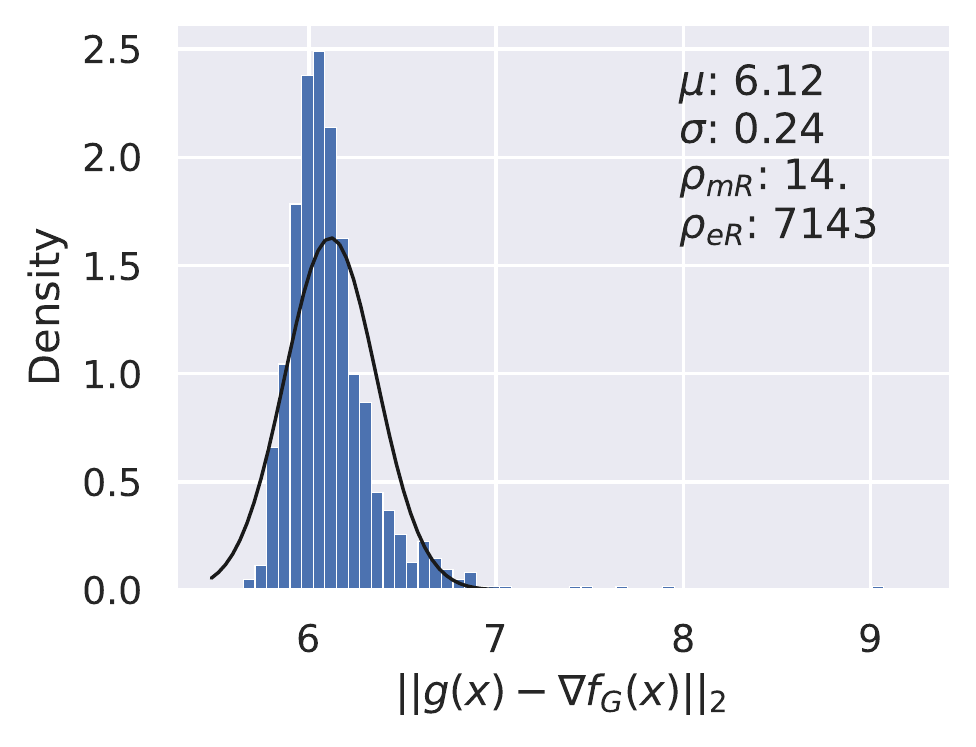} &   
        \hspace{-10pt}
        \includegraphics[width=\evofigscale\textwidth]{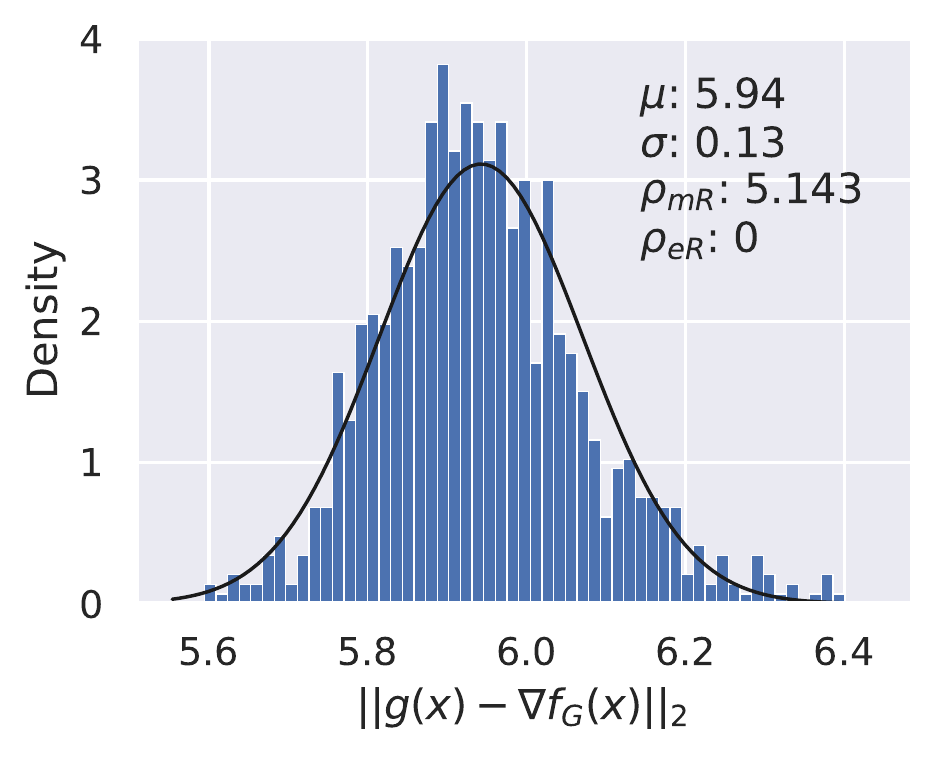} &
        \hspace{-10pt}
        \includegraphics[width=\evofigscale\textwidth]{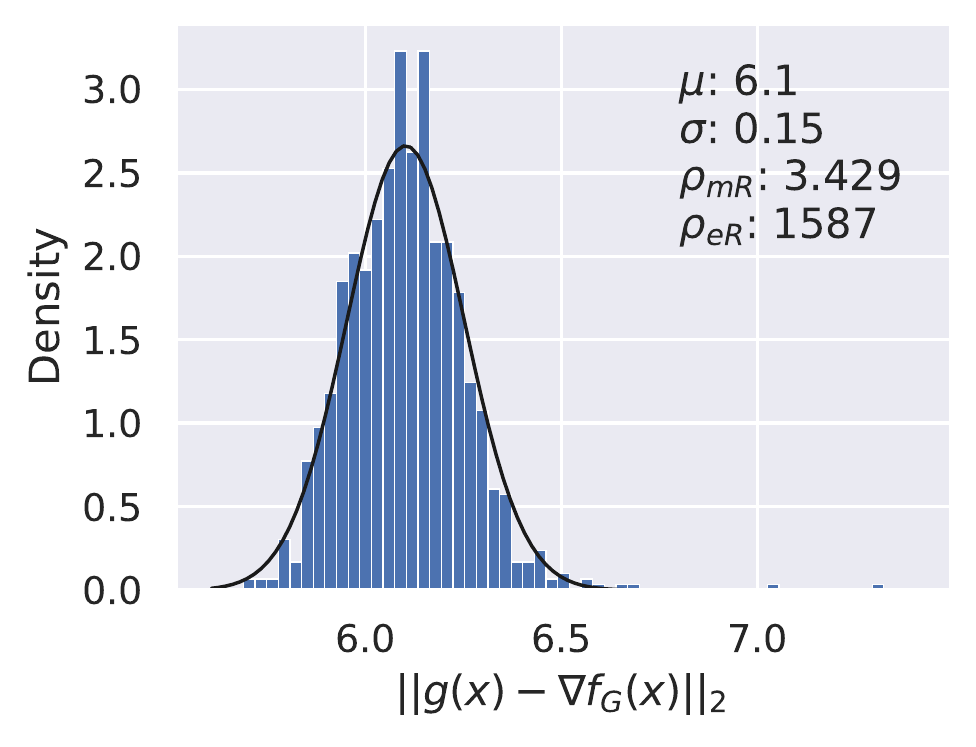} &   
        \hspace{-10pt}
        \includegraphics[width=\evofigscale\textwidth]{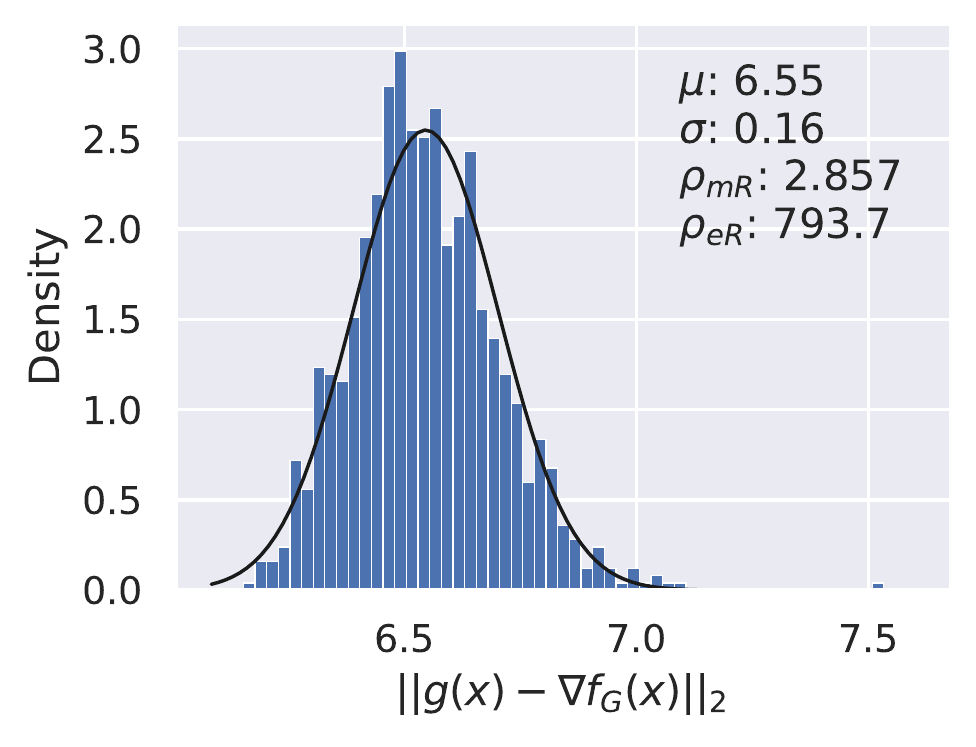} \\
    
        % Discriminator noise
        \raisebox{3.5\normalbaselineskip}[0pt][0pt]{
            \rotatebox[origin=c]{90}{Discriminator}
        }\hspace{-7pt} & 
        \hspace{-7pt}
        \includegraphics[width=\evofigscale\textwidth]{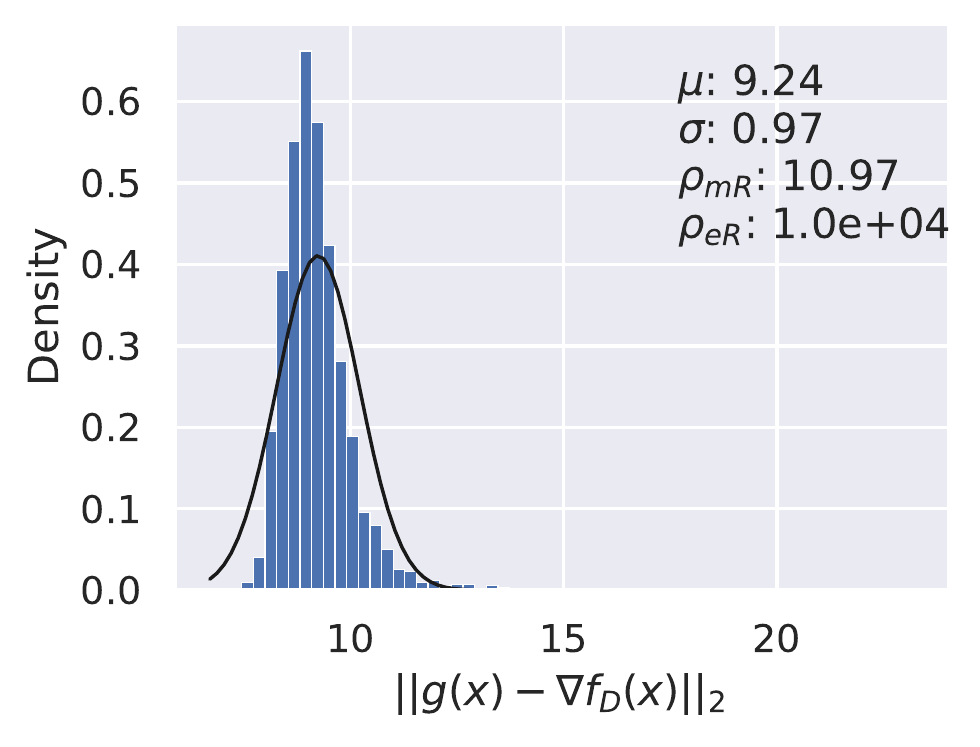} &   
        \hspace{-10pt}
        \includegraphics[width=\evofigscale\textwidth]{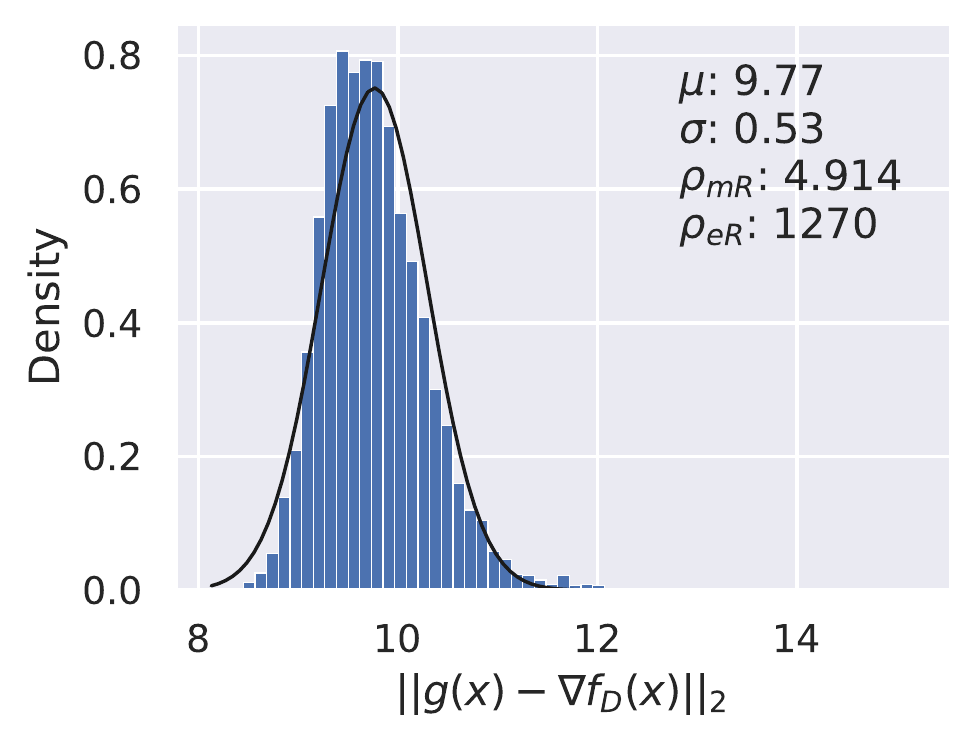} &
        \hspace{-10pt}
        \includegraphics[width=\evofigscale\textwidth]{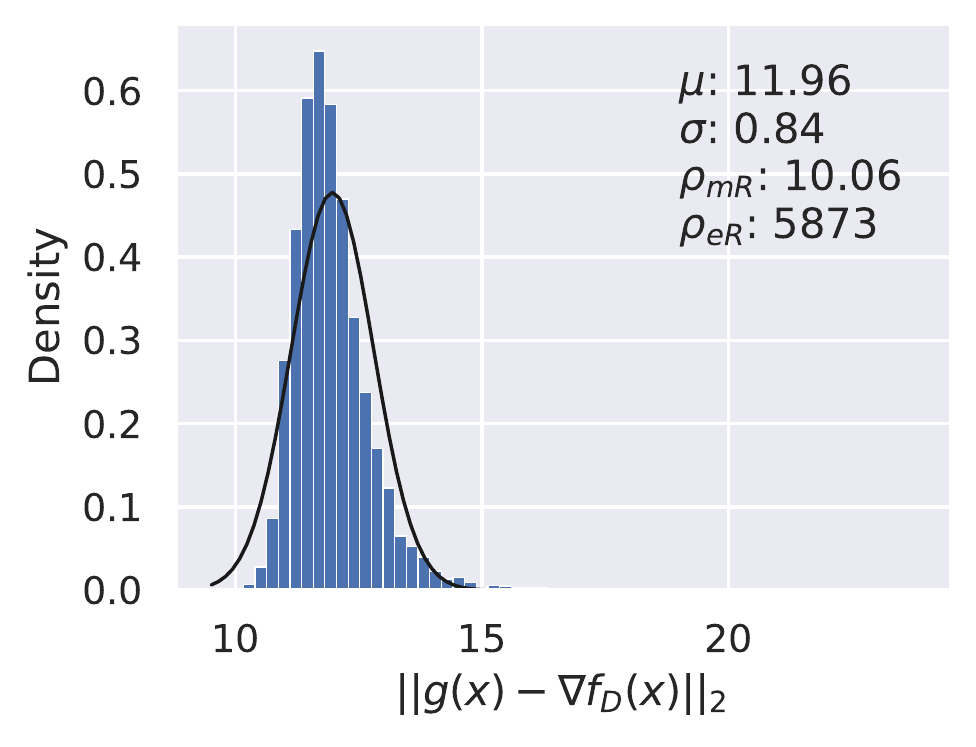} &   
        \hspace{-10pt}
        \includegraphics[width=\evofigscale\textwidth]{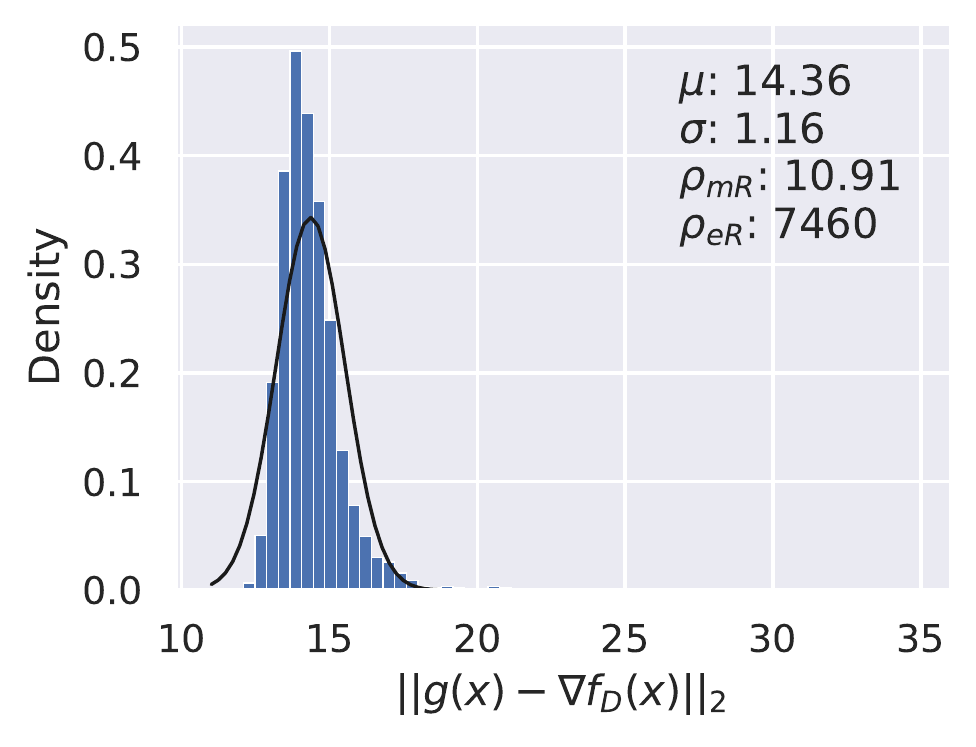} \\
        & {20000 steps} & {40000 steps} & {80000 steps} & {100000 steps} \\
    
    \end{tabular}
    \caption{\small Evolution of gradient noise histograms for a WGAN-GP model. The top two rows are trained with \algname{SGDA}, and the bottom two rows are trained with \algname{clipped-SGDA}.}
    \label{fig:wgangp_evo}
\end{figure}

\begin{ack}
This work was partially supported by a grant for research centers in the field of artificial intelligence, provided by the Analytical Center for the Government of the Russian Federation in accordance with the subsidy agreement (agreement identifier 000000D730321P5Q0002) and the agreement with the Moscow Institute of Physics and Technology dated November 1, 2021 No. 70-2021-00138.
The work by P. Dvurechensky was funded by the Deutsche Forschungsgemeinschaft (DFG, German Research Foundation) under Germany's Excellence Strategy – The Berlin Mathematics Research Center MATH+ (EXC-2046/1, project ID: 390685689).
\end{ack}

\bibliography{refs}

%%%%%%%%%%%%%%%%%%%%%%%%%%%%%%%%%%%%%%%%%%%%%%%%%%%%%%%%%%%%
\section*{Checklist}

\begin{enumerate}

\item For all authors...
\begin{enumerate}
  \item Do the main claims made in the abstract and introduction accurately reflect the paper's contributions and scope?
    \answerYes{}
  \item Did you describe the limitations of your work?
    \answerYes{We explicitly formulate all assumptions in \S~\ref{sec:techicalities}.}
  \item Did you discuss any potential negative societal impacts of your work?
    \answerNA{Our work is mainly theoretical.}
  \item Have you read the ethics review guidelines and ensured that your paper conforms to them?
    \answerYes{}
\end{enumerate}

\item If you are including theoretical results...
\begin{enumerate}
  \item Did you state the full set of assumptions of all theoretical results?
    \answerYes{We explicitly formulate all assumptions in \S~\ref{sec:techicalities}.}
        \item Did you include complete proofs of all theoretical results?
    \answerYes{Proofs are included in the Appendix.}
\end{enumerate}

\item If you ran experiments...
\begin{enumerate}
  \item Did you include the code, data, and instructions needed to reproduce the main experimental results (either in the supplemental material or as a URL)?
    \answerYes{}
  \item Did you specify all the training details (e.g., data splits, hyperparameters, how they were chosen)?
    \answerYes{}
        \item Did you report error bars (e.g., with respect to the random seed after running experiments multiple times)?
    \answerNo{Training GANs is computationally expensive and we did not aim at achieving SOTA FID.}
        \item Did you include the total amount of compute and the type of resources used (e.g., type of GPUs, internal cluster, or cloud provider)?
    \answerYes{}
\end{enumerate}

\item If you are using existing assets (e.g., code, data, models) or curating/releasing new assets...
\begin{enumerate}
  \item If your work uses existing assets, did you cite the creators?
    \answerYes{}
  \item Did you mention the license of the assets?
    \answerNA{}
  \item Did you include any new assets either in the supplemental material or as a URL?
    \answerNA{}
  \item Did you discuss whether and how consent was obtained from people whose data you're using/curating?
    \answerNA{}
  \item Did you discuss whether the data you are using/curating contains personally identifiable information or offensive content?
    \answerNA{}
\end{enumerate}

\item If you used crowdsourcing or conducted research with human subjects...
\begin{enumerate}
  \item Did you include the full text of instructions given to participants and screenshots, if applicable?
    \answerNA{}
  \item Did you describe any potential participant risks, with links to Institutional Review Board (IRB) approvals, if applicable?
    \answerNA{}
  \item Did you include the estimated hourly wage paid to participants and the total amount spent on participant compensation?
    \answerNA{}
\end{enumerate}

\end{enumerate}

%%%%%%%%%%%%%%%%%%%%%%%%%%%%%%%%%%%%%%%%%%%%%%%%%%%%%%%%%%%%
\newpage

\appendix
\tableofcontents
\newpage

\section{Further Related Work}\label{app:extra_related_work}

\textbf{Convergence in expectation.} Convergence in expectation of stochastic methods for solving \ref{eq:main_problem}s is relatively well-studied in the literature. In particular, versions of \algname{SEG} are studied under bounded variance \citep{beznosikov2020distributed, hsieh2020explore}, smoothness of stochastic realizations \citep{mishchenko2020revisiting}, and more refined assumptions unifying previously used ones \citep{gorbunov2022stochastic}. Recent advances on the in-expectation convergence of \algname{SGDA} are obtained in \citep{loizou2021stochastic, beznosikov2022stochastic}.

\textbf{Gradient clipping.} In the context of solving minimization problems, gradient clipping \citep{pascanu2013difficulty} and normalization \citep{hazan2015beyond} are known to have a number of favorable properties such as practical robustness to the rapid changes of the loss function \citep{Goodfellow-et-al-2016}, provable convergence for structured non-smooth problems with polynomial growth \cite{zhang2020gradient, mai2021stability} and for the problems with heavy-tailed noise in convex  \citep{nazin2019algorithms, gorbunov2020stochastic, gorbunov2021near} and non-convex cases \citep{zhang2020adaptive, cutkosky2021high}. Our work makes a further step towards a better understanding of gradient clipping and is the first to study the theoretical convergence of clipped first-order stochastic methods for \ref{eq:main_problem}s.

\textbf{Structured non-monotonicity.} There is a noticeable growing interest of the community in studying the theoretical convergence guarantees of deterministic methods for solving \ref{eq:main_problem} with non-monotone operators $F(x)$ having a certain structure, e.g., negative comonotonicty \citep{diakonikolas2021efficient, lee2021fast, bohm2022solving}, quasi-strong monotonicity \citep{song2020optimistic, mertikopoulos2019learning} and/or star-cocoercivity \citep{loizou2021stochastic, gorbunov2022extragradient, gorbunov2022stochastic, beznosikov2022stochastic}. In the context of stochastic \ref{eq:main_problem}s, \algname{SEG} (with different extrapolation and update stepsizes) is analyzed under negative comonotonicity by \citet{diakonikolas2021efficient} and under quasi-strong monotonicity by \citet{gorbunov2022stochastic}, while \algname{SGDA} is studied under quasi-strong monotonicity and/or star-cocoercivity by \citep{loizou2021stochastic, beznosikov2022stochastic}. These results establish in-expectation convergence rates. Our paper continues this line of works and provides the first high-probability analysis of stochastic methods for solving \ref{eq:main_problem}s with structured non-monotonicity.

\newpage

\section{Auxiliary Results}

\paragraph{Useful inequalities.} For all $a,b \in \R^d$ and $\alpha > 0$ the following relations hold:
\begin{eqnarray}
    2\langle a, b \rangle &=& \|a\|^2 + \|b\|^2 - \|a - b\|^2, \label{eq:inner_product_representation}\\
    \|a + b\|^2 &\leq& 2\|a\|^2 + 2\|b\|^2, \label{eq:a_plus_b}\\
    -\|a-b\|^2 &\leq& -\frac{1}{2}\|a\|^2 + \|b\|^2. \label{eq:a_minus_b} 
\end{eqnarray}

\paragraph{Bernstein inequality.} In our proofs, we rely on the following lemma known as {\it Bernstein inequality for martingale differences} \citep{bennett1962probability,dzhaparidze2001bernstein,freedman1975tail}.
\begin{lemma}\label{lem:Bernstein_ineq}
    Let the sequence of random variables $\{X_i\}_{i\ge 1}$ form a martingale difference sequence, i.e.\ $\EE\left[X_i\mid X_{i-1},\ldots, X_1\right] = 0$ for all $i \ge 1$. Assume that conditional variances $\sigma_i^2\eqdef\EE\left[X_i^2\mid X_{i-1},\ldots, X_1\right]$ exist and are bounded and assume also that there exists deterministic constant $c>0$ such that $|X_i| \le c$ almost surely for all $i\ge 1$. Then for all $b > 0$, $G > 0$ and $n\ge 1$
    \begin{equation}
        \PP\left\{\Big|\sum\limits_{i=1}^nX_i\Big| > b \text{ and } \sum\limits_{i=1}^n\sigma_i^2 \le G\right\} \le 2\exp\left(-\frac{b^2}{2G + \nicefrac{2cb}{3}}\right).
    \end{equation}
\end{lemma}

\paragraph{Bias and variance of clipped stochastic vector.} We also use the following properties of clipped stochastic estimators from \citep{gorbunov2020stochastic}.
\begin{lemma}[Simplified version of Lemma F.5 from \citep{gorbunov2020stochastic}]\label{lem:bias_variance}
    Let $X$ be a random vector in $\R^d$ and $\tX = \clip(X,\lambda)$. Then,
    \begin{equation}
        \left\|\tX - \EE[\tX]\right\| \leq 2\lambda. \label{eq:bound_X}
    \end{equation} 
    Moreover, if for some $\sigma \geq 0$
    \begin{equation}
        \EE[X] = x\in\R^d,\quad \EE[\|X - x\|^2] \leq \sigma^2 \label{eq:UBV_X}
    \end{equation}
    and $x \leq \nicefrac{\lambda}{2}$, then
    \begin{eqnarray}
        \left\|\EE[\tX] - x\right\| &\leq& \frac{4\sigma^2}{\lambda}, \label{eq:bias_X}\\
        \EE\left[\left\|\tX - x\right\|^2\right] &\leq& 18\sigma^2, \label{eq:distortion_X}\\
        \EE\left[\left\|\tX - \EE[\tX]\right\|^2\right] &\leq& 18\sigma^2. \label{eq:variance_X}
    \end{eqnarray}
\end{lemma}
\begin{proof}
    The proof of this lemma is identical to the original one, since \citet{gorbunov2020stochastic} rely only on $\tX = \clip(X,\lambda)$ to derive \eqref{eq:bound_X}, and to prove \eqref{eq:bias_X}-\eqref{eq:variance_X} they use only \eqref{eq:UBV_X}, $\tX = \clip(X,\lambda)$ and $x \leq \nicefrac{\lambda}{2}$
\end{proof}

\newpage

\section{Clipped Stochastic Extragradient: Missing Proofs and Details}\label{app:clipped_SEG_proofs}

\subsection{Monotone Case}

\begin{lemma}\label{lem:optimization_lemma_gap_SEG}
    Let Assumptions~\ref{as:UBV}, \ref{as:lipschitzness}, \ref{as:monotonicity} hold for $Q = B_{4R}(x^*)$, where $R \geq R_0 \eqdef \|x^0 - x^*\|$, and $\gamma_1 = \gamma_2 = \gamma$, $0 < \gamma \leq \nicefrac{1}{\sqrt{2}L}$. If $x^k$ and $\tx^k$ lie in $B_{4R}(x^*)$ for all $k = 0,1,\ldots, K$ for some $K\geq 0$, then for all $u \in B_{4R}(x^*)$ the iterates produced by \ref{eq:clipped_SEG} satisfy
    \begin{eqnarray}
        \langle F(u), \tx^K_{\avg} - u\rangle &\leq& \frac{\|x^0 - u\|^2 - \|x^{K+1} - u\|^2}{2\gamma(K+1)} + \frac{\gamma}{2(K+1)}\sum\limits_{k=0}^K\left(\|\theta_k\|^2 + 2\|\omega_k\|^2\right)\notag\\
        &&\quad + \frac{1}{K+1}\sum\limits_{k=0}^K\langle x^k - u - \gamma F(\tx^k), \theta_k\rangle, \label{eq:optimization_lemma_SEG}\\
        \tx^K_{\avg} &\eqdef& \frac{1}{K+1}\sum\limits_{k=0}^{K}\tx^k, \label{eq:tx_avg_SEG}\\
        \theta_k &\eqdef& F(\tx^k) - \tF_{\Bxi_2^k}(\tx^k), \label{eq:theta_k_SEG}\\
        \omega_k &\eqdef& F(x^k) - \tF_{\Bxi_1^k}(x^k). \label{eq:omega_k_SEG}
    \end{eqnarray}
\end{lemma}
\begin{proof}
    Using the update rule of \ref{eq:clipped_SEG}, for all $u \in B_{4R}(x^*)$ we obtain
    \begin{eqnarray*}
        \|x^{k+1} - u\|^2 &=& \|x^k - u\|^2 - 2\gamma \langle x^k - u,  \tF_{\Bxi_2^k}(\tx^k)\rangle + \gamma^2\|\tF_{\Bxi_2^k}(\tx^k)\|^2\\
        &=& \|x^k - u\|^2 -2\gamma \langle x^k - u, F(\tx^k) \rangle + 2\gamma \langle x^k - u, \theta_k \rangle\\
        &&\quad + \gamma^2\|F(\tx^k)\|^2 - 2\gamma^2 \langle F(\tx^k), \theta_k \rangle + \gamma^2\|\theta_k\|^2\\
        &=& \|x^k - u\|^2 -2\gamma \langle \tx^k - u, F(\tx^k) \rangle - 2\gamma \langle x^k - \tx^k, F(\tx^k) \rangle\\
        &&\quad + 2\gamma \langle x^k - u - \gamma F(\tx^k), \theta_k \rangle + \gamma^2\|F(\tx^k)\|^2 + \gamma^2\|\theta_k\|^2\\
        &\overset{\eqref{eq:monotonicity}}{\leq}& \|x^k - u\|^2 -2\gamma \langle \tx^k - u, F(u) \rangle - 2\gamma^2\langle \tF_{\Bxi_1^k}(x^k), F(\tx^k) \rangle\\
        &&\quad + 2\gamma \langle x^k - u - \gamma F(\tx^k), \theta_k \rangle + \gamma^2\|F(\tx^k)\|^2 + \gamma^2\|\theta_k\|^2\\
        &\overset{\eqref{eq:inner_product_representation}}{=}& \|x^k - u\|^2 -2\gamma \langle \tx^k - u, F(u) \rangle\\
        &&\quad + \gamma^2\| \tF_{\Bxi_1^k}(x^k) - F(\tx^k) \|^2 - \gamma^2 \|F(\tx^k)\|^2 - \gamma^2 \|\tF_{\Bxi_1^k}(x^k)\|^2\\
        &&\quad + 2\gamma \langle x^k - u - \gamma F(\tx^k), \theta_k \rangle + \gamma^2\|F(\tx^k)\|^2 + \gamma^2\|\theta_k\|^2\\
        &\overset{\eqref{eq:a_plus_b}}{\leq}& \|x^k - u\|^2 -2\gamma \langle \tx^k - u, F(u) \rangle\\
        &&\quad+ 2\gamma^2\|\omega_k\|^2 + 2\gamma^2\|F(x^k) - F(\tx^k)\|^2 - \gamma^2 \|\tF_{\Bxi_1^k}(x^k)\|^2\\
        &&\quad + 2\gamma \langle x^k - u - \gamma F(\tx^k), \theta_k \rangle + \gamma^2\|\theta_k\|^2\\
        &\overset{\eqref{eq:Lipschitzness}}{\leq}& \|x^k - u\|^2 -2\gamma \langle \tx^k - u, F(u) \rangle - \gamma^2\left(1 - 2\gamma^2L^2\right)\|\tF_{\Bxi_1^k}(x^k)\|^2 \\
        &&\quad + 2\gamma \langle x^k - u - \gamma F(\tx^k), \theta_k \rangle + \gamma^2\|\theta_k\|^2 + 2\gamma^2\|\omega_k\|^2
    \end{eqnarray*}
    where in the last step we additionally use $x^k - \tx^k = \gamma \tF_{\Bxi_1^k}(x^k)$ after the application of Lipschitzness of $F$. Since $\gamma \leq \nicefrac{1}{\sqrt{2}L}$, we have $\gamma^2\left(1 - 2\gamma^2L^2\right)\|\tF_{\Bxi_1^k}(x^k)\|^2 \geq 0$, implying
    \begin{eqnarray*}
        2\gamma \langle F(u), \tx^k - u \rangle &\leq& \|x^k - u\|^2 - \|x^{k+1} - u\|^2 + \gamma^2\left(\|\theta_k\|^2 + 2\|\omega_k\|^2\right)\\
        &&\quad + 2\gamma \langle x^k - u - \gamma F(\tx^k), \theta_k \rangle.
    \end{eqnarray*}
    Finally, we sum up the above inequality for $k = 0,1,\ldots, K$ and divide both sides of the result by $2\gamma(K+1)$:
    \begin{eqnarray*}
        \langle F(u), \tx^K_{\avg} - u\rangle &=& \frac{1}{K+1}\sum\limits_{k=0}^K \langle F(u), \tx^k - u \rangle\\
        &\leq& \frac{1}{2\gamma(K+1)}\sum\limits_{k=0}^K\left(\|x^k - u\|^2 - \|x^{k+1} - u\|^2\right) + \frac{\gamma}{2(K+1)}\sum\limits_{k=0}^K\|\theta_k\|^2\\
        &&\quad + \frac{1}{K+1}\sum\limits_{k=0}^K\langle x^k - u - \gamma F(\tx^k), \theta_k\rangle + \frac{\gamma}{K+1}\sum\limits_{k=0}^K\|\omega_k\|^2\\
        &=& \frac{\|x^0 - u\|^2 - \|x^{K+1} - u\|^2}{2\gamma(K+1)} + \frac{\gamma}{2(K+1)}\sum\limits_{k=0}^K\left(\|\theta_k\|^2 + 2\|\omega_k\|^2\right)\\
        &&\quad + \frac{1}{K+1}\sum\limits_{k=0}^K\langle x^k - u - \gamma F(\tx^k), \theta_k\rangle.
    \end{eqnarray*}
    This concludes the proof.
\end{proof}

\begin{theorem}\label{thm:main_result_gap_SEG}
    Let Assumptions~\ref{as:UBV}, \ref{as:lipschitzness}, \ref{as:monotonicity} hold for $Q = B_{4R}(x^*)$, where $R \geq R_0 \eqdef \|x^0 - x^*\|$, and\footnote{In this and further results, we have relatively large numerical constants in the conditions on step-sizes, batch-sizes, and clipping levels.  However, our main goal is deriving results in terms of $\cO(\cdot)$, where numerical constants are not taken into consideration. Although it is possible to significantly improve the dependence on numerical factors, we do not do it for the sake of proofs' simplicity.} $\gamma_1 = \gamma_2 = \gamma$,
    \begin{eqnarray}
        0< \gamma &\leq& \frac{1}{160L \ln \tfrac{6(K+1)}{\beta}}, \label{eq:gamma_SEG}\\
        \lambda_{1,k} = \lambda_{2,k} \equiv \lambda &=& \frac{R}{20\gamma \ln \tfrac{6(K+1)}{\beta}}, \label{eq:lambda_SEG}\\
        m_{1,k} = m_{2,k} \equiv m &\geq& \max\left\{1, \frac{10800 (K+1) \gamma^2\sigma^2 \ln\tfrac{6(K+1)}{\beta}}{R^2}\right\}, \label{eq:batch_SEG}
    \end{eqnarray}
    for some $K \geq 0$ and $\beta \in (0,1]$ such that $\ln \tfrac{6(K+1)}{\beta} \geq 1$. Then, after $K$ iterations the iterates produced by \ref{eq:clipped_SEG} with probability at least $1 - \beta$ satisfy 
    \begin{equation}
        \gap_R(\tx_{\avg}^K) \leq \frac{9R^2}{2\gamma(K+1)}, \label{eq:main_result}
    \end{equation}
    where $\tx_{\avg}^K$ is defined in \eqref{eq:tx_avg_SEG}.
\end{theorem}
\begin{proof}
    We introduce new notation: $R_k = \|x^k - x^*\|$ for all $k\geq 0$. The proof is based on deriving via induction that $R_k^2 \leq \tilde{C}R^2$ for some numerical constant $\tilde{C} > 0$. In particular, for each $k = 0,\ldots, K+1$ we define probability event $E_k$ as follows: inequalities
    \begin{gather}
        \underbrace{\max\limits_{u \in B_{R}(x^*)}\!\left\{\! \|x^0 - u\|^2 + 2\gamma \sum\limits_{l = 0}^{t-1} \langle x^l - u - \gamma F(\tx^l), \theta_l \rangle + \gamma^2 \sum\limits_{l=0}^{t-1}\left(\|\theta_l\|^2 + 2\|\omega_l\|^2\right)\!\right\}}_{A_t} \leq 9R^2, \label{eq:induction_inequality_1_SEG}\\
        \left\|\gamma\sum\limits_{l=0}^{t-1}\theta_l\right\| \leq R \label{eq:induction_inequality_2_SEG}
    \end{gather}
    hold for $t = 0,1,\ldots,k$ simultaneously. Our goal is to prove that $\PP\{E_k\} \geq  1 - \nicefrac{k\beta}{(K+1)}$ for all $k = 0,1,\ldots,K+1$. We use the induction to show this statement. For $k = 0$ the statement is trivial since $\|x^0 - u\|^2 \leq 2\|x^0 - x^*\|^2 + 2\|x^* - u\|^2 \leq 4R^2 \leq 9R^2$ and $\|\gamma\sum_{l=0}^{k-1}\theta_l\| = 0$ for any $u \in B_R(x^*)$. Next, assume that the statement holds for $k = T-1 \leq K$, i.e., we have $\PP\{E_{T-1}\} \geq 1 - \nicefrac{(T-1)\beta}{(K+1)}$. We need to prove that $\PP\{E_T\} \geq 1 - \nicefrac{T\beta}{(K+1)}$. First of all, we show that probability event $E_{T-1}$ implies $R_{t} \leq 3R$ for all $t = 0, 1, \ldots, T$. For $t = 0$ we already proved it. Next, assume that we have $R_{t} \leq 3R$ for all $t = 0,1,\ldots, t'$, where $t' < T$. Then, for all $t = 0, 1, \ldots, t'$ we have
    \begin{eqnarray}
        \|\tx^t - x^*\| &=& \|x^t - x^* - \gamma \tF_{\Bxi_1^t}(x^t)\| \leq \|x^t - x^*\| + \gamma \|\tF_{\Bxi_1^t}(x^t)\|\notag\\
        &\leq& \|x^t - x^*\| + \gamma\lambda \overset{\eqref{eq:lambda_SEG}}{\leq} 3R + \frac{R}{20\ln\frac{6(K+1)}{\beta}}  \leq 4R, \label{eq:gap_thm_SEG_technical_1}
    \end{eqnarray}
    i.e., $\tx^t \in B_{4R}(x^*)$. This means that the assumptions of Lemma~\ref{lem:optimization_lemma_gap_SEG} hold and we have that probability event $E_{T-1}$ implies
    \begin{eqnarray*}
        \max\limits_{u \in B_{R}(x^*)}\left\{2\gamma(t' + 1)\langle F(u), \tx^{t'}_{\avg} - u\rangle + \|x^{t'+1} - u\|^2\right\} &\\
        &\hspace{-3cm}\leq \max\limits_{u \in B_{R}(x^*)}\!\left\{\! \|x^0 - u\|^2 + 2\gamma \sum\limits_{l = 0}^{t'-1} \langle x^l - u - \gamma F(\tx^l), \theta_l \rangle\!\right\}\\
        & + \gamma^2 \sum\limits_{l=0}^{t'-1}\left(\|\theta_l\|^2 + 2\|\omega_l\|^2\right)\\
        &\hspace{-10.5cm} \overset{\eqref{eq:induction_inequality_1_SEG}}{\leq} 9R^2,
    \end{eqnarray*}
    meaning that
    \begin{eqnarray*}
        \|x^{t' + 1} - x^*\|^2 &\leq& \max\limits_{u \in B_{R}(x^*)}\left\{2\gamma(t' + 1)\langle F(u), \tx^{t'}_{\avg} - u\rangle + \|x^{t'+1} - u\|^2\right\}\\
        &\leq& 9R^2,
    \end{eqnarray*}
    i.e., $R_{t'+1} \leq 3R$. That is, we proved that probability event $E_{T-1}$ implies $R_{t} \leq 3R$ and
    \begin{equation}
        \max\limits_{u \in B_{R}(x^*)}\left\{2\gamma(t + 1)\langle F(u), \tx^{t}_{\avg} - u\rangle + \|x^{t+1} - u\|^2\right\} \leq 9R^2 \label{eq:gap_thm_SEG_technical_1_5}
    \end{equation}
    for all $t = 0, 1, \ldots, T$. Moreover, in view of \eqref{eq:gap_thm_SEG_technical_1} $E_{T-1}$ also implies that $\|\tx^t - x^*\| \leq 4R$ for all $t = 0, 1, \ldots, T$. Using this, we derive that $E_{T-1}$ implies
    \begin{eqnarray}
        \|x^t - x^* - \gamma F(\tx^t)\| &\leq& \|x^t - x^*\| + \gamma\|F(\tx^t)\| \overset{\eqref{eq:Lipschitzness}}{\leq} 3R + \gamma L\|\tx^t - x^*\|\notag\\
        &\overset{\eqref{eq:gap_thm_SEG_technical_1}}{\leq}& 3R + 4R\gamma L \overset{\eqref{eq:gamma_SEG}}{\leq} 5R, \label{eq:gap_thm_SEG_technical_2}
    \end{eqnarray}
    for all $t = 0, 1, \ldots, T$. Consider random vectors
    \begin{equation*}
        \eta_t = \begin{cases}x^t - x^* - \gamma F(\tx^t),& \text{if } \|x^t - x^* - \gamma F(\tx^t)\| \leq 5R,\\ 0,& \text{otherwise,} \end{cases}
    \end{equation*}
    for all $t = 0, 1, \ldots, T$. We notice that $\eta_t$ is bounded with probability $1$:
    \begin{equation}
        \|\eta_t\| \leq 5R  \label{eq:gap_thm_SEG_technical_3}
    \end{equation}
    for all $t = 0, 1, \ldots, T$. Moreover, in view of \eqref{eq:gap_thm_SEG_technical_2}, probability event $E_{T-1}$ implies $\eta_t = x^t - x^* - \gamma F(\tx^t)$ for all $t = 0, 1, \ldots, T$. Therefore, $E_{T-1}$ implies
    \begin{eqnarray}
        A_T &=& \max\limits_{u \in B_{R}(x^*)}\left\{ \|x^0 - u\|^2 + 2\gamma \sum\limits_{l = 0}^{T-1} \langle x^* - u, \theta_l \rangle \right\}\notag\\
        &&\quad + 2\gamma \sum\limits_{l = 0}^{T-1} \langle x^l - x^* - \gamma F(\tx^l), \theta_l \rangle + \gamma^2 \sum\limits_{l=0}^{T-1}\left(\|\theta_l\|^2 + 2\|\omega_l\|^2\right)\notag\\
        &\leq& 4R^2 + 2\gamma\max\limits_{u \in B_{R}(x^*)}\left\{\left\langle x^* - u, \sum\limits_{l = 0}^{T-1}\theta_l \right\rangle \right\}\notag\\
        &&\quad + 2\gamma \sum\limits_{l = 0}^{T-1} \langle \eta_l, \theta_l \rangle + \gamma^2 \sum\limits_{l=0}^{T-1}\left(\|\theta_l\|^2 + 2\|\omega_l\|^2\right)\notag\\
        &=& 4R^2 + 2\gamma R \left\|\sum\limits_{l = 0}^{T-1}\theta_l\right\| + 2\gamma \sum\limits_{l = 0}^{T-1} \langle \eta_l, \theta_l \rangle + \gamma^2 \sum\limits_{l=0}^{T-1}\left(\|\theta_l\|^2 + 2\|\omega_l\|^2\right), \notag
    \end{eqnarray}
    where $A_T$ is defined in \eqref{eq:induction_inequality_1_SEG}. To continue our derivation we introduce new notation:
    \begin{gather}
        \theta_l^u \eqdef \EE_{\Bxi_2^l}\left[\tF_{\Bxi_2^l}(\tx^l)\right] - \tF_{\Bxi_2^l}(\tx^l),\quad \theta_l^b \eqdef F(\tx^l) - \EE_{\Bxi_2^l}\left[\tF_{\Bxi_2^l}(\tx^l)\right], \label{eq:gap_thm_SEG_technical_4}\\
        \omega_l^u \eqdef \EE_{\Bxi_1^l}\left[\tF_{\Bxi_1^l}(x^l)\right] - \tF_{\Bxi_1^l}(x^l),\quad \omega_l^b \eqdef F(x^l) - \EE_{\Bxi_1^l}\left[\tF_{\Bxi_1^l}(x^l)\right], \label{eq:gap_thm_SEG_technical_5}
    \end{gather}
    for all $l = 0,\ldots, T-1$. By definition we have $\theta_l = \theta_l^u + \theta_l^b$, $\omega_l = \omega_l^u + \omega_l^b$ for all $l = 0,\ldots, T-1$. Using the introduced notation, we continue our derivation as follows: $E_{T-1}$ implies
    \begin{eqnarray}
        A_T &\overset{\eqref{eq:a_plus_b}}{\leq}& 4R^2 + 2\gamma R \left\|\sum\limits_{l = 0}^{T-1}\theta_l\right\| + \underbrace{2\gamma \sum\limits_{l = 0}^{T-1} \langle \eta_l, \theta_l^u \rangle}_{\circledOne} + \underbrace{2\gamma \sum\limits_{l = 0}^{T-1} \langle \eta_l, \theta_l^b \rangle}_{\circledTwo}\notag\\
        &&\quad + \underbrace{2\gamma^2 \sum\limits_{l=0}^{T-1}\left(\EE_{\Bxi_2^l}\left[\|\theta_l^u\|^2\right] + 2\EE_{\Bxi_1^l}\left[\|\omega_l^u\|^2\right]\right)}_{\circledThree} \notag\\
        &&\quad + \underbrace{2\gamma^2 \sum\limits_{l=0}^{T-1}\left(\|\theta_l^u\|^2 + 2\|\omega_l^u\|^2 - \EE_{\Bxi_2^l}\left[\|\theta_l^u\|^2\right] - 2\EE_{\Bxi_1^l}\left[\|\omega_l^u\|^2\right]\right)}_{\circledFour}\notag\\
        &&\quad + \underbrace{2\gamma^2 \sum\limits_{l=0}^{T-1}\left(\|\theta_l^b\|^2 + 2\|\omega_l^b\|^2\right)}_{\circledFive}\label{eq:gap_thm_SEG_technical_6}
    \end{eqnarray}
    
    The rest of the proof is based on deriving good enough upper bounds for $2\gamma R \left\|\sum_{l = 0}^{T-1}\theta_l\right\|, \circledOne, \circledTwo, \circledThree, \circledFour, \circledFive$, i.e., we want to prove that $2\gamma R \left\|\sum_{l = 0}^{T-1}\theta_l\right\| + \circledOne + \circledTwo + \circledThree + \circledFour + \circledFive \leq 5R^2$ with high probability.
    
    Before we move on, we need to derive some useful inequalities for operating with $\theta_l^u, \theta_l^b, \omega_l^u, \omega_l^b$. First of all, Lemma~\ref{lem:bias_variance} implies that
    \begin{equation}
        \|\theta_l^u\| \leq 2\lambda,\quad \|\omega_l^u\| \leq 2\lambda \label{eq:theta_omega_magnitude}
    \end{equation}
    for all $l = 0,1, \ldots, T-1$. Next, since $\{\xi_1^{i,l}\}_{i=1}^{m}$, $\{\xi_2^{i,l}\}_{i=1}^{m}$ are independently sampled from $\cD$, we have $\EE_{\Bxi_1^l}[F_{\Bxi_1^l}(x^l)] = F(x^l)$, $\EE_{\Bxi_2^l}[F_{\Bxi_2^l}(\tx^l)] = F(\tx^l)$, and 
    \begin{gather}
        \EE_{\Bxi_1^l}\left[\|F_{\Bxi_1^l}(x^l) - F(x^l)\|^2\right] = \frac{1}{m^2}\sum\limits_{i=1}^m \EE_{\xi_1^{i,l}}\left[\|F_{\xi_1^{i,l}}(x^l) - F(x^l)\|^2\right] \overset{\eqref{eq:UBV}}{\leq} \frac{\sigma^2}{m}, \notag\\
        \EE_{\Bxi_2^l}\left[\|F_{\Bxi_2^l}(\tx^l) - F(\tx^l)\|^2\right] = \frac{1}{m^2}\sum\limits_{i=1}^m \EE_{\xi_2^{i,l}}\left[\|F_{\xi_2^{i,l}}(\tx^l) - F(\tx^l)\|^2\right] \overset{\eqref{eq:UBV}}{\leq} \frac{\sigma^2}{m}, \notag
    \end{gather}
    for all $l = 0,1, \ldots, T-1$. Moreover, probability event $E_{T-1}$ implies
    \begin{gather*}
        \|F(x^l)\| \overset{\eqref{eq:Lipschitzness}}{\leq} L\|x^l - x^*\| \leq 3LR \overset{\eqref{eq:gamma_SEG}}{\leq} \frac{R}{40\gamma \ln\tfrac{6(K+1)}{\beta}} \overset{\eqref{eq:lambda_SEG}}{=} \frac{\lambda}{2},\\
        \|F(\tx^l)\| \overset{\eqref{eq:Lipschitzness}}{\leq} L\|\tx^l - x^*\| \overset{\eqref{eq:gap_thm_SEG_technical_1}}{\leq} 4LR \overset{\eqref{eq:gamma_SEG}}{\leq} \frac{R}{40\gamma \ln\tfrac{6(K+1)}{\beta}} \overset{\eqref{eq:lambda_SEG}}{=} \frac{\lambda}{2}
    \end{gather*}
    for all $l = 0,1, \ldots, T-1$. Therefore, in view of Lemma~\ref{lem:bias_variance}, $E_{T-1}$ implies that
    \begin{gather}
        \left\|\theta_l^b\right\| \leq \frac{4\sigma^2}{m\lambda},\quad \left\|\omega_l^b\right\| \leq \frac{4\sigma^2}{m\lambda}, \label{eq:bias_theta_omega}\\
        \EE_{\Bxi_2^l}\left[\left\|\theta_l\right\|^2\right] \leq \frac{18\sigma^2}{m},\quad \EE_{\Bxi_1^l}\left[\left\|\omega_l\right\|^2\right] \leq \frac{18\sigma^2}{m}, \label{eq:distortion_theta_omega}\\
        \EE_{\Bxi_2^l}\left[\left\|\theta_l^u\right\|^2\right] \leq \frac{18\sigma^2}{m},\quad \EE_{\Bxi_1^l}\left[\left\|\omega_l^u\right\|^2\right] \leq \frac{18\sigma^2}{m}, \label{eq:variance_theta_omega}
    \end{gather}
    for all $l = 0,1, \ldots, T-1$.
    
    \paragraph{Upper bound for $\circledOne$.} Since $\EE_{\Bxi_2^l}[\theta_l^u] = 0$, we have
    \begin{equation*}
        \EE_{\Bxi_2^l}\left[2\gamma \langle \eta_l, \theta_l^u \rangle\right] = 0.
    \end{equation*}
    Next, the summands in $\circledOne$ are bounded with probability $1$:
    \begin{equation}
        |2\gamma\langle \eta_l, \theta_l^u \rangle| \leq 2\gamma \|\eta_l\|\cdot \|\theta_l^u\| \overset{\eqref{eq:gap_thm_SEG_technical_3},\eqref{eq:theta_omega_magnitude}}{\leq} 20\gamma R\lambda \overset{\eqref{eq:lambda_SEG}}{=} \frac{R^2}{\ln\tfrac{6(K+1)}{\beta}} \eqdef c. \label{eq:gap_thm_SEG_technical_6_5}
    \end{equation}
    Moreover, these summands have bounded conditional variances $\sigma_l^2 \eqdef \EE_{\Bxi_2^l}\left[4\gamma^2 \langle \eta_l, \theta_l^u \rangle^2\right]$:
    \begin{equation}
        \sigma_l^2 \leq \EE_{\Bxi_2^l}\left[4\gamma^2 \|\eta_l\|^2\cdot \|\theta_l^u\|^2\right] \overset{\eqref{eq:gap_thm_SEG_technical_3}}{\leq} 100\gamma^2 R^2 \EE_{\Bxi_2^l}\left[\|\theta_l^u\|^2\right]. \label{eq:gap_thm_SEG_technical_7}
    \end{equation}
    That is, sequence $\{2\gamma \langle \eta_l, \theta_l^u \rangle\}_{l\geq 0}$ is a bounded martingale difference sequence having bounded conditional variances $\{\sigma_l^2\}_{l \geq 0}$. Applying the Bernstein's inequality (Lemma~\ref{lem:Bernstein_ineq}) with $X_l = 2\gamma \langle \eta_l, \theta_l^u \rangle$, $c$ defined in \eqref{eq:gap_thm_SEG_technical_6_5}, $b = R^2$, $G = \tfrac{R^4}{6\ln\tfrac{6(K+1)}{\beta}}$, we get that
    \begin{equation*}
        \PP\left\{|\circledOne| > R^2 \text{ and } \sum\limits_{l=0}^{T-1}\sigma_l^2 \leq \frac{R^4}{6\ln\tfrac{6(K+1)}{\beta}}\right\} \leq 2\exp\left(- \frac{b^2}{2G + \nicefrac{2cb}{3}}\right) = \frac{\beta}{3(K+1)}.
    \end{equation*}
    In other words, $\PP\{E_{\circledOne}\} \geq 1 - \tfrac{\beta}{3(K+1)}$, where probability event $E_{\circledOne}$ is defined as
    \begin{equation}
        E_{\circledOne} = \left\{\text{either} \quad \sum\limits_{l=0}^{T-1}\sigma_l^2 > \frac{R^4}{6\ln\tfrac{6(K+1)}{\beta}}\quad \text{or}\quad |\circledOne| \leq R^2\right\}. \label{eq:bound_1_gap_SEG}
    \end{equation}
    Moreover, we notice here that probability event $E_{T-1}$ implies that
    \begin{eqnarray}
        \sum\limits_{l=0}^{T-1}\sigma_l^2 \overset{\eqref{eq:gap_thm_SEG_technical_7}}{\leq} 100\gamma^2 R^2 \sum\limits_{l=0}^{T-1} \EE_{\Bxi_2^l}\left[\|\theta_l^u\|^2\right] \overset{\eqref{eq:variance_theta_omega}, T \leq K+1}{\leq} \frac{1800(K+1)\gamma^2 R^2 \sigma^2}{m} \overset{\eqref{eq:batch_SEG}}{\leq} \frac{R^4}{6\ln\tfrac{6(K+1)}{\beta}}. \label{eq:bound_1_variances_gap_SEG}
    \end{eqnarray}
    
    \paragraph{Upper bound for $\circledTwo$.} Probability event $E_{T-1}$ implies
    \begin{eqnarray}
        \circledTwo &\leq& 2\gamma \sum\limits_{l=0}^{T-1}\|\eta_l\| \cdot \|\theta_l^b\| \overset{\eqref{eq:gap_thm_SEG_technical_3}, \eqref{eq:bias_theta_omega}, T \leq K+1}{\leq} \frac{40(K+1)\gamma R\sigma^2}{m\lambda} \notag \\
        &\overset{\eqref{eq:lambda_SEG}}{=}& \frac{40(K+1)\gamma^2 \sigma^2 \ln\tfrac{6(K+1)}{\beta}}{m} \overset{\eqref{eq:batch_SEG}}{\leq} R^2. \label{eq:bound_2_variances_gap_SEG}
    \end{eqnarray}
    
    \paragraph{Upper bound for $\circledThree$.} Probability event $E_{T-1}$ implies
    \begin{eqnarray}
        2\gamma^2 \sum\limits_{l=0}^{T-1}\EE_{\Bxi_2^l}[\|\theta_l^u\|^2] &\overset{\eqref{eq:distortion_theta_omega}, T \leq K+1}{\leq}& \frac{36\gamma^2 (K+1)\sigma^2}{m} \overset{\eqref{eq:batch_SEG}}{\leq} \frac{1}{12} R^2, \label{eq:sum_theta_squared_bound_gap_SEG}\\
        4\gamma^2 \sum\limits_{l=0}^{T-1}\EE_{\Bxi_1^l}[\|\omega_l^u\|^2] &\overset{\eqref{eq:distortion_theta_omega}, T \leq K+1}{\leq}& \frac{72\gamma^2 (K+1)\sigma^2}{m} \overset{\eqref{eq:batch_SEG}}{\leq} \frac{1}{12} R^2, \label{eq:sum_omega_squared_bound_gap_SEG}\\
        \circledThree &\overset{\eqref{eq:sum_theta_squared_bound_gap_SEG}, \eqref{eq:sum_omega_squared_bound_gap_SEG}}{\leq}& \frac{1}{6}R^2. \label{eq:bound_3_variances_gap_SEG}
    \end{eqnarray}
    
    \paragraph{Upper bound for $\circledFour$.} First of all,
    \begin{equation*}
        2\gamma^2\EE_{\Bxi_1^l,\Bxi_2^l}\left[\|\theta_l^u\|^2 + 2\|\omega_l^u\|^2 - \EE_{\Bxi_2^l}\left[\|\theta_l^u\|^2\right] - 2\EE_{\Bxi_1^l}\left[\|\omega_l^u\|^2\right]\right] = 0.
    \end{equation*}
    Next, the summands in $\circledFour$ are bounded with probability $1$:
    \begin{eqnarray}
        2\gamma^2\left|\|\theta_l^u\|^2 + 2\|\omega_l^u\|^2 - \EE_{\Bxi_2^l}\left[\|\theta_l^u\|^2\right] - 2\EE_{\Bxi_1^l}\left[\|\omega_l^u\|^2\right] \right| &\leq& 2\gamma^2 \|\theta_l^u\|^2 + 2\gamma^2\EE_{\Bxi_2^l}\left[\|\theta_l^u\|^2\right]\notag\\
        &&\quad + 4\gamma^2 \|\omega_l^u\|^2 + 4\gamma^2\EE_{\Bxi_1^l}\left[\|\omega_l^u\|^2\right] \notag\\
        &\overset{\eqref{eq:theta_omega_magnitude}}{\leq}& 48\gamma^2 \lambda^2 \notag\\
        &\overset{\eqref{eq:lambda_SEG}}{\leq}& \frac{R^2}{6\ln\tfrac{6(K+1)}{\beta}} \eqdef c.\label{eq:gap_thm_SEG_technical_6_5_1}
    \end{eqnarray}
    Moreover, these summands have bounded conditional variances $\widetilde\sigma_l^2 \eqdef 4\gamma^4\EE_{\Bxi_1^l, \Bxi_2^l}\left[\left|\|\theta_l^u\|^2 + 2\|\omega_l^u\|^2 - \EE_{\Bxi_2^l}\left[\|\theta_l^u\|^2\right] - 2\EE_{\Bxi_1^l}\left[\|\omega_l^u\|^2\right] \right|^2\right]$:
    \begin{eqnarray}
        \widetilde\sigma_l^2 &\overset{\eqref{eq:gap_thm_SEG_technical_6_5_1}}{\leq}& \frac{\gamma^2R^2}{3\ln\tfrac{6(K+1)}{\beta}}\EE_{\Bxi_1^l,\Bxi_2^l}\left[\left|\|\theta_l^u\|^2 + 2\|\omega_l^u\|^2 - \EE_{\Bxi_2^l}\left[\|\theta_l^u\|^2\right] - 2\EE_{\Bxi_1^l}\left[\|\omega_l^u\|^2\right] \right|\right]\notag\\
        &\leq& \frac{2\gamma^2R^2}{3\ln\tfrac{6(K+1)}{\beta}} \EE_{\Bxi_1^l,\Bxi_2^l}\left[\|\theta_l^u\|^2 + 2\|\omega_l^u\|^2 \right]. \label{eq:gap_thm_SEG_technical_7_1}
    \end{eqnarray}
    That is, sequence $\left\{2\gamma^2\left(\|\theta_l^u\|^2 + 2\|\omega_l^u\|^2 - \EE_{\Bxi_2^l}\left[\|\theta_l^u\|^2\right] - 2\EE_{\Bxi_1^l}\left[\|\omega_l^u\|^2\right]\right)\right\}_{l\geq 0}$ is a bounded martingale difference sequence having bounded conditional variances $\{\widetilde\sigma_l^2\}_{l \geq 0}$. Applying the Bernstein's inequality (Lemma~\ref{lem:Bernstein_ineq}) with $X_l = 2\gamma^2\left(\|\theta_l^u\|^2 + 2\|\omega_l^u\|^2 - \EE_{\Bxi_2^l}\left[\|\theta_l^u\|^2\right] - 2\EE_{\Bxi_1^l}\left[\|\omega_l^u\|^2\right]\right)$, $c$ defined in \eqref{eq:gap_thm_SEG_technical_6_5_1}, $b = \tfrac{1}{6}R^2$, $G = \tfrac{R^4}{216\ln\tfrac{6(K+1)}{\beta}}$, we get that
    \begin{equation*}
        \PP\left\{|\circledFour| > \frac{1}{6}R^2 \text{ and } \sum\limits_{l=0}^{T-1}\widetilde\sigma_l^2 \leq \frac{R^4}{216\ln\tfrac{6(K+1)}{\beta}}\right\} \leq 2\exp\left(- \frac{b^2}{2G + \nicefrac{2cb}{3}}\right) = \frac{\beta}{3(K+1)}.
    \end{equation*}
    In other words, $\PP\{E_{\circledFour}\} \geq 1 - \tfrac{\beta}{3(K+1)}$, where probability event $E_{\circledFour}$ is defined as
    \begin{equation}
        E_{\circledFour} = \left\{\text{either} \quad \sum\limits_{l=0}^{T-1}\widetilde\sigma_l^2 > \frac{R^4}{216\ln\tfrac{6(K+1)}{\beta}}\quad \text{or}\quad |\circledFour| \leq \frac{1}{6}R^2\right\}. \label{eq:bound_1_gap_SEG_1}
    \end{equation}
     Moreover, we notice here that probability event $E_{T-1}$ implies that
    \begin{eqnarray}
        \sum\limits_{l=0}^{T-1}\widetilde\sigma_l^2 &\overset{\eqref{eq:gap_thm_SEG_technical_7_1}}{\leq}& \frac{2\gamma^2R^2}{3\ln\tfrac{6(K+1)}{\beta}} \sum\limits_{l=0}^{T-1} \EE_{\Bxi_1^l,\Bxi_2^l}\left[\|\theta_l^u\|^2 + 2\|\omega_l^u\|^2 \right]\notag\\
        &\overset{\eqref{eq:variance_theta_omega}, T \leq K+1}{\leq}& \frac{36(K+1)\gamma^2 R^2 \sigma^2}{m} \overset{\eqref{eq:batch_SEG}}{\leq} \frac{R^4}{216\ln\tfrac{6(K+1)}{\beta}}. \label{eq:bound_1_variances_gap_SEG_1}
    \end{eqnarray}
    
    \paragraph{Upper bound for $\circledFive$.} Probability event $E_{T-1}$ implies
    \begin{eqnarray}
        \circledFive &=& 2\gamma^2 \sum\limits_{l=0}^{T-1}\left(\|\theta_l^b\|^2 + 2\|\omega_l^b\|^2\right) \overset{\eqref{eq:bias_theta_omega}, T\leq K+1}{\leq} \frac{96\gamma^2\sigma^4 (K+1)}{m^2\lambda^2}  \notag\\
        &\overset{\eqref{eq:lambda_SEG}}{=}&  \frac{38400 \gamma^4 \sigma^4 (K+1) \ln^2\tfrac{6(K+1)}{\beta}}{m^2 R^2} \overset{\eqref{eq:batch_SEG}}{\leq} \frac{1}{6}R^2. \label{eq:nsjcbdjhcbfjdhfbj}
    \end{eqnarray}

    \paragraph{Upper bound for $2\gamma R \left\|\sum_{l=0}^{T-1} \theta_l\right\|$.} To handle this term, we introduce new notation:
    \begin{equation*}
        \zeta_l = \begin{cases} \gamma \sum\limits_{r=0}^{l-1}\theta_r,& \text{if } \left\|\gamma \sum\limits_{r=0}^{l-1}\theta_r\right\| \leq R,\\ 0, & \text{otherwise} \end{cases}
    \end{equation*}
    for $l = 1, 2, \ldots, T-1$. By definition, we have
    \begin{equation}
        \|\zeta_l\| \leq R.  \label{eq:gap_thm_SEG_technical_8}
    \end{equation}
    Therefore, in view of \eqref{eq:induction_inequality_2_SEG}, probability event $E_{T-1}$ implies
    \begin{eqnarray}
        2\gamma R\left\|\sum\limits_{l = 0}^{T-1}\theta_l\right\| &=& 2R\sqrt{\gamma^2\left\|\sum\limits_{l = 0}^{T-1}\theta_l\right\|^2}\notag\\
        &=& 2R\sqrt{\gamma^2\sum\limits_{l=0}^{T-1}\|\theta_l\|^2 + 2\gamma\sum\limits_{l=0}^{T-1}\left\langle \gamma\sum\limits_{r=0}^{l-1} \theta_r, \theta_l \right\rangle} \notag\\
        &=& 2R \sqrt{\gamma^2\sum\limits_{l=0}^{T-1}\|\theta_l\|^2 + 2\gamma \sum\limits_{l=0}^{T-1} \langle \zeta_l, \theta_l\rangle} \notag\\
        &\overset{\eqref{eq:gap_thm_SEG_technical_4}}{\leq}& 2R \sqrt{\circledThree + \circledFour + \circledFive + \underbrace{2\gamma \sum\limits_{l=0}^{T-1} \langle \zeta_l, \theta_l^u\rangle}_{\circledSix} + \underbrace{2\gamma \sum\limits_{l=0}^{T-1} \langle \zeta_l, \theta_l^b}_{\circledSeven}\rangle}. \label{eq:norm_sum_theta_bound_gap_SEG}
    \end{eqnarray}
    Following similar steps as before, we bound $\circledSix$ and $\circledSeven$.
    
    \paragraph{Upper bound for $\circledSix$.} Since $\EE_{\Bxi_2^l}[\theta_l^u] = 0$, we have
    \begin{equation*}
        \EE_{\Bxi_2^l}\left[2\gamma \langle \zeta_l, \theta_l^u \rangle\right] = 0.
    \end{equation*}
    Next, the summands in $\circledFour$ are bounded with probability $1$:
    \begin{equation}
        |2\gamma\langle \zeta_l, \theta_l^u \rangle| \leq 2\gamma \|\eta_l\|\cdot \|\theta_l^u\| \overset{\eqref{eq:gap_thm_SEG_technical_8},\eqref{eq:theta_omega_magnitude}}{\leq} 4\gamma R\lambda \overset{\eqref{eq:lambda_SEG}}{\leq} \frac{R^2}{4\ln\tfrac{6(K+1)}{\beta}} \eqdef c. \label{eq:gap_thm_SEG_technical_8_5}
    \end{equation}
    Moreover, these summands have bounded conditional variances $\hat\sigma_l^2 \eqdef \EE_{\Bxi_2^l}\left[4\gamma^2 \langle \zeta_l, \theta_l^u \rangle^2\right]$:
    \begin{equation}
        \hat\sigma_l^2 \leq \EE_{\Bxi_2^l}\left[4\gamma^2 \|\zeta_l\|^2\cdot \|\theta_l^u\|^2\right] \overset{\eqref{eq:gap_thm_SEG_technical_8}}{\leq} 4\gamma^2 R^2 \EE_{\Bxi_2^l}\left[\|\theta_l^u\|^2\right]. \label{eq:gap_thm_SEG_technical_9}
    \end{equation}
    That is, sequence $\{2\gamma \langle \zeta_l, \theta_l^u \rangle\}_{l\geq 0}$ is a bounded martingale difference sequence having bounded conditional variances $\{\hat\sigma_l^2\}_{l \geq 0}$. Applying Bernstein's inequality (Lemma~\ref{lem:Bernstein_ineq}) with $X_l = 2\gamma \langle \zeta_l, \theta_l^u \rangle$, $c$ defined in \eqref{eq:gap_thm_SEG_technical_6_5}, $b = \tfrac{R^2}{4}$, $G = \tfrac{R^4}{96\ln\tfrac{6(K+1)}{\beta}}$, we get that
    \begin{equation*}
        \PP\left\{|\circledFive| > \frac{1}{4}R^2 \text{ and } \sum\limits_{l=0}^{T-1}\hat\sigma_l^2 \leq \frac{R^4}{96\ln\tfrac{4(K+1)}{\beta}}\right\} \leq 2\exp\left(- \frac{b^2}{2G + \nicefrac{2cb}{3}}\right) = \frac{\beta}{3(K+1)}.
    \end{equation*}
    In other words, $\PP\{E_{\circledSix}\} \geq 1 - \tfrac{\beta}{3(K+1)}$, where probability event $E_{\circledSix}$ is defined as
    \begin{equation}
        E_{\circledSix} = \left\{ \text{either} \quad \sum\limits_{l=0}^{T-1}\hat\sigma_l^2 > \frac{R^4}{96\ln\tfrac{6(K+1)}{\beta}}\quad \text{or}\quad |\circledFive| \leq \frac{1}{4}R^2\right\}. \label{eq:bound_4_gap_SEG}
    \end{equation}
    Moreover, we notice here that probability event $E_{T-1}$ implies that
    \begin{eqnarray}
        \sum\limits_{l=0}^{T-1}\hat\sigma_l^2 \overset{\eqref{eq:gap_thm_SEG_technical_9}}{\leq} 4\gamma^2 R^2 \sum\limits_{l=0}^{T-1} \EE_{\Bxi_2^l}\left[\|\theta_l^u\|^2\right] \overset{\eqref{eq:variance_theta_omega}, T \leq K+1}{\leq} \frac{72(K+1)\gamma^2 R^2 \sigma^2}{m} \overset{\eqref{eq:batch_SEG}}{\leq} \frac{R^4}{96\ln\tfrac{6(K+1)}{\beta}}. \label{eq:bound_4_variances_gap_SEG}
    \end{eqnarray}
    
    \paragraph{Upper bound for $\circledSeven$.} Probability event $E_{T-1}$ implies
    \begin{eqnarray}
        \circledSeven &\leq& 2\gamma \sum\limits_{l=0}^{T-1}\|\zeta_l\| \cdot \|\theta_l^b\| \overset{\eqref{eq:gap_thm_SEG_technical_8}, \eqref{eq:bias_theta_omega}, T \leq K+1}{\leq} \frac{8(K+1)\gamma R\sigma^2}{m\lambda} \notag \\
        &\overset{\eqref{eq:lambda_SEG}}{=}& \frac{160(K+1)\gamma^2 \sigma^2 \ln\tfrac{6(K+1)}{\beta}}{m} \overset{\eqref{eq:batch_SEG}}{\leq} \frac{1}{4}R^2. \label{eq:bound_5_variances_gap_SEG}
    \end{eqnarray}
    
    \paragraph{Final derivation.} Putting all bounds together, we get that $E_{T-1}$ implies
    \begin{gather*}
        A_T \overset{\eqref{eq:gap_thm_SEG_technical_6}}{\leq} 4R^2 + 2\gamma R\left\|\sum\limits_{l=0}^{T-1} \theta_l\right\| + \circledOne + \circledTwo + \circledThree + \circledFour + \circledFive,\\
        2\gamma R\left\|\sum\limits_{l=0}^{T-1} \theta_l\right\| \overset{\eqref{eq:norm_sum_theta_bound_gap_SEG}}{\leq} 2R \sqrt{\circledThree + \circledFour + \circledFive + \circledSix + \circledSeven},\\
        \circledTwo \overset{\eqref{eq:bound_2_variances_gap_SEG}}{\leq} R^2,\quad \circledThree \overset{\eqref{eq:bound_3_variances_gap_SEG}}{\leq} \frac{1}{6}R^2,\quad \circledFive \overset{\eqref{eq:nsjcbdjhcbfjdhfbj}}{\leq} \frac{1}{6}R^2,\quad \circledSeven \overset{\eqref{eq:bound_5_variances_gap_SEG}}{\leq} \frac{1}{4}R^2,\\
        \sum\limits_{l=0}^{T-1}\sigma_l^2 \overset{\eqref{eq:bound_1_variances_gap_SEG}}{\leq}  \frac{R^4}{6\ln\tfrac{6(K+1)}{\beta}},\quad \sum\limits_{l=0}^{T-1}\widetilde\sigma_l^2 \overset{\eqref{eq:bound_1_variances_gap_SEG_1}}{\leq} \frac{R^4}{216\ln\tfrac{6(K+1)}{\beta}},\quad \sum\limits_{l=0}^{T-1}\hat\sigma_l^2 \overset{\eqref{eq:bound_4_variances_gap_SEG}}{\leq}  \frac{R^4}{96\ln\tfrac{6(K+1)}{\beta}}.
    \end{gather*}
    Moreover, in view of \eqref{eq:bound_1_gap_SEG}, \eqref{eq:bound_1_gap_SEG_1}, \eqref{eq:bound_4_gap_SEG}, and our induction assumption, we have
    \begin{gather*}
        \PP\{E_{T-1}\} \geq 1 - \frac{(T-1)\beta}{K+1},\\
        \PP\{E_{\circledOne}\} \geq 1 - \frac{\beta}{3(K+1)}, \quad \PP\{E_{\circledFour}\} \geq 1 - \frac{\beta}{3(K+1)}, \quad \PP\{E_{\circledSix}\} \geq 1 - \frac{\beta}{3(K+1)} ,
    \end{gather*}
    where probability events $E_{\circledOne}$, $E_{\circledFour}$, and $E_{\circledSix}$ are defined as
    \begin{eqnarray}
        E_{\circledOne}&=& \left\{\text{either} \quad \sum\limits_{l=0}^{T-1}\sigma_l^2 > \frac{R^4}{6\ln\tfrac{6(K+1)}{\beta}}\quad \text{or}\quad |\circledOne| \leq R^2\right\},\notag\\
        E_{\circledFour}&=& \left\{\text{either} \quad \sum\limits_{l=0}^{T-1}\widetilde\sigma_l^2 > \frac{R^4}{216\ln\tfrac{6(K+1)}{\beta}}\quad \text{or}\quad |\circledFour| \leq \frac{1}{6}R^2\right\},\notag\\
        E_{\circledSix}&=& \left\{\text{either} \quad \sum\limits_{l=0}^{T-1}\hat\sigma_l^2 > \frac{R^4}{96\ln\tfrac{6(K+1)}{\beta}}\quad \text{or}\quad |\circledSix| \leq \frac{1}{4}R^2\right\}.\notag
    \end{eqnarray}
    Putting all of these inequalities together, we obtain that probability event $E_{T-1} \cap E_{\circledOne} \cap E_{\circledFour} \cap E_{\circledSix}$ implies
    \begin{eqnarray}
        \left\|\gamma\sum\limits_{l=0}^{T-1} \theta_l\right\| &\leq& \sqrt{\frac{1}{6}R^2 + \frac{1}{6}R^2 + \frac{1}{6}R^2 + \frac{1}{4}R^2 + \frac{1}{4}R^2} = R, \label{eq:gap_thm_SEG_technical_10} \\
        A_T &\leq& 4R^2 + 2R\sqrt{\frac{1}{6}R^2 + \frac{1}{6}R^2 + \frac{1}{6}R^2 + \frac{1}{4}R^2 + \frac{1}{4}R^2}\notag\\
        &&\quad + R^2 + R^2 + \frac{1}{6}R^2 + \frac{1}{6}R^2 + \frac{1}{6}R^2\notag\\
        &\leq& 9R^2. \label{eq:gap_thm_SEG_technical_11}
    \end{eqnarray}
    Moreover, union bound for the probability events implies
    \begin{equation}
        \PP\{E_T\} \geq \PP\{E_{T-1} \cap E_{\circledOne} \cap E_{\circledFour} \cap E_{\circledSix}\} = 1 - \PP\{\overline{E}_{T-1} \cup \overline{E}_{\circledOne} \cup \overline{E}_{\circledFour} \cup \overline{E}_{\circledSix}\} \geq 1 - \frac{T\beta}{K+1}.
    \end{equation}
    
    This is exactly what we wanted to prove (see the paragraph after inequalities \eqref{eq:induction_inequality_1_SEG}, \eqref{eq:induction_inequality_2_SEG}). Therefore, for all $k = 0, 1, \ldots, K+1$ we have $\PP\{E_k\} \geq 1 - \nicefrac{k\beta}{(K+1)}$., i.e., for $k = K+1$ we have that with probability at least $1 - \beta$ inequality
    \begin{eqnarray*}
        \gap_{R}(\tx^{K}_{\avg}) &=& \max\limits_{u \in B_{R}(x^*)}\left\{\langle F(u), \tx^{K}_{\avg} - u\rangle\right\}\\
        &\leq& \frac{1}{2\gamma(K + 1)}\max\limits_{u \in B_{R}(x^*)}\left\{2\gamma(K + 1)\langle F(u), \tx^{t}_{\avg} - u\rangle + \|x^{K+1} - u\|^2\right\}\\
        &\overset{\eqref{eq:gap_thm_SEG_technical_1_5}}{\leq}& \frac{9R^2}{2\gamma(K+1)}
    \end{eqnarray*}
    holds. This concludes the proof.
\end{proof}

\begin{corollary}\label{cor:main_result_gap_SEG}
    Let the assumptions of Theorem~\ref{thm:main_result_gap_SEG} hold. Then, the following statements hold.
    \begin{enumerate}
        \item \textbf{Large stepsize/large batch.} The choice of stepsize and batchsize
        \begin{equation}
            \gamma = \frac{1}{160L \ln \tfrac{6(K+1)}{\beta}},\quad m = \max\left\{1, \frac{27 (K+1) \sigma^2}{64L^2R^2 \ln\tfrac{6(K+1)}{\beta}}\right\} \label{eq:gap_SEG_large_step_large_batch}
        \end{equation}
        satisfies conditions \eqref{eq:gamma_SEG} and \eqref{eq:batch_SEG}. With such choice of $\gamma, m$, and the choice of $\lambda$ as in \eqref{eq:lambda_SEG}, the iterates produced by \ref{eq:clipped_SEG} after $K$ iterations with probability at least $1-\beta$ satisfy
        \begin{equation}
            \gap_R(\tx_{\avg}^K) \leq \frac{720 LR^2 \ln \tfrac{6(K+1)}{\beta}}{K+1}. \label{eq:main_result_gap_SEG_large_batch}
        \end{equation}
        In particular, to guarantee $\gap_R(\tx_{\avg}^K) \leq \varepsilon$ with probability at least $1-\beta$ for some $\varepsilon > 0$ \ref{eq:clipped_SEG} requires,
        \begin{gather}
            \cO\left(\frac{LR^2}{\varepsilon} \ln\left(\frac{LR^2}{\varepsilon\beta}\right)\right) \text{ iterations}, \label{eq:gap_SEG_iteration_complexity_large_batch}\\
            \cO\left(\max\left\{\frac{LR^2}{\varepsilon}, \frac{\sigma^2R^2}{\varepsilon^2}\right\} \ln\left(\frac{LR^2}{\varepsilon\beta}\right)\right) \text{ oracle calls}. \label{eq:gap_SEG_oracle_complexity_large_batch} 
        \end{gather}
        
        \item \textbf{Small stepsize/small batch.} The choice of stepsize and batchsize
        \begin{equation}
            \gamma = \min\left\{\frac{1}{160L \ln \tfrac{6(K+1)}{\beta}}, \frac{R}{60\sigma \sqrt{3(K+1) \ln\tfrac{6(K+1)}{\beta}}}\right\},\quad m = 1 \label{eq:gap_SEG_small_step_small_batch}
        \end{equation}
        satisfies conditions \eqref{eq:gamma_SEG} and \eqref{eq:batch_SEG}. With such choice of $\gamma, m$, and the choice of $\lambda$ as in \eqref{eq:lambda_SEG}, the iterates produced by \ref{eq:clipped_SEG} after $K$ iterations with probability at least $1-\beta$ satisfy
        \begin{equation}
            \gap_R(\tx_{\avg}^K) \leq \max\left\{\frac{720 LR^2 \ln \tfrac{6(K+1)}{\beta}}{K+1}, \frac{270\sigma R \sqrt{\ln\tfrac{6(K+1)}{\beta}}}{\sqrt{K+1}}\right\}. \label{eq:main_result_gap_SEG_small_batch}
        \end{equation}
        In particular, to guarantee $\gap_R(\tx_{\avg}^K) \leq \varepsilon$ with probability at least $1-\beta$ for some $\varepsilon > 0$, \ref{eq:clipped_SEG} requires
        \begin{equation}
            \cO\left(\max\left\{\frac{LR^2}{\varepsilon} \ln\left(\frac{LR^2}{\varepsilon\beta}\right), \frac{\sigma^2R^2}{\varepsilon^2}\ln\left(\frac{\sigma^2R^2}{\varepsilon^2\beta}\right)\right\}\right) \text{ iterations/oracle calls}. \label{eq:gap_SEG_iteration_oracle_complexity_small_batch} 
        \end{equation}
    \end{enumerate}
\end{corollary}
\begin{proof}
    \begin{enumerate}
        \item \textbf{Large stepsize/large batch.} First of all, we verify that the choice of $\gamma$ and $m$ from \eqref{eq:gap_SEG_large_step_large_batch} satisfies conditions \eqref{eq:gamma_SEG} and \eqref{eq:batch_SEG}: \eqref{eq:gamma_SEG} trivially holds and \eqref{eq:batch_SEG} holds since
        \begin{equation*}
            m = \max\left\{1, \frac{27 (K+1) \sigma^2}{64L^2R^2 \ln\tfrac{6(K+1)}{\beta}}\right\} = \max\left\{1, \frac{10800 (K+1) \gamma^2\sigma^2 \ln\tfrac{6(K+1)}{\beta}}{R^2}\right\}.
        \end{equation*}
        Therefore, applying Theorem~\ref{thm:main_result_gap_SEG}, we derive that with probability at least $1-\beta$
        \begin{equation*}
            \gap_R(\tx_{\avg}^K) \leq \frac{9R^2}{2\gamma(K+1)} \overset{\eqref{eq:gap_SEG_large_step_large_batch}}{=} \frac{720 LR^2 \ln \tfrac{4(K+1)}{\beta}}{K+1}.
        \end{equation*}
        To guarantee $\gap_R(\tx_{\avg}^K) \leq \varepsilon$, we choose $K$ in such a way that the right-hand side of the above inequality is smaller than $\varepsilon$ that gives
        \begin{eqnarray*}
            K = \cO\left(\frac{LR^2}{\varepsilon} \ln\left(\frac{LR^2}{\varepsilon\beta}\right)\right).
        \end{eqnarray*}
        The total number of oracle calls equals
        \begin{eqnarray*}
            2m(K+1) &\overset{\eqref{eq:gap_SEG_large_step_large_batch}}{=}&  2\max\left\{K+1, \frac{27 (K+1)^2 \sigma^2}{64L^2R^2 \ln\tfrac{6(K+1)}{\beta}}\right\}\\
            &=& \cO\left(\max\left\{\frac{LR^2}{\varepsilon}, \frac{\sigma^2R^2}{\varepsilon^2}\right\} \ln\left(\frac{LR^2}{\varepsilon\beta}\right)\right).
        \end{eqnarray*}
        
        \item \textbf{Small stepsize/small batch.} First of all, we verify that the choice of $\gamma$ and $m$ from \eqref{eq:gap_SEG_large_step_large_batch} satisfies conditions \eqref{eq:gamma_SEG} and \eqref{eq:batch_SEG}:
        \begin{eqnarray*}
            \gamma &=& \min\left\{\frac{1}{160L \ln \tfrac{6(K+1)}{\beta}}, \frac{R}{60\sigma \sqrt{3(K+1) \ln\tfrac{6(K+1)}{\beta}}}\right\} \leq \frac{1}{160L \ln \tfrac{6(K+1)}{\beta}},\\
            m &=& 1 \overset{\eqref{eq:gap_SEG_small_step_small_batch}}{\geq} \frac{10800 (K+1) \gamma^2\sigma^2 \ln\tfrac{6(K+1)}{\beta}}{R^2}.
        \end{eqnarray*}
        Therefore, applying Theorem~\ref{thm:main_result_gap_SEG}, we derive that with probability at least $1-\beta$
        \begin{eqnarray*}
            \gap_R(\tx_{\avg}^K) &\leq& \frac{9R^2}{2\gamma(K+1)}\\
            &\overset{\eqref{eq:gap_SEG_small_step_small_batch}}{=}& \max\left\{\frac{720 LR^2 \ln \tfrac{6(K+1)}{\beta}}{K+1}, \frac{270\sigma R \sqrt{\ln\tfrac{6(K+1)}{\beta}}}{\sqrt{K+1}}\right\}.
        \end{eqnarray*}
        To guarantee $\gap_R(\tx_{\avg}^K) \leq \varepsilon$, we choose $K$ in such a way that the right-hand side of the above inequality is smaller than $\varepsilon$ that gives
        \begin{eqnarray*}
            K = \cO\left(\max\left\{\frac{LR^2}{\varepsilon} \ln\left(\frac{LR^2}{\varepsilon\beta}\right), \frac{\sigma^2R^2}{\varepsilon^2}\ln\left(\frac{\sigma^2R^2}{\varepsilon^2\beta}\right)\right\}\right).
        \end{eqnarray*}
        The total number of oracle calls equals $2m(K+1) = 2(K+1)$.
    \end{enumerate}
\end{proof}

\subsection{Star-Negative Comonotone Case}

\begin{lemma}\label{lem:optimization_lemma_SEG}
    Let Assumptions~\ref{as:lipschitzness}, \ref{as:negative_comon} hold for $Q = B_{3R}(x^*) = \{x\in\R^d\mid \|x - x^*\| \leq 3R\}$, where $R \geq R_0 \eqdef \|x^0 - x^*\|$, and $\gamma_2 + 2\rho < \gamma_1 \leq \nicefrac{1}{(2L)}$. If $x^k$ and $\tx^k$ lie in $B_{3R_0}(x^*)$ for all $k = 0,1,\ldots, K$ for some $K\geq 0$, then the iterates produced by \ref{eq:clipped_SEG} satisfy
    \begin{eqnarray}
        \frac{\gamma_1\gamma_2}{4(K+1)}\sum\limits_{k=0}^K\|F(x^k)\|^2 &\leq& \frac{\|x^0 - x^*\|^2 - \|x^{K+1} - x^*\|^2}{K+1}\notag\\
        &&\quad + \frac{1}{K+1}\sum\limits_{k=0}^K\left(\gamma_2^2\|\omega_k\|^2 + 3\gamma_1\gamma_2\|\omega_k\|^2\right)\notag\\
        &&\quad + \frac{2\gamma_2}{K+1}\sum\limits_{k=0}^K\langle x^k - x^* - \gamma_2 F(\tx^k), \theta_k\rangle\label{eq:optimization_lemma_norm_SEG}
    \end{eqnarray}
    where $\theta_k,\omega_k$ are defined in \eqref{eq:theta_k_SEG}, \eqref{eq:omega_k_SEG}.
\end{lemma}
\begin{proof}
Using the update rule of \ref{eq:clipped_SEG}, we obtain
    \begin{eqnarray*}
        \|x^{k+1} - x^*\|^2 &=& \|x^k - x^*\|^2 - 2\gamma_2 \langle x^k - x^*,  \tF_{\Bxi_2^k}(\tx^k)\rangle + \gamma_2^2\|\tF_{\Bxi_2^k}(\tx^k)\|^2\\
        &=& \|x^k - x^*\|^2 -2\gamma_2 \langle x^k - x^*, F(\tx^k) \rangle + 2\gamma_2 \langle x^k - x^*, \theta_k \rangle\\
        &&\quad + \gamma_2^2\|F(\tx^k)\|^2 - 2\gamma_2^2 \langle F(\tx^k), \theta_k \rangle + \gamma_2^2\|\theta_k\|^2\\
        &=& \|x^k - x^*\|^2 -2\gamma_2 \langle \tx^k - x^*, F(\tx^k) \rangle - 2\gamma_2 \langle x^k - \tx^k, F(\tx^k) \rangle\\
        &&\quad + 2\gamma_2 \langle x^k - x^* - \gamma_2 F(\tx^k), \theta_k \rangle + \gamma_2^2\|F(\tx^k)\|^2 + \gamma^2\|\theta_k\|^2\\
        &\overset{\eqref{eq:negative_comon}}{\leq}& \|x^k - x^*\|^2 + 2\gamma_2\rho\|\tF(\tx^k)\|^2 - 2\gamma_1\gamma_2\langle \tF_{\Bxi_1^k}(x^k), F(\tx^k) \rangle\\
        &&\quad + 2\gamma_2 \langle x^k - x^* - \gamma_2 F(\tx^k), \theta_k \rangle + \gamma_2^2\|F(\tx^k)\|^2 + \gamma_2^2\|\theta_k\|^2\\
        &\overset{\eqref{eq:inner_product_representation}}{=}& \|x^k - x^*\|^2 + \gamma_1\gamma_2\| \tF_{\Bxi_1^k}(x^k) - F(\tx^k) \|^2 - \gamma_1\gamma_2 \|F(\tx^k)\|^2 - \gamma_1\gamma_2 \|\tF_{\Bxi_1^k}(x^k)\|^2\\
        &&\quad + 2\gamma_2 \langle x^k - x^* - \gamma_2 F(\tx^k), \theta_k \rangle + \gamma_2\left(2\rho + \gamma_2\right)\|F(\tx^k)\|^2 + \gamma_2^2\|\theta_k\|^2\\
        &\overset{\eqref{eq:a_plus_b}}{\leq}& \|x^k - x^*\|^2 + 2\gamma_1\gamma_2\|\omega_k\|^2 + 2\gamma_1\gamma_2\|F(x^k) - F(\tx^k)\|^2 - \gamma_1\gamma_2 \|\tF_{\Bxi_1^k}(x^k)\|^2\\
        &&\quad + 2\gamma_2 \langle x^k - x^* - \gamma_2 F(\tx^k), \theta_k \rangle + \gamma_2\left(2\rho + \gamma_2 - \gamma_1\right)\|F(\tx^k)\|^2 + \gamma_2^2\|\theta_k\|^2\\
        &\overset{\eqref{eq:Lipschitzness}}{\leq}& \|x^k - x^*\|^2 - \gamma_1\gamma_2\left(1 - 2\gamma_1^2L^2\right)\|\tF_{\Bxi_1^k}(x^k)\|^2 \\
        &&\quad + 2\gamma_2 \langle x^k - x^* - \gamma_2 F(\tx^k), \theta_k \rangle + \gamma_2^2\|\theta_k\|^2 + 2\gamma_1\gamma_2\|\omega_k\|^2,
    \end{eqnarray*}
    where in the last step we additionally use $x^k - \tx^k = \gamma_1 \tF_{\Bxi_1^k}(x^k)$ after the application of Lipschitzness of $F$ and we use our assumption on $\gamma_1, \gamma_2, \rho$: $\gamma_2 + 2\rho \leq \gamma_1$. Since $\gamma_1 \leq \nicefrac{1}{(2L)}$, we have $\gamma_1\gamma_2\left(1 - 2\gamma_1^2L^2\right)\|\tF_{\Bxi_1^k}(x^k)\|^2 \geq 0$ and, using \eqref{eq:a_minus_b} with $\alpha = 1$, we derive
    \begin{eqnarray*}
        \|x^{k+1} - x^*\|^2 &\leq& \|x^k - x^*\|^2 - \frac{\gamma_1\gamma_2}{2}\left(1 - 2\gamma_1^2L^2\right)\|F(x^k)\|^2 \\
        &&\quad + 2\gamma_2 \langle x^k - x^* - \gamma_2 F(\tx^k), \theta_k \rangle + \gamma_2^2\|\theta_k\|^2 + 2\gamma_1\gamma_2\left(\frac{3}{2} - \gamma_1^2 L^2\right)\|\omega_k\|^2.
    \end{eqnarray*}
    Rearranging the terms and using $\tfrac{3}{2} - \gamma_1^2 L^2 \leq \tfrac{3}{2}$, $1 - 2\gamma_1^2 L^2 \geq \nicefrac{1}{2}$, we derive
    \begin{eqnarray*}
        \frac{\gamma_1\gamma_2}{4}\|F(x^k)\|^2 &\leq& \|x^k - x^*\|^2 - \|x^{k+1} - x^*\|^2  + \left(\gamma_2^2\|\theta_k\|^2 + 3\gamma_1\gamma_2\|\omega_k\|^2\right)\\
        &&\quad + 2\gamma_2 \langle x^k - x^* - 2\gamma_2 F(\tx^k),  \theta_k\rangle.
    \end{eqnarray*}
    Finally, we sum up the above inequality for $k = 0,1,\ldots, K$ and divide both sides of the result by $(K+1)$:
    \begin{eqnarray*}
         \frac{\gamma_1\gamma_2}{4(K+1)}\sum\limits_{k=0}^K\|F(x^k)\|^2 &\leq& \frac{1}{K+1}\sum\limits_{k=0}^K\left(\|x^k - x^*\|^2 - \|x^{k+1} - x^*\|^2\right) + \frac{\gamma_2^2}{K+1}\sum\limits_{k=0}^K\|\theta_k\|^2\\
        &&\quad + \frac{2\gamma_2}{K+1}\sum\limits_{k=0}^K\langle x^k - x^* - \gamma_2 F(\tx^k), \theta_k\rangle + \frac{3\gamma_1\gamma_2}{K+1}\sum\limits_{k=0}^K\|\omega_k\|^2\\
        &=& \frac{\|x^0 - x^*\|^2 - \|x^{K+1} - x^*\|^2}{K+1}\\
        &&\quad + \frac{1}{K+1}\sum\limits_{k=0}^K\left(\gamma_2^2\|\theta_k\|^2 + 3\gamma_1\gamma_2\|\omega_k\|^2\right)\\
        &&\quad + \frac{2\gamma_2}{K+1}\sum\limits_{k=0}^K\langle x^k - x^* - \gamma_2 F(\tx^k), \theta_k\rangle.
    \end{eqnarray*}
    This finishes the proof.
\end{proof}

\begin{theorem}\label{thm:main_result_avg_sq_norm_SEG}
    Let Assumptions~\ref{as:UBV}, \ref{as:lipschitzness}, \ref{as:negative_comon} hold for $Q = B_{3R}(x^*)$, where $R \geq R_0 \eqdef \|x^0 - x^*\|$, and
    \begin{gather}
        \gamma_2 + 2\rho \leq \gamma_1 \leq \frac{1}{160L \ln \tfrac{6(K+1)}{\beta}}, \label{eq:gamma_SEG_neg_mon}\\
        \lambda_{1,k} \equiv \lambda_1 = \frac{R}{20\gamma_1 \ln \tfrac{6(K+1)}{\beta}},\quad \lambda_{1,k} \equiv \lambda_2 = \frac{R}{20\gamma_2 \ln \tfrac{6(K+1)}{\beta}},  \label{eq:lambda_SEG_neg_mon}\\
        m_{1,k} \equiv m_1 \geq \max\left\{1, \frac{216 \max\{\gamma_1\gamma_2(K+1), \sqrt{\gamma_1^3\gamma_2(K+1)}\ln\tfrac{6(K+1)}{\beta}\}\sigma^2}{R^2}\right\}, \label{eq:batch_1_SEG_neg_mon}\\
        m_{2,k} \equiv m_2 \geq \max\left\{1, \frac{3240 (K+1) \gamma_2^2\sigma^2 \ln\tfrac{6(K+1)}{\beta}}{R^2}\right\}, \label{eq:batch_2_SEG_neg_mon}
    \end{gather}
    for some $K \geq 0$ and $\beta \in (0,1]$ such that $\ln \tfrac{6(K+1)}{\beta} \geq 1$. Then, after $K$ iterations the iterates produced by \ref{eq:clipped_SEG} with probability at least $1 - \beta$ satisfy 
    \begin{equation}
        \frac{1}{K+1}\sum\limits_{k=0}^{K}\|F(x^k)\|^2 \leq \frac{36 R^2}{\gamma_1 \gamma_2 (K+1)}. \label{eq:main_result_neg_mon}
    \end{equation}
\end{theorem}
\begin{proof}
    As in the proof of Theorem~\ref{thm:main_result_gap_SEG}, we use the following notation: $R_k = \|x^k - x^*\|^2$, $k \geq 0$. We will derive \eqref{eq:main_result_neg_mon} by induction. In particular, for each $k = 0,\ldots, K+1$ we define probability event $E_k$ as follows: inequalities
    \begin{equation}
        R_t^2 \leq 4R^2 \label{eq:induction_inequality_neg_mon_SEG}
    \end{equation}
    hold for $t = 0,1,\ldots,k$ simultaneously. Our goal is to prove that $\PP\{E_k\} \geq  1 - \nicefrac{k\beta}{(K+1)}$ for all $k = 0,1,\ldots,K+1$. We use the induction to show this statement. For $k = 0$ the statement is trivial since $R_0^2 \leq 4R^2$ by definition. Next, assume that the statement holds for $k = T-1 \leq K$, i.e., we have $\PP\{E_{T-1}\} \geq 1 - \nicefrac{(T-1)\beta}{(K+1)}$. We need to prove that $\PP\{E_T\} \geq 1 - \nicefrac{T\beta}{(K+1)}$. First of all, since $R_t^2 \leq 4 R^2$, we have $x^t \in B_{2R}(x^*)$. Operator $F$ is $L$-Lipschitz on $B_{3R}(x^*)$. Therefore, probability event $E_{T-1}$ implies
    \begin{eqnarray}
        \|F(x^t)\| &\leq& L\|x^t - x^*\| \overset{\eqref{eq:induction_inequality_neg_mon_SEG}}{\leq} 2LR \overset{\eqref{eq:gamma_SEG_neg_mon},\eqref{eq:lambda_SEG_neg_mon}}{\leq} \frac{\lambda_1}{2}. \label{eq:operator_bound_x_t_SEG_neg_mon}
    \end{eqnarray}
    and
    \begin{eqnarray}
        \|\omega_t\|^2 &\overset{\eqref{eq:a_plus_b}}{\leq}& 2\|\tF_{\Bxi_1}(x^t)\|^2 + 2\|F(x^t)\|^2 \overset{\eqref{eq:operator_bound_x_t_SEG_neg_mon}}{\leq} \frac{5}{2}\lambda_1^2 \overset{\eqref{eq:lambda_SEG_neg_mon}}{\leq} \frac{R^2}{4\gamma_1^2} \label{eq:omega_bound_x_t_SEG_neg_mon}
    \end{eqnarray}
    for all $t = 0, 1, \ldots, T-1$.
    
    Next, we show that probability event $E_{T-1}$ implies $\|\tx^t - x^*\| \leq 3R$ and derive useful inequalities related to $\theta_t$ for all $t = 0, 1, \ldots, T-1$. Indeed, due to Lipschitzness of $F$ probability event $E_{T-1}$ implies
    \begin{eqnarray}
        \|\tx^t - x^*\|^2 &=& \|x^t - x^* - \gamma_1 \tF_{\Bxi_1}(x^t)\|^2 \overset{\eqref{eq:a_plus_b}}{\leq}  2\|x^t - x^*\|^2 + 2\gamma_1^2\|\tF_{\Bxi_1}(x^t)\|^2 \notag \\
        &\overset{\eqref{eq:a_plus_b}}{\leq}& 2R_t^2 + 4\gamma_1^2\|F(x^t)\|^2 + 4\gamma_1^2\|\omega_t\|^2 \notag \\
        &\overset{\eqref{eq:Lipschitzness}}{\leq}& 2(1 + 2\gamma_1^2 L^2)R_t^2 + 4\gamma_1^2\|\omega_t\|^2 \notag \\
        &\overset{\eqref{eq:gamma_SEG_neg_mon},\eqref{eq:omega_bound_x_t_SEG_neg_mon}}{\leq}& 7R^2 \leq 9R^2 \label{eq:tilde_x_distance_SEG_neg_mon}
    \end{eqnarray}
    and
    \begin{eqnarray}
        \|F(\tx^t)\| &\leq& L\|\tx^t - x^*\| \leq \sqrt{7}LR \overset{\eqref{eq:gamma_SEG_neg_mon},\eqref{eq:lambda_SEG_neg_mon}}{\leq} \frac{\lambda_2}{2} \label{eq:operator_bound_tx_t_SEG_neg_mon}
    \end{eqnarray}
    for all $t = 0, 1, \ldots, T-1$.
    
    That is, $E_{T-1}$ implies that $x^t, \tx^t \in B_{3R}(x^*)$ for all $t = 0, 1, \ldots, T-1$. Applying Lemma~\ref{lem:optimization_lemma_SEG}, we get that probability event $E_{T-1}$ implies
    \begin{eqnarray}
        \frac{\gamma_1\gamma_2}{4T} \sum\limits_{l=0}^{T-1}\|F(x^l)\|^2 &\leq& \frac{R^2 - R_T^2}{T} + \frac{2\gamma_2}{T}\sum\limits_{l=0}^{T-1}\langle x^l - x^* - \gamma_2 F(\tx^l), \theta_l\rangle\notag\\
        &&\quad + \frac{1}{T}\sum\limits_{l=0}^{T-1}\left(\gamma_2^2\|\theta_l\|^2 + 3\gamma_1\gamma_2\|\omega_l\|^2\right) \label{eq:SEG_neg_mon_useful_inequality_for_the_end} \\
        R_{T}^2 &\leq& R^2 + 2\gamma_2\sum\limits_{l=0}^{T-1}\langle x^l - x^* - \gamma_2 F(\tx^l), \theta_l\rangle  + \sum\limits_{l=0}^{T-1}\left(\gamma_2^2\|\theta_l\|^2 + 3\gamma_1\gamma_2\|\omega_l\|^2\right). \notag
    \end{eqnarray}
    To estimate the sums in the right-hand side, we introduce new vectors:
    \begin{gather}
        \eta_t = \begin{cases} x^t - x^* - \gamma_2 F(\tx^t),& \text{if } \|x^t - x^* - \gamma_2 F(\tx^t)\| \leq \sqrt{7}(1 + \gamma_2 L)R,\\ 0,& \text{otherwise}, \end{cases} \label{eq:eta_t_SEG_neg_mon}
    \end{gather}
    for $t = 0, 1, \ldots, T-1$. First of all, we point out that vectors $\zeta_t$ and $\eta_t$ are bounded with probability $1$, i.e., with probability $1$
    \begin{equation}
        \|\eta_t\| \leq \sqrt{7}(1 + \gamma_2 L)R \label{eq:eta_t_bound_SEG_neg_mon} 
    \end{equation}
    for all $t = 0, 1, \ldots, T-1$. Next, we notice that $E_{T-1}$ implies
    \begin{eqnarray*}
        \|x^t - x^* - \gamma_2 F(\tx^t)\| &\leq& \|x^t - x^*\| + \gamma_2 \|F(\tx^t)\|\\
        &\overset{\eqref{eq:tilde_x_distance_SEG_neg_mon},\eqref{eq:operator_bound_tx_t_SEG_neg_mon}}{\leq}& \sqrt{7}(1 + \gamma_2 L)R
    \end{eqnarray*}
    for $t = 0, 1, \ldots, T-1$, i.e., probability event $E_{T-1}$ implies $\eta_t = x^t - x^* - \gamma_2 F(\tx^t)$ for all $t = 0,1,\ldots,T-1$. Therefore, $E_{T-1}$ implies
    \begin{eqnarray}
        R_{T}^2 &\leq& R^2 + 2\gamma_2\sum\limits_{l=0}^{T-1}\langle \eta_l, \theta_l\rangle  + \sum\limits_{l=0}^{T-1}\left(\gamma_2^2\|\theta_l\|^2 + 3\gamma_1\gamma_2\|\omega_l\|^2\right). \notag
    \end{eqnarray}
    As in the monotone case, to continue the derivation, we introduce vectors $\theta_l^u, \theta_l^b, \omega_l^u, \omega_l^b$ defined as
    \begin{gather}
        \theta_l^u \eqdef \EE_{\Bxi_2^l}\left[\tF_{\Bxi_2^l}(\tx^l)\right] - \tF_{\Bxi_2^l}(\tx^l),\quad \theta_l^b \eqdef F(\tx^l) - \EE_{\Bxi_2^l}\left[\tF_{\Bxi_2^l}(\tx^l)\right], \label{eq:theta_unbias_bias_SEG_neg_mon}\\
        \omega_l^u \eqdef \EE_{\Bxi_1^l}\left[\tF_{\Bxi_1^l}(x^l)\right] - \tF_{\Bxi_1^l}(x^l),\quad \theta_l^b \eqdef F(x^l) - \EE_{\Bxi_1^l}\left[\tF_{\Bxi_1^l}(x^l)\right], \label{eq:omega_unbias_bias_SEG_neg_mon}
    \end{gather}
    for all $l = 0,\ldots, T-1$. By definition we have $\theta_l = \theta_l^u + \theta_l^b$, $\omega_l = \omega_l^u + \omega_l^b$ for all $l = 0,\ldots, T-1$. Using the introduced notation, we continue our derivation as follows: $E_{T-1}$ implies
    \begin{eqnarray}
        R_T^2 &\overset{\eqref{eq:a_plus_b}}{\leq}& R^2 + \underbrace{2\gamma_2\sum\limits_{l=0}^{T-1}\langle \eta_l, \theta_l^u\rangle}_{\circledOne} + \underbrace{2\gamma_2\sum\limits_{l=0}^{T-1}\langle \eta_l, \theta_l^b\rangle}_{\circledTwo} + \underbrace{2\gamma_2^2 \sum\limits_{l = 0}^{T-1}\EE_{\Bxi_{2}^l}\left[\|\theta_l^u\|^2\right]}_{\circledThree} \notag\\
        &&\quad + \underbrace{2\gamma_2^2 \sum\limits_{l = 0}^{T-1}\left(\|\theta_l^u\|^2 - \EE_{\Bxi_{2}^l}\left[\|\theta_l^u\|^2\right]\right)}_{\circledFour}  + \underbrace{2\gamma_2^2 \sum\limits_{l = 0}^{T-1}\|\theta_l^b\|^2}_{\circledFive} + \underbrace{6\gamma_1\gamma_2 \sum\limits_{l = 0}^{T-1}\EE_{\Bxi_{1}^l}\left[\|\omega_l^u\|^2\right]}_{\circledSix} \notag\\
        &&\quad + \underbrace{6\gamma_1\gamma_2 \sum\limits_{l = 0}^{T-1}\left(\|\omega_l^u\|^2 - \EE_{\Bxi_{1}^l}\left[\|\omega_l^u\|^2\right]\right)}_{\circledSeven}  + \underbrace{6\gamma_1\gamma_2 \sum\limits_{l = 0}^{T-1}\|\omega_l^b\|^2}_{\circledEight}. \label{eq:SEG_neg_mon_12345678_bound}
    \end{eqnarray}
    The rest of the proof is based on deriving good enough upper bounds for $\circledOne, \circledTwo, \circledThree, \circledFour, \circledFive, \circledSix, \circledSeven, \circledEight$, i.e., we want to prove that $\circledOne + \circledTwo + \circledThree + \circledFour + \circledFive + \circledSix + \circledSeven + \circledEight \leq 8R^2$ with high probability.
    
    Before we move on, we need to derive some useful inequalities for operating with $\theta_l^u, \theta_l^b, \omega_l^u, \omega_l^b$. First of all, Lemma~\ref{lem:bias_variance} implies that
    \begin{equation}
        \|\theta_l^u\| \leq 2\lambda_2,\quad \|\omega_l^u\| \leq 2\lambda_1 \label{eq:theta_omega_magnitude_neg_mon}
    \end{equation}
    for all $l = 0,1, \ldots, T-1$. Next, since $\{\xi_1^{i,l}\}_{i=1}^{m_1}$, $\{\xi_2^{i,l}\}_{i=1}^{m_2}$ are independently sampled from $\cD$, we have $\EE_{\Bxi_1^l}[F_{\Bxi_1^l}(x^l)] = F(x^l)$, $\EE_{\Bxi_2^l}[F_{\Bxi_2^l}(\tx^l)] = F(\tx^l)$, and 
    \begin{gather}
        \EE_{\Bxi_1^l}\left[\|F_{\Bxi_1^l}(x^l) - F(x^l)\|^2\right] = \frac{1}{m_1^2}\sum\limits_{i=1}^{m_1} \EE_{\xi_1^{i,l}}\left[\|F_{\xi_1^{i,l}}(x^l) - F(x^l)\|^2\right] \overset{\eqref{eq:UBV}}{\leq} \frac{\sigma^2}{m_1}, \notag\\
        \EE_{\Bxi_2^l}\left[\|F_{\Bxi_2^l}(\tx^l) - F(\tx^l)\|^2\right] = \frac{1}{m_2^2}\sum\limits_{i=1}^{m_2} \EE_{\xi_2^{i,l}}\left[\|F_{\xi_2^{i,l}}(\tx^l) - F(\tx^l)\|^2\right] \overset{\eqref{eq:UBV}}{\leq} \frac{\sigma^2}{m_2}, \notag
    \end{gather}
    for all $l = 0,1, \ldots, T-1$. Moreover, as we already derived, probability event $E_{T-1}$ implies that $\|F(x^l)\| \leq \nicefrac{\lambda_l}{2}$ and $\|F(\tx^l)\| \leq \nicefrac{\lambda_l}{2}$ for all $l = 0,1, \ldots, T-1$ (see \eqref{eq:operator_bound_x_t_SEG_neg_mon} and \eqref{eq:operator_bound_tx_t_SEG_neg_mon}). Therefore, in view of Lemma~\ref{lem:bias_variance}, $E_{T-1}$ implies that
    \begin{gather}
        \left\|\theta_l^b\right\| \leq \frac{4\sigma^2}{m_2\lambda_2},\quad \left\|\omega_l^b\right\| \leq \frac{4\sigma^2}{m_1\lambda_1}, \label{eq:bias_theta_omega_neg_mon}\\
        \EE_{\Bxi_2^l}\left[\left\|\theta_l\right\|^2\right] \leq \frac{18\sigma^2}{m_2},\quad \EE_{\Bxi_1^l}\left[\left\|\omega_l\right\|^2\right] \leq \frac{18\sigma^2}{m_1}, \label{eq:distortion_theta_omega_neg_mon}\\
        \EE_{\Bxi_2^l}\left[\left\|\theta_l^u\right\|^2\right] \leq \frac{18\sigma^2}{m_2},\quad \EE_{\Bxi_1^l}\left[\left\|\omega_l^u\right\|^2\right] \leq \frac{18\sigma^2}{m_1}, \label{eq:variance_theta_omega_neg_mon}
    \end{gather}
    for all $l = 0,1, \ldots, T-1$.
    
    \paragraph{Upper bound for $\circledOne$.} Since $\EE_{\Bxi_2^l}[\theta_l^u] = 0$, we have
    \begin{equation*}
        \EE_{\Bxi_2^l}\left[2\gamma_2\langle \eta_l, \theta_l^u \rangle\right] = 0.
    \end{equation*}
    Next, the summands in $\circledOne$ are bounded with probability $1$:
    \begin{eqnarray}
        |2\gamma_2\langle \eta_l, \theta_l^u \rangle | \leq 2\gamma_2  \|\eta_l\|\cdot \|\theta_l^u\| 
        \overset{\eqref{eq:eta_t_bound_SEG_neg_mon},\eqref{eq:theta_omega_magnitude_neg_mon}}{\leq} 4\sqrt{7}\gamma_2 (1 + \gamma_2 L)  R \lambda_l \overset{\eqref{eq:gamma_SEG_neg_mon},\eqref{eq:lambda_SEG_neg_mon}}{\leq} \frac{R^2}{\ln\tfrac{6(K+1)}{\beta}} \eqdef c. \label{eq:SEG_neg_mon_technical_1_1}
    \end{eqnarray}
    Moreover, these summands have bounded conditional variances $\sigma_l^2 \eqdef \EE_{\Bxi_2^l}\left[4\gamma_2^2 \langle \eta_l, \theta_l^u \rangle^2\right]$:
    \begin{eqnarray}
        \sigma_l^2 \leq \EE_{\Bxi_2^l}\left[4\gamma_2^2 \|\eta_l\|^2\cdot \|\theta_l^u\|^2\right] \overset{\eqref{eq:eta_t_bound_SEG_neg_mon}}{\leq} 28\gamma_2^2 (1 + \gamma_2 L)^2R^2 \EE_{\Bxi_2^l}\left[\|\theta_l^u\|^2\right] \overset{\eqref{eq:gamma_SEG_neg_mon}}{\leq} 30\gamma_2^2R^2 \EE_{\Bxi_2^l}\left[\|\theta_l^u\|^2\right]. \label{eq:SEG_neg_mon_technical_1_2}
    \end{eqnarray}
    That is, sequence $\{2\gamma_2 \langle \eta_l, \theta_l^u \rangle\}_{l\geq 0}$ is a bounded martingale difference sequence having bounded conditional variances $\{\sigma_l^2\}_{l \geq 0}$. Applying Bernstein's inequality (Lemma~\ref{lem:Bernstein_ineq}) with $X_l = 2\gamma_2 \langle \eta_l, \theta_l^u \rangle$, $c$ defined in \eqref{eq:SEG_neg_mon_technical_1_1}, $b = R^2$, $G = \tfrac{R^4}{6\ln\frac{6(K+1)}{\beta}}$, we get that
    \begin{equation*}
        \PP\left\{|\circledOne| > R^2 \text{ and } \sum\limits_{l=0}^{T-1}\sigma_l^2 \leq \frac{R^4}{6\ln\tfrac{6(K+1)}{\beta}}\right\} \leq 2\exp\left(- \frac{b^2}{2G + \nicefrac{2cb}{3}}\right) = \frac{\beta}{3(K+1)}.
    \end{equation*}
    In other words, $\PP\{E_{\circledOne}\} \geq 1 - \tfrac{\beta}{3(K+1)}$, where probability event $E_{\circledOne}$ is defined as
    \begin{equation}
        E_{\circledOne} = \left\{\text{either} \quad \sum\limits_{l=0}^{T-1}\sigma_l^2 > \frac{R^4}{6\ln\tfrac{6(K+1)}{\beta}}\quad \text{or}\quad |\circledOne| \leq R^2\right\}. \label{eq:bound_1_SEG_neg_mon}
    \end{equation}
    Moreover, we notice here that probability event $E_{T-1}$ implies that
    \begin{eqnarray}
        \sum\limits_{l=0}^{T-1}\sigma_l^2 &\overset{\eqref{eq:SEG_neg_mon_technical_1_2}}{\leq}& 30\gamma_2^2R^2\sum\limits_{l=0}^{T-1} \EE_{\Bxi_2^l}\left[\|\theta_l^u\|^2\right] \overset{\eqref{eq:variance_theta_omega_neg_mon}, T \leq K+1}{\leq} \frac{540\gamma_2^2 R^2 \sigma^2 (K+1)}{m_2} \overset{\eqref{eq:batch_2_SEG_neg_mon}}{\leq} \frac{R^4}{6\ln\tfrac{6(K+1)}{\beta}}. \label{eq:bound_1_variances_SEG_neg_mon}
    \end{eqnarray}

    \paragraph{Upper bound for $\circledTwo$.} Probability event $E_{T-1}$ implies
    \begin{eqnarray}
        \circledTwo &\leq& 2\gamma_2 \sum\limits_{l=0}^{T-1}\|\eta_l\| \cdot \|\theta_l^b\| \overset{\eqref{eq:eta_t_bound_SEG_neg_mon},\eqref{eq:bias_theta_omega_neg_mon}, T \leq K+1}{\leq} \frac{8\sqrt{7}\gamma_2(1+\gamma_2 L)\sigma^2 R(K+1)}{m_2\lambda_2}\notag\\
        &\overset{\eqref{eq:gamma_SEG_neg_mon},\eqref{eq:lambda_SEG_neg_mon}}{=}& \frac{161\sqrt{7}\gamma_2^2\sigma^2(K+1) \ln \frac{6(K+1)}{\beta}}{m_2} \overset{\eqref{eq:batch_2_SEG_neg_mon}}{\leq} R^2. \label{eq:bound_2_SEG_neg_mon}
    \end{eqnarray}

    \paragraph{Upper bound for $\circledThree$.} Probability event $E_{T-1}$ implies
    \begin{eqnarray}
        \circledThree =  2\gamma_2^2 \sum\limits_{l = 0}^{T-1}\EE_{\Bxi_{2}^l}\left[\|\theta_l^u\|^2\right] \overset{\eqref{eq:variance_theta_omega_neg_mon}, T \leq K+1}{\leq} \frac{36\gamma_2^2\sigma^2 (K+1)}{m_2} \overset{\eqref{eq:batch_2_SEG_neg_mon}}{\leq} R^2. \label{eq:bound_3_SEG_neg_mon}
    \end{eqnarray}

    \paragraph{Upper bound for $\circledFour$.} We have
    \begin{equation*}
        2\gamma_2^2\EE_{\Bxi_2^l}\left[\|\theta_l^u\|^2 - \EE_{\Bxi_{2}^l}\left[\|\theta_l^u\|^2\right]\right] = 0.
    \end{equation*}
    Next, the summands in $\circledFour$ are bounded with probability $1$:
    \begin{eqnarray}
        2\gamma_2^2 \left|\|\theta_l^u\|^2 - \EE_{\Bxi_{2}^l}\left[\|\theta_l^u\|^2\right] \right| &\leq& 2\gamma_2^2\left( \|\theta_l^u\|^2 + \EE_{\Bxi_{2}^l}\left[\|\theta_l^u\|^2\right] \right) 
        \overset{\eqref{eq:theta_omega_magnitude_neg_mon}}{\leq} 16\gamma_2^2 \lambda_2^2 \notag\\
        &\overset{\eqref{eq:lambda_SEG_neg_mon}}{\leq}& \frac{R^2}{\ln\tfrac{6(K+1)}{\beta}} \eqdef c. \label{eq:SEG_neg_mon_technical_4_1}
    \end{eqnarray}
    Moreover, these summands have bounded conditional variances $\widetilde\sigma_l^2 \eqdef 4\gamma_2^4\EE_{\Bxi_2^l}\left[\left(\|\theta_l^u\|^2 - \EE_{\Bxi_{2}^l}\left[\|\theta_l^u\|^2\right]\right)^2\right]$:
    \begin{eqnarray}
        \widetilde\sigma_l^2 \overset{\eqref{eq:SEG_neg_mon_technical_4_1}}{\leq} \frac{2\gamma_2^2 R^2}{\ln \frac{6(K+1)}{\beta}} \EE_{\Bxi_2^l}\left[\left| \|\theta_l^u\|^2 - \EE_{\Bxi_{2}^l}\left[\|\theta_l^u\|^2\right] \right|\right]  \leq \frac{4\gamma_2^2 R^2}{\ln \frac{6(K+1)}{\beta}} \EE_{\Bxi_2^l}\left[\|\theta_l^u\|^2\right] \label{eq:SEG_neg_mon_technical_4_2}
    \end{eqnarray}
    That is, sequence $\{\|\theta_l^u\|^2 - \EE_{\Bxi_{2}^l}[\|\theta_l^u\|^2]\}_{l\geq 0}$ is a bounded martingale difference sequence having bounded conditional variances $\{\widetilde\sigma_l^2\}_{l \geq 0}$. Applying Bernstein's inequality (Lemma~\ref{lem:Bernstein_ineq}) with $X_l = \|\theta_l^u\|^2 - \EE_{\Bxi_{2}^l}[\|\theta_l^u\|^2]$, $c$ defined in \eqref{eq:SEG_neg_mon_technical_4_1}, $b = R^2$, $G = \tfrac{R^4}{6\ln\frac{6(K+1)}{\beta}}$, we get that
    \begin{equation*}
        \PP\left\{|\circledFour| > R^2 \text{ and } \sum\limits_{l=0}^{T-1}\widetilde\sigma_l^2 \leq \frac{R^4}{6\ln\tfrac{6(K+1)}{\beta}}\right\} \leq 2\exp\left(- \frac{b^2}{2G + \nicefrac{2cb}{3}}\right) = \frac{\beta}{3(K+1)}.
    \end{equation*}
    In other words, $\PP\{E_{\circledFour}\} \geq 1 - \tfrac{\beta}{3(K+1)}$, where probability event $E_{\circledFour}$ is defined as
    \begin{equation}
        E_{\circledFour} = \left\{\text{either} \quad \sum\limits_{l=0}^{T-1}\widetilde\sigma_l^2 > \frac{R^4}{6\ln\tfrac{6(K+1)}{\beta}}\quad \text{or}\quad |\circledFour| \leq R^2\right\}. \label{eq:bound_4_SEG_neg_mon}
    \end{equation}
    Moreover, we notice here that probability event $E_{T-1}$ implies that
    \begin{eqnarray}
        \sum\limits_{l=0}^{T-1}\widetilde\sigma_l^2 &\overset{\eqref{eq:SEG_neg_mon_technical_4_2}}{\leq}& \frac{4\gamma_2^2 R^2}{\ln \frac{6(K+1)}{\beta}}\sum\limits_{l=0}^{T-1} \EE_{\Bxi_2^l}\left[\|\theta_l^u\|^2\right] \overset{\eqref{eq:variance_theta_omega_neg_mon}, T \leq K+1}{\leq} \frac{72\gamma_2^2 R^2 \sigma^2 (K+1)}{m_2 \ln\tfrac{6(K+1)}{\beta}} \notag\\
        &\overset{\eqref{eq:batch_2_SEG_neg_mon}}{\leq}& \frac{R^4}{6\ln\tfrac{6(K+1)}{\beta}}. \label{eq:bound_4_variances_SEG_neg_mon}
    \end{eqnarray}

    \paragraph{Upper bound for $\circledFive$.} Probability event $E_{T-1}$ implies
    \begin{eqnarray}
        \circledFive =  2\gamma_2^2 \sum\limits_{l = 0}^{T-1}\|\theta_l^b\|^2  \overset{\eqref{eq:bias_theta_omega_neg_mon}, T \leq K+1}{\leq} \frac{32\gamma_2^2 \sigma^4 (K+1)}{m_2^2 \lambda_2^2} \overset{\eqref{eq:lambda_SEG_neg_mon}}{=}  \frac{12800 \gamma_2^4 \sigma^4 (K+1) \ln^2 \frac{6(K+1)}{\beta}}{m_2^2 R^2}  \overset{\eqref{eq:batch_2_SEG_neg_mon}}{\leq} R^2. \label{eq:bound_5_SEG_neg_mon}
    \end{eqnarray}

    \paragraph{Upper bound for $\circledSix$.} Probability event $E_{T-1}$ implies
    \begin{eqnarray}
        \circledSix =  6\gamma_1\gamma_2 \sum\limits_{l = 0}^{T-1}\EE_{\Bxi_{1}^l}\left[\|\omega_l^u\|^2\right] \overset{\eqref{eq:variance_theta_omega_neg_mon}, T \leq K+1}{\leq} \frac{108\gamma_1\gamma_2\sigma^2 (K+1)}{m_1} \overset{\eqref{eq:batch_1_SEG_neg_mon}}{\leq} R^2. \label{eq:bound_6_SEG_neg_mon}
    \end{eqnarray}

    \paragraph{Upper bound for $\circledSeven$.} We have
    \begin{equation*}
        6\gamma_1\gamma_2\EE_{\Bxi_1^l}\left[\|\omega_l^u\|^2 - \EE_{\Bxi_{1}^l}\left[\|\omega_l^u\|^2\right]\right] = 0.
    \end{equation*}
    Next, the summands in $\circledSeven$ are bounded with probability $1$:
    \begin{eqnarray}
        6\gamma_1\gamma_2 \left|\|\omega_l^u\|^2 - \EE_{\Bxi_{1}^l}\left[\|\omega_l^u\|^2\right] \right| &\leq& 6\gamma_1\gamma_2\left( \|\omega_l^u\|^2 + \EE_{\Bxi_{1}^l}\left[\|\omega_l^u\|^2\right] \right) 
        \overset{\eqref{eq:theta_omega_magnitude_neg_mon}}{\leq} 48\gamma_1\gamma_2 \lambda_1^2 \notag\\
        &\overset{\eqref{eq:lambda_SEG_neg_mon}}{\leq}& \frac{\gamma_2R^2}{\gamma_1\ln\tfrac{6(K+1)}{\beta}} \overset{\gamma_2 \leq \gamma_1}{\leq} \frac{R^2}{\ln\tfrac{6(K+1)}{\beta}} \eqdef c. \label{eq:SEG_neg_mon_technical_7_1}
    \end{eqnarray}
    Moreover, these summands have bounded conditional variances $\widehat\sigma_l^2 \eqdef 36\gamma_1^2\gamma_2^2\EE_{\Bxi_1^l}\left[\left(\|\omega_l^u\|^2 - \EE_{\Bxi_{1}^l}\left[\|\omega_l^u\|^2\right]\right)^2\right]$:
    \begin{eqnarray}
        \widehat\sigma_l^2 \overset{\eqref{eq:SEG_neg_mon_technical_7_1}}{\leq} \frac{6\gamma_2^2 R^2}{\ln \frac{6(K+1)}{\beta}} \EE_{\Bxi_1^l}\left[\left| \|\omega_l^u\|^2 - \EE_{\Bxi_{1}^l}\left[\|\omega_l^u\|^2\right] \right|\right]  \leq \frac{12\gamma_2^2 R^2}{\ln \frac{6(K+1)}{\beta}} \EE_{\Bxi_1^l}\left[\|\omega_l^u\|^2\right] \label{eq:SEG_neg_mon_technical_7_2}
    \end{eqnarray}
    That is, sequence $\{\|\omega_l^u\|^2 - \EE_{\Bxi_{1}^l}[\|\omega_l^u\|^2]\}_{l\geq 0}$ is a bounded martingale difference sequence having bounded conditional variances $\{\widehat\sigma_l^2\}_{l \geq 0}$. Applying Bernstein's inequality (Lemma~\ref{lem:Bernstein_ineq}) with $X_l = \|\omega_l^u\|^2 - \EE_{\Bxi_{1}^l}[\|\omega_l^u\|^2]$, $c$ defined in \eqref{eq:SEG_neg_mon_technical_7_1}, $b = R^2$, $G = \tfrac{R^4}{6\ln\frac{6(K+1)}{\beta}}$, we get that
    \begin{equation*}
        \PP\left\{|\circledSeven| > R^2 \text{ and } \sum\limits_{l=0}^{T-1}\widehat\sigma_l^2 \leq \frac{R^4}{6\ln\tfrac{6(K+1)}{\beta}}\right\} \leq 2\exp\left(- \frac{b^2}{2G + \nicefrac{2cb}{3}}\right) = \frac{\beta}{3(K+1)}.
    \end{equation*}
    In other words, $\PP\{E_{\circledSeven}\} \geq 1 - \tfrac{\beta}{3(K+1)}$, where probability event $E_{\circledSeven}$ is defined as
    \begin{equation}
        E_{\circledSeven} = \left\{\text{either} \quad \sum\limits_{l=0}^{T-1}\widehat\sigma_l^2 > \frac{R^4}{6\ln\tfrac{6(K+1)}{\beta}}\quad \text{or}\quad |\circledSeven| \leq R^2\right\}. \label{eq:bound_7_SEG_neg_mon}
    \end{equation}
    Moreover, we notice here that probability event $E_{T-1}$ implies that
    \begin{eqnarray}
        \sum\limits_{l=0}^{T-1}\widehat\sigma_l^2 &\overset{\eqref{eq:SEG_neg_mon_technical_7_2}}{\leq}& \frac{12\gamma_2^2 R^2}{\ln \frac{6(K+1)}{\beta}}\sum\limits_{l=0}^{T-1} \EE_{\Bxi_1^l}\left[\|\omega_l^u\|^2\right] \overset{\eqref{eq:variance_theta_omega_neg_mon}, T \leq K+1}{\leq} \frac{216\gamma_2^2 R^2 \sigma^2 (K+1)}{m_1 \ln\tfrac{6(K+1)}{\beta}} \notag\\
        &\overset{\eqref{eq:batch_1_SEG_neg_mon}}{\leq}& \frac{R^4}{6\ln\tfrac{6(K+1)}{\beta}}. \label{eq:bound_7_variances_SEG_neg_mon}
    \end{eqnarray}

    \paragraph{Upper bound for $\circledEight$.} Probability event $E_{T-1}$ implies
    \begin{eqnarray}
        \circledEight &=&  6\gamma_1\gamma_2 \sum\limits_{l = 0}^{T-1}\|\omega_l^b\|^2  \overset{\eqref{eq:bias_theta_omega_neg_mon}, T \leq K+1}{\leq} \frac{96\gamma_1\gamma_2 \sigma^4 (K+1)}{m_1^2 \lambda_1^2} \notag\\
        &\overset{\eqref{eq:lambda_SEG_neg_mon}}{=}& \frac{38400 \gamma_1^3\gamma_2 \sigma^4 (K+1) \ln^2 \frac{6(K+1)}{\beta}}{m_1^2 R^2}  \overset{\eqref{eq:batch_1_SEG_neg_mon}}{\leq} R^2. \label{eq:bound_8_SEG_neg_mon}
    \end{eqnarray}

    \paragraph{Final derivation.} Putting all bounds together, we get that $E_{T-1}$ implies
    \begin{gather*}
        R_T^2 \overset{\eqref{eq:SEG_neg_mon_12345678_bound}}{\leq} R^2 + \circledOne + \circledTwo + \circledThree + \circledFour + \circledFive + \circledSix + \circledSeven + \circledEight,\\
        \circledTwo \overset{\eqref{eq:bound_2_SEG_neg_mon}}{\leq} R^2,\quad \circledThree \overset{\eqref{eq:bound_3_SEG_neg_mon}}{\leq} R^2,\quad \circledFive \overset{\eqref{eq:bound_5_SEG_neg_mon}}{\leq} R^2,\quad \circledSix \overset{\eqref{eq:bound_6_SEG_neg_mon}}{\leq} R^2,\quad \circledEight \overset{\eqref{eq:bound_8_SEG_neg_mon}}{\leq} R^2,\\
        \sum\limits_{l=0}^{T-1}\sigma_l^2 \overset{\eqref{eq:bound_1_variances_SEG_neg_mon}}{\leq}  \frac{R^4}{6\ln\tfrac{6(K+1)}{\beta}},\quad \sum\limits_{l=0}^{T-1}\widetilde\sigma_l^2 \overset{\eqref{eq:bound_4_variances_SEG_neg_mon}}{\leq} \frac{R^4}{6\ln\tfrac{6(K+1)}{\beta}},\quad \sum\limits_{l=0}^{T-1}\widehat\sigma_l^2 \overset{\eqref{eq:bound_7_variances_SEG_neg_mon}}{\leq}  \frac{R^4}{6\ln\tfrac{6(K+1)}{\beta}}.
    \end{gather*}
    Moreover, in view of \eqref{eq:bound_1_SEG_neg_mon}, \eqref{eq:bound_4_SEG_neg_mon}, \eqref{eq:bound_7_SEG_neg_mon}, and our induction assumption, we have
    \begin{gather*}
        \PP\{E_{T-1}\} \geq 1 - \frac{(T-1)\beta}{K+1},\\
        \PP\{E_{\circledOne}\} \geq 1 - \frac{\beta}{3(K+1)}, \quad \PP\{E_{\circledFour}\} \geq 1 - \frac{\beta}{3(K+1)}, \quad \PP\{E_{\circledSeven}\} \geq 1 - \frac{\beta}{3(K+1)} ,
    \end{gather*}
    where probability events $E_{\circledOne}$, $E_{\circledFour}$, and $E_{\circledSeven}$ are defined as
    \begin{eqnarray}
        E_{\circledOne}&=& \left\{\text{either} \quad \sum\limits_{l=0}^{T-1}\sigma_l^2 > \frac{R^4}{6\ln\tfrac{6(K+1)}{\beta}}\quad \text{or}\quad |\circledOne| \leq R^2\right\},\notag\\
        E_{\circledFour}&=& \left\{\text{either} \quad \sum\limits_{l=0}^{T-1}\widetilde\sigma_l^2 > \frac{R^4}{6\ln\tfrac{6(K+1)}{\beta}}\quad \text{or}\quad |\circledFour| \leq R^2\right\},\notag\\
        E_{\circledSeven}&=& \left\{\text{either} \quad \sum\limits_{l=0}^{T-1}\widehat\sigma_l^2 > \frac{R^4}{6\ln\tfrac{6(K+1)}{\beta}}\quad \text{or}\quad |\circledSeven| \leq R^2\right\}.\notag
    \end{eqnarray}
    Putting all of these inequalities together, we obtain that probability event $E_{T-1} \cap E_{\circledOne} \cap E_{\circledFour} \cap E_{\circledSeven}$ implies
    \begin{eqnarray*}
        R_T^2 &\overset{\eqref{eq:SEG_neg_mon_12345678_bound}}{\leq}& R^2 + \circledOne + \circledTwo + \circledThree + \circledFour + \circledFive + \circledSix + \circledSeven + \circledEight\\
        &\leq& 9R^2.
    \end{eqnarray*}
    Moreover, union bound for the probability events implies
    \begin{equation}
        \PP\{E_T\} \geq \PP\{E_{T-1} \cap E_{\circledOne} \cap E_{\circledFour} \cap E_{\circledSeven}\} = 1 - \PP\{\overline{E}_{T-1} \cup \overline{E}_{\circledOne} \cup \overline{E}_{\circledFour} \cup \overline{E}_{\circledSeven}\} \geq 1 - \frac{T\beta}{K+1}.
    \end{equation}
    This is exactly what we wanted to prove (see the paragraph after inequality \eqref{eq:induction_inequality_neg_mon_SEG}). In particular, $E_{K+1}$ implies
    \begin{eqnarray*}
        \frac{1}{K+1}\sum\limits_{k=0}^{K+1}\|F(x^k)\|^2 &\overset{\eqref{eq:SEG_neg_mon_useful_inequality_for_the_end}}{\leq}& \frac{4(R^2 - R_{K+1}^2)}{\gamma_1\gamma_2(K+1)}  + \frac{4\left(\circledOne + \circledTwo + \circledThree + \circledFour + \circledFive + \circledSix + \circledSeven + \circledEight\right)}{\gamma_1\gamma_2(K+1)}\\
        &\leq& \frac{36R^2}{\gamma_1\gamma_2 (K+1)}.
    \end{eqnarray*}
    This finishes the proof.
\end{proof}

\begin{corollary}\label{cor:main_result_avg_sq_norm_SEG}
    Let the assumptions of Theorem~\ref{thm:main_result_avg_sq_norm_SEG} hold and
    \begin{equation}
        \rho \leq \frac{1}{640L \ln\frac{6(K+1)}{\beta}}. \label{eq:rho_condition}
    \end{equation}
    Then, the choice of step-sizes and batch-sizes
    \begin{equation}
        2\gamma_2 = \gamma_1 = \frac{1}{160L \ln \tfrac{6(K+1)}{\beta}},\quad m_1 = m_2 = \max\left\{1, \frac{81 (K+1) \sigma^2}{640 L^2R^2 \ln\tfrac{6(K+1)}{\beta}}\right\} \label{eq:neg_mon_SEG_large_step_large_batch}
    \end{equation}
    satisfies conditions \eqref{eq:gamma_SEG_neg_mon}, \eqref{eq:batch_1_SEG_neg_mon}, \eqref{eq:batch_2_SEG_neg_mon}. With such choice of $\gamma, m_1, m_2$, and the choice of $\lambda_1,\lambda_2$ as in \eqref{eq:lambda_SEG_neg_mon}, the iterates produced by \ref{eq:clipped_SEG} after $K$ iterations with probability at least $1-\beta$ satisfy
    \begin{equation}
        \frac{1}{K+1}\sum\limits_{k=0}^{K}\|F(x^k)\|^2 \leq \frac{1 843 200 L^2R^2 \ln^2 \tfrac{6(K+1)}{\beta}}{K+1}. \label{eq:main_result_neg_mon_SEG_large_batch}
    \end{equation}
    In particular, to guarantee $\frac{1}{K+1}\sum_{k=0}^{K}\|F(x^k)\|^2 \leq \varepsilon$ with probability at least $1-\beta$ for some $\varepsilon > 0$ \ref{eq:clipped_SEG} requires,
    \begin{gather}
        \cO\left(\frac{L^2R^2}{\varepsilon} \ln^2\left(\frac{L^2R^2}{\varepsilon\beta}\right)\right) \text{ iterations}, \label{eq:gap_SEG_iteration_complexity_large_batch}\\
        \cO\left(\max\left\{\frac{L^2R^2}{\varepsilon}\ln^2\left(\frac{L^2R^2}{\varepsilon\beta}\right), \frac{L^2\sigma^2R^2}{\varepsilon^2}\ln^3\left(\frac{LR^2}{\varepsilon\beta}\right)\right\} \right) \text{ oracle calls}. \label{eq:neg_mon_SEG_oracle_complexity_large_batch} 
    \end{gather}
\end{corollary}
\begin{proof}
    First of all, we verify that the choice of $\gamma_1, \gamma_2$ and $m_1, m_2$ from \eqref{eq:neg_mon_SEG_large_step_large_batch} satisfies conditions \eqref{eq:gamma_SEG_neg_mon}, \eqref{eq:batch_1_SEG_neg_mon}, \eqref{eq:batch_2_SEG_neg_mon}. Inequality \eqref{eq:gamma_SEG_neg_mon} holds since
    \begin{eqnarray*}
        \gamma_2 + 2\rho \overset{\eqref{eq:neg_mon_SEG_large_step_large_batch}}{=} \frac{1}{320 L \ln\frac{6(K+1)}{\beta}} + 2\rho \overset{\eqref{eq:rho_condition}}{\leq} \frac{1}{320 L \ln\frac{6(K+1)}{\beta}} + \frac{1}{320 L \ln\frac{6(K+1)}{\beta}} \overset{\eqref{eq:neg_mon_SEG_large_step_large_batch}}{=} \gamma_1
    \end{eqnarray*}
    and \eqref{eq:batch_1_SEG_neg_mon}, \eqref{eq:batch_2_SEG_neg_mon} are satisfied since
    \begin{eqnarray*}
        m_1 &=& \max\left\{1, \frac{81 (K+1) \sigma^2}{640 L^2R^2 \ln\tfrac{6(K+1)}{\beta}}\right\}\\
        &\geq& \max\left\{1, \frac{216 \max\{\gamma_1\gamma_2(K+1), \sqrt{\gamma_1^3\gamma_2(K+1)}\ln\tfrac{6(K+1)}{\beta}\}\sigma^2}{R^2}\right\},\\
        m_2 &=& \max\left\{1, \frac{81 (K+1) \sigma^2}{640 L^2R^2 \ln\tfrac{6(K+1)}{\beta}}\right\} \geq \max\left\{1, \frac{3240 (K+1) \gamma_2^2\sigma^2 \ln\tfrac{6(K+1)}{\beta}}{R^2}\right\}.
    \end{eqnarray*}
    Therefore, applying Theorem~\ref{thm:main_result_avg_sq_norm_SEG}, we derive that with probability at least $1-\beta$
    \begin{equation*}
        \frac{1}{K+1}\sum\limits_{k=0}^{K}\|F(x^k)\|^2 \leq \frac{36R^2}{\gamma_1\gamma_2(K+1)} \overset{\eqref{eq:neg_mon_SEG_large_step_large_batch}}{=} \frac{1 843 200 L^2R^2 \ln^2 \tfrac{6(K+1)}{\beta}}{K+1}.
    \end{equation*}
    To guarantee $\frac{1}{K+1}\sum_{k=0}^{K}\|F(x^k)\|^2 \leq \varepsilon$, we choose $K$ in such a way that the right-hand side of the above inequality is smaller than $\varepsilon$ that gives
    \begin{eqnarray*}
        K = \cO\left(\frac{L^2R^2}{\varepsilon} \ln^2\left(\frac{L^2R^2}{\varepsilon\beta}\right)\right).
    \end{eqnarray*}
    The total number of oracle calls equals
    \begin{eqnarray*}
        2m(K+1) &\overset{\eqref{eq:neg_mon_SEG_large_step_large_batch}}{=}&  2\max\left\{K+1, \frac{81 (K+1)^2 \sigma^2}{640L^2R^2 \ln\tfrac{6(K+1)}{\beta}}\right\}\\
        &=& \cO\left(\max\left\{\frac{L^2R^2}{\varepsilon}\ln^2\left(\frac{L^2R^2}{\varepsilon\beta}\right), \frac{L^2\sigma^2R^2}{\varepsilon^2}\ln^3\left(\frac{L^2R^2}{\varepsilon\beta}\right)\right\} \right).
    \end{eqnarray*}
\end{proof}

\subsection{Quasi-Strongly Monotone Case}

\begin{lemma}\label{lem:optimization_lemma_str_mon_SEG}
    Let Assumptions~\ref{as:lipschitzness}, \ref{as:str_monotonicity} hold for $Q = B_{3R}(x^*) = \{x\in\R^d\mid \|x - x^*\| \leq 3R\}$, where $R \geq R_0 \eqdef \|x^0 - x^*\|$, and $\gamma_1 = \gamma_2 = \gamma$, $0 < \gamma \leq \nicefrac{1}{2(L+2\mu)}$. If $x^k$ and $\tx^k \eqdef x^k - \gamma F(x^k)$ lie in $B_{3R}(x^*)$ for all $k = 0,1,\ldots, K$ for some $K\geq 0$, then the iterates produced by \ref{eq:clipped_SEG} satisfy
    \begin{eqnarray}
        \|x^{K+1} - x^*\|^2 &\leq& (1 - \gamma \mu)^{K+1}\|x^0 - x^*\|^2 - 4\gamma^3 \mu \sum\limits_{k=0}^K (1-\gamma\mu)^{K-k}\langle F(x^k), \omega_k \rangle\notag\\
        &&\quad + 2\gamma \sum\limits_{k=0}^K (1-\gamma\mu)^{K-k} \langle x^k - x^* - \gamma F(\tx^k), \theta_k \rangle\notag\\
        &&\quad + \gamma^2 \sum\limits_{k=0}^K (1-\gamma\mu)^{K-k} \left(\|\theta_k\|^2 + 4\|\omega_k\|^2\right), \label{eq:optimization_lemma_SEG_str_mon}
    \end{eqnarray}
    where $\theta_k, \omega_k$ are defined in \eqref{eq:theta_k_SEG}, \eqref{eq:omega_k_SEG}.
\end{lemma}
\begin{proof}
    Using the update rule of \ref{eq:clipped_SEG}, we obtain
    \begin{eqnarray*}
        \|x^{k+1} - x^*\|^2 &=& \|x^k - x^*\|^2 - 2\gamma \langle x^k - x^*,  \tF_{\Bxi_2^k}(\tx^k)\rangle + \gamma^2\|\tF_{\Bxi_2^k}(\tx^k)\|^2\\
        &=& \|x^k - x^*\|^2 -2\gamma \langle x^k - x^*, F(\tx^k) \rangle + 2\gamma \langle x^k - x^*, \theta_k \rangle\\
        &&\quad + \gamma^2\|F(\tx^k)\|^2 - 2\gamma^2 \langle F(\tx^k), \theta_k \rangle + \gamma^2\|\theta_k\|^2\\
        &=& \|x^k - x^*\|^2 -2\gamma \langle \tx^k - x^*, F(\tx^k) \rangle - 2\gamma \langle x^k - \tx^k, F(\tx^k) \rangle\\
        &&\quad + 2\gamma \langle x^k - x^* - \gamma F(\tx^k), \theta_k \rangle + \gamma^2\|F(\tx^k)\|^2 + \gamma^2\|\theta_k\|^2.
    \end{eqnarray*}
    Since $F$ is $\mu$-quasi strongly monotone, we have
    \begin{eqnarray*}
        -2\gamma \langle \tx^k - x^*, F(\tx^k) \rangle &\leq& -2\gamma\mu \|\tx^k - x^*\|^2 \overset{\eqref{eq:a_minus_b}}{\leq} - \gamma \mu \|x^k - x^*\|^2 + 2\gamma\mu\|\tx^k - x^k\|^2\\
        &=& - \gamma \mu \|x^k - x^*\|^2 + 2\gamma^3\mu\|\tF_{\Bxi_1}(x^k)\|^2\\
        &=& - \gamma \mu \|x^k - x^*\|^2 + 2\gamma^3\mu\|F(x^k)\|^2 - 4\gamma^3 \mu \langle F(x^k), \omega_k \rangle + 2\gamma^3\mu \|\omega_k\|^2.
    \end{eqnarray*}
    Moreover, $- 2\gamma \langle x^k - \tx^k, F(\tx^k) \rangle$ can be rewritten as
    \begin{eqnarray*}
        - 2\gamma \langle x^k - \tx^k, F(\tx^k) \rangle &=& -2\gamma^2 \langle \tF_{\Bxi_1}(x^k), F(\tx^k) \rangle\\
        &=& \gamma^2 \|\tF_{\Bxi_1}(x^k) - F(x^k)\|^2 - \gamma^2 \|\tF_{\Bxi_1}(x^k)\|^2 - \gamma^2 \|F(\tx^k)\|^2.
    \end{eqnarray*}
    Putting all together, we get
    \begin{eqnarray*}
        \|x^{k+1} - x^*\|^2 &\leq& (1 - \gamma\mu)\|x^k - x^*\|^2 + 2\gamma^3\mu\|F(x^k)\|^2 - 4\gamma^3 \mu \langle F(x^k), \omega_k \rangle + 2\gamma^3\mu \|\omega_k\|^2\\
        &&\quad + \gamma^2 \|\tF_{\Bxi_1}(x^k) - F(x^k)\|^2 - \gamma^2 \|\tF_{\Bxi_1}(x^k)\|^2 - \gamma^2 \|F(\tx^k)\|^2\\
        &&\quad + 2\gamma \langle x^k - x^* - \gamma F(\tx^k), \theta_k \rangle + \gamma^2\|F(\tx^k)\|^2 + \gamma^2\|\theta_k\|^2\\
        &\overset{\eqref{eq:a_plus_b}}{\leq}& (1 - \gamma\mu)\|x^k - x^*\|^2 + 2\gamma^3\mu\|F(x^k)\|^2 - 4\gamma^3 \mu \langle F(x^k), \omega_k \rangle + 2\gamma^3\mu \|\omega_k\|^2\\
        &&\quad + 2\gamma^2 \|\omega_k\|^2 + 2\gamma^2\|F(x^k) - F(\tx^k)\|^2  - \gamma^2 \|\tF_{\Bxi_1}(x^k)\|^2\\
        &&\quad + 2\gamma \langle x^k - x^* - \gamma F(\tx^k), \theta_k \rangle + \gamma^2\|\theta_k\|^2\\
        &\overset{\eqref{eq:Lipschitzness}}{\leq}& (1 - \gamma\mu)\|x^k - x^*\|^2 + 2\gamma^3\mu\|F(x^k)\|^2 - 4\gamma^3 \mu \langle F(x^k), \omega_k \rangle\\
        &&\quad + 2\gamma^2(1 + \gamma\mu)\|\omega_k\|^2 - \gamma^2(1 - 2\gamma^2 L^2)\|\tF_{\Bxi_1}(x^k)\|^2\\
        &&\quad + 2\gamma \langle x^k - x^* - \gamma F(\tx^k), \theta_k \rangle + \gamma^2\|\theta_k\|^2\\
        &\overset{\eqref{eq:a_minus_b}}{\leq}& (1 - \gamma\mu)\|x^k - x^*\|^2 - \gamma^2\left(\frac{1}{2} - \gamma^2 L^2 - 2\gamma\mu\right)\|F(x^k)\|^2\\
        &&\quad - 4\gamma^3 \mu \langle F(x^k), \omega_k \rangle + \gamma^2(3 - 2\gamma^2 L^2 + 2\gamma \mu)\|\omega_k\|^2\\
        &&\quad + 2\gamma \langle x^k - x^* - \gamma F(\tx^k), \theta_k \rangle + \gamma^2\|\theta_k\|^2\\
        &\leq& (1 - \gamma\mu)\|x^k - x^*\|^2 - 4\gamma^3 \mu \langle F(x^k), \omega_k \rangle\\
        &&\quad + 2\gamma \langle x^k - x^* - \gamma F(\tx^k), \theta_k \rangle + \gamma^2\left(\|\theta_k\|^2 + 4\|\omega_k\|^2\right),
    \end{eqnarray*}
    where in the last step we apply $0 < \gamma \leq \nicefrac{1}{2(L+2\mu)}$. Unrolling the recurrence, we obtain \eqref{eq:optimization_lemma_SEG_str_mon}.
\end{proof}

\begin{theorem}\label{thm:main_result_str_mon_SEG}
    Let Assumptions~\ref{as:UBV}, \ref{as:lipschitzness}, \ref{as:str_monotonicity}, hold for $Q = B_{3R}(x^*) = \{x\in\R^d\mid \|x - x^*\| \leq 3R\}$, where $R \geq R_0 \eqdef \|x^0 - x^*\|$, and $\gamma_1 = \gamma_2 = \gamma$,
    \begin{eqnarray}
        0< \gamma &\leq& \frac{1}{650 L \ln \tfrac{6(K+1)}{\beta}}, \label{eq:gamma_SEG_str_mon}\\
        \lambda_{1,k} = \lambda_{2,k} = \lambda_k &=& \frac{\exp(-\gamma\mu(1 + \nicefrac{k}{2}))R}{120\gamma \ln \tfrac{6(K+1)}{\beta}}, \label{eq:lambda_SEG_str_mon}\\
        m_{1,k} = m_{2,k} = m_k &\geq& \max\left\{1, \frac{264600\gamma^2 (K+1) \sigma^2\ln \tfrac{6(K+1)}{\beta}}{\exp(-\gamma\mu k) R^2}\right\}, \label{eq:batch_SEG_str_mon}
    \end{eqnarray}
    for some $K \geq 0$ and $\beta \in (0,1]$ such that $\ln \tfrac{6(K+1)}{\beta} \geq 1$. Then, after $K$ iterations the iterates produced by \ref{eq:clipped_SEG} with probability at least $1 - \beta$ satisfy 
    \begin{equation}
        \|x^{K+1} - x^*\|^2 \leq 2\exp(-\gamma\mu(K+1))R^2. \label{eq:main_result_str_mon}
    \end{equation}
\end{theorem}
\begin{proof}
    As in the proof of Theorem~\ref{thm:main_result_gap_SEG}, we use the following notation: $R_k = \|x^k - x^*\|^2$, $k \geq 0$. We will derive \eqref{eq:main_result_str_mon} by induction. In particular, for each $k = 0,\ldots, K+1$ we define probability event $E_k$ as follows: inequalities
    \begin{equation}
        R_t^2 \leq 2 \exp(-\gamma\mu t) R^2 \label{eq:induction_inequality_str_mon_SEG}
    \end{equation}
    hold for $t = 0,1,\ldots,k$ simultaneously. Our goal is to prove that $\PP\{E_k\} \geq  1 - \nicefrac{k\beta}{(K+1)}$ for all $k = 0,1,\ldots,K+1$. We use the induction to show this statement. For $k = 0$ the statement is trivial since $R_0^2 \leq 2R^2$ by definition. Next, assume that the statement holds for $k = T-1 \leq K$, i.e., we have $\PP\{E_{T-1}\} \geq 1 - \nicefrac{(T-1)\beta}{(K+1)}$. We need to prove that $\PP\{E_T\} \geq 1 - \nicefrac{T\beta}{(K+1)}$. First of all, since $R_t^2 \leq 2\exp(-\gamma\mu t) R^2 \leq 9R^2$, we have $x^t \in B_{3R}(x^*)$. Operator $F$ is $L$-Lipschitz on $B_{3R}(x^*)$. Therefore, probability event $E_{T-1}$ implies
    \begin{eqnarray}
        \|F(x^t)\| &\leq& L\|x^t - x^*\| \overset{\eqref{eq:induction_inequality_str_mon_SEG}}{\leq} \sqrt{2}L\exp(- \nicefrac{\gamma\mu t}{2})R \overset{\eqref{eq:gamma_SEG_str_mon},\eqref{eq:lambda_SEG_str_mon}}{\leq} \frac{\lambda_t}{2}. \label{eq:operator_bound_x_t_SEG_str_mon}
    \end{eqnarray}
    and
    \begin{eqnarray}
        \|\omega_t\|^2 &\overset{\eqref{eq:a_plus_b}}{\leq}& 2\|\tF_{\Bxi_1}(x^t)\|^2 + 2\|F(x^t)\|^2 \overset{\eqref{eq:operator_bound_x_t_SEG_str_mon}}{\leq} \frac{5}{2}\lambda_t^2 \overset{\eqref{eq:lambda_SEG_str_mon}}{\leq} \frac{\exp(-\gamma\mu t)R^2}{4\gamma^2} \label{eq:omega_bound_x_t_SEG_str_mon}
    \end{eqnarray}
    for all $t = 0, 1, \ldots, T-1$. 
    
    Next, we show that probability event $E_{T-1}$ implies $\|\tx^t - x^*\| \leq 3R$ and derive useful inequalities related to $\theta_t$ for all $t = 0, 1, \ldots, T-1$. Indeed, due to Lipschitzness of $F$ probability event $E_{T-1}$ implies
    \begin{eqnarray}
        \|\tx^t - x^*\|^2 &=& \|x^t - x^* - \gamma \tF_{\Bxi_1}(x^t)\|^2 \overset{\eqref{eq:a_plus_b}}{\leq}  2\|x^t - x^*\|^2 + 2\gamma^2\|\tF_{\Bxi_1}(x^t)\|^2 \notag \\
        &\overset{\eqref{eq:a_plus_b}}{\leq}& 2R_t^2 + 4\gamma^2\|F(x^t)\|^2 + 4\gamma^2\|\omega_t\|^2 \notag \\
        &\overset{\eqref{eq:Lipschitzness}}{\leq}& 2(1 + 2\gamma^2 L^2)R_t^2 + 4\gamma^2\|\omega_t\|^2 \notag \\
        &\overset{\eqref{eq:gamma_SEG_str_mon},\eqref{eq:omega_bound_x_t_SEG_str_mon}}{\leq}& 7\exp(-\gamma\mu t) R^2 \leq 9R^2 \label{eq:tilde_x_distance_SEG_str_mon}
    \end{eqnarray}
    and
    \begin{eqnarray}
        \|F(\tx^t)\| &\leq& L\|\tx^t - x^*\| \leq \sqrt{7}L\exp(- \nicefrac{\gamma\mu t}{2})R \overset{\eqref{eq:gamma_SEG_str_mon},\eqref{eq:lambda_SEG_str_mon}}{\leq} \frac{\lambda_t}{2} \label{eq:operator_bound_tx_t_SEG_str_mon}
    \end{eqnarray}
    for all $t = 0, 1, \ldots, T-1$.
    
    That is, $E_{T-1}$ implies that $x^t, \tx^t \in B_{3R}(x^*)$ for all $t = 0, 1, \ldots, T-1$. Applying Lemma~\ref{lem:optimization_lemma_str_mon_SEG} and $(1 - \gamma\mu)^T \leq \exp(-\gamma\mu T)$, we get that probability event $E_{T-1}$ implies
    \begin{eqnarray}
        R_T^2 &\leq& \exp(-\gamma\mu T)R^2 - 4\gamma^3 \mu \sum\limits_{l=0}^{T-1} (1-\gamma\mu)^{T-1-l}\langle F(x^l), \omega_l \rangle\notag\\
        &&\quad + 2\gamma \sum\limits_{l=0}^{T-1} (1-\gamma\mu)^{T-1-l} \langle x^l - x^* - \gamma F(\tx^l), \theta_l \rangle\notag\\
        &&\quad + \gamma^2 \sum\limits_{l=0}^{T-1} (1-\gamma\mu)^{T-1-l} \left(\|\theta_l\|^2 + 4\|\omega_l\|^2\right). \notag
    \end{eqnarray}
    To estimate the sums in the right-hand side, we introduce new vectors:
    \begin{gather}
        \zeta_t = \begin{cases} F(x^t),& \text{if } \|F(x^t)\| \leq \sqrt{2}L\exp(-\nicefrac{\gamma\mu t}{2})R,\\ 0,& \text{otherwise}, \end{cases} \label{eq:zeta_t_SEG_str_mon}\\
        \eta_t = \begin{cases} x^t - x^* - \gamma F(\tx^t),& \text{if } \|x^t - x^* - \gamma F(\tx^t)\| \leq \sqrt{7}(1 + \gamma L) \exp(- \nicefrac{\gamma\mu t}{2})R,\\ 0,& \text{otherwise}, \end{cases} \label{eq:eta_t_SEG_str_mon}
    \end{gather}
    for $t = 0, 1, \ldots, T-1$. First of all, we point out that vectors $\zeta_t$ and $\eta_t$ are bounded with probability $1$, i.e., with probability $1$
    \begin{equation}
        \|\zeta_t\| \leq \sqrt{2}L\exp(-\nicefrac{\gamma\mu t}{2})R,\quad \|\eta_t\| \leq \sqrt{7}(1 + \gamma L)\exp(-\nicefrac{\gamma\mu t}{2})R \label{eq:zeta_t_eta_t_bound_SEG_str_mon} 
    \end{equation}
    for all $t = 0, 1, \ldots, T-1$. Next, we notice that $E_{T-1}$ implies $\|F(x^t)\| \leq \sqrt{2}L\exp(-\nicefrac{\gamma\mu t}{2})R$ (due to \eqref{eq:operator_bound_x_t_SEG_str_mon}) and
    \begin{eqnarray*}
        \|x^t - x^* - \gamma F(\tx^t)\| &\leq& \|x^t - x^*\| + \gamma \|F(\tx^t)\|\\
        &\overset{\eqref{eq:tilde_x_distance_SEG_str_mon},\eqref{eq:operator_bound_tx_t_SEG_str_mon}}{\leq}& \sqrt{7}(1 + \gamma L)\exp(-\nicefrac{\gamma\mu t}{2})R
    \end{eqnarray*}
    for $t = 0, 1, \ldots, T-1$, i.e., probability event $E_{T-1}$ implies $\zeta_t = F(x^t)$ and $\eta_t = x^t - x^* - \gamma F(\tx^t)$ for all $t = 0,1,\ldots,T-1$. Therefore, $E_{T-1}$ implies 
    \begin{eqnarray}
        R_T^2 &\leq& \exp(-\gamma\mu T)R^2 - 4\gamma^3 \mu \sum\limits_{l=0}^{T-1} (1-\gamma\mu)^{T-1-l}\langle \zeta_l, \omega_l \rangle\notag\\
        &&\quad + 2\gamma \sum\limits_{l=0}^{T-1} (1-\gamma\mu)^{T-1-l} \langle \eta_l, \theta_l \rangle + \gamma^2 \sum\limits_{l=0}^{T-1} (1-\gamma\mu)^{T-1-l} \left(\|\theta_l\|^2 + 4\|\omega_l\|^2\right). \notag
    \end{eqnarray}
    As in the monotone case, to continue the derivation, we introduce vectors $\theta_l^u, \theta_l^b, \omega_l^u, \omega_l^b$ defined as
    \begin{gather}
        \theta_l^u \eqdef \EE_{\Bxi_2^l}\left[\tF_{\Bxi_2^l}(\tx^l)\right] - \tF_{\Bxi_2^l}(\tx^l),\quad \theta_l^b \eqdef F(\tx^l) - \EE_{\Bxi_2^l}\left[\tF_{\Bxi_2^l}(\tx^l)\right], \label{eq:theta_unbias_bias_SEG_str_mon}\\
        \omega_l^u \eqdef \EE_{\Bxi_1^l}\left[\tF_{\Bxi_1^l}(x^l)\right] - \tF_{\Bxi_1^l}(x^l),\quad \omega_l^b \eqdef F(x^l) - \EE_{\Bxi_1^l}\left[\tF_{\Bxi_1^l}(x^l)\right], \label{eq:omega_unbias_bias_SEG_str_mon}
    \end{gather}
    for all $l = 0,\ldots, T-1$. By definition we have $\theta_l = \theta_l^u + \theta_l^b$, $\omega_l = \omega_l^u + \omega_l^b$ for all $l = 0,\ldots, T-1$. Using the introduced notation, we continue our derivation as follows: $E_{T-1}$ implies
    \begin{eqnarray}
        R_T^2 &\overset{\eqref{eq:a_plus_b}}{\leq}& \exp(-\gamma\mu T) R^2 \underbrace{- 4\gamma^3 \mu \sum\limits_{l=0}^{T-1} (1-\gamma\mu)^{T-1-l}\langle \zeta_l, \omega_l^u \rangle}_{\circledOne} \underbrace{- 4\gamma^3 \mu \sum\limits_{l=0}^{T-1} (1-\gamma\mu)^{T-1-l}\langle \zeta_l, \omega_l^b \rangle}_{\circledTwo}\notag\\
        &&\quad + \underbrace{2\gamma \sum\limits_{l=0}^{T-1} (1-\gamma\mu)^{T-1-l} \langle \eta_l, \theta_l^u \rangle}_{\circledThree} + \underbrace{2\gamma \sum\limits_{l=0}^{T-1} (1-\gamma\mu)^{T-1-l} \langle \eta_l, \theta_l^b \rangle}_{\circledFour} \notag\\
        &&\quad + \underbrace{2\gamma^2 \sum\limits_{l=0}^{T-1} (1-\gamma\mu)^{T-1-l} \left(\EE_{\Bxi_2^l}\left[\|\theta_l^u\|^2\right] + 4\EE_{\Bxi_1^l}\left[\|\omega_l^u\|^2\right]\right)}_{\circledFive} \notag\\
        &&\quad + \underbrace{2\gamma^2 \sum\limits_{l=0}^{T-1} (1-\gamma\mu)^{T-1-l} \left(\|\theta_l^u\|^2 + 4\|\omega_l^u\|^2 -\EE_{\Bxi_2^l}\left[\|\theta_l^u\|^2\right] - 4\EE_{\Bxi_1^l}\left[\|\omega_l^u\|^2\right]\right)}_{\circledSix}\notag\\
        &&\quad + \underbrace{2\gamma^2 \sum\limits_{l=0}^{T-1} (1-\gamma\mu)^{T-1-l} \left(\|\theta_l^b\|^2 + 4\|\omega_l^b\|^2\right)}_{\circledSeven}. \label{eq:SEG_str_mon_1234567_bound}
    \end{eqnarray}
    The rest of the proof is based on deriving good enough upper bounds for $\circledOne, \circledTwo, \circledThree, \circledFour, \circledFive, \circledSix, \circledSeven$, i.e., we want to prove that $\circledOne + \circledTwo + \circledThree + \circledFour + \circledFive + \circledSix + \circledSeven \leq \exp(-\gamma\mu T) R^2$ with high probability.
    
    Before we move on, we need to derive some useful inequalities for operating with $\theta_l^u, \theta_l^b, \omega_l^u, \omega_l^b$. First of all, Lemma~\ref{lem:bias_variance} implies that
    \begin{equation}
        \|\theta_l^u\| \leq 2\lambda_l,\quad \|\omega_l^u\| \leq 2\lambda_l \label{eq:theta_omega_magnitude_str_mon}
    \end{equation}
    for all $l = 0,1, \ldots, T-1$. Next, since $\{\xi_1^{i,l}\}_{i=1}^{m_l}$, $\{\xi_2^{i,l}\}_{i=1}^{m_l}$ are independently sampled from $\cD$, we have $\EE_{\Bxi_1^l}[F_{\Bxi_1^l}(x^l)] = F(x^l)$, $\EE_{\Bxi_2^l}[F_{\Bxi_2^l}(\tx^l)] = F(\tx^l)$, and 
    \begin{gather}
        \EE_{\Bxi_1^l}\left[\|F_{\Bxi_1^l}(x^l) - F(x^l)\|^2\right] = \frac{1}{m_l^2}\sum\limits_{i=1}^{m_l} \EE_{\xi_1^{i,l}}\left[\|F_{\xi_1^{i,l}}(x^l) - F(x^l)\|^2\right] \overset{\eqref{eq:UBV}}{\leq} \frac{\sigma^2}{m_l}, \notag\\
        \EE_{\Bxi_2^l}\left[\|F_{\Bxi_2^l}(\tx^l) - F(\tx^l)\|^2\right] = \frac{1}{m_l^2}\sum\limits_{i=1}^{m_l} \EE_{\xi_2^{i,l}}\left[\|F_{\xi_2^{i,l}}(\tx^l) - F(\tx^l)\|^2\right] \overset{\eqref{eq:UBV}}{\leq} \frac{\sigma^2}{m_l}, \notag
    \end{gather}
    for all $l = 0,1, \ldots, T-1$. Moreover, as we already derived, probability event $E_{T-1}$ implies that $\|F(x^l)\| \leq \nicefrac{\lambda_l}{2}$ and $\|F(\tx^l)\| \leq \nicefrac{\lambda_l}{2}$ for all $l = 0,1, \ldots, T-1$ (see \eqref{eq:operator_bound_x_t_SEG_str_mon} and \eqref{eq:operator_bound_tx_t_SEG_str_mon}). Therefore, in view of Lemma~\ref{lem:bias_variance}, $E_{T-1}$ implies that
    \begin{gather}
        \left\|\theta_l^b\right\| \leq \frac{4\sigma^2}{m_l\lambda_l},\quad \left\|\omega_l^b\right\| \leq \frac{4\sigma^2}{m_l\lambda_l}, \label{eq:bias_theta_omega_str_mon}\\
        \EE_{\Bxi_2^l}\left[\left\|\theta_l\right\|^2\right] \leq \frac{18\sigma^2}{m_l},\quad \EE_{\Bxi_1^l}\left[\left\|\omega_l\right\|^2\right] \leq \frac{18\sigma^2}{m_l}, \label{eq:distortion_theta_omega_str_mon}\\
        \EE_{\Bxi_2^l}\left[\left\|\theta_l^u\right\|^2\right] \leq \frac{18\sigma^2}{m_l},\quad \EE_{\Bxi_1^l}\left[\left\|\omega_l^u\right\|^2\right] \leq \frac{18\sigma^2}{m_l}, \label{eq:variance_theta_omega_str_mon}
    \end{gather}
    for all $l = 0,1, \ldots, T-1$.
    
    \paragraph{Upper bound for $\circledOne$.} Since $\EE_{\Bxi_1^l}[\omega_l^u] = 0$, we have
    \begin{equation*}
        \EE_{\Bxi_1^l}\left[-4\gamma^3\mu (1-\gamma\mu)^{T-1-l} \langle \zeta_l, \omega_l^u \rangle\right] = 0.
    \end{equation*}
    Next, the summands in $\circledOne$ are bounded with probability $1$:
    \begin{eqnarray}
        |-4\gamma^3\mu (1-\gamma\mu)^{T-1-l} \langle \zeta_l, \omega_l^u \rangle | &\leq& 4\gamma^3\mu \exp(-\gamma\mu (T - 1 - l)) \|\zeta_l\|\cdot \|\omega_l^u\|\notag\\
        &\overset{\eqref{eq:zeta_t_eta_t_bound_SEG_str_mon},\eqref{eq:theta_omega_magnitude_str_mon}}{\leq}& 8\sqrt{2}\gamma^3\mu L \exp(-\gamma\mu (T - 1 - \nicefrac{l}{2})) R \lambda_l\notag\\
        &\overset{\eqref{eq:gamma_SEG_str_mon},\eqref{eq:lambda_SEG_str_mon}}{\leq}& \frac{\exp(-\gamma\mu T)R^2}{7\ln\tfrac{6(K+1)}{\beta}}  \eqdef c. \label{eq:SEG_str_mon_technical_1_1}
    \end{eqnarray}
    Moreover, these summands have bounded conditional variances $\sigma_l^2 \eqdef \EE_{\Bxi_1^l}\left[16\gamma^6\mu^2 (1-\gamma\mu)^{2T-2-2l} \langle \zeta_l, \omega_l^u \rangle^2\right]$:
    \begin{eqnarray}
        \sigma_l^2 &\leq& \EE_{\Bxi_1^l}\left[16\gamma^6\mu^2 \exp(-\gamma\mu (2T - 2 - 2l)) \|\zeta_l\|^2\cdot \|\omega_l^u\|^2\right]\notag\\
        &\overset{\eqref{eq:zeta_t_eta_t_bound_SEG_str_mon}}{\leq}& 36\gamma^6 \mu^2L^2 \exp(-\gamma\mu (2T - 2 - l)) R^2 \EE_{\Bxi_1^l}\left[\|\omega_l^u\|^2\right]\notag\\
        &\overset{\eqref{eq:gamma_SEG_str_mon}}{\leq}& \frac{4\gamma^2\exp(-\gamma\mu(2T - l))R^2}{2809\ln\tfrac{6(K+1)}{\beta}} \EE_{\Bxi_1^l}\left[\|\omega_l^u\|^2\right]. \label{eq:SEG_str_mon_technical_1_2}
    \end{eqnarray}
    That is, sequence $\{-4\gamma^3\mu (1-\gamma\mu)^{T-1-l} \langle \zeta_l, \omega_l^u \rangle\}_{l\geq 0}$ is a bounded martingale difference sequence having bounded conditional variances $\{\sigma_l^2\}_{l \geq 0}$. Applying Bernstein's inequality (Lemma~\ref{lem:Bernstein_ineq}) with $X_l = -4\gamma^3\mu (1-\gamma\mu)^{T-1-l} \langle \zeta_l, \omega_l^u \rangle$, $c$ defined in \eqref{eq:SEG_str_mon_technical_1_1}, $b = \tfrac{1}{7}\exp(-\gamma\mu T) R^2$, $G = \tfrac{\exp(-2 \gamma\mu T) R^4}{294\ln\frac{6(K+1)}{\beta}}$, we get that
    \begin{equation*}
        \PP\left\{|\circledOne| > \frac{1}{7}\exp(-\gamma\mu T) R^2 \text{ and } \sum\limits_{l=0}^{T-1}\sigma_l^2 \leq \frac{\exp(-2\gamma\mu T) R^4}{294\ln\tfrac{6(K+1)}{\beta}}\right\} \leq 2\exp\left(- \frac{b^2}{2G + \nicefrac{2cb}{3}}\right) = \frac{\beta}{3(K+1)}.
    \end{equation*}
    In other words, $\PP\{E_{\circledOne}\} \geq 1 - \tfrac{\beta}{3(K+1)}$, where probability event $E_{\circledOne}$ is defined as
    \begin{equation}
        E_{\circledOne} = \left\{\text{either} \quad \sum\limits_{l=0}^{T-1}\sigma_l^2 > \frac{\exp(-2\gamma\mu T) R^4}{294\ln\tfrac{6(K+1)}{\beta}}\quad \text{or}\quad |\circledOne| \leq \frac{1}{7}\exp(-\gamma\mu T) R^2\right\}. \label{eq:bound_1_SEG_str_mon}
    \end{equation}
    Moreover, we notice here that probability event $E_{T-1}$ implies that
    \begin{eqnarray}
        \sum\limits_{l=0}^{T-1}\sigma_l^2 &\overset{\eqref{eq:SEG_str_mon_technical_1_2}}{\leq}& \frac{4\gamma^2\exp(-2\gamma\mu T)R^2}{2809 \ln\tfrac{6(K+1)}{\beta}} \sum\limits_{l=0}^{T-1} \frac{\EE_{\Bxi_1^l}\left[\|\omega_l^u\|^2\right]}{\exp(-\gamma\mu l)}\notag\\ &\overset{\eqref{eq:variance_theta_omega_str_mon}, T \leq K+1}{\leq}& \frac{72\gamma^2\exp(-2\gamma\mu T) R^2 \sigma^2}{2809 \ln\tfrac{6(K+1)}{\beta}} \sum\limits_{l=0}^{K} \frac{1}{m_l\exp(-\gamma\mu l)}\notag\\
        &\overset{\eqref{eq:batch_SEG_str_mon}}{\leq}& \frac{\exp(-2\gamma\mu T)R^4}{294\ln\tfrac{6(K+1)}{\beta}}. \label{eq:bound_1_variances_SEG_str_mon}
    \end{eqnarray}

    \paragraph{Upper bound for $\circledTwo$.} Probability event $E_{T-1}$ implies
    \begin{eqnarray}
        \circledTwo &\leq& 4\gamma^3 \mu \sum\limits_{l=0}^{T-1} \exp(-\gamma\mu (T-1-l)) \|\zeta_l\| \cdot \|\omega_l^b\| \notag\\
        &\overset{\eqref{eq:zeta_t_eta_t_bound_SEG_str_mon},\eqref{eq:bias_theta_omega_str_mon}}{\leq}& 16\sqrt{2} \exp(-\gamma\mu (T-1)) \gamma^3 \mu L R \sum\limits_{l=0}^{T-1} \frac{\sigma^2}{m_l \lambda_l \exp(-\nicefrac{\gamma\mu l}{2})} \notag\\
        &\overset{\eqref{eq:lambda_SEG_str_mon}}{=}& 1920\sqrt{2} \exp(-\gamma\mu (T-2)) \gamma^4 \mu L \sum\limits_{l=0}^{T-1} \frac{\sigma^2 \ln\tfrac{6(K+1)}{\beta}}{m_l\exp(-\gamma\mu l)} \notag\\
        &\overset{\eqref{eq:gamma_SEG_str_mon},\eqref{eq:batch_SEG_str_mon}, T \leq K+1}{\leq}& \frac{1}{7}\exp(-\gamma\mu T) R^2. \label{eq:bound_2_SEG_str_mon}
    \end{eqnarray}

    \paragraph{Upper bound for $\circledThree$.} Since $\EE_{\Bxi_2^l}[\theta_l^u] = 0$, we have
    \begin{equation*}
        \EE_{\Bxi_2^l}\left[2\gamma (1-\gamma\mu)^{T-1-l} \langle \eta_l, \theta_l^u \rangle\right] = 0.
    \end{equation*}
    Next, the summands in $\circledThree$ are bounded with probability $1$:
    \begin{eqnarray}
        |2\gamma (1-\gamma\mu)^{T-1-l} \langle \eta_l, \theta_l^u \rangle | &\leq& 2\gamma\exp(-\gamma\mu (T - 1 - l)) \|\eta_l\|\cdot \|\theta_l^u\|\notag\\
        &\overset{\eqref{eq:zeta_t_eta_t_bound_SEG_str_mon},\eqref{eq:theta_omega_magnitude_str_mon}}{\leq}& 4\sqrt{7}\gamma (1 + \gamma L) \exp(-\gamma\mu (T - 1 - \nicefrac{l}{2})) R \lambda_l\notag\\
        &\overset{\eqref{eq:gamma_SEG_str_mon},\eqref{eq:lambda_SEG_str_mon}}{\leq}& \frac{\exp(-\gamma\mu T)R^2}{7\ln\tfrac{6(K+1)}{\beta}} \eqdef c. \label{eq:SEG_str_mon_technical_3_1}
    \end{eqnarray}
    Moreover, these summands have bounded conditional variances $\widetilde\sigma_l^2 \eqdef \EE_{\Bxi_2^l}\left[4\gamma^2 (1-\gamma\mu)^{2T-2-2l} \langle \eta_l, \theta_l^u \rangle^2\right]$:
    \begin{eqnarray}
        \widetilde\sigma_l^2 &\leq& \EE_{\Bxi_2^l}\left[4\gamma^2\exp(-\gamma\mu (2T - 2 - 2l)) \|\eta_l\|^2\cdot \|\theta_l^u\|^2\right]\notag\\
        &\overset{\eqref{eq:zeta_t_eta_t_bound_SEG_str_mon}}{\leq}& 49\gamma^2 (1 + \gamma L)^2 \exp(-\gamma\mu (2T - 2 - l)) R^2 \EE_{\Bxi_2^l}\left[\|\theta_l^u\|^2\right]\notag\\
        &\overset{\eqref{eq:gamma_SEG_str_mon}}{\leq}& 50\gamma^2\exp(-\gamma\mu (2T - l))R^2 \EE_{\Bxi_2^l}\left[\|\theta_l^u\|^2\right]. \label{eq:SEG_str_mon_technical_3_2}
    \end{eqnarray}
    That is, sequence $\{2\gamma (1-\gamma\mu)^{T-1-l} \langle \eta_l, \theta_l^u \rangle\}_{l\geq 0}$ is a bounded martingale difference sequence having bounded conditional variances $\{\widetilde\sigma_l^2\}_{l \geq 0}$. Applying Bernstein's inequality (Lemma~\ref{lem:Bernstein_ineq}) with $X_l = 2\gamma (1-\gamma\mu)^{T-1-l} \langle \eta_l, \theta_l^u \rangle$, $c$ defined in \eqref{eq:SEG_str_mon_technical_3_1}, $b = \tfrac{1}{7}\exp(-\gamma\mu T) R^2$, $G = \tfrac{\exp(- 2\gamma\mu T) R^4}{294\ln\frac{6(K+1)}{\beta}}$, we get that
    \begin{equation*}
        \PP\left\{|\circledThree| > \frac{1}{7}\exp(-\gamma\mu T) R^2 \text{ and } \sum\limits_{l=0}^{T-1}\widetilde\sigma_l^2 \leq \frac{\exp(- 2\gamma\mu T) R^4}{294\ln\tfrac{6(K+1)}{\beta}}\right\} \leq 2\exp\left(- \frac{b^2}{2G + \nicefrac{2cb}{3}}\right) = \frac{\beta}{3(K+1)}.
    \end{equation*}
    In other words, $\PP\{E_{\circledThree}\} \geq 1 - \tfrac{\beta}{3(K+1)}$, where probability event $E_{\circledThree}$ is defined as
    \begin{equation}
        E_{\circledThree} = \left\{\text{either} \quad \sum\limits_{l=0}^{T-1}\widetilde\sigma_l^2 > \frac{\exp(- 2\gamma\mu T) R^4}{294\ln\tfrac{6(K+1)}{\beta}}\quad \text{or}\quad |\circledThree| \leq \frac{1}{7}\exp(-\gamma\mu T) R^2\right\}. \label{eq:bound_3_SEG_str_mon}
    \end{equation}
    Moreover, we notice here that probability event $E_{T-1}$ implies that
    \begin{eqnarray}
        \sum\limits_{l=0}^{T-1}\widetilde\sigma_l^2 &\overset{\eqref{eq:SEG_str_mon_technical_3_2}}{\leq}& 50\gamma^2\exp(- 2\gamma\mu T)R^2\sum\limits_{l=0}^{T-1} \frac{\EE_{\Bxi_2^l}\left[\|\theta_l^u\|^2\right]}{\exp(-\gamma\mu l)}\notag\\ &\overset{\eqref{eq:variance_theta_omega_str_mon}, T \leq K+1}{\leq}& 900\gamma^2\exp(-2\gamma\mu T) R^2 \sigma^2 \sum\limits_{l=0}^{K} \frac{1}{m_l\exp(-\gamma\mu l)}\notag\\
        &\overset{\eqref{eq:batch_SEG_str_mon}}{\leq}& \frac{\exp(-2\gamma\mu T)R^4}{294\ln\tfrac{6(K+1)}{\beta}}. \label{eq:bound_3_variances_SEG_str_mon}
    \end{eqnarray}

    \paragraph{Upper bound for $\circledFour$.} Probability event $E_{T-1}$ implies
    \begin{eqnarray}
        \circledFour &\leq& 2\gamma \exp(-\gamma\mu (T-1)) \sum\limits_{l=0}^{T-1} \frac{\|\eta_l\|\cdot \|\theta_l^b\|}{\exp(-\gamma\mu l)}\notag\\
        &\overset{\eqref{eq:zeta_t_eta_t_bound_SEG_str_mon}, \eqref{eq:bias_theta_omega_str_mon}}{\leq}& 8\sqrt{7} \gamma (1+\gamma L) \exp(-\gamma\mu (T-1)) R \sum\limits_{l=0}^{T-1} \frac{\sigma^2}{m_l \lambda_l \exp(-\nicefrac{\gamma\mu l}{2})}\notag\\
        &\overset{\eqref{eq:lambda_SEG_str_mon}}{\leq}& 960\sqrt{7} \gamma^2(1+\gamma L) \exp(-\gamma\mu (T-2)) \sum\limits_{l=0}^{T-1} \frac{\sigma^2 \ln\tfrac{6(K+1)}{\beta}}{m_l \exp(-\gamma\mu l)}\notag \\
        &\overset{\eqref{eq:batch_SEG_str_mon}, T \leq K+1}{\leq}& \frac{1}{7}\exp(-\gamma\mu T) R^2. \label{eq:bound_4_SEG_str_mon}
    \end{eqnarray}

    \paragraph{Upper bound for $\circledFive$.} Probability event $E_{T-1}$ implies
    \begin{eqnarray}
        \circledFive &=& 2\gamma^2 \exp(-\gamma\mu (T-1)) \sum\limits_{l=0}^{T-1} \frac{\EE_{\Bxi_2^l}\left[\|\theta_l^u\|^2\right] + 4\EE_{\Bxi_1^l}\left[\|\omega_l^u\|^2\right]}{\exp(-\gamma\mu l)} \notag\\
        &\overset{\eqref{eq:variance_theta_omega_str_mon}}{\leq}& 180\gamma^2\exp(-\gamma\mu (T-1)) \sum\limits_{l=0}^{T-1} \frac{\sigma^2}{m_l \exp(-\gamma\mu l)} \notag\\
        &\overset{\eqref{eq:batch_SEG_str_mon}, T \leq K+1}{\leq}& \frac{1}{7} \exp(-\gamma\mu T) R^2. \label{eq:bound_5_SEG_str_mon}
    \end{eqnarray}

    \paragraph{Upper bound for $\circledSix$.} First of all, we have
    \begin{equation*}
        2\gamma^2 (1-\gamma\mu)^{T-1-l}\EE_{\Bxi_1^l, \Bxi_2^l}\left[\|\theta_l^u\|^2 + 4\|\omega_l^u\|^2 -\EE_{\Bxi_2^l}\left[\|\theta_l^u\|^2\right] - 4\EE_{\Bxi_1^l}\left[\|\omega_l^u\|^2\right]\right] = 0.
    \end{equation*}
    Next, the summands in $\circledSix$ are bounded with probability $1$:
    \begin{eqnarray}
        2\gamma^2 (1-\gamma\mu)^{T-1-l}\left| \|\theta_l^u\|^2 + 4\|\omega_l^u\|^2 -\EE_{\Bxi_2^l}\left[\|\theta_l^u\|^2\right] - 4\EE_{\Bxi_1^l}\left[\|\omega_l^u\|^2\right] \right| 
        &\overset{\eqref{eq:theta_omega_magnitude_str_mon}}{\leq}& \frac{80\gamma^2 \exp(-\gamma\mu T) \lambda_l^2}{\exp(-\gamma\mu (1+l))}\notag\\
        &\overset{\eqref{eq:lambda_SEG_str_mon}}{\leq}& \frac{\exp(-\gamma\mu T)R^2}{7\ln\tfrac{6(K+1)}{\beta}}\notag\\
        &\eqdef& c. \label{eq:SEG_str_mon_technical_6_1}
    \end{eqnarray}
    Moreover, these summands have bounded conditional variances $\widehat\sigma_l^2 \eqdef \EE_{\Bxi_1^l,\Bxi_2^l}\left[4\gamma^4 (1-\gamma\mu)^{2T-2-2l} \left| \|\theta_l^u\|^2 + 4\|\omega_l^u\|^2 -\EE_{\Bxi_2^l}\left[\|\theta_l^u\|^2\right] - 4\EE_{\Bxi_1^l}\left[\|\omega_l^u\|^2\right] \right|^2\right]$:
    \begin{eqnarray}
        \widehat\sigma_l^2 &\overset{\eqref{eq:SEG_str_mon_technical_6_1}}{\leq}& \frac{2\gamma^2\exp(-2\gamma\mu T)R^2}{7\exp(-\gamma\mu (1+l))\ln\tfrac{6(K+1)}{\beta}} \EE_{\Bxi_1^l,\Bxi_2^l}\left[\left| \|\theta_l^u\|^2 + 4\|\omega_l^u\|^2 -\EE_{\Bxi_2^l}\left[\|\theta_l^u\|^2\right] - 4\EE_{\Bxi_1^l}\left[\|\omega_l^u\|^2\right] \right|\right]\notag\\
        &\leq& \frac{4\gamma^2\exp(-2\gamma\mu T)R^2}{7\exp(-\gamma\mu (1+l))\ln\tfrac{6(K+1)}{\beta}} \EE_{\Bxi_1^l,\Bxi_2^l}\left[\|\theta_l^u\|^2 + 4\|\omega_l^u\|^2\right]. \label{eq:SEG_str_mon_technical_6_2}
    \end{eqnarray}
    That is, sequence $\left\{2\gamma^2 (1-\gamma\mu)^{T-1-l}\left( \|\theta_l^u\|^2 + 4\|\omega_l^u\|^2 -\EE_{\Bxi_2^l}\left[\|\theta_l^u\|^2\right] - 4\EE_{\Bxi_1^l}\left[\|\omega_l^u\|^2\right]\right)\right\}_{l\geq 0}$ is a bounded martingale difference sequence having bounded conditional variances $\{\widehat\sigma_l^2\}_{l \geq 0}$. Applying Bernstein's inequality (Lemma~\ref{lem:Bernstein_ineq}) with $X_l = 2\gamma^2 (1-\gamma\mu)^{T-1-l}\left( \|\theta_l^u\|^2 + 4\|\omega_l^u\|^2 -\EE_{\Bxi_2^l}\left[\|\theta_l^u\|^2\right] - 4\EE_{\Bxi_1^l}\left[\|\omega_l^u\|^2\right]\right)$, $c$ defined in \eqref{eq:SEG_str_mon_technical_6_1}, $b = \tfrac{1}{7}\exp(-\gamma\mu T) R^2$, $G = \tfrac{\exp(-2\gamma\mu T) R^4}{294\ln\frac{6(K+1)}{\beta}}$, we get that
    \begin{equation*}
        \PP\left\{|\circledSix| > \frac{1}{7}\exp(-\gamma\mu T) R^2 \text{ and } \sum\limits_{l=0}^{T-1}\widehat\sigma_l^2 \leq \frac{\exp(-2\gamma\mu T) R^4}{294\ln\frac{6(K+1)}{\beta}}\right\} \leq 2\exp\left(- \frac{b^2}{2G + \nicefrac{2cb}{3}}\right) = \frac{\beta}{3(K+1)}.
    \end{equation*}
    In other words, $\PP\{E_{\circledSix}\} \geq 1 - \tfrac{\beta}{3(K+1)}$, where probability event $E_{\circledSix}$ is defined as
    \begin{equation}
        E_{\circledSix} = \left\{\text{either} \quad \sum\limits_{l=0}^{T-1}\widehat\sigma_l^2 > \frac{\exp(-2\gamma\mu T) R^4}{294\ln\tfrac{6(K+1)}{\beta}}\quad \text{or}\quad |\circledSix| \leq \frac{1}{7}\exp(-\gamma\mu T) R^2\right\}. \label{eq:bound_6_SEG_str_mon}
    \end{equation}
    Moreover, we notice here that probability event $E_{T-1}$ implies that
    \begin{eqnarray}
        \sum\limits_{l=0}^{T-1}\widehat\sigma_l^2 &\overset{\eqref{eq:SEG_str_mon_technical_6_2}}{\leq}& \frac{4\gamma^2\exp(-\gamma\mu (2T-1))R^2}{7\ln\tfrac{6(K+1)}{\beta}} \sum\limits_{l=0}^{T-1} \frac{\EE_{\Bxi_1^l,\Bxi_2^l}\left[\|\theta_l^u\|^2 + 4\|\omega_l^u\|^2\right]}{\exp(-\gamma\mu l)}\notag\\ &\overset{\eqref{eq:variance_theta_omega_str_mon}, T \leq K+1}{\leq}& \frac{360\gamma^2\exp(-\gamma\mu (2T-1)) R^2 \sigma^2}{7\ln\tfrac{6(K+1)}{\beta}} \sum\limits_{l=0}^{K} \frac{1}{m_l\exp(-\gamma\mu l)}\notag\\
        &\overset{\eqref{eq:batch_SEG_str_mon}}{\leq}& \frac{\exp(-2\gamma\mu T)R^4}{294\ln\tfrac{6(K+1)}{\beta}}. \label{eq:bound_6_variances_SEG_str_mon}
    \end{eqnarray}

    \paragraph{Upper bound for $\circledSeven$.} Probability event $E_{T-1}$ implies
    \begin{eqnarray}
        \circledSeven &=&  2\gamma^2 \sum\limits_{l=0}^{T-1} \exp(-\gamma\mu (T-1-l)) \left(\|\theta_l^b\|^2 + 4\|\omega_l^b\|^2\right)\notag\\
        &\overset{\eqref{eq:bias_theta_omega_str_mon}}{\leq}& 160\gamma^2 \exp(-\gamma\mu (T-1)) \sum\limits_{l=0}^{T-1} \frac{\sigma^4}{m_l^2 \lambda_l^2 \exp(-\gamma\mu l)} \notag\\
        &\overset{\eqref{eq:lambda_SEG_str_mon}}{=}& 2304000 \gamma^4 \exp(-\gamma\mu (T-3)) \sum\limits_{l=0}^{T-1} \frac{\sigma^4 \ln^2\tfrac{6(K+1)}{\beta}}{m_l^2R^2\exp(-2\gamma\mu l)} \notag\\
        &\overset{\eqref{eq:batch_SEG_str_mon}, T \leq K+1}{\leq}& \frac{1}{7}\exp(-\gamma\mu T) R^2. \label{eq:bound_7_SEG_str_mon}
    \end{eqnarray}

    \paragraph{Final derivation.} Putting all bounds together, we get that $E_{T-1}$ implies
    \begin{gather*}
        R_T^2 \overset{\eqref{eq:SEG_str_mon_1234567_bound}}{\leq} \exp(-\gamma\mu T) R^2 + \circledOne + \circledTwo + \circledThree + \circledFour + \circledFive + \circledSix + \circledSeven,\\
        \circledTwo \overset{\eqref{eq:bound_2_SEG_str_mon}}{\leq} \frac{1}{7}\exp(-\gamma\mu T)R^2,\quad \circledFour \overset{\eqref{eq:bound_4_SEG_str_mon}}{\leq} \frac{1}{7}\exp(-\gamma\mu T)R^2,\\ \circledFive \overset{\eqref{eq:bound_5_SEG_str_mon}}{\leq} \frac{1}{7}\exp(-\gamma\mu T)R^2,\quad \circledSeven \overset{\eqref{eq:bound_7_SEG_str_mon}}{\leq} \frac{1}{7}\exp(-\gamma\mu T)R^2,\\
        \sum\limits_{l=0}^{T-1}\sigma_l^2 \overset{\eqref{eq:bound_1_variances_SEG_str_mon}}{\leq}  \frac{\exp(-2\gamma\mu T)R^4}{294\ln\tfrac{6(K+1)}{\beta}},\quad \sum\limits_{l=0}^{T-1}\widetilde\sigma_l^2 \overset{\eqref{eq:bound_3_variances_SEG_str_mon}}{\leq} \frac{\exp(-2\gamma\mu T)R^4}{294\ln\tfrac{6(K+1)}{\beta}},\quad \sum\limits_{l=0}^{T-1}\widehat\sigma_l^2 \overset{\eqref{eq:bound_6_variances_SEG_str_mon}}{\leq}  \frac{\exp(-2\gamma\mu T)R^4}{294\ln\tfrac{6(K+1)}{\beta}}.
    \end{gather*}
    Moreover, in view of \eqref{eq:bound_1_SEG_str_mon}, \eqref{eq:bound_3_SEG_str_mon}, \eqref{eq:bound_6_SEG_str_mon}, and our induction assumption, we have
    \begin{gather*}
        \PP\{E_{T-1}\} \geq 1 - \frac{(T-1)\beta}{K+1},\\
        \PP\{E_{\circledOne}\} \geq 1 - \frac{\beta}{3(K+1)}, \quad \PP\{E_{\circledThree}\} \geq 1 - \frac{\beta}{3(K+1)}, \quad \PP\{E_{\circledSix}\} \geq 1 - \frac{\beta}{3(K+1)} ,
    \end{gather*}
    where probability events $E_{\circledOne}$, $E_{\circledThree}$, and $E_{\circledSix}$ are defined as
    \begin{eqnarray}
        E_{\circledOne}&=& \left\{\text{either} \quad \sum\limits_{l=0}^{T-1}\sigma_l^2 > \frac{\exp(-2\gamma\mu T) R^4}{294\ln\tfrac{6(K+1)}{\beta}}\quad \text{or}\quad |\circledOne| \leq \frac{1}{7}\exp(-\gamma\mu T) R^2\right\},\notag\\
        E_{\circledThree}&=& \left\{\text{either} \quad \sum\limits_{l=0}^{T-1}\widetilde\sigma_l^2 > \frac{\exp(-2\gamma\mu T) R^4}{294\ln\tfrac{6(K+1)}{\beta}}\quad \text{or}\quad |\circledThree| \leq \frac{1}{7}\exp(-\gamma\mu T) R^2\right\},\notag\\
        E_{\circledSix}&=& \left\{\text{either} \quad \sum\limits_{l=0}^{T-1}\widehat\sigma_l^2 > \frac{\exp(-2\gamma\mu T) R^4}{294\ln\tfrac{6(K+1)}{\beta}}\quad \text{or}\quad |\circledSix| \leq \frac{1}{7}\exp(-\gamma\mu T) R^2\right\}.\notag
    \end{eqnarray}
    Putting all of these inequalities together, we obtain that probability event $E_{T-1} \cap E_{\circledOne} \cap E_{\circledThree} \cap E_{\circledSix}$ implies
    \begin{eqnarray*}
        R_T^2 &\overset{\eqref{eq:SEG_str_mon_1234567_bound}}{\leq}& \exp(-\gamma\mu T) R^2 + \circledOne + \circledTwo + \circledThree + \circledFour + \circledFive + \circledSix + \circledSeven\\
        &\leq& 2\exp(-\gamma\mu T) R^2.
    \end{eqnarray*}
    Moreover, union bound for the probability events implies
    \begin{equation}
        \PP\{E_T\} \geq \PP\{E_{T-1} \cap E_{\circledOne} \cap E_{\circledThree} \cap E_{\circledSix}\} = 1 - \PP\{\overline{E}_{T-1} \cup \overline{E}_{\circledOne} \cup \overline{E}_{\circledThree} \cup \overline{E}_{\circledSix}\} \geq 1 - \frac{T\beta}{K+1}.
    \end{equation}
    This is exactly what we wanted to prove (see the paragraph after inequality \eqref{eq:induction_inequality_str_mon_SEG}). In particular, with probability at least $1 - \beta$ satisfy we have
    \begin{equation}
        \|x^{K+1} - x^*\|^2 \leq 2\exp(-\gamma\mu (K+1))R^2, \notag
    \end{equation}
    which finishes the proof.
\end{proof}

\begin{corollary}\label{cor:main_result_SEG_str_mon}
    Let the assumptions of Theorem~\ref{thm:main_result_str_mon_SEG} hold. Then, the following statements hold.
    \begin{enumerate}
        \item \textbf{Large stepsize/large batch.} The choice of stepsize and batchsize
        \begin{equation}
            \gamma = \frac{1}{650 L \ln \tfrac{6(K+1)}{\beta}},\quad m_k = \max\left\{1, \frac{264600\gamma^2 (K+1) \sigma^2\ln \tfrac{6(K+1)}{\beta}}{\exp(-\gamma\mu k) R^2}\right\} \label{eq:str_mon_SEG_large_step_large_batch}
        \end{equation}
        satisfies conditions \eqref{eq:gamma_SEG_str_mon} and \eqref{eq:batch_SEG_str_mon}. With such choice of $\gamma, m_k$, and the choice of $\lambda_k$ as in \eqref{eq:lambda_SEG_str_mon}, the iterates produced by \ref{eq:clipped_SEG} after $K$ iterations with probability at least $1-\beta$ satisfy
        \begin{equation}
            \|x^{K+1} - x^*\|^2 \leq 2\exp\left( - \frac{\mu(K+1)}{650 L \ln \tfrac{6(K+1)}{\beta}}\right)R^2. \label{eq:main_result_str_mon_SEG_large_batch}
        \end{equation}
        In particular, to guarantee $\|x^{K+1} - x^*\|^2 \leq \varepsilon$ with probability at least $1-\beta$ for some $\varepsilon > 0$ \ref{eq:clipped_SEG} requires
        \begin{gather}
            \cO\left(\frac{L}{\mu} \ln\left(\frac{R^2}{\varepsilon}\right)\ln\left(\frac{L}{\mu \beta}\ln\left(\frac{R^2}{\varepsilon}\right)\right)\right) \text{ iterations}, \label{eq:str_mon_SEG_iteration_complexity_large_batch}\\
            \cO\left(\max\left\{\frac{L}{\mu}, \frac{\sigma^2}{\mu^2 \varepsilon}\right\} \ln\left(\frac{R^2}{\varepsilon}\right) \ln\left(\frac{L}{\mu\beta}\ln\left(\frac{R^2}{\varepsilon}\right)\right)\right) \text{ oracle calls}. \label{eq:str_mon_SEG_oracle_complexity_large_batch} 
        \end{gather}
        
        \item \textbf{Small stepsize/small batch.} The choice of stepsize and batchsize
        \begin{equation}
            \gamma = \min\left\{\frac{1}{650 L \ln \tfrac{6(K+1)}{\beta}}, \frac{\ln\left(B_K\right)}{\mu (K+1)}\right\},\quad m_k \equiv 1 \label{eq:str_mon_SEG_small_step_small_batch}
        \end{equation}
        satisfies conditions \eqref{eq:gamma_SEG_str_mon} and \eqref{eq:batch_SEG_str_mon}, where $B_K = \max\left\{2, \frac{(K+1)\mu^2 R^2}{264600\sigma^2\ln\left(\frac{6(K+1)}{\beta}\right)\ln^2(B_K)}\right\} = \cO\left(\max\left\{2, \frac{(K+1)\mu^2 R^2}{264600\sigma^2\ln\left(\frac{6(K+1)}{\beta}\right)\ln^2\left(\max\left\{2, \frac{(K+1)\mu^2 R^2}{264600\sigma^2\ln\left(\frac{6(K+1)}{\beta}\right)}\right\}\right)}\right\}\right)$. With such choice of $\gamma, m_k$, and the choice of $\lambda_k$ as in \eqref{eq:lambda_SEG_str_mon}, the iterates produced by \ref{eq:clipped_SEG} after $K$ iterations with probability at least $1-\beta$ satisfy
        \begin{equation}
            \|x^{K+1} - x^*\|^2 \leq \max\left\{2\exp\left( - \frac{\mu(K+1)}{650 L \ln \tfrac{6(K+1)}{\beta}}\right)R^2, \frac{529200\sigma^2\ln\left(\frac{6(K+1)}{\beta}\right) \ln^2 (B_K)}{\mu^2(K+1)}\right\}. \label{eq:main_result_str_mon_SEG_small_batch}
        \end{equation}
        In particular, to guarantee $\|x^{K+1} - x^*\|^2 \leq \varepsilon$ with probability at least $1-\beta$ for some $\varepsilon > 0$ \ref{eq:clipped_SEG} requires
        \begin{equation}
            \cO\left(\max\left\{\frac{L}{\mu} \ln\left(\frac{R^2}{\varepsilon}\right)\ln\left(\frac{L}{\mu\beta}\ln\left(\frac{R^2}{\varepsilon}\right)\right), \frac{\sigma^2}{\mu^2 \varepsilon}\ln\left(\frac{\sigma^2}{\mu^2 \varepsilon\beta}\right)\ln^2\left(B_\varepsilon\right)\right\} \right) \label{eq:str_mon_SEG_iteration_oracle_complexity_small_batch} 
        \end{equation}
        iterations/oracle calls, where
        \begin{equation*}
            B_\varepsilon = \max\left\{2, \frac{R^2}{\varepsilon \ln\left(\frac{\sigma^2}{\mu^2 \varepsilon\beta}\right) \ln^2\left(\max\left\{2 , \frac{R^2}{\varepsilon \ln\left(\frac{\sigma^2}{\mu^2 \varepsilon\beta}\right)}\right\}\right)} \right\}.
        \end{equation*}
    \end{enumerate}
\end{corollary}
\begin{proof}
    \begin{enumerate}
        \item \textbf{Large stepsize/large batch.} First of all, it is easy to see that the choice of $\gamma$ and $m_k$ from \eqref{eq:str_mon_SEG_large_step_large_batch} satisfies conditions \eqref{eq:gamma_SEG_str_mon} and \eqref{eq:batch_SEG_str_mon}. Therefore, applying Theorem~\ref{thm:main_result_str_mon_SEG}, we derive that with probability at least $1-\beta$
        \begin{equation*}
            \|x^{K+1} - x^*\|^2 \leq 2\exp(-\gamma\mu(K+1))R^2 \overset{\eqref{eq:str_mon_SEG_large_step_large_batch}}{=} 2\exp\left(- \frac{\mu (K+1)}{650 L \ln \tfrac{6(K+1)}{\beta}}\right)R^2.
        \end{equation*}
        To guarantee $\|x^{K+1} - x^*\|^2 \leq \varepsilon$, we choose $K$ in such a way that the right-hand side of the above inequality is smaller than $\varepsilon$ that gives
        \begin{eqnarray*}
            K = \cO\left(\frac{L}{\mu} \ln\left(\frac{R^2}{\varepsilon}\right)\ln\left(\frac{L}{\mu \beta}\ln\left(\frac{R^2}{\varepsilon}\right)\right)\right).
        \end{eqnarray*}
        The total number of oracle calls equals
        \begin{eqnarray*}
            \sum\limits_{k=0}^{K}2m_k &\overset{\eqref{eq:str_mon_SEG_large_step_large_batch}}{=}&  2\sum\limits_{k=0}^{K}\max\left\{1, \frac{264600\gamma^2 (K+1) \sigma^2\ln \tfrac{6(K+1)}{\beta}}{\exp(-\gamma\mu k) R^2}\right\}\\
            &=& \cO\left(\max\left\{K,\frac{\gamma(K+1)\exp(\gamma\mu(K+1))\sigma^2\ln \tfrac{6(K+1)}{\beta}}{\mu R^2}\right\}\right)\\
            &=& \cO\left(\max\left\{\frac{L}{\mu}, \frac{\sigma^2}{\mu^2 \varepsilon}\right\} \ln\left(\frac{R^2}{\varepsilon}\right) \ln\left(\frac{L}{\mu\beta}\ln\left(\frac{R^2}{\varepsilon}\right)\right)\right).
        \end{eqnarray*}
        
        \item \textbf{Small stepsize/small batch.} First of all, we verify that the choice of $\gamma$ and $m_k$ from \eqref{eq:str_mon_SEG_small_step_small_batch} satisfies conditions \eqref{eq:gamma_SEG_str_mon} and \eqref{eq:batch_SEG_str_mon}: \eqref{eq:gamma_SEG_str_mon} trivially holds and \eqref{eq:batch_SEG_str_mon} holds since for all $k = 0,\ldots, K$
        \begin{eqnarray*}
            \frac{264600\gamma^2 (K+1) \sigma^2\ln \tfrac{6(K+1)}{\beta}}{\exp(-\gamma\mu k) R^2} &\leq& \frac{264600\gamma^2 (K+1) \sigma^2\ln \tfrac{6(K+1)}{\beta}}{\exp(-\gamma\mu (K+1))R^2} \\
            &\overset{\eqref{eq:str_mon_SEG_small_step_small_batch}}{\leq}& \frac{264600\ln^2\left(B_K\right) \exp(\gamma\mu(K+1)) \sigma^2\ln \tfrac{6(K+1)}{\beta}}{\mu^2 (K+1) R^2}\\
            &\overset{\eqref{eq:str_mon_SEG_small_step_small_batch}}{\leq}& 1.
        \end{eqnarray*}
        Therefore, applying Theorem~\ref{thm:main_result_str_mon_SEG}, we derive that with probability at least $1-\beta$
        \begin{eqnarray*}
            \|x^{K+1} - x^*\|^2 &\leq& 2\exp(-\gamma\mu(K+1))R^2\\
            &\overset{\eqref{eq:str_mon_SEG_small_step_small_batch}}{=}& \max\left\{2\exp\left(- \frac{\mu (K+1)}{650 L \ln \tfrac{6(K+1)}{\beta}}\right)R^2, \frac{2R^2}{B_K}\right\}\\
            &=& \max\left\{2\exp\left(- \frac{\mu (K+1)}{650 L \ln \tfrac{6(K+1)}{\beta}}\right)R^2, \frac{529200\sigma^2\ln\left(\frac{6(K+1)}{\beta}\right) \ln^2 (B_K)}{\mu^2(K+1)}\right\}.
        \end{eqnarray*}
        To guarantee $\|x^{K+1} - x^*\|^2 \leq \varepsilon$, we choose $K$ in such a way that the right-hand side of the above inequality is smaller than $\varepsilon$ that gives $K$ of the order
        \begin{eqnarray*}
            \cO\left(\max\left\{\frac{L}{\mu} \ln\left(\frac{R^2}{\varepsilon}\right)\ln\left(\frac{L}{\mu\beta}\ln\left(\frac{R^2}{\varepsilon}\right)\right), \frac{\sigma^2}{\mu^2 \varepsilon}\ln\left(\frac{\sigma^2}{\mu^2 \varepsilon\beta}\right)\ln^2\left(B_\varepsilon\right)\right\} \right),
        \end{eqnarray*}
        where
        \begin{equation*}
            B_\varepsilon = \max\left\{2, \frac{R^2}{\varepsilon \ln\left(\frac{\sigma^2}{\mu^2 \varepsilon\beta}\right) \ln^2\left(\max\left\{2 , \frac{R^2}{\varepsilon \ln\left(\frac{\sigma^2}{\mu^2 \varepsilon\beta}\right)}\right\}\right)} \right\}.
        \end{equation*}
        The total number of oracle calls equals $\sum_{k=0}^K 2m_k = 2(K+1)$.
    \end{enumerate}
\end{proof}

\newpage

\section{Clipped Stochastic Gradient Descent-Ascent: Missing Proofs and Details}\label{app:clipped_SGDA_proofs}

\subsection{Monotone Star-Cocoercive Case}

\begin{lemma}\label{lem:optimization_lemma_SGDA_gap}
    Let Assumption~\ref{as:monotonicity} hold for $Q = B_{2R}(x^*)$, where $R \geq R_0 \eqdef \|x^0 - x^*\|$ and $0 < \gamma \leq \nicefrac{2}{\ell}$. If $x^k$ lies in $B_{2R}(x^*)$ for all $k = 0,1,\ldots, K$ for some $K\geq 0$, then for all $u \in B_{3R}(x^*)$ the iterates produced by \ref{eq:clipped_SGDA} satisfy
    \begin{eqnarray}
        2\gamma\langle F(u), x^K_{\avg} - u\rangle &\leq& \frac{\|x^0 - u\|^2 - \|x^{K+1} - u\|^2}{K+1}\notag\\
        &&\quad + \frac{2\gamma}{K+1}\sum\limits_{k=0}^K\langle x^k - u - \gamma F(x^k), \omega_k\rangle\notag\\
        &&\quad + \frac{\gamma^2}{K+1}\sum\limits_{k=0}^K\left(\|F(x^k)\|^2 + \|\omega_k\|^2\right),  \label{eq:optimization_lemma_gap_SGDA}\\
        x_{\avg}^K &\eqdef& \frac{1}{K+1}\sum\limits_{k=0}^{K} x^k, \label{eq:x_avg_K_SGDA}\\
        \omega_k &\eqdef& F(x^k) - \tF_{\Bxi^k}(x^k). \label{eq:omega_k_SGDA}
    \end{eqnarray}
\end{lemma}
\begin{proof}
    Using the update rule of \ref{eq:clipped_SGDA}, we obtain
    \begin{eqnarray*}
        \|x^{k+1} - u\|^2 &=& \|x^k - u\|^2 - 2\gamma \langle x^k - u,  \tF_{\Bxi^k}(x^k)\rangle + \gamma^2\|\tF_{\Bxi^k}(x^k)\|^2\\
        &=& \|x^k - u\|^2 -2\gamma \langle x^k - u, F(x^k) \rangle + 2\gamma \langle x^k - u, \omega_k \rangle\\
        &&\quad + \gamma^2\|F(x^k)\|^2 - 2\gamma^2 \langle F(x^k), \omega_k \rangle + \gamma^2\|\omega_k\|^2\\
         &\overset{\eqref{eq:monotonicity}}{\leq}& \|x^k - u\|^2 - 2\gamma \langle x^k - u, F(u) \rangle  + 2\gamma \langle x^k - u - \gamma F(x^k), \omega_k \rangle\\
        &&\quad + \gamma^2\left(\|F(x^k)\|^2 + \|\omega_k\|^2\right).
    \end{eqnarray*}
    Rearranging the terms, we derive
    \begin{eqnarray*}
        2\gamma \langle F(u), x^k - u \rangle &\leq& \|x^k - u\|^2 - \|x^{k+1} - u\|^2  + 2\gamma \langle x^k - u - \gamma F(x^k), \omega_k \rangle\\
        &&\quad + \gamma^2\left(\|F(x^k)\|^2 + \|\omega_k\|^2\right).
    \end{eqnarray*}
    Finally, we sum up the above inequality for $k = 0,1,\ldots, K$ and divide both sides of the result by $(K+1)$:
    \begin{eqnarray*}
          2\gamma\langle F(u), x^K_{\avg} - u\rangle &\leq& \frac{1}{K+1}\sum\limits_{k=0}^K\left(\|x^k - u\|^2 - \|x^{k+1} - u\|^2\right)\\
        &&\quad + \frac{2\gamma}{K+1}\sum\limits_{k=0}^K\langle x^k - u - \gamma F(x^k), \omega_k\rangle\\
        &&\quad + \frac{\gamma^2}{K+1}\sum\limits_{k=0}^K\left(\|F(x^k)\|^2 + \|\omega_k\|^2\right)\\
        &=& \frac{\|x^0 - u\|^2 - \|x^{K+1} - u\|^2}{K+1}\\
        &&\quad + \frac{2\gamma}{K+1}\sum\limits_{k=0}^K\langle x^k - u - \gamma F(x^k), \omega_k\rangle\\
        &&\quad + \frac{\gamma^2}{K+1}\sum\limits_{k=0}^K\left(\|F(x^k)\|^2 + \|\omega_k\|^2\right).
    \end{eqnarray*}
    This finishes the proof.
\end{proof}

We also derive the following lemma, which we use in the analysis of the star-cocoercive case as well.

\begin{lemma}\label{lem:optimization_lemma_SGDA}
    Let Assumption~\ref{as:star_cocoercivity} hold for $Q = B_{2R}(x^*)$, where $R \geq R_0 \eqdef \|x^0 - x^*\|$ and $0 < \gamma \leq \nicefrac{2}{\ell}$. If $x^k$ lies in $B_{2R}(x^*)$ for all $k = 0,1,\ldots, K$ for some $K\geq 0$, then the iterates produced by \ref{eq:clipped_SGDA} satisfy
    \begin{eqnarray}
        \frac{\gamma}{K+1}\left(\frac{2}{\ell} - \gamma\right)\sum\limits_{k=0}^K\|F(x^k)\|^2 &\leq& \frac{\|x^0 - x^*\|^2 - \|x^{K+1} - x^*\|^2}{K+1}\notag\\
        &&\quad + \frac{2\gamma}{K+1}\sum\limits_{k=0}^K\langle x^k - x^* - \gamma F(x^k), \omega_k\rangle\notag\\
        &&\quad + \frac{\gamma^2}{K+1}\sum\limits_{k=0}^K\|\omega_k\|^2,  \label{eq:optimization_lemma_norm_SGDA}
    \end{eqnarray}
    where $\omega_k$ is defined in \eqref{eq:omega_k_SGDA}.
\end{lemma}
\begin{proof}
    Using the update rule of \ref{eq:clipped_SGDA}, we obtain
    \begin{eqnarray*}
        \|x^{k+1} - x^*\|^2 &=& \|x^k - x^*\|^2 - 2\gamma \langle x^k - x^*,  \tF_{\Bxi^k}(x^k)\rangle + \gamma^2\|\tF_{\Bxi^k}(x^k)\|^2\\
        &=& \|x^k - x^*\|^2 -2\gamma \langle x^k - x^*, F(x^k) \rangle + 2\gamma \langle x^k - x^*, \omega_k \rangle\\
        &&\quad + \gamma^2\|F(x^k)\|^2 - 2\gamma^2 \langle F(x^k), \omega_k \rangle + \gamma^2\|\omega_k\|^2\\
         &\overset{\eqref{eq:star_cocoercivity}}{\leq}& \|x^k - x^*\|^2  + 2\gamma \langle x^k - x^*, \omega_k \rangle - 2\gamma^2 \langle F(x^k), \omega_k \rangle\\
        &&\quad + \gamma \left(\gamma - \frac{2}{\ell}\right)\|F(x^k)\|^2 + \gamma^2\|\omega_k\|^2.
    \end{eqnarray*}
    Since $0 < \gamma \leq \nicefrac{2}{\ell}$, we have $\gamma\left(\nicefrac{2}{\ell} - \gamma\right)\|F(x^k)\|^2 \geq 0$ and, rearranging the terms, we derive
    \begin{eqnarray*}
        \gamma\left(\frac{2}{\ell} - \gamma\right)\|F(x^k)\|^2 &\leq& \|x^k - x^*\|^2 - \|x^{k+1} - x^*\|^2  + 2\gamma \langle x^k - x^*, \omega_k \rangle\\
        &&\quad -2 \gamma^2\langle F(x^k), \omega_k \rangle + \gamma^2\|\omega_k\|^2.
    \end{eqnarray*}
    Finally, we sum up the above inequality for $k = 0,1,\ldots, K$ and divide both sides of the result by $(K+1)$:
    \begin{eqnarray*}
         \frac{\gamma}{K+1}\left(\frac{2}{\ell} - \gamma\right)\sum\limits_{k=0}^K\|F(x^k)\|^2 &\leq& \frac{1}{K+1}\sum\limits_{k=0}^K\left(\|x^k - x^*\|^2 - \|x^{k+1} - x^*\|^2\right) + \frac{\gamma^2}{K+1}\sum\limits_{k=0}^K\|\omega_k\|^2\\
        &&\quad + \frac{2\gamma}{K+1}\sum\limits_{k=0}^K\langle x^k - x^*, \omega_k\rangle - \frac{2\gamma^2}{K+1}\sum\limits_{k=0}^K\langle F(x^k),\omega_k\rangle\\
        &=& \frac{\|x^0 - x^*\|^2 - \|x^{K+1} - x^*\|^2}{K+1} + \frac{\gamma^2}{K+1}\sum\limits_{k=0}^K\|\omega_k\|^2\\
        &&\quad + \frac{2\gamma}{K+1}\sum\limits_{k=0}^K\langle x^k - x^*, \omega_k\rangle - \frac{2\gamma^2}{K+1}\sum\limits_{k=0}^K\langle F(x^k),\omega_k\rangle.
    \end{eqnarray*}
    This finishes the proof.
\end{proof}

\begin{theorem}\label{thm:main_result_SGDA}
    Let Assumptions~\ref{as:UBV}, \ref{as:monotonicity}, \ref{as:star_cocoercivity}, hold for $Q = B_{2R}(x^*)$, where $R \geq R_0 \eqdef \|x^0 - x^*\|$, and
    \begin{gather}
        \gamma \leq \frac{1}{170 \ell \ln \tfrac{6(K+1)}{\beta}}, \label{eq:gamma_gap_SGDA}\\
        \lambda = \frac{R}{60\gamma \ln \tfrac{6(K+1)}{\beta}},  \label{eq:lambda_gap_SGDA}\\
        m \geq \max\left\{1, \frac{97200 (K+1) \gamma^2\sigma^2 \ln\tfrac{6(K+1)}{\beta}}{R^2}\right\}, \label{eq:batch_gap_SGDA}
    \end{gather}
    for some $K \geq 0$ and $\beta \in (0,1]$ such that $\ln \tfrac{6(K+1)}{\beta} \geq 1$. Then, after $K$ iterations the iterates produced by \ref{eq:clipped_SGDA} with probability at least $1 - \beta$ satisfy 
    \begin{equation}
        \gap_R(x_{\avg}^K) \leq \frac{9R^2}{2\gamma(K+1)}. \label{eq:main_result_gap_SGDA}
    \end{equation}
\end{theorem}
\begin{proof}
    We introduce new notation: $R_k = \|x^k - x^*\|$ for all $k\geq 0$. The proof is based on the induction. In particular, for each $k = 0,\ldots, K+1$ we define the probability event $E_k$ as follows: inequalities
    \begin{gather}
        \|x^t - x^*\|^2 \leq 2R^2 \quad \text{and}\quad \gamma\left\|\sum\limits_{l=0}^{t-1} \omega_l\right\| \leq R \label{eq:induction_inequality_SGDA_gap}
    \end{gather}
    hold for $t = 0,1,\ldots,k$ simultaneously. Our goal is to prove that $\PP\{E_k\} \geq  1 - \nicefrac{k\beta}{(K+1)}$ for all $k = 0,1,\ldots,K+1$. We use the induction to show this statement. For $k = 0$ the statement is trivial since $R_0^2 \leq 2R^2$ by definition and $\sum_{l=0}^{-1} \omega_l = 0$. Next, assume that the statement holds for $k = T \leq K$, i.e., we have $\PP\{E_{T}\} \geq 1 - \nicefrac{T\beta}{(K+1)}$. We need to prove that $\PP\{E_{T+1}\} \geq 1 - \nicefrac{(T+1)\beta}{(K+1)}$. Let us notice that probability event $E_{T}$ implies $x^t \in B_{2R}(x^*)$ for all $t = 0, 1, \ldots, T$. This means that the assumptions of Lemma~\ref{lem:optimization_lemma_SGDA} hold and we have that probability event $E_{T}$ implies ($\gamma < \nicefrac{1}{\ell}$)
    \begin{eqnarray}
        \frac{\gamma}{\ell(T+1)}\sum\limits_{t=0}^T\|F(x^t)\|^2 &\leq& \frac{\|x^0 - x^*\|^2 - \|x^{T+1} - x^*\|^2}{T+1}\notag\\
        &&\quad + \frac{2\gamma}{T+1}\sum\limits_{t=0}^T\langle x^t - x^* - \gamma F(x^t), \omega_t\rangle\notag\\
        &&\quad + \frac{\gamma^2}{T+1}\sum\limits_{t=0}^T\|\omega_t\|^2
        \label{eq:thm_SGDA_technical_gap_0}
    \end{eqnarray}
    and
    \begin{eqnarray}
        \|F(x^t)\| &\overset{\eqref{eq:star_cocoercivity}}{\leq}& \ell\|x^t - x^*\| \overset{\eqref{eq:induction_inequality_SGDA_gap}}{\leq} \sqrt{2}\ell R \overset{\eqref{eq:gamma_gap_SGDA},\eqref{eq:lambda_gap_SGDA}}{\leq} \frac{\lambda}{2} \label{eq:operator_bound_tx_t_SGDA_gap}
    \end{eqnarray}
    for all $t = 0, 1, \ldots, T$.
From \eqref{eq:thm_SGDA_technical_gap_0} we have    
    \begin{eqnarray*}
       R_{T+1}^2 \leq R_0^2 + 2\gamma \sum\limits_{t=0}^T\langle x^t - x^* - \gamma F(x^t), \omega_t\rangle + \gamma^2\sum\limits_{t=0}^T\|\omega_t\|^2.
    \end{eqnarray*}
    Next, we notice that
    \begin{eqnarray}
        \|x^t - x^* - \gamma F(x^t)\| &\leq& \|x^t - x^*\| + \gamma\|F(x^t)\| \overset{\eqref{eq:star_cocoercivity},\eqref{eq:induction_inequality_SGDA_gap}}{\leq} 2R + \gamma \ell\|x^t - x^*\|\notag\\
        &\overset{\eqref{eq:induction_inequality_SGDA_gap}}{\leq}& 2R + 2R\gamma \ell \overset{\eqref{eq:gamma_gap_SGDA}}{\leq} 3R, \label{eq:thm_SGDA_technical_gap_2}
    \end{eqnarray}
    for all $t = 0, 1, \ldots, T$. Consider random vectors
    \begin{equation*}
        \eta_t = \begin{cases}x^t - x^* - \gamma F(x^t),& \text{if } \|x^t - x^* - \gamma F(x^t)\| \leq 3R,\\ 0,& \text{otherwise,} \end{cases}
    \end{equation*}
    for all $t = 0, 1, \ldots, T$. We notice that $\eta_t$ is bounded with probability $1$:
    \begin{equation}
        \|\eta_t\| \leq 3R  \label{eq:thm_SGDA_technical_gap_3}
    \end{equation}
    for all $t = 0, 1, \ldots, T$. Moreover, in view of \eqref{eq:thm_SGDA_technical_gap_2}, probability event $E_{T}$ implies $\eta_t = x^t - x^* - \gamma F(x^t)$ for all $t = 0, 1, \ldots, T$. Therefore, $E_{T}$ implies
    \begin{eqnarray*}
       R_{T+1}^2 \leq R^2 + 2\gamma\sum\limits_{t=0}^T\langle \eta_t, \omega_t\rangle + \gamma^2\sum\limits_{t=0}^T\|\omega_t\|^2.
    \end{eqnarray*}
    To continue our derivation we introduce new notation:
    \begin{gather}
        \omega_t^u \eqdef \EE_{\Bxi^t}\left[\tF_{\Bxi^t}(x^t)\right] - \tF_{\Bxi^t}(x^t),\quad \omega_t^b \eqdef F(x^t) - \EE_{\Bxi^t}\left[\tF_{\Bxi^t}(x^t)\right] \label{eq:thm_SGDA_technical_gap_4}
    \end{gather}
    By definition we have $\omega_t = \omega_t^u + \omega_t^b$ for all $t = 0,\ldots, T$. Using the introduced notation, we continue our derivation as follows: $E_{T}$ implies
    \begin{eqnarray}
        R_{T+1}^2 &\leq& R^2 + \underbrace{2\gamma \sum\limits_{t = 0}^{T} \langle \eta_t, \omega_t^u \rangle}_{\circledOne} + \underbrace{2\gamma \sum\limits_{t = 0}^{T} \langle \eta_t, \omega_t^b \rangle}_{\circledTwo} + \underbrace{2\gamma^2 \sum\limits_{t=0}^{T}\left(\EE_{\Bxi^t}\left[\|\omega_t^u\|^2\right] \right)}_{\circledThree} \notag\\
        &&\quad + \underbrace{2\gamma^2 \sum\limits_{t=0}^{T}\left(\|\omega_t^u\|^2 - \EE_{\Bxi^t}\left[\|\omega_t^u\|^2\right]\right)}_{\circledFour} + \underbrace{2\gamma^2 \sum\limits_{t=0}^{T}\left(\|\omega_t^b\|^2\right)}_{\circledFive}.\label{eq:thm_SGDA_technical_gap_5}
    \end{eqnarray}
    We emphasize that the above inequality does not rely on monotonicity of $F$.
    
    As we notice above, $E_T$ implies $x^t \in B_{2R}(x^*)$ for all $t = 0, 1, \ldots, T$. This means that the assumptions of Lemma~\ref{lem:optimization_lemma_SGDA_gap} hold and we have that probability event $E_{T}$ implies
    \begin{eqnarray}
        2\gamma(T+1)\gap_R(x^T_{\avg}) &\leq& \max\limits_{u\in B_R(x^*)}\left\{\|x^0 - u\|^2 + 2\gamma \sum\limits_{t=0}^T\langle x^t - u - \gamma F(x^t), \omega_t\rangle\right\}\notag\\
        &&\quad + \gamma^2\sum\limits_{t=0}^T\left(\|F(x^t)\|^2 + \|\omega_t\|^2\right),  \notag\\
        &=& \max\limits_{u\in B_R(x^*)}\left\{\|x^0 - u\|^2 + 2\gamma \sum\limits_{t=0}^T\langle x^* - u, \omega_t\rangle\right\}\notag\\
        &&\quad + 2\gamma \sum\limits_{t=0}^T\langle x^t - x^* - \gamma F(x^t), \omega_t\rangle\notag\\
        &&\quad + \gamma^2\sum\limits_{t=0}^T\left(\|F(x^t)\|^2 + \|\omega_t\|^2\right).  \notag
    \end{eqnarray}
    We notice that $E_T$ implies $\eta_t = x^t - x^* - \gamma F(x^t)$ for all $t = 0, 1, \ldots, T$ as well as \eqref{eq:thm_SGDA_technical_gap_0} and $\gamma < \nicefrac{1}{\ell}$. Therefore, probability event $E_T$ implies
    \begin{eqnarray}
        2\gamma(T+1)\gap_R(x^T_{\avg}) &\leq& \max\limits_{u\in B_R(x^*)}\left\{\|x^0 - u\|^2\right\} + 2\gamma \max\limits_{u\in B_R(x^*)}\left\{\sum\limits_{t=0}^T\langle x^* - u, \omega_t\rangle\right\}\notag\\
        &&\quad + 2\gamma \sum\limits_{t=0}^T\langle \eta_t, \omega_t\rangle + \frac{\gamma}{\ell}\sum\limits_{t=0}^T\|F(x^t)\|^2 + \gamma^2\sum\limits_{t=0}^T\|\omega_t\|^2 \notag\\
        &\leq& 4R^2 + 2\gamma \max\limits_{u\in B_R(x^*)}\left\{\left\langle x^* - u, \sum\limits_{t=0}^T\omega_t\right\rangle\right\} \notag\\
        &&\quad + R^2 + 4\gamma \sum\limits_{t=0}^T\langle \eta_t, \omega_t\rangle + 2\gamma^2\sum\limits_{t=0}^T\|\omega_t\|^2\notag\\
        &\leq& 5R^2 + 2\gamma R\left\|\sum\limits_{t=0}^T\omega_t\right\| + 2\cdot\left(\circledOne + \circledTwo + \circledThree + \circledFour + \circledFive\right),\label{eq:thm_SGDA_technical_gap_5_1}
    \end{eqnarray}
    where $\circledOne, \circledTwo, \circledThree, \circledFour, \circledFive$ are defined in \eqref{eq:thm_SGDA_technical_gap_5}.
    
    The rest of the proof is based on deriving good enough upper bounds for $\circledOne, \circledTwo, \circledThree, \circledFour, \circledFive$, i.e., we want to prove that $\circledOne + \circledTwo + \circledThree + \circledFour + \circledFive \leq R^2$ and $2\gamma R\left\|\sum_{t=0}^T\omega_t\right\| \leq 2R^2$ with high probability.
    
    Before we move on, we need to derive some useful inequalities for operating with $\omega_t^u, \omega_t^b$. First of all, Lemma~\ref{lem:bias_variance} implies that
    \begin{equation}
        \|\omega_t^u\| \leq 2\lambda \label{eq:theta_omega_magnitude_SGDA}
    \end{equation}
    for all $t = 0,1, \ldots, T$. Next, since $\{\xi^{i,t}\}_{i=1}^{m}$ are independently sampled from $\cD$, we have $\EE_{\Bxi^t}[F_{\Bxi^t}(x^t)] = F(x^t)$, and 
    \begin{gather}
        \EE_{\Bxi^t}\left[\|F_{\Bxi^t}(x^t) - F(x^t)\|^2\right] = \frac{1}{m^2}\sum\limits_{i=1}^m \EE_{\xi^{i,t}}\left[\|F_{\xi^{i,t}}(x^t) - F(x^t)\|^2\right] \overset{\eqref{eq:UBV}}{\leq} \frac{\sigma^2}{m}, \notag
    \end{gather}
    for all $l = 0,1, \ldots, T$. Therefore, in view of Lemma~\ref{lem:bias_variance}, $E_{T}$ implies that
    \begin{gather}
         \left\|\omega_t^b\right\| \leq \frac{4\sigma^2}{m\lambda}, \label{eq:bias_theta_omega_SGDA}\\
         \EE_{\Bxi^t}\left[\left\|\omega_t\right\|^2\right] \leq \frac{18\sigma^2}{m}, \label{eq:distortion_theta_omega_SGDA}\\
         \EE_{\Bxi^t}\left[\left\|\omega_t^u\right\|^2\right] \leq \frac{18\sigma^2}{m} \label{eq:variance_theta_omega_SGDA}
    \end{gather}
    for all $l = 0,1, \ldots, T$.
    
    \paragraph{Upper bound for $\circledOne$.} Since $\EE_{\Bxi^t}[\omega_t^u] = 0$, we have
    \begin{equation*}
        \EE_{\Bxi^t}\left[2\gamma\langle \eta_t, \omega_t^u \rangle\right] = 0.
    \end{equation*}
    Next, the summands in $\circledOne$ are bounded with probability $1$:
    \begin{eqnarray}
        |2\gamma\langle \eta_t, \omega_t^u \rangle | \leq 2\gamma  \|\eta_t\|\cdot \|\omega_t^u\| 
        \overset{\eqref{eq:thm_SGDA_technical_gap_3},\eqref{eq:theta_omega_magnitude_SGDA}}{\leq} 12 \gamma R \lambda \overset{\eqref{eq:lambda_gap_SGDA}}{\leq} \frac{R^2}{5\ln\tfrac{6(K+1)}{\beta}} \eqdef c. \label{eq:SGDA_neg_mon_technical_1_1}
    \end{eqnarray}
    Moreover, these summands have bounded conditional variances $\sigma_t^2 \eqdef \EE_{\Bxi^t}\left[4\gamma^2 \langle \eta_t, \omega_t^u \rangle^2\right]$:
    \begin{eqnarray}
        \sigma_t^2 \leq \EE_{\Bxi^t}\left[4\gamma^2 \|\eta_t\|^2\cdot \|\omega_t^u\|^2\right] \overset{\eqref{eq:thm_SGDA_technical_gap_3}}{\leq} 36\gamma^2 R^2 \EE_{\Bxi^t}\left[\|\omega_t^u\|^2\right]. \label{eq:SGDA_neg_mon_technical_1_2}
    \end{eqnarray}
    That is, sequence $\{2\gamma \langle \eta_t, \omega_t^u \rangle\}_{t\geq 0}$ is a bounded martingale difference sequence having bounded conditional variances $\{\sigma_t^2\}_{t \geq 0}$. Applying Bernstein's inequality (Lemma~\ref{lem:Bernstein_ineq}) with $X_t = 2\gamma \langle \eta_t, \omega_t^u \rangle$, $c$ defined in \eqref{eq:SGDA_neg_mon_technical_1_1}, $b = \frac{R^2}{5}$, $G = \tfrac{R^4}{150\ln\frac{6(K+1)}{\beta}}$, we get that
    \begin{equation*}
        \PP\left\{|\circledOne| > \frac{R^2}{5} \text{ and } \sum\limits_{t=0}^{T}\sigma_t^2 \leq \frac{R^4}{150\ln\tfrac{6(K+1)}{\beta}}\right\} \leq 2\exp\left(- \frac{b^2}{2G + \nicefrac{2cb}{3}}\right) = \frac{\beta}{3(K+1)}.
    \end{equation*}
    In other words, $\PP\{E_{\circledOne}\} \geq 1 - \tfrac{\beta}{3(K+1)}$, where probability event $E_{\circledOne}$ is defined as
    \begin{equation}
        E_{\circledOne} = \left\{\text{either} \quad \sum\limits_{t=0}^{T}\sigma_t^2 > \frac{R^4}{150\ln\tfrac{6(K+1)}{\beta}}\quad \text{or}\quad |\circledOne| \leq \frac{R^2}{5}\right\}. \label{eq:bound_1_SGDA_neg_mon}
    \end{equation}
    Moreover, we notice here that probability event $E_{T}$ implies that
    \begin{eqnarray}
        \sum\limits_{t=0}^{T}\sigma_t^2 \overset{\eqref{eq:SGDA_neg_mon_technical_1_2}}{\leq} 36\gamma^2R^2\sum\limits_{t=0}^{T} \EE_{\Bxi^t}\left[\|\omega_t^u\|^2\right] \overset{\eqref{eq:variance_theta_omega_SGDA}, T \leq K+1}{\leq} \frac{648\gamma^2 R^2 \sigma^2 (K+1)}{m} \overset{\eqref{eq:batch_gap_SGDA}}{\leq} \frac{R^4}{150\ln\tfrac{6(K+1)}{\beta}}. \label{eq:bound_1_variances_SGDA_neg_mon}
    \end{eqnarray}
    
    \paragraph{Upper bound for $\circledTwo$.} Probability event $E_{T}$ implies
    \begin{eqnarray}
        \circledTwo &\leq& 2\gamma \sum\limits_{t=0}^{T}\|\eta_l\| \cdot \|\omega_t^b\| \overset{\eqref{eq:thm_SGDA_technical_gap_3},\eqref{eq:bias_theta_omega_SGDA}, T \leq K+1}{\leq} \frac{24\gamma\sigma^2 R(K+1)}{m\lambda}\notag\\
        &\overset{\eqref{eq:lambda_gap_SGDA}}{=}& \frac{1440\gamma^2\sigma^2(K+1) \ln \frac{6(K+1)}{\beta}}{m} \overset{\eqref{eq:batch_gap_SGDA}}{\leq} \frac{R^2}{5}. \label{eq:bound_2_SGDA_neg_mon}
    \end{eqnarray}

    \paragraph{Upper bound for $\circledThree$.} Probability event $E_{T}$ implies
    \begin{eqnarray}
        \circledThree =  2\gamma^2 \sum\limits_{t = 0}^{T}\EE_{\Bxi^t}\left[\|\omega_t^u\|^2\right] \overset{\eqref{eq:variance_theta_omega_SGDA}, T \leq K+1}{\leq} \frac{36\gamma^2\sigma^2 (K+1)}{m} \overset{\eqref{eq:batch_gap_SGDA}}{\leq} \frac{R^2}{5}. \label{eq:bound_3_SGDA_neg_mon}
    \end{eqnarray}
    
    \paragraph{Upper bound for $\circledFour$.} We have
    \begin{equation*}
        2\gamma^2\EE_{\Bxi^t}\left[\|\omega_t^u\|^2 - \EE_{\Bxi^t}\left[\|\omega_t^u\|^2\right]\right] = 0.
    \end{equation*}
    Next, the summands in $\circledFour$ are bounded with probability $1$:
    \begin{eqnarray}
        2\gamma^2 \left|\|\omega_t^u\|^2 - \EE_{\Bxi^t}\left[\|\omega_t^u\|^2\right] \right| &\leq& 2\gamma^2\left( \|\omega_t^u\|^2 + \EE_{\Bxi^t}\left[\|\omega_t^u\|^2\right] \right) 
        \overset{\eqref{eq:theta_omega_magnitude_SGDA}}{\leq} 16\gamma^2 \lambda^2 \notag\\
        &\overset{\eqref{eq:lambda_gap_SGDA}}{\leq}& \frac{R^2}{225\ln\tfrac{6(K+1)}{\beta}} \leq \frac{R^2}{5\ln\tfrac{6(K+1)}{\beta}} \eqdef c. \label{eq:SGDA_neg_mon_technical_4_1}
    \end{eqnarray}
    Moreover, these summands have bounded conditional variances $\widetilde\sigma_t^2 \eqdef 4\gamma^4\EE_{\Bxi^t}\left[\left(\|\omega_t^u\|^2 - \EE_{\Bxi^t}\left[\|\omega_t^u\|^2\right]\right)^2\right]$:
    \begin{eqnarray}
        \widetilde\sigma_t^2 \overset{\eqref{eq:SGDA_neg_mon_technical_4_1}}{\leq} \frac{2\gamma^2 R^2}{225\ln \frac{6(K+1)}{\beta}} \EE_{\Bxi^t}\left[\left| \|\omega_t^u\|^2 - \EE_{\Bxi^t}\left[\|\omega_t^u\|^2\right] \right|\right]  \leq \frac{4\gamma^2 R^2}{225\ln \frac{6(K+1)}{\beta}} \EE_{\Bxi^t}\left[\|\omega_t^u\|^2\right]. \label{eq:SGDA_neg_mon_technical_4_2}
    \end{eqnarray}
    That is, sequence $\{\|\omega_t^u\|^2 - \EE_{\Bxi^t}[\|\omega_t^u\|^2]\}_{t\geq 0}$ is a bounded martingale difference sequence having bounded conditional variances $\{\widetilde\sigma_t^2\}_{t \geq 0}$. Applying Bernstein's inequality (Lemma~\ref{lem:Bernstein_ineq}) with $X_t = \|\omega_t^u\|^2 - \EE_{\Bxi^t}[\|\omega_t^u\|^2]$, $c$ defined in \eqref{eq:SGDA_neg_mon_technical_4_1}, $b = \frac{R^2}{5}$, $G = \tfrac{R^4}{150\ln\frac{6(K+1)}{\beta}}$, we get that
    \begin{equation*}
        \PP\left\{|\circledFour| > \frac{R^2}{5} \text{ and } \sum\limits_{t=0}^{T}\widetilde\sigma_t^2 \leq \frac{R^4}{150\ln\tfrac{6(K+1)}{\beta}}\right\} \leq 2\exp\left(- \frac{b^2}{2G + \nicefrac{2cb}{3}}\right) = \frac{\beta}{3(K+1)}.
    \end{equation*}
    In other words, $\PP\{E_{\circledFour}\} \geq 1 - \tfrac{\beta}{3(K+1)}$, where probability event $E_{\circledFour}$ is defined as
    \begin{equation}
        E_{\circledFour} = \left\{\text{either} \quad \sum\limits_{t=0}^{T}\widetilde\sigma_t^2 > \frac{R^4}{150\ln\tfrac{6(K+1)}{\beta}}\quad \text{or}\quad |\circledFour| \leq \frac{R^2}{5}\right\}. \label{eq:bound_4_SGDA_neg_mon}
    \end{equation}
    Moreover, we notice here that probability event $E_{T}$ implies that
    \begin{eqnarray}
        \sum\limits_{t=0}^{T}\widetilde\sigma_t^2 &\overset{\eqref{eq:SGDA_neg_mon_technical_4_2}}{\leq}& \frac{4\gamma^2 R^2}{225\ln \frac{6(K+1)}{\beta}}\sum\limits_{t=0}^{T} \EE_{\Bxi^t}\left[\|\omega_t^u\|^2\right] \overset{\eqref{eq:variance_theta_omega_SGDA}, T \leq K+1}{\leq} \frac{8\gamma^2 R^2 \sigma^2 (K+1)}{25m \ln\tfrac{6(K+1)}{\beta}} \notag\\
        &\overset{\eqref{eq:batch_gap_SGDA}}{\leq}& \frac{R^4}{150\ln\tfrac{6(K+1)}{\beta}}. \label{eq:bound_4_variances_SGDA_neg_mon}
    \end{eqnarray}

    \paragraph{Upper bound for $\circledFive$.} Probability event $E_{T}$ implies
    \begin{eqnarray}
        \circledFive &=&  2\gamma^2 \sum\limits_{t = 0}^{T}\|\omega_t^b\|^2  \overset{\eqref{eq:bias_theta_omega_SGDA}, T \leq K+1}{\leq} \frac{32\gamma^2 \sigma^4 (K+1)}{m^2 \lambda^2} \overset{\eqref{eq:lambda_gap_SGDA}}{=}  \frac{115200 \gamma^4 \sigma^4 (K+1) \ln^2 \frac{6(K+1)}{\beta}}{m^2 R^2} \notag\\ &\overset{\eqref{eq:batch_gap_SGDA}}{\leq}& \frac{R^2}{5}. \label{eq:bound_5_SGDA_neg_mon}
    \end{eqnarray}

    \paragraph{Upper bound for $\gamma \left\|\sum_{t=0}^T\omega_t\right\|$.} To handle this term, we introduce new notation:
    \begin{equation*}
        \zeta_l = \begin{cases} \gamma \sum\limits_{r=0}^{l-1}\omega_r,& \text{if } \left\|\gamma \sum\limits_{r=0}^{l-1}\omega_r\right\| \leq R,\\ 0, & \text{otherwise} \end{cases}
    \end{equation*}
    for $l = 1, 2, \ldots, T-1$. By definition, we have
    \begin{equation}
        \|\zeta_l\| \leq R.  \label{eq:gap_thm_SGDA_technical_8}
    \end{equation}
    Therefore, in view of \eqref{eq:induction_inequality_SGDA_gap}, probability event $E_{T}$ implies
    \begin{eqnarray}
        \gamma\left\|\sum\limits_{l = 0}^{T}\omega_l\right\| &=& \sqrt{\gamma^2\left\|\sum\limits_{l = 0}^{T}\omega_l\right\|^2}\notag\\
        &=& \sqrt{\gamma^2\sum\limits_{l=0}^{T}\|\omega_l\|^2 + 2\gamma\sum\limits_{l=0}^{T}\left\langle \gamma\sum\limits_{r=0}^{l-1} \omega_r, \omega_l \right\rangle} \notag\\
        &=& \sqrt{\gamma^2\sum\limits_{l=0}^{T}\|\omega_l\|^2 + 2\gamma \sum\limits_{l=0}^{T} \langle \zeta_l, \omega_l\rangle} \notag\\
        &\overset{\eqref{eq:thm_SGDA_technical_gap_5}}{\leq}& \sqrt{\circledThree + \circledFour + \circledFive + \underbrace{2\gamma \sum\limits_{l=0}^{T} \langle \zeta_l, \omega_l^u\rangle}_{\circledSix} + \underbrace{2\gamma \sum\limits_{l=0}^{T} \langle \zeta_l, \omega_l^b}_{\circledSeven}\rangle}. \label{eq:norm_sum_omega_bound_gap_SGDA}
    \end{eqnarray}
    Following similar steps as before, we bound $\circledSix$ and $\circledSeven$.

    \paragraph{Upper bound for $\circledSix$.} Since $\EE_{\Bxi^t}[\omega_t^u] = 0$, we have
    \begin{equation*}
        \EE_{\Bxi^t}\left[2\gamma\langle \zeta_t, \omega_t^u \rangle\right] = 0.
    \end{equation*}
    Next, the summands in $\circledSix$ are bounded with probability $1$:
    \begin{eqnarray}
        |2\gamma\langle \zeta_t, \omega_t^u \rangle | \leq 2\gamma  \|\zeta_t\|\cdot \|\omega_t^u\| 
        \overset{\eqref{eq:gap_thm_SGDA_technical_8},\eqref{eq:theta_omega_magnitude_SGDA}}{\leq} 4 \gamma R \lambda \overset{\eqref{eq:lambda_gap_SGDA}}{\leq} \frac{R^2}{5\ln\tfrac{6(K+1)}{\beta}} \eqdef c. \label{eq:SGDA_neg_mon_technical_6_1}
    \end{eqnarray}
    Moreover, these summands have bounded conditional variances $\widehat\sigma_t^2 \eqdef \EE_{\Bxi^t}\left[4\gamma^2 \langle \zeta_t, \omega_t^u \rangle^2\right]$:
    \begin{eqnarray}
        \widehat\sigma_t^2 \leq \EE_{\Bxi^t}\left[4\gamma^2 \|\zeta_t\|^2\cdot \|\omega_t^u\|^2\right] \overset{\eqref{eq:thm_SGDA_technical_gap_3}}{\leq} 4\gamma^2 R^2 \EE_{\Bxi^t}\left[\|\omega_t^u\|^2\right]. \label{eq:SGDA_neg_mon_technical_6_2}
    \end{eqnarray}
    That is, sequence $\{2\gamma \langle \zeta_t, \omega_t^u \rangle\}_{t\geq 0}$ is a bounded martingale difference sequence having bounded conditional variances $\{\widehat\sigma_t^2\}_{t \geq 0}$. Applying Bernstein's inequality (Lemma~\ref{lem:Bernstein_ineq}) with $X_t = 2\gamma \langle \zeta_t, \omega_t^u \rangle$, $c$ defined in \eqref{eq:SGDA_neg_mon_technical_1_1}, $b = \frac{R^2}{5}$, $G = \tfrac{R^4}{150\ln\frac{6(K+1)}{\beta}}$, we get that
    \begin{equation*}
        \PP\left\{|\circledSix| > \frac{R^2}{5} \text{ and } \sum\limits_{t=0}^{T}\widehat\sigma_t^2 \leq \frac{R^4}{150\ln\tfrac{6(K+1)}{\beta}}\right\} \leq 2\exp\left(- \frac{b^2}{2G + \nicefrac{2cb}{3}}\right) = \frac{\beta}{3(K+1)}.
    \end{equation*}
    In other words, $\PP\{E_{\circledSix}\} \geq 1 - \tfrac{\beta}{3(K+1)}$, where probability event $E_{\circledSix}$ is defined as
    \begin{equation}
        E_{\circledSix} = \left\{\text{either} \quad \sum\limits_{t=0}^{T}\widehat\sigma_t^2 > \frac{R^4}{150\ln\tfrac{6(K+1)}{\beta}}\quad \text{or}\quad |\circledSix| \leq \frac{R^2}{5}\right\}. \label{eq:bound_6_SGDA_neg_mon}
    \end{equation}
    Moreover, we notice here that probability event $E_{T}$ implies that
    \begin{eqnarray}
        \sum\limits_{t=0}^{T}\widehat\sigma_t^2 \overset{\eqref{eq:SGDA_neg_mon_technical_6_2}}{\leq} 4\gamma^2R^2\sum\limits_{t=0}^{T} \EE_{\Bxi^t}\left[\|\omega_t^u\|^2\right] \overset{\eqref{eq:variance_theta_omega_SGDA}, T \leq K+1}{\leq} \frac{72\gamma^2 R^2 \sigma^2 (K+1)}{m} \overset{\eqref{eq:batch_gap_SGDA}}{\leq} \frac{R^4}{150\ln\tfrac{6(K+1)}{\beta}}. \label{eq:bound_6_variances_SGDA_neg_mon}
    \end{eqnarray}
    
    \paragraph{Upper bound for $\circledSeven$.} Probability event $E_{T}$ implies
    \begin{eqnarray}
        \circledSeven &\leq& 2\gamma \sum\limits_{t=0}^{T}\|\zeta_t\| \cdot \|\omega_t^b\| \overset{\eqref{eq:gap_thm_SGDA_technical_8},\eqref{eq:bias_theta_omega_SGDA}, T \leq K+1}{\leq} \frac{8\gamma\sigma^2 R(K+1)}{m\lambda}\notag\\
        &\overset{\eqref{eq:lambda_gap_SGDA}}{=}& \frac{480\gamma^2\sigma^2(K+1) \ln \frac{6(K+1)}{\beta}}{m} \overset{\eqref{eq:batch_gap_SGDA}}{\leq} \frac{R^2}{5}. \label{eq:bound_7_SGDA_neg_mon}
    \end{eqnarray}

    \paragraph{Final derivation.} Putting all bounds together, we get that $E_{T}$ implies
    \begin{gather*}
        R_{T+1}^2 \overset{\eqref{eq:thm_SGDA_technical_gap_5}}{\leq} R^2 + \circledOne + \circledTwo + \circledThree + \circledFour + \circledFive ,\\
        2\gamma(T+1)\gap_R(x^T_{\avg}) \overset{\eqref{eq:thm_SGDA_technical_gap_5_1}}{\leq} 5R^2 + 2\gamma R\left\|\sum\limits_{t=0}^T\omega_t\right\| + 2\cdot\left(\circledOne + \circledTwo + \circledThree + \circledFour + \circledFive\right), \\
        \gamma\left\|\sum\limits_{l = 0}^{T}\omega_l\right\| \overset{\eqref{eq:norm_sum_omega_bound_gap_SGDA}}{\leq} \sqrt{\circledThree + \circledFour + \circledFive + \circledSix + \circledSeven},\\
        \circledTwo \overset{\eqref{eq:bound_2_SGDA_neg_mon}}{\leq} \frac{R^2}{5},\quad \circledThree \overset{\eqref{eq:bound_3_SGDA_neg_mon}}{\leq} \frac{R^2}{5},\quad \circledFive \overset{\eqref{eq:bound_5_SGDA_neg_mon}}{\leq} \frac{R^2}{5},\quad \circledSeven \overset{\eqref{eq:bound_7_SGDA_neg_mon}}{\leq} \frac{R^2}{5},\\
        \sum\limits_{t=0}^{T}\sigma_t^2 \overset{\eqref{eq:bound_1_variances_SGDA_neg_mon}}{\leq}  \frac{R^4}{150\ln\tfrac{6(K+1)}{\beta}},\quad \sum\limits_{t=0}^{T}\widetilde\sigma_t^2 \overset{\eqref{eq:bound_4_variances_SGDA_neg_mon}}{\leq} \frac{R^4}{150\ln\tfrac{6(K+1)}{\beta}},\quad \sum\limits_{t=0}^{T}\widehat\sigma_t^2 \overset{\eqref{eq:bound_6_variances_SGDA_neg_mon}}{\leq} \frac{R^4}{150\ln\tfrac{6(K+1)}{\beta}}.
    \end{gather*}
    Moreover, in view of \eqref{eq:bound_1_SGDA_neg_mon}, \eqref{eq:bound_4_SGDA_neg_mon}, \eqref{eq:bound_7_SGDA_neg_mon}, and our induction assumption, we have
    \begin{gather*}
        \PP\{E_{T}\} \geq 1 - \frac{T\beta}{K+1},\\
        \PP\{E_{\circledOne}\} \geq 1 - \frac{\beta}{3(K+1)}, \quad \PP\{E_{\circledFour}\} \geq 1 - \frac{\beta}{3(K+1)}, \quad \PP\{E_{\circledSix}\} \geq 1 - \frac{\beta}{3(K+1)},
    \end{gather*}
    where probability events $E_{\circledOne}$, $E_{\circledFour}$, and $E_{\circledSix}$ are defined as
    \begin{eqnarray}
        E_{\circledOne}&=& \left\{\text{either} \quad \sum\limits_{t=0}^{T}\sigma_t^2 > \frac{R^4}{150\ln\tfrac{6(K+1)}{\beta}}\quad \text{or}\quad |\circledOne| \leq \frac{R^2}{5}\right\},\notag\\
        E_{\circledFour}&=& \left\{\text{either} \quad \sum\limits_{t=0}^{T}\widetilde\sigma_t^2 > \frac{R^4}{150\ln\tfrac{6(K+1)}{\beta}}\quad \text{or}\quad |\circledFour| \leq \frac{R^2}{5}\right\},\notag\\
        E_{\circledSix}&=& \left\{\text{either} \quad \sum\limits_{t=0}^{T}\widehat\sigma_t^2 > \frac{R^4}{150\ln\tfrac{6(K+1)}{\beta}}\quad \text{or}\quad |\circledSix| \leq \frac{R^2}{5}\right\}.\notag
    \end{eqnarray}
    Putting all of these inequalities together, we obtain that probability event $E_{T} \cap E_{\circledOne} \cap E_{\circledFour} \cap E_{\circledSix}$ implies
    \begin{eqnarray*}
        R_{T+1}^2 &\leq& R^2 + \circledOne + \circledTwo + \circledThree + \circledFour + \circledFive \leq 2R^2,\\
        \gamma\left\|\sum\limits_{l = 0}^{T}\omega_l\right\| &\leq& \sqrt{\circledThree + \circledFour + \circledFive + \circledSix + \circledSeven} \leq R,\\
        2\gamma(T+1)\gap_R(x^T_{\avg}) &\leq& 5R^2 + 2\gamma R\left\|\sum\limits_{t=0}^T\omega_t\right\| + 2\cdot\left(\circledOne + \circledTwo + \circledThree + \circledFour + \circledFive\right)\\
        &\leq& 9R^2.
    \end{eqnarray*}
    Moreover, union bound for the probability events implies
    \begin{equation}
        \PP\{E_{T+1}\} \geq \PP\{E_{T} \cap E_{\circledOne} \cap E_{\circledFour} \cap E_{\circledSix}\} = 1 - \PP\{\overline{E}_{T} \cup \overline{E}_{\circledOne} \cup \overline{E}_{\circledFour} \cup \overline{E}_{\circledSix}\} \geq 1 - \frac{T\beta}{K+1}. \notag
    \end{equation}
    This is exactly what we wanted to prove (see the paragraph after inequality \eqref{eq:induction_inequality_SGDA_gap}). In particular, $E_{K}$ implies
    \begin{equation*}
        \gap_R(x_{\avg}^K) \leq \frac{9R^2}{2\gamma(K+1)},
    \end{equation*}
    which finishes the proof.
\end{proof}

\begin{corollary}\label{cor:main_result_SGDA}
    Let the assumptions of Theorem~\ref{thm:main_result_SGDA} hold. Then, the following statements hold.
    \begin{enumerate}
        \item \textbf{Large stepsize/large batch.} The choice of stepsize and batchsize
        \begin{equation}
            \gamma = \frac{1}{170\ell \ln \tfrac{6(K+1)}{\beta}},\quad m = \max\left\{1, \frac{972 (K+1) \sigma^2}{289\ell^2R^2 \ln\tfrac{6(K+1)}{\beta}}\right\} \label{eq:SGDA_large_step_large_batch}
        \end{equation}
        satisfies conditions \eqref{eq:gamma_gap_SGDA} and \eqref{eq:batch_gap_SGDA}. With such choice of $\gamma, m$, and the choice of $\lambda$ as in \eqref{eq:lambda_gap_SGDA}, the iterates produced by \ref{eq:clipped_SGDA} after $K$ iterations with probability at least $1-\beta$ satisfy
        \begin{equation}
            \gap(x_{\avg}^K) \leq \frac{765 \ell R^2 \ln \tfrac{6(K+1)}{\beta}}{K+1}. \label{eq:main_result_SGDA_large_batch}
        \end{equation}
        In particular, to guarantee $\gap(x_{\avg}^K) \leq \varepsilon$ with probability at least $1-\beta$ for some $\varepsilon > 0$ \ref{eq:clipped_SGDA} requires,
        \begin{gather}
            \cO\left(\frac{\ell R^2}{\varepsilon} \ln\left(\frac{\ell R^2}{\varepsilon\beta}\right)\right) \text{ iterations}, \label{eq:SGDA_iteration_complexity_large_batch}\\
            \cO\left(\max\left\{\frac{\ell R^2}{\varepsilon}, \frac{\sigma^2R^2}{\varepsilon^2}\right\} \ln\left(\frac{\ell R^2}{\varepsilon\beta}\right)\right) \text{ oracle calls}. \label{eq:SGDA_oracle_complexity_large_batch} 
        \end{gather}
        
        \item \textbf{Small stepsize/small batch.} The choice of stepsize and batchsize
        \begin{equation}
            \gamma = \min\left\{\frac{1}{170\ell \ln \tfrac{6(K+1)}{\beta}}, \frac{R}{180\sigma \sqrt{3(K+1) \ln\tfrac{6(K+1)}{\beta}}}\right\},\quad m = 1 \label{eq:SGDA_small_step_small_batch}
        \end{equation}
        satisfies conditions \eqref{eq:gamma_gap_SGDA} and \eqref{eq:batch_gap_SGDA}. With such choice of $\gamma, m$, and the choice of $\lambda$ as in \eqref{eq:lambda_gap_SGDA}, the iterates produced by \ref{eq:clipped_SGDA} after $K$ iterations with probability at least $1-\beta$ satisfy
        \begin{equation}
            \gap(x_{\avg}^K) \leq \max\left\{\frac{765 \ell R^2 \ln \tfrac{6(K+1)}{\beta}}{K+1}, \frac{810\sigma R \sqrt{3\ln\tfrac{6(K+1)}{\beta}}}{\sqrt{K+1}}\right\}. \label{eq:main_result_SGDA_small_batch}
        \end{equation}
        In particular, to guarantee $\gap(x_{\avg}^K) \leq \varepsilon$ with probability at least $1-\beta$ for some $\varepsilon > 0$, \ref{eq:clipped_SGDA} requires
        \begin{equation}
            \cO\left(\max\left\{\frac{\ell R^2}{\varepsilon} \ln\left(\frac{\ell R^2}{\varepsilon\beta}\right), \frac{\sigma^2R^2}{\varepsilon^2}\ln\left(\frac{\sigma^2R^2}{\varepsilon^2\beta}\right)\right\}\right) \text{ iterations/oracle calls}. \label{eq:SGDA_iteration_oracle_complexity_small_batch} 
        \end{equation}
    \end{enumerate}
\end{corollary}
\begin{proof}
    \begin{enumerate}
        \item \textbf{Large stepsize/large batch.} First of all, we verify that the choice of $\gamma$ and $m$ from \eqref{eq:SGDA_large_step_large_batch} satisfies conditions \eqref{eq:gamma_gap_SGDA} and \eqref{eq:batch_gap_SGDA}: \eqref{eq:gamma_gap_SGDA} trivially holds and \eqref{eq:batch_gap_SGDA} holds since
        \begin{equation*}
            m = \max\left\{1, \frac{972 (K+1) \sigma^2}{289\ell^2R^2 \ln\tfrac{6(K+1)}{\beta}}\right\} = \max\left\{1, \frac{97200 (K+1) \gamma^2\sigma^2 \ln\tfrac{6(K+1)}{\beta}}{R^2}\right\}.
        \end{equation*}
        Therefore, applying Theorem~\ref{thm:main_result_SGDA}, we derive that with probability at least $1-\beta$
        \begin{equation*}
            \gap(x_{\avg}^K) \leq \frac{9 R^2}{2\gamma(K+1)} \overset{\eqref{eq:SGDA_large_step_large_batch}}{\leq} \frac{765 \ell R^2 \ln \tfrac{6(K+1)}{\beta}}{K+1}.
        \end{equation*}
        To guarantee $\gap(x_{\avg}^K) \leq \varepsilon$, we choose $K$ in such a way that the right-hand side of the above inequality is smaller than $\varepsilon$ that gives
        \begin{eqnarray*}
            K = \cO\left(\frac{\ell R^2}{\varepsilon} \ln\left(\frac{\ell R^2}{\varepsilon\beta}\right)\right).
        \end{eqnarray*}
        The total number of oracle calls equals
        \begin{eqnarray*}
            m(K+1) &\overset{\eqref{eq:SGDA_large_step_large_batch}}{=}&  \max\left\{K+1, \frac{972 (K+1)^2 \sigma^2}{289\ell^2R^2 \ln\tfrac{6(K+1)}{\beta}}\right\}\\
            &=& \cO\left(\max\left\{\frac{\ell R^2}{\varepsilon}, \frac{\sigma^2R^2}{\varepsilon^2}\right\} \ln\left(\frac{\ell R^2}{\varepsilon\beta}\right)\right).
        \end{eqnarray*}
        
        \item \textbf{Small stepsize/small batch.} First of all, we verify that the choice of $\gamma$ and $m$ from \eqref{eq:SGDA_small_step_small_batch} satisfies conditions \eqref{eq:gamma_gap_SGDA} and \eqref{eq:batch_gap_SGDA}:
        \begin{eqnarray*}
            \gamma &=& \min\left\{\frac{1}{170\ell \ln \tfrac{6(K+1)}{\beta}}, \frac{R}{180\sigma \sqrt{3(K+1) \ln\tfrac{6(K+1)}{\beta}}}\right\} \leq \frac{1}{170\ell \ln \tfrac{6(K+1)}{\beta}},\\
            m &=& 1 \overset{\eqref{eq:SGDA_small_step_small_batch}}{\geq} \frac{97200 (K+1) \gamma^2\sigma^2 \ln\tfrac{6(K+1)}{\beta}}{R^2}.
        \end{eqnarray*}
        Therefore, applying Theorem~\ref{thm:main_result_SGDA}, we derive that with probability at least $1-\beta$
        \begin{eqnarray*}
            \gap(x_{\avg}^K) &\leq& \frac{9 R^2}{2\gamma(K+1)}\\
            &\overset{\eqref{eq:SGDA_small_step_small_batch}}{=}& \max\left\{\frac{765 \ell R^2 \ln \tfrac{6(K+1)}{\beta}}{K+1}, \frac{810\sigma R \sqrt{3\ln\tfrac{6(K+1)}{\beta}}}{\sqrt{K+1}}\right\}.
        \end{eqnarray*}
        To guarantee $\gap(x_{\avg}^K) \leq \varepsilon$, we choose $K$ in such a way that the right-hand side of the above inequality is smaller than $\varepsilon$ that gives
        \begin{eqnarray*}
            K = \cO\left(\max\left\{\frac{\ell R^2}{\varepsilon} \ln\left(\frac{\ell R^2}{\varepsilon\beta}\right), \frac{\sigma^2R^2}{\varepsilon^2}\ln\left(\frac{\sigma^2R^2}{\varepsilon^2\beta}\right)\right\}\right).
        \end{eqnarray*}
        The total number of oracle calls equals $K+1$.
    \end{enumerate}
\end{proof}

\subsection{Star-Cocoercive Case}

\begin{theorem}\label{thm:main_result_SGDA_sq_norm}
    Let Assumptions~\ref{as:UBV}, \ref{as:star_cocoercivity}, hold for $Q = B_{2R}(x^*)$, where $R \geq R_0 \eqdef \|x^0 - x^*\|$, and
    \begin{gather}
        \gamma \leq \frac{1}{170 \ell \ln \tfrac{4(K+1)}{\beta}}, \label{eq:gamma_comon_SGDA}\\
        \lambda = \frac{R}{60\gamma \ln \tfrac{4(K+1)}{\beta}},  \label{eq:lambda_comon_SGDA}\\
        m \geq \max\left\{1, \frac{97200 (K+1) \gamma^2\sigma^2 \ln\tfrac{4(K+1)}{\beta}}{R^2}\right\}, \label{eq:batch_comon_SGDA}
    \end{gather}
    for some $K \geq 0$ and $\beta \in (0,1]$ such that $\ln \tfrac{4(K+1)}{\beta} \geq 1$. Then, after $K$ iterations the iterates produced by \ref{eq:clipped_SGDA} with probability at least $1 - \beta$ satisfy 
    \begin{equation}
        \frac{1}{K+1}\sum\limits_{k=0}^{K}\|F(x^k)\|^2 \leq \frac{2\ell R^2}{\gamma(K+1)}. \label{eq:main_result_comon_SGDA}
    \end{equation}
\end{theorem}
\begin{proof}
    We introduce new notation: $R_k = \|x^k - x^*\|$ for all $k\geq 0$. The proof is based on deriving via induction that $R_k^2 \leq CR^2$ for some numerical constant $C > 0$. In particular, for each $k = 0,\ldots, K+1$ we define probability event $E_k$ as follows: inequalities
    \begin{gather}
        \|x^t - x^*\|^2 \leq 2R^2, \label{eq:induction_inequality_SGDA}
    \end{gather}
    hold for $t = 0,1,\ldots,k$ simultaneously. Our goal is to prove that $\PP\{E_k\} \geq  1 - \nicefrac{k\beta}{(K+1)}$ for all $k = 0,1,\ldots,K+1$. We notice that inequalities \eqref{eq:thm_SGDA_technical_gap_0} and \eqref{eq:thm_SGDA_technical_gap_5} are derived without assuming monotonicity of $F$. Therefore, following exactly the same step as in the proof of Theorem~\ref{thm:main_result_SGDA} (up to the replacement of $\ln \frac{6(K+1)}{\beta}$ by $\ln \frac{4(K+1)}{\beta}$), we get that
    \begin{gather*}
        R_{T+1}^2 \overset{\eqref{eq:thm_SGDA_technical_gap_5}}{\leq} R^2 + \circledOne + \circledTwo + \circledThree + \circledFour + \circledFive ,\\
        \circledTwo \overset{\eqref{eq:bound_2_SGDA_neg_mon}}{\leq} \frac{R^2}{5},\quad \circledThree \overset{\eqref{eq:bound_3_SGDA_neg_mon}}{\leq} \frac{R^2}{5},\quad \circledFive \overset{\eqref{eq:bound_5_SGDA_neg_mon}}{\leq} \frac{R^2}{5},\\
        \sum\limits_{t=0}^{T}\sigma_t^2 \overset{\eqref{eq:bound_1_variances_SGDA_neg_mon}}{\leq}  \frac{R^4}{150\ln\tfrac{4(K+1)}{\beta}},\quad \sum\limits_{t=0}^{T}\widetilde\sigma_t^2 \overset{\eqref{eq:bound_4_variances_SGDA_neg_mon}}{\leq} \frac{R^4}{150\ln\tfrac{4(K+1)}{\beta}}.
    \end{gather*}
    Moreover, in view of \eqref{eq:bound_1_SGDA_neg_mon}, \eqref{eq:bound_4_SGDA_neg_mon}, and our induction assumption, we have
    \begin{gather*}
        \PP\{E_{T}\} \geq 1 - \frac{T\beta}{K+1},\\
        \PP\{E_{\circledOne}\} \geq 1 - \frac{\beta}{2(K+1)}, \quad \PP\{E_{\circledFour}\} \geq 1 - \frac{\beta}{2(K+1)},
    \end{gather*}
    where probability events $E_{\circledOne}$, and $E_{\circledFour}$ are defined as
    \begin{eqnarray}
        E_{\circledOne}&=& \left\{\text{either} \quad \sum\limits_{t=0}^{T}\sigma_t^2 > \frac{R^4}{150\ln\tfrac{4(K+1)}{\beta}}\quad \text{or}\quad |\circledOne| \leq \frac{R^2}{5}\right\},\notag\\
        E_{\circledFour}&=& \left\{\text{either} \quad \sum\limits_{t=0}^{T}\widetilde\sigma_t^2 > \frac{R^4}{150\ln\tfrac{4(K+1)}{\beta}}\quad \text{or}\quad |\circledFour| \leq \frac{R^2}{5}\right\}.\notag
    \end{eqnarray}
    Putting all of these inequalities together, we obtain that probability event $E_{T-1} \cap E_{\circledOne} \cap E_{\circledFour}$ implies
    \begin{eqnarray*}
        R_{T+1}^2 \leq R^2 + \circledOne + \circledTwo + \circledThree + \circledFour + \circledFive \leq 2R^2.
    \end{eqnarray*}
    Moreover, union bound for the probability events implies
    \begin{equation}
        \PP\{E_{T+1}\} \geq \PP\{E_{T} \cap E_{\circledOne} \cap E_{\circledFour}\} = 1 - \PP\{\overline{E}_{T} \cup \overline{E}_{\circledOne} \cup \overline{E}_{\circledFour}\} \geq 1 - \frac{T\beta}{K+1}.
    \end{equation}
    This is exactly what we wanted to prove (see the paragraph after inequality \eqref{eq:induction_inequality_SGDA}). In particular, $E_{K}$ implies
    \begin{eqnarray*}
        \frac{1}{K+1}\sum\limits_{k=0}^{K}\|F(x^k)\|^2 &\overset{\eqref{eq:thm_SGDA_technical_gap_0}}{\leq}& \frac{\ell(R^2 - R_{K+1}^2)}{\gamma(K+1)}  + \frac{\ell(\circledOne + \circledTwo + \circledThree + \circledFour + \circledFive)}{\gamma(K+1)}\\
        &\leq& \frac{2\ell R^2}{\gamma (K+1)}.
    \end{eqnarray*}
    This finishes the proof.
\end{proof}

\begin{corollary}\label{cor:main_result_SGDA_sq_norm}
    Let the assumptions of Theorem~\ref{thm:main_result_SGDA_sq_norm} hold. Then, the following statements hold.
    \begin{enumerate}
        \item \textbf{Large stepsize/large batch.} The choice of stepsize and batchsize
        \begin{equation}
            \gamma = \frac{1}{170\ell \ln \tfrac{4(K+1)}{\beta}},\quad m = \max\left\{1, \frac{972 (K+1) \sigma^2}{289\ell^2R^2 \ln\tfrac{4(K+1)}{\beta}}\right\} \label{eq:sq_norm_SGDA_large_step_large_batch}
        \end{equation}
        satisfies conditions \eqref{eq:gamma_comon_SGDA} and \eqref{eq:batch_comon_SGDA}. With such choice of $\gamma, m$, and the choice of $\lambda$ as in \eqref{eq:lambda_comon_SGDA}, the iterates produced by \ref{eq:clipped_SGDA} after $K$ iterations with probability at least $1-\beta$ satisfy
        \begin{equation}
            \frac{1}{K+1}\sum\limits_{k=0}^K \|F(x^k)\|^2 \leq \frac{340 \ell^2 R^2 \ln \tfrac{4(K+1)}{\beta}}{K+1}. \label{eq:main_result_sq_norm_SGDA_large_batch}
        \end{equation}
        In particular, to guarantee $\frac{1}{K+1}\sum_{k=0}^K \|F(x^k)\|^2 \leq \varepsilon$ with probability at least $1-\beta$ for some $\varepsilon > 0$ \ref{eq:clipped_SGDA} requires,
        \begin{gather}
            \cO\left(\frac{\ell^2 R^2}{\varepsilon} \ln\left(\frac{\ell^2 R^2}{\varepsilon\beta}\right)\right) \text{ iterations}, \label{eq:sq_norm_SGDA_iteration_complexity_large_batch}\\
            \cO\left(\max\left\{\frac{\ell^2 R^2}{\varepsilon}, \frac{\ell^2\sigma^2R^2}{\varepsilon^2}\right\} \ln\left(\frac{\ell^2 R^2}{\varepsilon\beta}\right)\right) \text{ oracle calls}. \label{eq:sq_norm_SGDA_oracle_complexity_large_batch} 
        \end{gather}
        
        \item \textbf{Small stepsize/small batch.} The choice of stepsize and batchsize
        \begin{equation}
            \gamma = \min\left\{\frac{1}{170\ell \ln \tfrac{4(K+1)}{\beta}}, \frac{R}{180\sigma \sqrt{3(K+1) \ln\tfrac{4(K+1)}{\beta}}}\right\},\quad m = 1 \label{eq:sq_norm_SGDA_small_step_small_batch}
        \end{equation}
        satisfies conditions \eqref{eq:gamma_comon_SGDA} and \eqref{eq:batch_comon_SGDA}. With such choice of $\gamma, m$, and the choice of $\lambda$ as in \eqref{eq:lambda_comon_SGDA}, the iterates produced by \ref{eq:clipped_SGDA} after $K$ iterations with probability at least $1-\beta$ satisfy
        \begin{equation}
            \frac{1}{K+1}\sum\limits_{k=0}^K \|F(x^k)\|^2 \leq \max\left\{\frac{340 \ell^2 R^2 \ln \tfrac{4(K+1)}{\beta}}{K+1}, \frac{360\ell\sigma R \sqrt{3\ln\tfrac{4(K+1)}{\beta}}}{\sqrt{K+1}}\right\}. \label{eq:main_result_sq_norm_SGDA_small_batch}
        \end{equation}
        In particular, to guarantee $\frac{1}{K+1}\sum_{k=0}^K \|F(x^k)\|^2 \leq \varepsilon$ with probability at least $1-\beta$ for some $\varepsilon > 0$, \ref{eq:clipped_SGDA} requires
        \begin{equation}
            \cO\left(\max\left\{\frac{\ell^2 R^2}{\varepsilon} \ln\left(\frac{\ell^2 R^2}{\varepsilon\beta}\right), \frac{\ell^2\sigma^2R^2}{\varepsilon^2}\ln\left(\frac{\ell^2\sigma^2R^2}{\varepsilon^2\beta}\right)\right\}\right) \text{ iterations/oracle calls}. \label{eq:sq_norm_SGDA_iteration_oracle_complexity_small_batch} 
        \end{equation}
    \end{enumerate}
\end{corollary}
\begin{proof}
    \begin{enumerate}
        \item \textbf{Large stepsize/large batch.} First of all, we verify that the choice of $\gamma$ and $m$ from \eqref{eq:sq_norm_SGDA_large_step_large_batch} satisfies conditions \eqref{eq:gamma_comon_SGDA} and \eqref{eq:batch_comon_SGDA}: \eqref{eq:gamma_comon_SGDA} trivially holds and \eqref{eq:batch_comon_SGDA} holds since
        \begin{equation*}
            m = \max\left\{1, \frac{972 (K+1) \sigma^2}{289\ell^2R^2 \ln\tfrac{4(K+1)}{\beta}}\right\} = \max\left\{1, \frac{97200 (K+1) \gamma^2\sigma^2 \ln\tfrac{4(K+1)}{\beta}}{R^2}\right\}.
        \end{equation*}
        Therefore, applying Theorem~\ref{thm:main_result_SGDA_sq_norm}, we derive that with probability at least $1-\beta$
        \begin{equation*}
            \frac{1}{K+1}\sum\limits_{k=0}^K \|F(x^k)\|^2 \leq \frac{2\ell R^2}{\gamma(K+1)} \overset{\eqref{eq:sq_norm_SGDA_large_step_large_batch}}{\leq} \frac{340 \ell^2 R^2 \ln \tfrac{4(K+1)}{\beta}}{K+1}.
        \end{equation*}
        To guarantee $\frac{1}{K+1}\sum_{k=0}^K \|F(x^k)\|^2 \leq \varepsilon$, we choose $K$ in such a way that the right-hand side of the above inequality is smaller than $\varepsilon$ that gives
        \begin{eqnarray*}
            K = \cO\left(\frac{\ell^2 R^2}{\varepsilon} \ln\left(\frac{\ell^2 R^2}{\varepsilon\beta}\right)\right).
        \end{eqnarray*}
        The total number of oracle calls equals
        \begin{eqnarray*}
            m(K+1) &\overset{\eqref{eq:sq_norm_SGDA_large_step_large_batch}}{=}&  \max\left\{K+1, \frac{972 (K+1)^2 \sigma^2}{289\ell^2R^2 \ln\tfrac{4(K+1)}{\beta}}\right\}\\
            &=& \cO\left(\max\left\{\frac{\ell^2 R^2}{\varepsilon}, \frac{\ell^2\sigma^2R^2}{\varepsilon^2}\right\} \ln\left(\frac{\ell^2 R^2}{\varepsilon\beta}\right)\right).
        \end{eqnarray*}
        
        \item \textbf{Small stepsize/small batch.} First of all, we verify that the choice of $\gamma$ and $m$ from \eqref{eq:sq_norm_SGDA_small_step_small_batch} satisfies conditions \eqref{eq:gamma_comon_SGDA} and \eqref{eq:batch_comon_SGDA}:
        \begin{eqnarray*}
            \gamma &=& \min\left\{\frac{1}{170\ell \ln \tfrac{4(K+1)}{\beta}}, \frac{R}{180\sigma \sqrt{3(K+1) \ln\tfrac{4(K+1)}{\beta}}}\right\} \leq \frac{1}{170\ell \ln \tfrac{4(K+1)}{\beta}},\\
            m &=& 1 \overset{\eqref{eq:sq_norm_SGDA_small_step_small_batch}}{\geq} \frac{97200 (K+1) \gamma^2\sigma^2 \ln\tfrac{4(K+1)}{\beta}}{R^2}.
        \end{eqnarray*}
        Therefore, applying Theorem~\ref{thm:main_result_SGDA_sq_norm}, we derive that with probability at least $1-\beta$
        \begin{eqnarray*}
            \frac{1}{K+1}\sum\limits_{k=0}^K \|F(x^k)\|^2 &\leq& \frac{2\ell R^2}{\gamma(K+1)}\\
            &\overset{\eqref{eq:sq_norm_SGDA_small_step_small_batch}}{=}& \max\left\{\frac{340 \ell^2 R^2 \ln \tfrac{4(K+1)}{\beta}}{K+1}, \frac{360\ell\sigma R \sqrt{3\ln\tfrac{4(K+1)}{\beta}}}{\sqrt{K+1}}\right\}.
        \end{eqnarray*}
        To guarantee $\frac{1}{K+1}\sum_{k=0}^K \|F(x^k)\|^2 \leq \varepsilon$, we choose $K$ in such a way that the right-hand side of the above inequality is smaller than $\varepsilon$ that gives
        \begin{eqnarray*}
            K = \cO\left(\max\left\{\frac{\ell^2 R^2}{\varepsilon} \ln\left(\frac{\ell^2 R^2}{\varepsilon\beta}\right), \frac{\ell^2\sigma^2R^2}{\varepsilon^2}\ln\left(\frac{\ell^2\sigma^2R^2}{\varepsilon^2\beta}\right)\right\}\right).
        \end{eqnarray*}
        The total number of oracle calls equals $K+1$.
    \end{enumerate}
\end{proof}

\subsection{Quasi-Strongly Monotone Star-Cocoercive Case}

\begin{lemma}\label{lem:optimization_lemma_str_mon_SGDA}
    Let Assumptions~\ref{as:str_monotonicity}, \ref{as:star_cocoercivity} hold for $Q = B_{2R}(x^*)$, where $R \geq R_0 \eqdef \|x^0 - x^*\|$, and $0 < \gamma \leq \nicefrac{1}{\ell}$. If $x^k$ lies in $B_{2R}(x^*)$ for all $k = 0,1,\ldots, K$ for some $K\geq 0$, then the iterates produced by \ref{eq:clipped_SGDA} satisfy
    \begin{eqnarray}
        \|x^{K+1} - x^*\|^2 &\leq& (1 - \gamma \mu)^{K+1}\|x^0 - x^*\|^2 + 2\gamma \sum\limits_{k=0}^K (1-\gamma\mu)^{K-k}\langle x^k - x^* - \gamma F(x^k), \omega_k \rangle\notag\\
        &&\quad + \gamma^2 \sum\limits_{k=0}^K (1-\gamma\mu)^{K-k} \|\omega_k\|^2, \label{eq:optimization_lemma_SGDA_str_mon}
    \end{eqnarray}
    where $\omega_k$ is defined in \eqref{eq:omega_k_SGDA}.
\end{lemma}
\begin{proof}
     Using the update rule of \ref{eq:clipped_SGDA}, we obtain
     \begin{eqnarray*}
        \|x^{k+1} - x^*\|^2 &=& \|x^k - x^*\|^2 - 2\gamma \langle x^k - x^*,  \tF_{\Bxi^k}(x^k)\rangle + \gamma^2\|\tF_{\Bxi^k}(x^k)\|^2\\
        &=& \|x^k - x^*\|^2 -2\gamma \langle x^k - x^*, F(x^k) \rangle + 2\gamma \langle x^k - x^*, \omega_k \rangle\\
        &&\quad + \gamma^2\|F(x^k)\|^2 - 2\gamma^2 \langle F(x^k), \omega_k \rangle + \gamma^2\|\omega_k\|^2\\
        &=& \|x^k - x^*\|^2 + 2\gamma \langle x^k - x^* - \gamma F(x^k), \omega_k \rangle\\
        &&\quad  - 2\gamma \langle x^k - x^*, F(x^k) \rangle + \gamma^2\|F(x^k)\|^2 + \gamma^2\|\omega_k\|^2\\
        &\overset{\eqref{eq:star_cocoercivity}}{\leq}& \|x^k - x^*\|^2 + 2\gamma \langle x^k - x^* - \gamma F(x^k), \omega_k \rangle\\
        &&\quad  - 2\gamma \langle x^k - x^*, F(x^k) \rangle + \gamma^2\ell\langle x^k-x^*,F(x^k)\rangle + \gamma^2\|\omega_k\|^2\\
        &=& \|x^k - x^*\|^2 + 2\gamma \langle x^k - x^* - \gamma F(x^k), \omega_k \rangle\\
        &&\quad  - 2\gamma \left(1 - \frac{\gamma \ell}{2}\right) \langle x^k - x^*, F(x^k) \rangle + \gamma^2\|\omega_k\|^2\\
        &\overset{\eqref{eq:str_monotonicity}, \gamma \leq \frac{1}{\ell}}{\leq}&\|x^k - x^*\|^2 + 2\gamma \langle x^k - x^* - \gamma F(x^k), \omega_k \rangle\\
        &&\quad  - 2\gamma \mu \left(1 - \frac{\gamma \ell}{2}\right) \|x^k - x^*\|^2 + \gamma^2\|\omega_k\|^2\\
        &\overset{\gamma \leq \frac{1}{\ell}}{\leq}&(1-\gamma \mu)\|x^k - x^*\|^2 + 2\gamma \langle x^k - x^* - \gamma F(x^k), \omega_k \rangle + \gamma^2\|\omega_k\|^2.
    \end{eqnarray*}
    Unrolling the recurrence, we obtain \eqref{eq:optimization_lemma_SGDA_str_mon}.
\end{proof}

\begin{theorem}\label{thm:main_result_SGDA_str_mon}
    Let Assumptions~\ref{as:UBV}, \ref{as:str_monotonicity}, \ref{as:star_cocoercivity} hold for $Q = B_{2R}(x^*) = \{x\in\R^d\mid \|x - x^*\| \leq 2R\}$, where $R \geq R_0 \eqdef \|x^0 - x^*\|$, and 
    \begin{eqnarray}
        0< \gamma &\leq& \frac{1}{400 \ell \ln \tfrac{4(K+1)}{\beta}}, \label{eq:gamma_SGDA_str_mon}\\
        \lambda_k &=& \frac{\exp(-\gamma\mu(1 + \nicefrac{k}{2}))R}{120\gamma \ln \tfrac{4(K+1)}{\beta}}, \label{eq:lambda_SGDA_str_mon}\\
        m_k &\geq& \max\left\{1, \frac{27000\gamma^2 (K+1) \sigma^2\ln \tfrac{4(K+1)}{\beta}}{\exp(-\gamma\mu k) R^2}\right\}, \label{eq:batch_SGDA_str_mon}
    \end{eqnarray}
    for some $K \geq 0$ and $\beta \in (0,1]$ such that $\ln \tfrac{4(K+1)}{\beta} \geq 1$. Then, after $K$ iterations the iterates produced by \ref{eq:clipped_SGDA} with probability at least $1 - \beta$ satisfy 
    \begin{equation}
        \|x^{K+1} - x^*\|^2 \leq 2\exp(-\gamma\mu(K+1))R^2. \label{eq:main_result_str_mon_SGDA}
    \end{equation}
\end{theorem}
\begin{proof}
     As in the proof of Theorem~\ref{thm:main_result_SGDA}, we use the following notation: $R_k = \|x^k - x^*\|^2$, $k \geq 0$. We will derive \eqref{eq:main_result_str_mon_SGDA} by induction. In particular, for each $k = 0,\ldots, K+1$ we define probability event $E_k$ as follows: inequalities
      \begin{equation}
        R_t^2 \leq 2 \exp(-\gamma\mu t) R^2 \label{eq:induction_inequality_str_mon_SGDA}
    \end{equation}
    hold for $t = 0,1,\ldots,k$ simultaneously. Our goal is to prove that $\PP\{E_k\} \geq  1 - \nicefrac{k\beta}{(K+1)}$ for all $k = 0,1,\ldots,K+1$. We use the induction to show this statement. For $k = 0$ the statement is trivial since $R_0^2 \leq 2R^2$ by definition. Next, assume that the statement holds for $k = T \leq K$, i.e., we have $\PP\{E_{T}\} \geq 1 - \nicefrac{T\beta}{(K+1)}$. We need to prove that $\PP\{E_{T+1}\} \geq 1 - \nicefrac{(T+1)\beta}{(K+1)}$. First of all, since $R_t^2 \leq 2\exp(-\gamma\mu t) R^2 \leq 2R^2$, we have $x^t \in B_{2R}(x^*)$. Operator $F$ is  $\ell$-star-cocoercive  on $B_{2R}(x^*)$. Therefore, probability event $E_{T}$ implies
    \begin{eqnarray}
        \|F(x^t)\| &\overset{\eqref{eq:star_cocoercivity}}{\leq}& \ell\|x^t - x^*\| \overset{\eqref{eq:induction_inequality_str_mon_SGDA}}{\leq} \sqrt{2}\ell\exp(- \nicefrac{\gamma\mu t}{2})R \overset{\eqref{eq:gamma_SGDA_str_mon},\eqref{eq:lambda_SGDA_str_mon}}{\leq} \frac{\lambda_t}{2}. \label{eq:operator_bound_x_t_SGDA_str_mon}
    \end{eqnarray}
    and
    \begin{eqnarray}
        \|\omega_t\|^2 &\overset{\eqref{eq:a_plus_b}}{\leq}& 2\|\tF_{\Bxi}(x^t)\|^2 + 2\|F(x^t)\|^2 \overset{\eqref{eq:operator_bound_x_t_SGDA_str_mon}}{\leq} \frac{5}{2}\lambda_t^2 \overset{\eqref{eq:lambda_SGDA_str_mon}}{\leq} \frac{\exp(-\gamma\mu t)R^2}{4\gamma^2} \label{eq:omega_bound_x_t_SGDA_str_mon}
    \end{eqnarray}
    for all $t = 0, 1, \ldots, T$.
    
    Applying Lemma~\ref{lem:optimization_lemma_str_mon_SGDA} and $(1 - \gamma\mu)^T \leq \exp(-\gamma\mu T)$, we get that probability event $E_{T}$ implies
    \begin{eqnarray}
        R_T^2 &\leq& \exp(-\gamma\mu T)R^2 + 2\gamma \sum\limits_{t=0}^{T} (1-\gamma\mu)^{T-t} \langle x^t - x^* - \gamma F(x^t), \omega_t \rangle\notag\\
        &&\quad + \gamma^2 \sum\limits_{t=0}^{T} (1-\gamma\mu)^{T-t} \|\omega_t\|^2. \notag
    \end{eqnarray}
    To estimate the sums in the right-hand side, we introduce new vectors:
    \begin{gather}
        \eta_t = \begin{cases} x^t - x^* - \gamma F(x^t),& \text{if } \|x^t - x^* - \gamma F(x^t)\| \leq \sqrt{2}(1 + \gamma \ell) \exp(- \nicefrac{\gamma\mu t}{2})R,\\ 0,& \text{otherwise}, \end{cases} \label{eq:eta_t_SEG_str_mon}
    \end{gather}
    for $t = 0, 1, \ldots, T$. First of all, we point out that vector $\eta_t$ is bounded with probability $1$, i.e., with probability $1$
    \begin{equation}
        \|\eta_t\| \leq \sqrt{2}(1 + \gamma \ell)\exp(-\nicefrac{\gamma\mu t}{2})R \label{eq:eta_t_bound_SGDA_str_mon} 
    \end{equation}
    for all $t = 0, 1, \ldots, T$. Next, we notice that $E_{T}$ implies $\|F(x^t)\| \leq \sqrt{2}\ell\exp(-\nicefrac{\gamma\mu t}{2})R$ (due to \eqref{eq:operator_bound_x_t_SGDA_str_mon}) for $t = 0, 1, \ldots, T$, i.e., probability event $E_{T}$ implies $\eta_t = x^t - x^* - \gamma F(x^t)$ for all $t = 0,1,\ldots,T$. Therefore, $E_{T}$ implies 
    \begin{eqnarray}
        R_T^2 &\leq& \exp(-\gamma\mu T)R^2 + 2\gamma \sum\limits_{t=0}^{T} (1-\gamma\mu)^{T-t} \langle \eta_t, \omega_t \rangle\notag\\
        &&\quad + \gamma^2 \sum\limits_{t=0}^{T} (1-\gamma\mu)^{T-t} \|\omega_t\|^2. \notag
    \end{eqnarray}
    As in the monotone case, to continue the derivation, we introduce vectors $\omega_t^u, \omega_t^b$ defined as
    \begin{gather}
        \omega_t^u \eqdef \EE_{\Bxi^t}\left[\tF_{\Bxi^t}(x^t)\right] - \tF_{\Bxi^t}(x^t),\quad \omega_t^b \eqdef F(x^t) - \EE_{\Bxi^t}\left[\tF_{\Bxi^t}(x^t)\right], \label{eq:omega_unbias_bias_SGDA_str_mon}
    \end{gather}
    for all $t = 0,\ldots, T$. By definition we have $\omega_t = \omega_t^u + \omega_t^b$ for all $t = 0,\ldots, T$. Using the introduced notation, we continue our derivation as follows: $E_{T}$ implies
    \begin{eqnarray}
        R_T^2 &\overset{\eqref{eq:a_plus_b}}{\leq}& \exp(-\gamma\mu T) R^2 + \underbrace{2\gamma \sum\limits_{t=0}^{T} (1-\gamma\mu)^{T-t} \langle \eta_t, \omega_t^u \rangle}_{\circledOne} + \underbrace{2\gamma \sum\limits_{t=0}^{T} (1-\gamma\mu)^{T-t} \langle \eta_t, \omega_t^b \rangle}_{\circledTwo} \notag\\
        &&\quad + \underbrace{2\gamma^2 \sum\limits_{t=0}^{T} (1-\gamma\mu)^{T-t} \EE_{\Bxi^t}\left[\|\omega_t^u\|^2\right] }_{\circledThree} + \underbrace{2\gamma^2 \sum\limits_{t=0}^{T} (1-\gamma\mu)^{T-t} \left(\|\omega_t^u\|^2 -\EE_{\Bxi^t}\left[\|\omega_t^u\|^2\right] \right)}_{\circledFour}\notag\\
        &&\quad + \underbrace{2\gamma^2 \sum\limits_{t=0}^{T} (1-\gamma\mu)^{T-t} \left(\|\omega_t^b\|^2\right)}_{\circledFive}. \label{eq:SGDA_str_mon_12345_bound}
    \end{eqnarray}
    The rest of the proof is based on deriving good enough upper bounds for $\circledOne, \circledTwo, \circledThree, \circledFour, \circledFive$, i.e., we want to prove that $\circledOne + \circledTwo + \circledThree + \circledFour + \circledFive \leq \exp(-\gamma\mu T) R^2$ with high probability.
    
    Before we move on, we need to derive some useful inequalities for operating with $\omega_t^u, \omega_t^b$. First of all, Lemma~\ref{lem:bias_variance} implies that
    \begin{equation}
         \|\omega_t^u\| \leq 2\lambda_t \label{eq:omega_magnitude_str_mon_SGDA}
    \end{equation}
    for all $t = 0,1, \ldots, T$. Next, since $\{\xi^{i,t}\}_{i=1}^{m_t}$ are independently sampled from $\cD$, we have $\EE_{\Bxi^t}[F_{\Bxi^t}(x^t)] = F(x^t)$, and 
    \begin{gather}
        \EE_{\Bxi^t}\left[\|F_{\Bxi^t}(x^t) - F(x^t)\|^2\right] = \frac{1}{m_t^2}\sum\limits_{i=1}^{m_t} \EE_{\xi^{i,t}}\left[\|F_{\xi^{i,t}}(x^t) - F(x^t)\|^2\right] \overset{\eqref{eq:UBV}}{\leq} \frac{\sigma^2}{m_t}, \notag
    \end{gather}
    for all $l = 0,1, \ldots, T$. Moreover, as we already derived, probability event $E_{T}$ implies that $\|F(x^t)\| \leq \nicefrac{\lambda_t}{2}$ for all $t = 0,1, \ldots, T$ (see \eqref{eq:operator_bound_x_t_SGDA_str_mon}). Therefore, in view of Lemma~\ref{lem:bias_variance}, $E_{T}$ implies that
    \begin{gather}
         \left\|\omega_t^b\right\| \leq \frac{4\sigma^2}{m_t\lambda_t}, \label{eq:bias_theta_omega_str_mon_SGDA}\\
         \EE_{\Bxi^t}\left[\left\|\omega_t\right\|^2\right] \leq \frac{18\sigma^2}{m_t}, \label{eq:distortion_theta_omega_str_mon_SGDA}\\
         \EE_{\Bxi^t}\left[\left\|\omega_t^u\right\|^2\right] \leq \frac{18\sigma^2}{m_t}, \label{eq:variance_theta_omega_str_mon_SGDA}
    \end{gather}
    for all $l = 0,1, \ldots, T$.
    
    \paragraph{Upper bound for $\circledOne$.} Since $\EE_{\Bxi^t}[\omega_t^u] = 0$, we have
    \begin{equation*}
        \EE_{\Bxi^t}\left[2\gamma (1-\gamma\mu)^{T-t} \langle \eta_t, \omega_t^u \rangle\right] = 0.
    \end{equation*}
    Next, the summands in $\circledOne$ are bounded with probability $1$:
    \begin{eqnarray}
        |2\gamma (1-\gamma\mu)^{T-t} \langle \eta_t, \omega_t^u \rangle | &\leq& 2\gamma\exp(-\gamma\mu (T - t)) \|\eta_t\|\cdot \|\omega_t^u\|\notag\\
        &\overset{\eqref{eq:eta_t_bound_SGDA_str_mon},\eqref{eq:omega_magnitude_str_mon_SGDA}}{\leq}& 4\sqrt{2}\gamma (1 + \gamma \ell) \exp(-\gamma\mu (T - \nicefrac{t}{2})) R \lambda_t\notag\\
        &\overset{\eqref{eq:gamma_SGDA_str_mon},\eqref{eq:lambda_SGDA_str_mon}}{\leq}& \frac{\exp(-\gamma\mu T)R^2}{5\ln\tfrac{4(K+1)}{\beta}} \eqdef c. \label{eq:SGDA_str_mon_technical_3_1}
    \end{eqnarray}
    Moreover, these summands have bounded conditional variances $\sigma_t^2 \eqdef \EE_{\Bxi^t}\left[4\gamma^2 (1-\gamma\mu)^{2T-2t} \langle \eta_t, \omega_t^u \rangle^2\right]$:
    \begin{eqnarray}
        \sigma_t^2 &\leq& \EE_{\Bxi^t}\left[4\gamma^2\exp(-\gamma\mu (2T - 2t)) \|\eta_t\|^2\cdot \|\omega_t^u\|^2\right]\notag\\
        &\overset{\eqref{eq:eta_t_bound_SGDA_str_mon}}{\leq}& 8\gamma^2 (1 + \gamma \ell)^2 \exp(-\gamma\mu (2T - t)) R^2 \EE_{\Bxi^t}\left[\|\omega_t^u\|^2\right]\notag\\
        &\overset{\eqref{eq:gamma_SGDA_str_mon}}{\leq}& 10\gamma^2\exp(-\gamma\mu (2T - t))R^2 \EE_{\Bxi^t}\left[\|\omega_t^u\|^2\right]. \label{eq:SGDA_str_mon_technical_3_2}
    \end{eqnarray}
    That is, sequence $\{2\gamma (1-\gamma\mu)^{T-t} \langle \eta_t, \omega_t^u \rangle\}_{t\geq 0}$ is a bounded martingale difference sequence having bounded conditional variances $\{\sigma_t^2\}_{t \geq 0}$. Applying Bernstein's inequality (Lemma~\ref{lem:Bernstein_ineq}) with $X_t = 2\gamma (1-\gamma\mu)^{T-t} \langle \eta_t, \omega_t^u \rangle$, $c$ defined in \eqref{eq:SGDA_str_mon_technical_3_1}, $b = \tfrac{1}{5}\exp(-\gamma\mu T) R^2$, $G = \tfrac{\exp(- 2\gamma\mu T) R^4}{150\ln\frac{4(K+1)}{\beta}}$, we get that
    \begin{equation*}
        \PP\left\{|\circledOne| > \frac{1}{5}\exp(-\gamma\mu T) R^2 \text{ and } \sum\limits_{t=0}^{T}\sigma_t^2 \leq \frac{\exp(- 2\gamma\mu T) R^4}{150\ln\tfrac{4(K+1)}{\beta}}\right\} \leq 2\exp\left(- \frac{b^2}{2G + \nicefrac{2cb}{3}}\right) = \frac{\beta}{2(K+1)}.
    \end{equation*}
    In other words, $\PP\{E_{\circledOne}\} \geq 1 - \tfrac{\beta}{2(K+1)}$, where probability event $E_{\circledOne}$ is defined as
    \begin{equation}
        E_{\circledOne} = \left\{\text{either} \quad \sum\limits_{t=0}^{T}\sigma_t^2 > \frac{\exp(- 2\gamma\mu T) R^4}{150\ln\tfrac{4(K+1)}{\beta}}\quad \text{or}\quad |\circledOne| \leq \frac{1}{5}\exp(-\gamma\mu T) R^2\right\}. \label{eq:bound_3_SGDA_str_mon}
    \end{equation}
    Moreover, we notice here that probability event $E_{T}$ implies that
    \begin{eqnarray}
        \sum\limits_{t=0}^{T}\sigma_t^2 &\overset{\eqref{eq:SGDA_str_mon_technical_3_2}}{\leq}& 10\gamma^2\exp(- 2\gamma\mu T)R^2\sum\limits_{t=0}^{T} \frac{\EE_{\Bxi^t}\left[\|\omega_t^u\|^2\right]}{\exp(-\gamma\mu t)}\notag\\ &\overset{\eqref{eq:variance_theta_omega_str_mon_SGDA}, T \leq K+1}{\leq}& 180\gamma^2\exp(-2\gamma\mu T) R^2 \sigma^2 \sum\limits_{t=0}^{K} \frac{1}{m_t\exp(-\gamma\mu t)}\notag\\
        &\overset{\eqref{eq:batch_SGDA_str_mon}}{\leq}& \frac{\exp(-2\gamma\mu T)R^4}{150\ln\tfrac{6(K+1)}{\beta}}. \label{eq:bound_3_variances_SGDA_str_mon}
    \end{eqnarray}
    
    \paragraph{Upper bound for $\circledTwo$.} Probability event $E_{T}$ implies
    \begin{eqnarray}
        \circledTwo &\leq& 2\gamma \exp(-\gamma\mu T) \sum\limits_{t=0}^{T} \frac{\|\eta_t\|\cdot \|\omega_t^b\|}{\exp(-\gamma\mu t)}\notag\\
        &\overset{\eqref{eq:eta_t_bound_SGDA_str_mon}, \eqref{eq:bias_theta_omega_str_mon_SGDA}}{\leq}& 8\sqrt{2} \gamma (1+\gamma \ell) \exp(-\gamma\mu T) R \sum\limits_{t=0}^{T} \frac{\sigma^2}{m_t \lambda_t \exp(-\nicefrac{\gamma\mu t}{2})}\notag\\
        &\overset{\eqref{eq:lambda_SGDA_str_mon}}{\leq}& 960\sqrt{2} \gamma^2(1+\gamma \ell) \exp(-\gamma\mu (T-1)) \sum\limits_{t=0}^{T} \frac{\sigma^2 \ln\tfrac{4(K+1)}{\beta}}{m_t \exp(-\gamma\mu t)}\notag \\
        &\overset{\eqref{eq:batch_SGDA_str_mon}, T \leq K+1}{\leq}& \frac{1}{5}\exp(-\gamma\mu T) R^2. \label{eq:bound_4_SGDA_str_mon}
    \end{eqnarray}
    
    \paragraph{Upper bound for $\circledThree$.} Probability event $E_{T}$ implies
    \begin{eqnarray}
        \circledThree &=& 2\gamma^2 \exp(-\gamma\mu T) \sum\limits_{t=0}^{T} \frac{\EE_{\Bxi^t}\left[\|\omega_t^u\|^2\right] }{\exp(-\gamma\mu t)} \notag\\
        &\overset{\eqref{eq:variance_theta_omega_str_mon_SGDA}}{\leq}& 36\gamma^2\exp(-\gamma\mu T) \sum\limits_{t=0}^{T} \frac{\sigma^2}{m_t \exp(-\gamma\mu t)} \notag\\
        &\overset{\eqref{eq:batch_SGDA_str_mon}, T \leq K+1}{\leq}& \frac{1}{5} \exp(-\gamma\mu T) R^2. \label{eq:bound_5_SGDA_str_mon}
    \end{eqnarray}
    
    \paragraph{Upper bound for $\circledFour$.} First of all, we have
    \begin{equation*}
        2\gamma^2 (1-\gamma\mu)^{T-t}\EE_{\Bxi^t}\left[\|\omega_t^u\|^2 - \EE_{\Bxi^t}\left[\|\omega_t^u\|^2\right]\right] = 0.
    \end{equation*}
    Next, the summands in $\circledFour$ are bounded with probability $1$:
    \begin{eqnarray}
        2\gamma^2 (1-\gamma\mu)^{T-t}\left| \|\omega_t^u\|^2 - \EE_{\Bxi^t}\left[\|\omega_t^u\|^2\right] \right| 
        &\overset{\eqref{eq:omega_magnitude_str_mon_SGDA}}{\leq}& \frac{16\gamma^2 \exp(-\gamma\mu T) \lambda_t^2}{\exp(-\gamma\mu t)}\notag\\
        &\overset{\eqref{eq:lambda_SGDA_str_mon}}{\leq}& \frac{\exp(-\gamma\mu T)R^2}{5\ln\tfrac{4(K+1)}{\beta}}\notag\\
        &\eqdef& c. \label{eq:SGDA_str_mon_technical_6_1}
    \end{eqnarray}
    Moreover, these summands have bounded conditional variances $\widetilde\sigma_t^2 \eqdef \EE_{\Bxi^t}\left[4\gamma^4 (1-\gamma\mu)^{2T-2t} \left| \|\omega_t^u\|^2 -\EE_{\Bxi^t}\left[\|\omega_t^u\|^2\right] \right|^2\right]$:
    \begin{eqnarray}
        \widetilde\sigma_t^2 &\overset{\eqref{eq:SGDA_str_mon_technical_6_1}}{\leq}& \frac{2\gamma^2\exp(-2\gamma\mu T)R^2}{5\exp(-\gamma\mu t)\ln\tfrac{4(K+1)}{\beta}} \EE_{\Bxi^t}\left[\left| \|\omega_t^u\|^2 -\EE_{\Bxi^t}\left[\|\omega_t^u\|^2\right] \right|\right]\notag\\
        &\leq& \frac{4\gamma^2\exp(-2\gamma\mu T)R^2}{5\exp(-\gamma\mu t)\ln\tfrac{4(K+1)}{\beta}} \EE_{\Bxi^t}\left[\|\omega_t^u\|^2\right]. \label{eq:SGDA_str_mon_technical_6_2}
    \end{eqnarray}
    That is, sequence $\left\{2\gamma^2 (1-\gamma\mu)^{T-t}\left( \|\omega_t^u\|^2 - \EE_{\Bxi^t}\left[\|\omega_t^u\|^2\right] \right)\right\}_{t\geq 0}$ is a bounded martingale difference sequence having bounded conditional variances $\{\widetilde\sigma_t^2\}_{t \geq 0}$. Applying Bernstein's inequality (Lemma~\ref{lem:Bernstein_ineq}) with $X_t = 2\gamma^2 (1-\gamma\mu)^{T-t}\left( \|\omega_t^u\|^2  -\EE_{\Bxi^t}\left[\|\omega_t^u\|^2\right] \right)$, $c$ defined in \eqref{eq:SGDA_str_mon_technical_6_1}, $b = \tfrac{1}{5}\exp(-\gamma\mu T) R^2$, $G = \tfrac{\exp(-2\gamma\mu T) R^4}{150\ln\frac{4(K+1)}{\beta}}$, we get that
    \begin{equation*}
        \PP\left\{|\circledFour| > \frac{1}{5}\exp(-\gamma\mu T) R^2 \text{ and } \sum\limits_{t=0}^{T}\widetilde\sigma_t^2 \leq \frac{\exp(-2\gamma\mu T) R^4}{150\ln\frac{4(K+1)}{\beta}}\right\} \leq 2\exp\left(- \frac{b^2}{2G + \nicefrac{2cb}{3}}\right) = \frac{\beta}{2(K+1)}.
    \end{equation*}
    In other words, $\PP\{E_{\circledFour}\} \geq 1 - \tfrac{\beta}{2(K+1)}$, where probability event $E_{\circledFour}$ is defined as
    \begin{equation}
        E_{\circledFour} = \left\{\text{either} \quad \sum\limits_{t=0}^{T}\widetilde\sigma_t^2 > \frac{\exp(-2\gamma\mu T) R^4}{150\ln\tfrac{4(K+1)}{\beta}}\quad \text{or}\quad |\circledFour| \leq \frac{1}{5}\exp(-\gamma\mu T) R^2\right\}. \label{eq:bound_6_SGDA_str_mon}
    \end{equation}
    Moreover, we notice here that probability event $E_{T}$ implies that
    \begin{eqnarray}
        \sum\limits_{t=0}^{T}\widetilde\sigma_t^2 &\overset{\eqref{eq:SGDA_str_mon_technical_6_2}}{\leq}& \frac{4\gamma^2\exp(-2\gamma\mu T)R^2}{5\ln\tfrac{4(K+1)}{\beta}} \sum\limits_{t=0}^{T} \frac{\EE_{\Bxi^t}\left[\|\omega_t^u\|^2 \right]}{\exp(-\gamma\mu t)}\notag\\ &\overset{\eqref{eq:variance_theta_omega_str_mon_SGDA}, T \leq K+1}{\leq}& \frac{72\gamma^2\exp(-2\gamma\mu T) R^2 \sigma^2}{5\ln\tfrac{4(K+1)}{\beta}} \sum\limits_{t=0}^{K} \frac{1}{m_t\exp(-\gamma\mu t)}\notag\\
        &\overset{\eqref{eq:batch_SGDA_str_mon}}{\leq}& \frac{\exp(-2\gamma\mu T)R^4}{150\ln\tfrac{4(K+1)}{\beta}}. \label{eq:bound_6_variances_SGDA_str_mon}
    \end{eqnarray}
    
    \paragraph{Upper bound for $\circledFive$.} Probability event $E_{T}$ implies
    \begin{eqnarray}
        \circledFive &=&  2\gamma^2 \sum\limits_{t=0}^{T} \exp(-\gamma\mu (T-t)) \left(\|\omega_t^b\|^2 \right)\notag\\
        &\overset{\eqref{eq:bias_theta_omega_str_mon_SGDA}}{\leq}& 32\gamma^2 \exp(-\gamma\mu T) \sum\limits_{t=0}^{T} \frac{\sigma^4}{m_t^2 \lambda_t^2 \exp(-\gamma\mu t)} \notag\\
        &\overset{\eqref{eq:lambda_SGDA_str_mon}}{=}& 460800 \gamma^4 \exp(-\gamma\mu (T-2)) \sum\limits_{t=0}^{T} \frac{\sigma^4 \ln^2\tfrac{4(K+1)}{\beta}}{m_t^2R^2\exp(-2\gamma\mu t)} \notag\\
        &\overset{\eqref{eq:batch_SGDA_str_mon}, T \leq K+1}{\leq}& \frac{1}{5}\exp(-\gamma\mu T) R^2. \label{eq:bound_7_SGDA_str_mon}
    \end{eqnarray}
    
    \paragraph{Final derivation.} Putting all bounds together, we get that $E_{T}$ implies
    \begin{gather*}
        R_T^2 \overset{\eqref{eq:SGDA_str_mon_12345_bound}}{\leq} \exp(-\gamma\mu T) R^2 + \circledOne + \circledTwo + \circledThree + \circledFour + \circledFive,\\
        \circledTwo \overset{\eqref{eq:bound_4_SGDA_str_mon}}{\leq} \frac{1}{5}\exp(-\gamma\mu T)R^2,\\ \circledThree \overset{\eqref{eq:bound_5_SGDA_str_mon}}{\leq} \frac{1}{5}\exp(-\gamma\mu T)R^2,\quad \circledFive \overset{\eqref{eq:bound_7_SGDA_str_mon}}{\leq} \frac{1}{5}\exp(-\gamma\mu T)R^2,\\
         \sum\limits_{t=0}^{T}\sigma_t^2 \overset{\eqref{eq:bound_3_variances_SGDA_str_mon}}{\leq} \frac{\exp(-2\gamma\mu T)R^4}{150\ln\tfrac{4(K+1)}{\beta}},\quad \sum\limits_{t=0}^{T}\widetilde\sigma_t^2 \overset{\eqref{eq:bound_6_variances_SGDA_str_mon}}{\leq}  \frac{\exp(-2\gamma\mu T)R^4}{150\ln\tfrac{4(K+1)}{\beta}}.
    \end{gather*}
    Moreover, in view of \eqref{eq:bound_3_SGDA_str_mon}, \eqref{eq:bound_6_SGDA_str_mon}, and our induction assumption, we have
    \begin{gather*}
        \PP\{E_{T}\} \geq 1 - \frac{T\beta}{K+1},\\
        \PP\{E_{\circledOne}\} \geq 1 - \frac{\beta}{2(K+1)}, \quad \PP\{E_{\circledFour}\} \geq 1 - \frac{\beta}{2(K+1)},
    \end{gather*}
    where probability events $E_{\circledOne}$, and $E_{\circledFour}$ are defined as
    \begin{eqnarray}
        E_{\circledOne}&=& \left\{\text{either} \quad \sum\limits_{t=0}^{T}\sigma_t^2 > \frac{\exp(-2\gamma\mu T) R^4}{150\ln\tfrac{4(K+1)}{\beta}}\quad \text{or}\quad |\circledOne| \leq \frac{1}{5}\exp(-\gamma\mu T) R^2\right\},\notag\\
        E_{\circledFour}&=& \left\{\text{either} \quad \sum\limits_{t=0}^{T}\widetilde\sigma_t^2 > \frac{\exp(-2\gamma\mu T) R^4}{150\ln\tfrac{4(K+1)}{\beta}}\quad \text{or}\quad |\circledFour| \leq \frac{1}{5}\exp(-\gamma\mu T) R^2\right\}.\notag
    \end{eqnarray}
    Putting all of these inequalities together, we obtain that probability event $E_{T} \cap E_{\circledOne} \cap E_{\circledFour}$ implies
    \begin{eqnarray*}
        R_T^2 &\overset{\eqref{eq:SGDA_str_mon_12345_bound}}{\leq}& \exp(-\gamma\mu T) R^2 + \circledOne + \circledTwo + \circledThree + \circledFour + \circledFive\\
        &\leq& 2\exp(-\gamma\mu T) R^2.
    \end{eqnarray*}
     Moreover, union bound for the probability events implies
    \begin{equation}
        \PP\{E_{T+1}\} \geq \PP\{E_{T} \cap E_{\circledOne} \cap E_{\circledFour}\} = 1 - \PP\{\overline{E}_{T} \cup \overline{E}_{\circledOne} \cup \overline{E}_{\circledFour}\} \geq 1 - \frac{(T+1)\beta}{K+1}.
    \end{equation}
    This is exactly what we wanted to prove (see the paragraph after inequality \eqref{eq:induction_inequality_str_mon_SGDA}). In particular, with probability at least $1 - \beta$ we have
    \begin{equation}
        \|x^{K+1} - x^*\|^2 \leq 2\exp(-\gamma\mu (K+1))R^2, \notag
    \end{equation}
    which finishes the proof.
\end{proof}

\begin{corollary}\label{cor:main_result_SGDA_str_mon}
    Let the assumptions of Theorem~\ref{thm:main_result_SGDA_str_mon} hold. Then, the following statements hold.
\begin{enumerate}
        \item \textbf{Large stepsize/large batch.} The choice of stepsize and batchsize
        \begin{equation}
            \gamma = \frac{1}{400 \ell \ln \tfrac{4(K+1)}{\beta}},\quad m_k = \max\left\{1, \frac{27000\gamma^2 (K+1) \sigma^2\ln \tfrac{4(K+1)}{\beta}}{\exp(-\gamma\mu k) R^2}\right\} \label{eq:str_mon_SGDA_large_step_large_batch}
        \end{equation}
        satisfies conditions \eqref{eq:gamma_SGDA_str_mon} and \eqref{eq:batch_SGDA_str_mon}. With such choice of $\gamma, m_k$, and the choice of $\lambda_k$ as in \eqref{eq:lambda_SGDA_str_mon}, the iterates produced by \ref{eq:clipped_SGDA} after $K$ iterations with probability at least $1-\beta$ satisfy
        \begin{equation}
            \|x^{K+1} - x^*\|^2 \leq 2\exp\left( - \frac{\mu(K+1)}{400 \ell \ln \tfrac{4(K+1)}{\beta}}\right)R^2. \label{eq:main_result_str_mon_SGDA_large_batch}
        \end{equation}
        In particular, to guarantee $\|x^{K+1} - x^*\|^2 \leq \varepsilon$ with probability at least $1-\beta$ for some $\varepsilon > 0$ \ref{eq:clipped_SGDA} requires
        \begin{gather}
            \cO\left(\frac{\ell}{\mu} \ln\left(\frac{R^2}{\varepsilon}\right)\ln\left(\frac{\ell}{\mu \beta}\ln\left(\frac{R^2}{\varepsilon}\right)\right)\right) \text{ iterations}, \label{eq:str_mon_SGDA_iteration_complexity_large_batch}\\
            \cO\left(\max\left\{\frac{\ell}{\mu}, \frac{\sigma^2}{\mu^2 \varepsilon}\right\} \ln\left(\frac{R^2}{\varepsilon}\right) \ln\left(\frac{\ell}{\mu\beta}\ln\left(\frac{R^2}{\varepsilon}\right)\right)\right) \text{ oracle calls}. \label{eq:str_mon_SGDA_oracle_complexity_large_batch} 
        \end{gather}
        
        \item \textbf{Small stepsize/small batch.} The choice of stepsize and batchsize
        \begin{equation}
            \gamma = \min\left\{\frac{1}{400 \ell \ln \tfrac{4(K+1)}{\beta}}, \frac{\ln\left(B_K\right)}{\mu (K+1)}\right\},\quad m_k \equiv 1 \label{eq:str_mon_SGDA_small_step_small_batch}
        \end{equation}
        satisfies conditions \eqref{eq:gamma_SGDA_str_mon} and \eqref{eq:batch_SGDA_str_mon}, where $B_K = \max\left\{2, \frac{(K+1)\mu^2 R^2}{27000\sigma^2\ln\left(\frac{4(K+1)}{\beta}\right)\ln^2(B_K)}\right\} = \cO\left(\max\left\{2, \frac{(K+1)\mu^2 R^2}{27000\sigma^2\ln\left(\frac{4(K+1)}{\beta}\right)\ln^2\left(\max\left\{2, \frac{(K+1)\mu^2 R^2}{27000\sigma^2\ln\left(\frac{4(K+1)}{\beta}\right)}\right\}\right)}\right\}\right)$. With such choice of $\gamma, m_k$, and the choice of $\lambda_k$ as in \eqref{eq:lambda_SGDA_str_mon}, the iterates produced by \ref{eq:clipped_SGDA} after $K$ iterations with probability at least $1-\beta$ satisfy
        \begin{equation}
            \|x^{K+1} - x^*\|^2 \leq \max\left\{2\exp\left( - \frac{\mu(K+1)}{400 \ell \ln \tfrac{4(K+1)}{\beta}}\right)R^2, \frac{54000\sigma^2\ln\left(\frac{4(K+1)}{\beta}\right) \ln^2 (B_K)}{\mu^2(K+1)}\right\}. \label{eq:main_result_str_mon_SGDA_small_batch}
        \end{equation}
        In particular, to guarantee $\|x^{K+1} - x^*\|^2 \leq \varepsilon$ with probability at least $1-\beta$ for some $\varepsilon > 0$ \ref{eq:clipped_SGDA} requires
        \begin{equation}
            \cO\left(\max\left\{\frac{\ell}{\mu} \ln\left(\frac{R^2}{\varepsilon}\right)\ln\left(\frac{\ell}{\mu\beta}\ln\left(\frac{R^2}{\varepsilon}\right)\right), \frac{\sigma^2}{\mu^2 \varepsilon}\ln\left(\frac{\sigma^2}{\mu^2 \varepsilon\beta}\right)\ln^2\left(B_\varepsilon\right)\right\} \right) \label{eq:str_mon_SGDA_iteration_oracle_complexity_small_batch} 
        \end{equation}
        iterations/oracle calls, where
        \begin{equation*}
            B_\varepsilon = \max\left\{2, \frac{R^2}{\varepsilon \ln\left(\frac{\sigma^2}{\mu^2 \varepsilon\beta}\right) \ln^2\left(\max\left\{2 , \frac{R^2}{\varepsilon \ln\left(\frac{\sigma^2}{\mu^2 \varepsilon\beta}\right)}\right\}\right)} \right\}.
        \end{equation*}
    \end{enumerate}
\end{corollary}
\begin{proof}
     \begin{enumerate}
        \item \textbf{Large stepsize/large batch.} First of all, it is easy to see that the choice of $\gamma$ and $m_k$ from \eqref{eq:str_mon_SGDA_large_step_large_batch} satisfies conditions \eqref{eq:gamma_SGDA_str_mon} and \eqref{eq:batch_SGDA_str_mon}. Therefore, applying Theorem~\ref{thm:main_result_SGDA_str_mon}, we derive that with probability at least $1-\beta$
        \begin{equation*}
            \|x^{K+1} - x^*\|^2 \leq 2\exp(-\gamma\mu(K+1))R^2 \overset{\eqref{eq:str_mon_SGDA_large_step_large_batch}}{=} 2\exp\left(- \frac{\mu (K+1)}{400 \ell \ln \tfrac{4(K+1)}{\beta}}\right)R^2.
        \end{equation*}
        To guarantee $\|x^{K+1} - x^*\|^2 \leq \varepsilon$, we choose $K$ in such a way that the right-hand side of the above inequality is smaller than $\varepsilon$ that gives
        \begin{eqnarray*}
            K = \cO\left(\frac{\ell}{\mu} \ln\left(\frac{R^2}{\varepsilon}\right)\ln\left(\frac{\ell}{\mu \beta}\ln\left(\frac{R^2}{\varepsilon}\right)\right)\right).
        \end{eqnarray*}
        The total number of oracle calls equals
        \begin{eqnarray*}
            \sum\limits_{k=0}^{K}m_k &\overset{\eqref{eq:str_mon_SGDA_large_step_large_batch}}{=}&  \sum\limits_{k=0}^{K}\max\left\{1, \frac{27000\gamma^2 (K+1) \sigma^2\ln \tfrac{4(K+1)}{\beta}}{\exp(-\gamma\mu k) R^2}\right\}\\
            &=& \cO\left(\max\left\{K,\frac{\gamma(K+1)\exp(\gamma\mu(K+1))\sigma^2\ln \tfrac{4(K+1)}{\beta}}{\mu R^2}\right\}\right)\\
            &=& \cO\left(\max\left\{\frac{\ell}{\mu}, \frac{\sigma^2}{\mu^2 \varepsilon}\right\} \ln\left(\frac{R^2}{\varepsilon}\right) \ln\left(\frac{\ell}{\mu\beta}\ln\left(\frac{R^2}{\varepsilon}\right)\right)\right).
        \end{eqnarray*}
        
        \item \textbf{Small stepsize/small batch.} First of all, we verify that the choice of $\gamma$ and $m_k$ from \eqref{eq:str_mon_SGDA_small_step_small_batch} satisfies conditions \eqref{eq:gamma_SGDA_str_mon} and \eqref{eq:batch_SGDA_str_mon}: \eqref{eq:gamma_SGDA_str_mon} trivially holds and \eqref{eq:batch_SGDA_str_mon} holds since for all $k = 0,\ldots, K$
        \begin{eqnarray*}
            \frac{27000\gamma^2 (K+1) \sigma^2\ln \tfrac{4(K+1)}{\beta}}{\exp(-\gamma\mu k) R^2} &\leq& \frac{27000\gamma^2 (K+1) \sigma^2\ln \tfrac{4(K+1)}{\beta}}{\exp(-\gamma\mu (K+1))R^2} \\
            &\overset{\eqref{eq:str_mon_SGDA_small_step_small_batch}}{\leq}& \frac{27000\ln^2\left(B_K\right) \exp(\gamma\mu(K+1)) \sigma^2\ln \tfrac{4(K+1)}{\beta}}{\mu^2 (K+1) R^2}\\
            &\overset{\eqref{eq:str_mon_SGDA_small_step_small_batch}}{\leq}& 1.
        \end{eqnarray*}
        Therefore, applying Theorem~\ref{thm:main_result_SGDA_str_mon}, we derive that with probability at least $1-\beta$
        \begin{eqnarray*}
            \|x^{K+1} - x^*\|^2 &\leq& 2\exp(-\gamma\mu(K+1))R^2\\
            &\overset{\eqref{eq:str_mon_SGDA_small_step_small_batch}}{=}& \max\left\{2\exp\left(- \frac{\mu (K+1)}{400 \ell \ln \tfrac{4(K+1)}{\beta}}\right)R^2, \frac{2R^2}{B_K}\right\}\\
            &=& \max\left\{2\exp\left(- \frac{\mu (K+1)}{400 \ell \ln \tfrac{4(K+1)}{\beta}}\right)R^2, \frac{54000\sigma^2\ln\left(\frac{4(K+1)}{\beta}\right) \ln^2 (B_K)}{\mu^2(K+1)}\right\}.
        \end{eqnarray*}
        To guarantee $\|x^{K+1} - x^*\|^2 \leq \varepsilon$, we choose $K$ in such a way that the right-hand side of the above inequality is smaller than $\varepsilon$ that gives $K$ of the order
        \begin{eqnarray*}
            \cO\left(\max\left\{\frac{\ell}{\mu} \ln\left(\frac{R^2}{\varepsilon}\right)\ln\left(\frac{\ell}{\mu\beta}\ln\left(\frac{R^2}{\varepsilon}\right)\right), \frac{\sigma^2}{\mu^2 \varepsilon}\ln\left(\frac{\sigma^2}{\mu^2 \varepsilon\beta}\right)\ln^2\left(B_\varepsilon\right)\right\} \right),
        \end{eqnarray*}
        where
        \begin{equation*}
            B_\varepsilon = \max\left\{2, \frac{R^2}{\varepsilon \ln\left(\frac{\sigma^2}{\mu^2 \varepsilon\beta}\right) \ln^2\left(\max\left\{2 , \frac{R^2}{\varepsilon \ln\left(\frac{\sigma^2}{\mu^2 \varepsilon\beta}\right)}\right\}\right)} \right\}.
        \end{equation*}
        The total number of oracle calls equals $\sum_{k=0}^K m_k = (K+1)$.
    \end{enumerate}
\end{proof}

\newpage

\section{Extra Experiments}\label{app:extra_exps}

In this section, we provide more details for the experiments done in \S~\ref{sec:experiments}, as well as additional tables, figures, and image samples from some of our trained models.

\subsection{WGAN-GP}

In all cases, everything in the experimental setup other than learning rates and clip values remained constant.
We use the same ResNet architectures and training parameters specified in \cite{gulrajani2017improved}: the gradient penalty coefficient $\lambda_{GP}=10$, $n_{dis}=5$ where $n_{dis}$ is the number of discriminator steps for every generator step, and a learning rate decayed linearly to 0 over 100k steps.
The only exception is we double the feature map of the generator from 128 to 256 dimensions.
For all stochastic extragradient (\algname{SEG}) methods, we use the \algname{ExtraSGD} implementation provided by \cite{gidel2019variational}.
We alternate between exploration and update steps and do not treat the exploration steps as ``free'' -- this means we only have 50k parameter updates as opposed to 100k for all \algname{SGDA} methods (we decay the learning rate twice as fast such that it still reaches 0 after 50k parameter updates).

All of the hyperparameter sweeps performed for \algname{SGDA}, \algname{clipped-SGDA}, \algname{clipped-SEG}, \algname{clipped-SGDA} (coordinate), and \algname{clipped-SEG} (coordinate), as well as the associated best FID score obtained within the first 35k training steps, can be found in
Tables \ref{tab:hparam_sgda}, \ref{tab:hparam_clipval_sgda}, \ref{tab:hparam_clipnorm_seg}, \ref{tab:hparam_clipval_sgda}, and \ref{tab:hparam_clipval_seg} respectively.
\textbf{Bold} rows denote the hyperparameters that were trained for the full 100k steps and are henceforth referred to as the \textit{``best models''}.
For each of the methods tested, additional samples for the best models trained can be found in Figures \ref{fig:wgangp/samples/sgd}, \ref{fig:wgangp/samples/sgd_norm}, \ref{fig:wgangp/samples/extrasgd_norm}, \ref{fig:wgangp/samples/sgd_val}, \& \ref{fig:wgangp/samples/extrasgd_val}.
We also plot the evolution of the gradient noise histograms in Figures \ref{fig:app/evo/sgda}, \ref{fig:app/evo/clipnorm-sgda}, \ref{fig:app/evo/clipnorm-SEG}, \ref{fig:app/evo/clipval-SGDA}, \& \ref{fig:app/evo/clipval-SEG}.
We emphasize that our goal is not to get the best possible FID score (e.g. are often able to obtain marginally better FIDs by training for longer), but rather to compare the systematic differences in performance between the various unclipped and clipped methods. 
Therefore, log-space hyperparameter sweeps are appropriate for our experiments and we do not tune further.

\subsection{StyleGAN2}

We train on FFHQ downsampled to $128\times128$ pixels, and use the recommended StyleGAN2 hyperparameter configuration for this resolution: batch~size~$=32$, $\gamma=0.1024$, map~depth~$=2$, and channel~multiplier~$=16384$.
For both \algname{SGDA} and \algname{clipped-SGDA}, we sweep over a (roughly) log-scale of learning rates and clipping values; a summary of the hyperparmaters and best FID scores obtained Table~\ref{tab:stylegan2/hparam/sgd} and Table~\ref{tab:stylegan2/hparam/clipval-sgd} respectively.

Based on the results in Table~\ref{tab:stylegan2/hparam/clipval-sgd}, the best hyperparameters are lr=$0.35$ and clip=$0.0025$ which we then used to train our \textit{``best model''}.
We trained for longer, and decayed the learning rate twice (by a factor of $\times10$) when the FID plateaued or worsened. 
The best schedule we found was to scale the learning rate by $\times 0.1$ after 6000 kimgs (thousands of real images shown to the discriminator), by another $\times 0.1$ after 3600 kimgs, and then train until the FID begins increasing (for another 8000 kimgs) -- for a total of 17600 kimgs.
We did not explore different scale factors or other schedules (such as cosine annealing).
Additional samples for this model can be found in Figure~\ref{fig:stylegan/clippedsgda/moresamples}.

In general, we observe that every \algname{SGDA}-trained model for the wide range of learning rates we tested failed to improve the FID, while models trained with \algname{clipped-SGDA} (with appropriately set hyperparameters) are generally able to learn some meaningful features and improve the FID.
We show this behaviour in Figure \ref{fig:stylegan/FID} -- the FID scores for \algname{SGDA}-trained models fluctuate around $320$ and only generate noise such as the samples shown in Figure~\ref{fig:stylegan/sgda/moresamples}, which is in contrast to models trained with \algname{clipped-SGDA}.
Note that the range of the hyperparameter sweep is fairly narrow and favourable for \algname{clipped-SGDA}, while being quite wide for \algname{SGDA}.
The purpose for these parameter ranges is not to directly compare the parameter sweeps (which would unfairly favour \algname{clipped-SGDA}), but to show that in general \algname{SGDA} fails, while \algname{clipped-SGDA} is capable of learning.

\newpage

\begin{table}[]
    \centering
    \begin{minipage}{.46\linewidth}
        \caption{\algname{SGDA} hyperparameter sweep, and the best FID score obtained in 35k training steps.}
        \label{tab:hparam_sgda}
        \centering
        \medskip
        \begin{tabular}{ccc}
            \toprule
              G-LR &   D-LR &   FID \\
            \midrule
             6e-06 &  6e-06 & 233.3 \\
             2e-05 &  2e-05 & 177.2 \\
             2e-05 &  4e-05 & 183.4 \\
             2e-05 &  8e-05 & 187.3 \\
            0.0002 & 0.0002 &  85.6 \\
            \textbf{0.0002} & \textbf{0.0004} &  \textbf{82.8} \\
            0.0002 & 0.0008 & \texttt{NaN} \\
             0.002 &  0.002 & \texttt{NaN} \\
              0.02 &   0.02 & \texttt{NaN} \\
               0.2 &    0.2 & \texttt{NaN} \\
            \bottomrule
        \end{tabular}
    \end{minipage}
    \hfill
    \begin{minipage}{.46\linewidth}
        \caption{\algname{SEG} hyperparameter sweep, and the best FID score obtained in 35k training steps.}
        \label{tab:hparam_seg}
        \centering
        \medskip
        %
        %       SEG
        % 
        \begin{tabular}{ccc}
            \toprule
              G-LR &   D-LR &   FID \\
            \midrule
             6e-06 &  6e-06 & 236.1 \\
             2e-05 &  2e-05 & 208.6 \\
             2e-05 &  4e-05 & 213.7 \\
             \textbf{4e-05} &  \textbf{4e-05} & \textbf{176.5} \\
             4e-05 & 0.0001 & \texttt{NaN} \\
            0.0002 & 0.0002 & \texttt{NaN} \\
            0.0002 & 0.0004 & \texttt{NaN} \\
            0.0002 & 0.0008 & \texttt{NaN} \\
             0.002 &  0.002 & \texttt{NaN} \\
              0.02 &   0.02 & \texttt{NaN} \\
               0.2 &    0.2 & \texttt{NaN} \\
                 2 &      2 & \texttt{NaN} \\
            \bottomrule
        \end{tabular}

    \end{minipage}
\end{table}

\begin{table}[]
    \begin{minipage}{.46\linewidth}
        \centering
        \caption{\algname{clipped-SGDA} (norm) hyperparameter sweep, and the best FID score obtained in 35k training steps.}
        \label{tab:hparam_clipnorm_sgda}
        \medskip
        %
        %       normclip SGDA
        % 
        \begin{tabular}{ccccc}
        \toprule
         G-LR &  D-LR &  G-clip &  D-clip &   FID \\
        \midrule
        0.002 & 0.002 &     0.1 &     0.1 & 257.6 \\
        0.002 & 0.002 &       1 &       1 & 121.6 \\
        0.002 & 0.002 &      10 &      10 & 145.4 \\
         0.02 &  0.02 &     0.1 &     0.1 & 115.4 \\
         0.02 &  0.02 &       1 &       1 & 141.8 \\
         0.02 &  0.02 &      10 &      10 &  27.4 \\
          0.2 &   0.2 &     0.1 &     0.1 & 133.0 \\
          \textbf{0.2} &   \textbf{0.2} &       \textbf{1} &       \textbf{1} &  \textbf{26.3} \\
            2 &     2 &     0.1 &     0.1 &  26.1 \\
        \bottomrule
        \end{tabular}
    \end{minipage}
    \hfill
    \begin{minipage}{.46\linewidth}
        \centering
        \caption{\algname{clipped-SEG} (norm) hyperparameter sweep, and the best FID score obtained in 35k training steps (17.5k parameter updates).}
        \label{tab:hparam_clipnorm_seg}
        \medskip
        %
        %       normclip SEG
        % 
        \begin{tabular}{ccccc}
        \toprule
         G-LR &  D-LR &  G-clip &  D-clip &   FID \\
        \midrule
        0.002 & 0.002 &     0.1 &     0.1 & 232.5 \\
        0.002 & 0.002 &       1 &       1 & 150.5 \\
        0.002 & 0.002 &      10 &      10 & 192.7 \\
         0.02 &  0.02 &     0.1 &     0.1 & 161.0 \\
         0.02 &  0.02 &       1 &       1 & 160.3 \\
         0.02 &  0.02 &      10 &      10 &  39.3 \\
          0.2 &   0.2 &     0.1 &     0.1 & 160.0 \\
          \textbf{0.2} &   \textbf{0.2} &       \textbf{1} &       \textbf{1} &  \textbf{36.3} \\
            2 &     2 &    0.1 &     0.1 &  37.7 \\
        \bottomrule
        \end{tabular}
    \end{minipage} \\
\end{table}

\begin{table}[]
    \begin{minipage}{.46\linewidth}
        \centering
        \caption{\algname{clipped-SGDA} (coordinate) hyperparameter sweep, and the best FID score obtained in 35k training steps.}
        \label{tab:hparam_clipval_sgda}
        \medskip
        %
        %       valclip SGDA
        % 
        \begin{tabular}{ccccc}
        \toprule
          G-LR &   D-LR &  G-clip &  D-clip &   FID \\
        \midrule
        0.0002 & 0.0002 &   0.001 &   0.001 & 292.2 \\
        0.0002 & 0.0002 &    0.01 &    0.01 & 108.6 \\
        0.0002 & 0.0002 &     0.1 &     0.1 &  91.5 \\
         0.002 &  0.002 &   0.001 &   0.001 &  76.5 \\
         0.002 &  0.002 &    0.01 &    0.01 &  43.5 \\
         0.002 &  0.002 &     0.1 &     0.1 &  45.1 \\
          0.02 &   0.02 &   0.001 &   0.001 &  37.3 \\
          \textbf{0.02} &   \textbf{0.02} &    \textbf{0.01} &    \textbf{0.01} &  \textbf{26.7} \\
          0.02 &   0.02 &     0.1 &     0.1 &  34.7 \\
        \bottomrule
        \end{tabular}
    \end{minipage}
    \hfill
    \begin{minipage}{.46\linewidth}
        \centering
        \caption{\algname{clipped-SEG} (coordinate) hyperparameter sweep, and the best FID score obtained in 35k training steps (17.5k parameter updates).}
        \label{tab:hparam_clipval_seg}
        \medskip
        %
        %       valclip SEG
        % 
        \begin{tabular}{ccccc}
        \toprule
          G-LR &   D-LR &  G-clip &  D-clip &   FID \\
        \midrule
        0.0002 & 0.0002 &   0.001 &   0.001 & 298.7 \\
        0.0002 & 0.0002 &    0.01 &    0.01 & 146.5 \\
        0.0002 & 0.0002 &     0.1 &     0.1 & 158.4 \\
         0.002 &  0.002 &   0.001 &   0.001 & 112.8 \\
         0.002 &  0.002 &    0.01 &    0.01 &  52.7 \\
         0.002 &  0.002 &     0.1 &     0.1 &  66.5 \\
          0.02 &   0.02 &   0.001 &   0.001 &  43.5 \\
         \textbf{0.02} &   \textbf{0.02} &    \textbf{0.01} &    \textbf{0.01} &  \textbf{36.2} \\
          0.02 &   0.02 &     0.1 &     0.1 &  75.3 \\
        \bottomrule
        \end{tabular}
    \end{minipage}
\end{table}

\newpage

\begin{figure}[]
    \centering
    \hspace*{-1.8cm}       
    \includegraphics[width=1.07\textwidth]{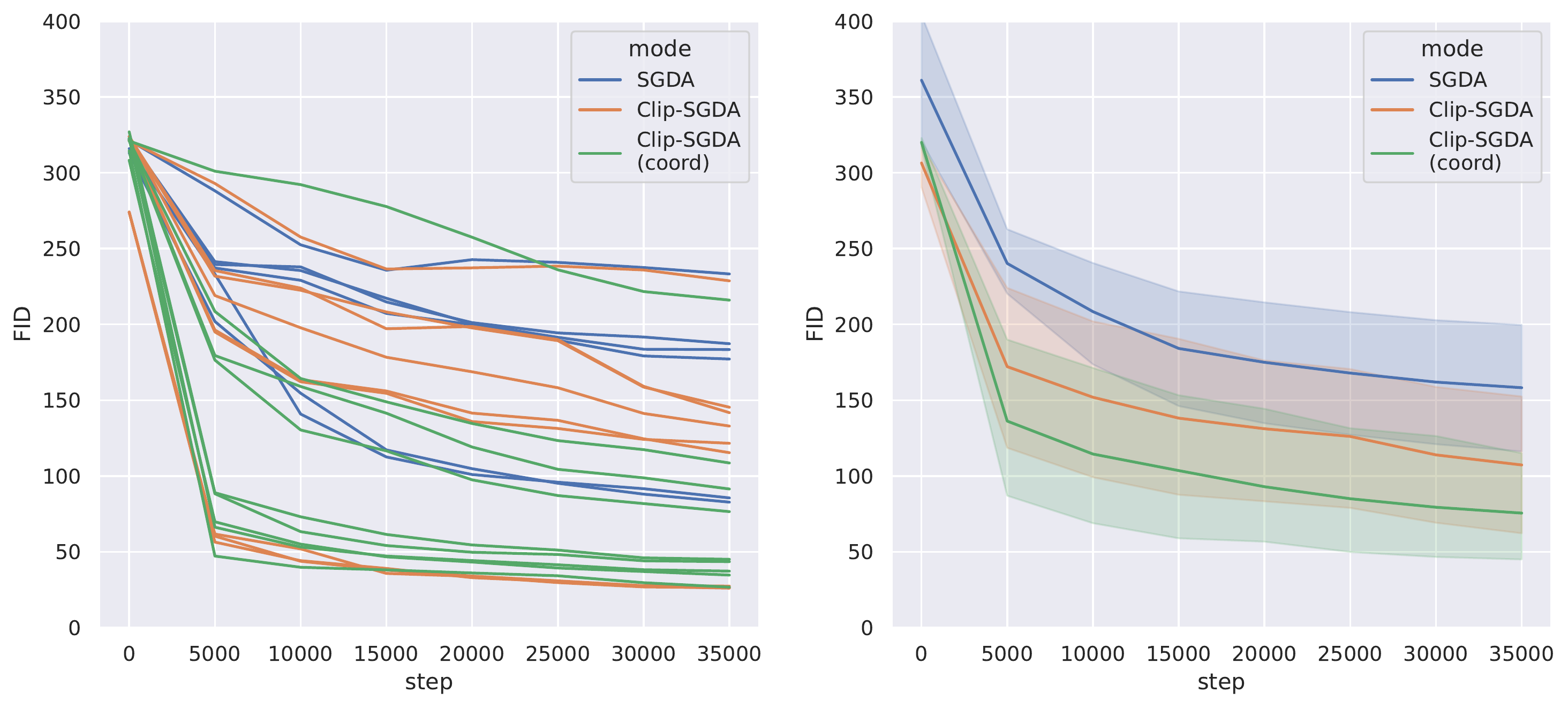}
    \caption{FID curves when training WGAN-GP for 35k steps with \algname{SGDA}, \algname{clipped-SGDA} (norm), and \algname{clipped-SGDA} (coordinate), corresponding to the hyperparameters in Tables~\ref{tab:hparam_sgda},~\ref{tab:hparam_clipnorm_sgda}~\&~\ref{tab:hparam_clipval_sgda} respectively.
    The left figure is the individual runs for each choice of hyperparameters, and the right is the mean and $95\%$ confidence interval of these runs.
    Note that 4 of 10 runs diverged (\texttt{NaN} loss) for \algname{SGDA}, which is not reflected in the mean FID for the right figure beyond the first step.
    }
    \label{fig:wgangp/sgda_FID}
\end{figure}

\begin{figure}[]
    \centering
    \hspace*{-1.8cm}       
    \includegraphics[width=1.07\textwidth]{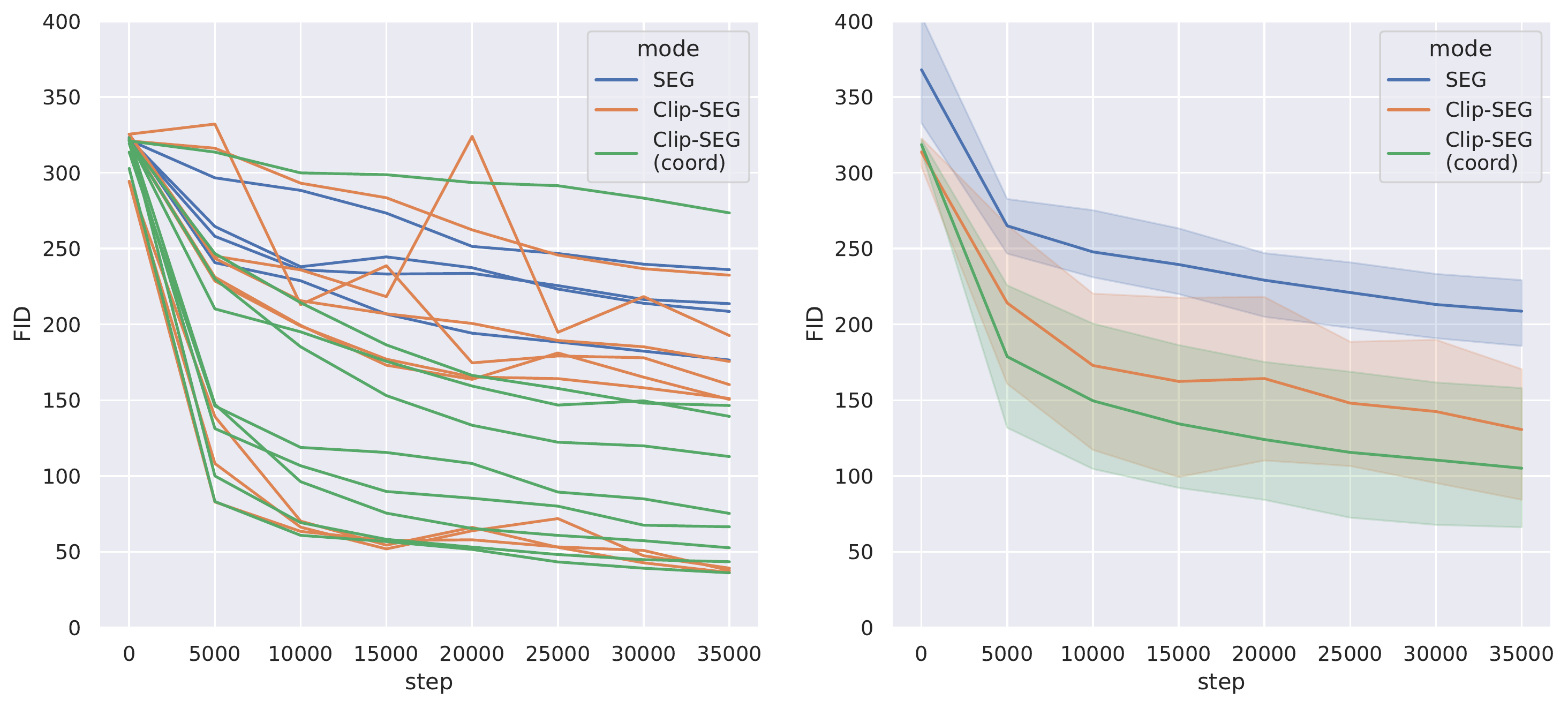}
    \caption{FID curves when training WGAN-GP for 35k steps with \algname{SEG}, \algname{clipped-SEG} (norm), and \algname{clipped-SEG} (coordinate), corresponding to the hyperparameters in Tables~\ref{tab:hparam_seg},~\ref{tab:hparam_clipnorm_seg}~\&~\ref{tab:hparam_clipval_seg} respectively.
    The left figure is the individual runs for each choice of hyperparameters, and the right is the mean and $95\%$ confidence interval of these runs.
    Note that 8 of 12 runs diverged (\texttt{NaN} loss) for \algname{SEG}, which is not reflected in the mean FID for the right figure beyond the first step.
    }
    \label{fig:wgangp/seg_FID}
\end{figure}

\newpage

\begin{figure}[]
    \centering
    \includegraphics[width=\textwidth]{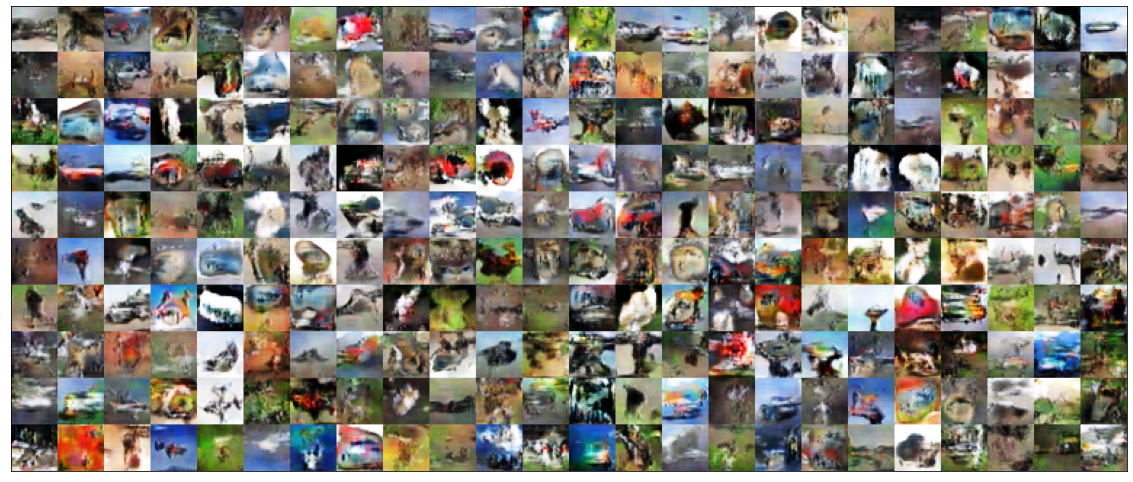}
    \caption{Samples generated from the best WGAN-GP model trained with \algname{SGDA}.}
    \label{fig:wgangp/samples/sgd}
\end{figure}
\begin{figure}[]
    \centering
    \includegraphics[width=\textwidth]{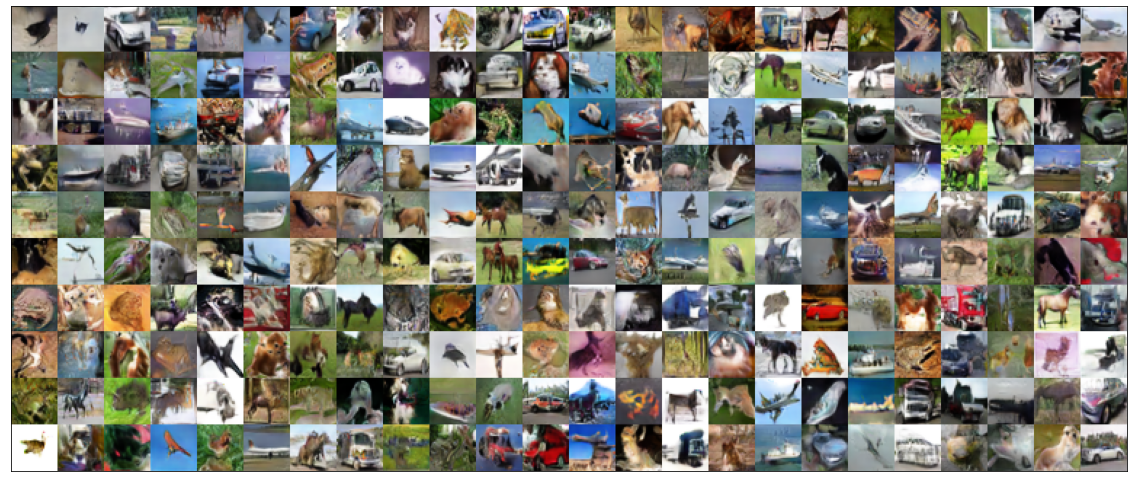}
    \caption{Samples generated from the best WGAN-GP model trained with \algname{clipped-SGDA}.}
    \label{fig:wgangp/samples/sgd_norm}
\end{figure}
\begin{figure}[]
    \centering
    \includegraphics[width=\textwidth]{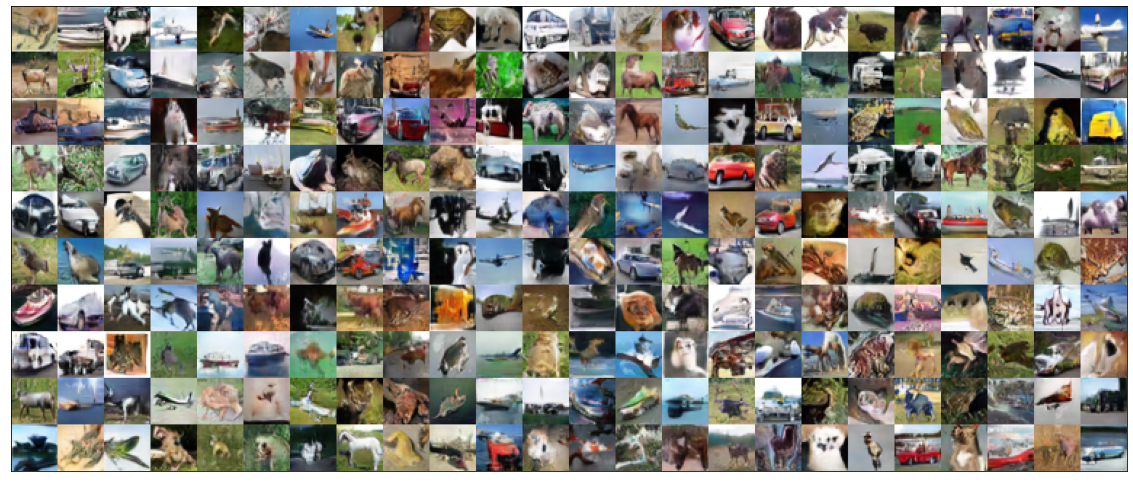}
    \caption{Samples generated from the best WGAN-GP model trained with \algname{clipped-SEG}.}
    \label{fig:wgangp/samples/extrasgd_norm}
\end{figure}
\begin{figure}[]
    \centering
    \includegraphics[width=\textwidth]{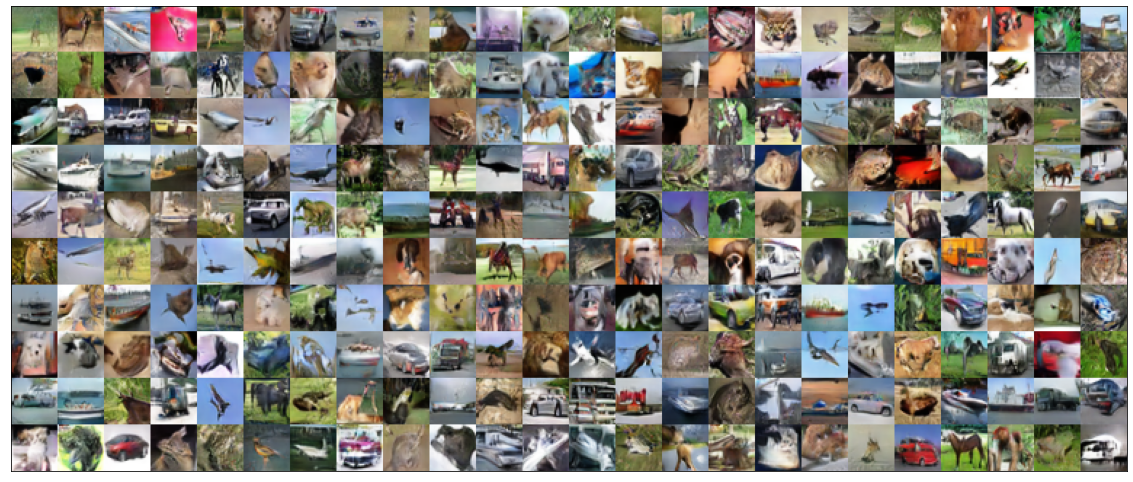}
    \caption{Samples generated from the best WGAN-GP model trained with \algname{clipped-SGDA} (coordinate clipping).}
    \label{fig:wgangp/samples/sgd_val}
\end{figure}
\begin{figure}[]
    \centering
    \includegraphics[width=\textwidth]{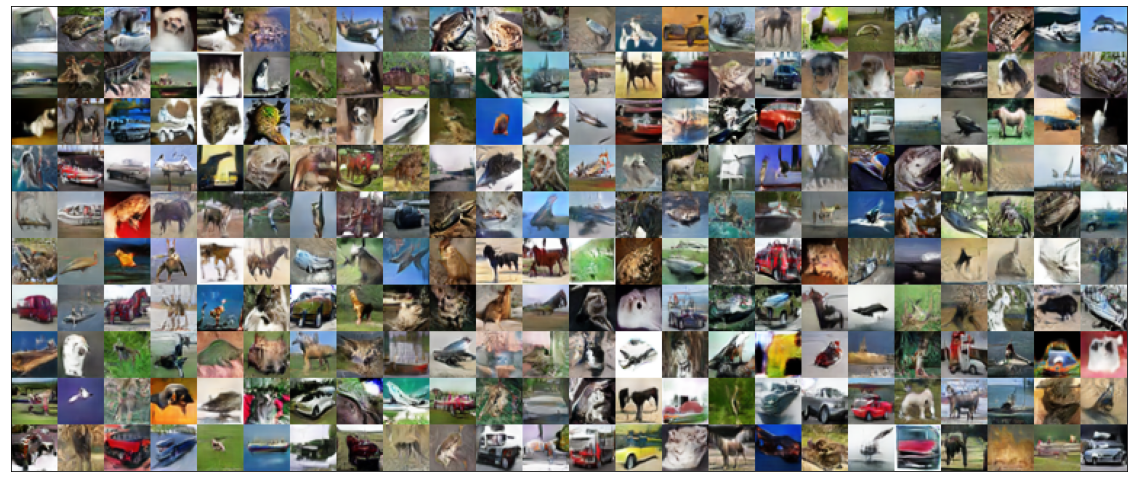}
    \caption{Samples generated from the best WGAN-GP model trained with \algname{clipped-SEG} (coordinate clipping).}
    \label{fig:wgangp/samples/extrasgd_val}
\end{figure}

\newpage

\begin{figure}[]
    % \hspace{-25pt}  
    % \centering
    \begin{tabular}{lcccc}
        % Generator noise
        \raisebox{3.5\normalbaselineskip}[0pt][0pt]{
            \rotatebox[origin=c]{90}{Generator}
        }\hspace{-7pt} & 
        \hspace{-7pt}
        \includegraphics[width=\evofigscale\textwidth]{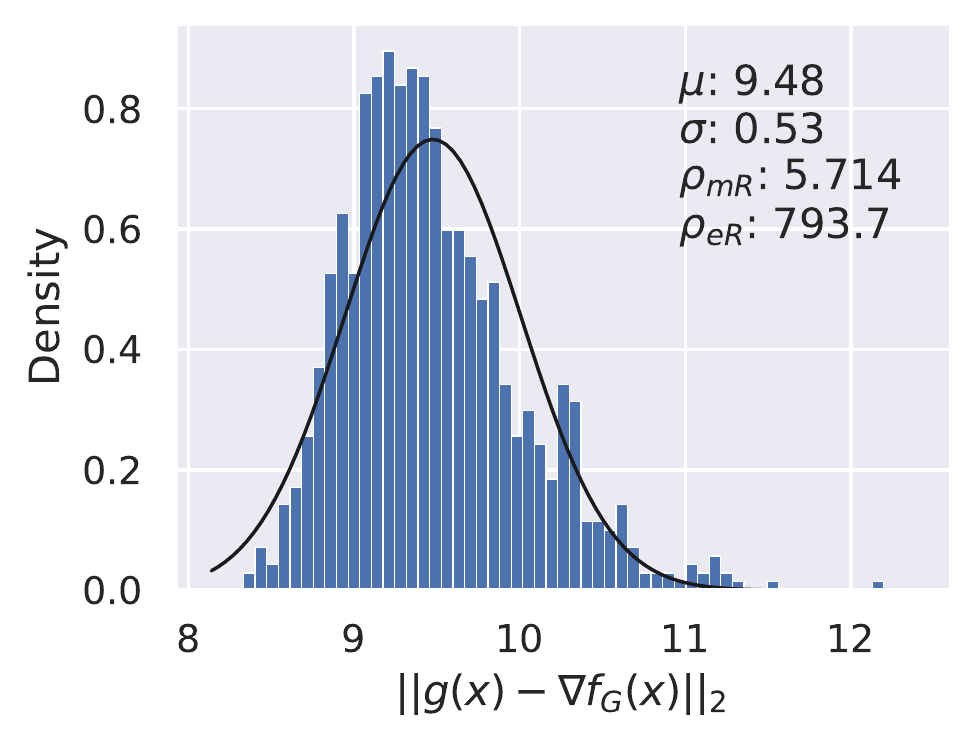} &   
        \hspace{-10pt}
        \includegraphics[width=\evofigscale\textwidth]{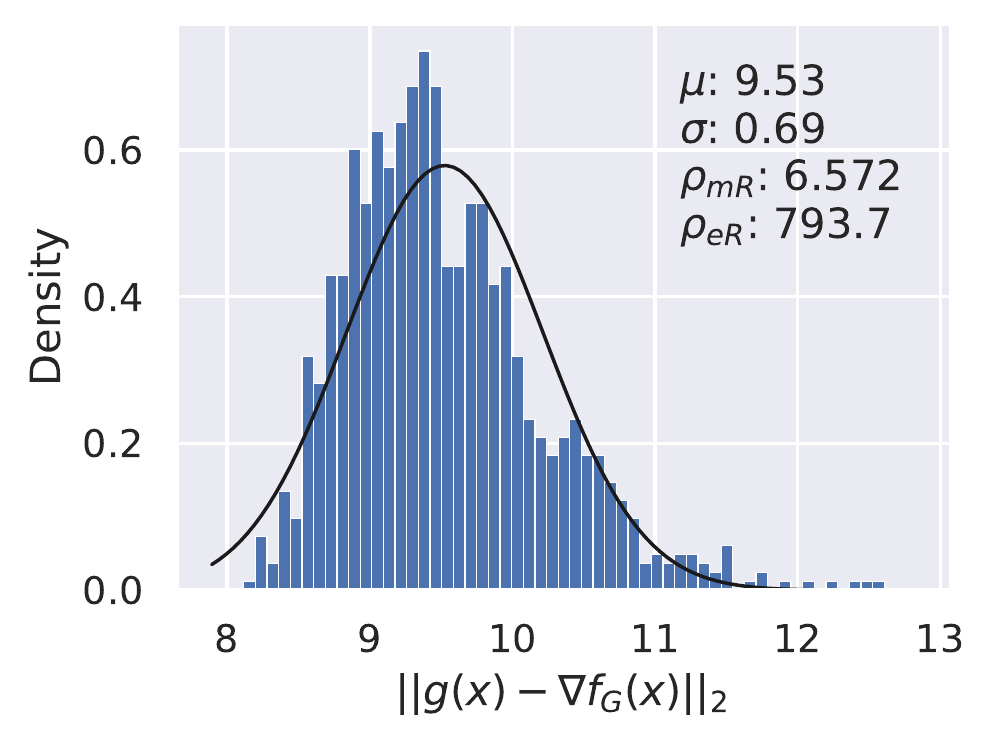} &
        \hspace{-10pt}
        \includegraphics[width=\evofigscale\textwidth]{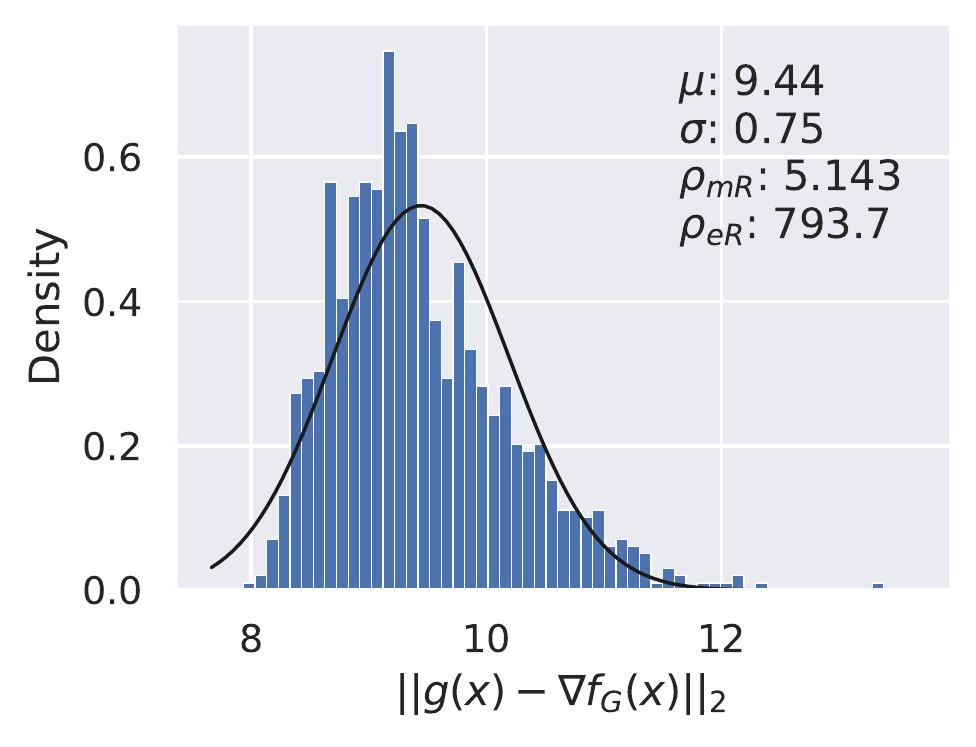} &   
        \hspace{-10pt}
        \includegraphics[width=\evofigscale\textwidth]{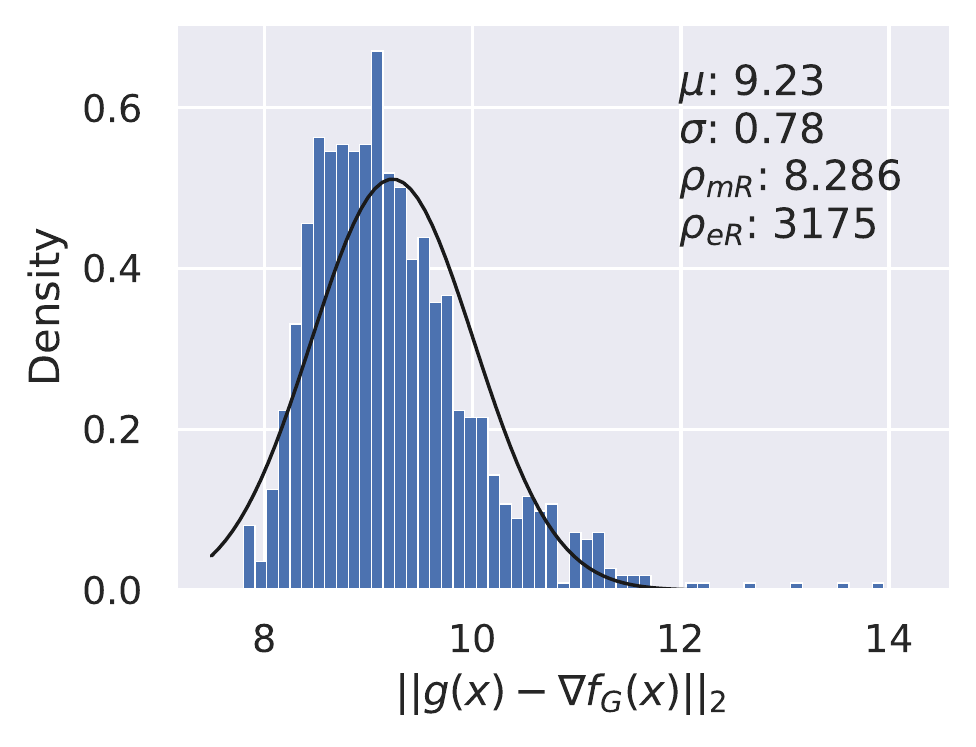} \\
    
        % Discriminator noise
        \raisebox{3.5\normalbaselineskip}[0pt][0pt]{
            \rotatebox[origin=c]{90}{Discriminator}
        }\hspace{-7pt} & 
        \hspace{-7pt}
        \includegraphics[width=\evofigscale\textwidth]{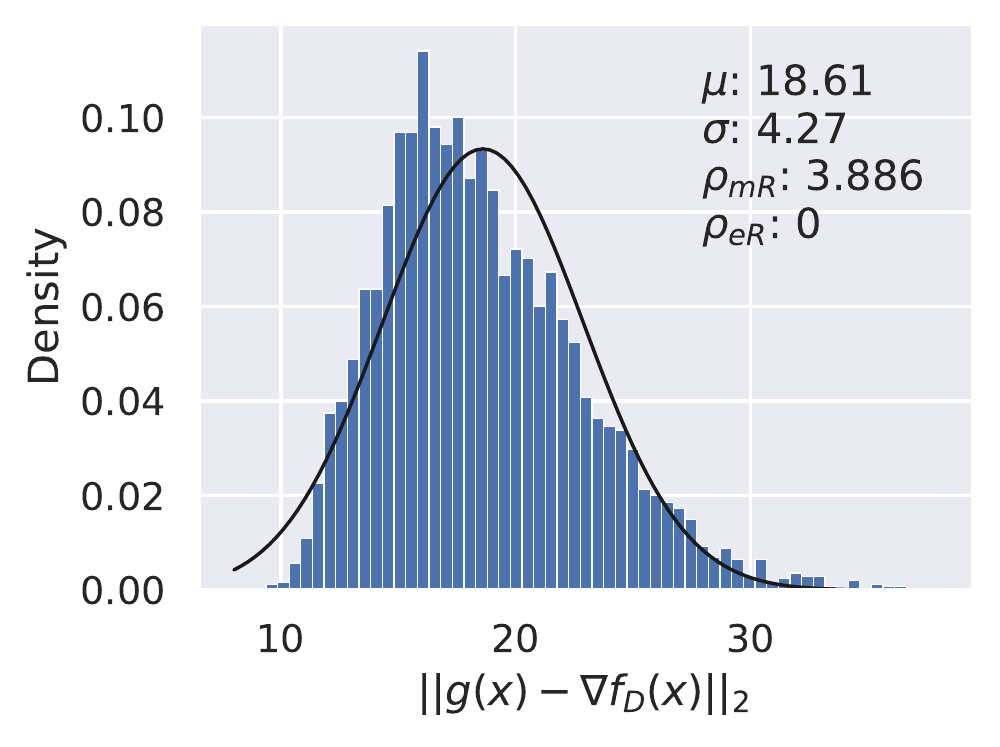} &   
        \hspace{-10pt}
        \includegraphics[width=\evofigscale\textwidth]{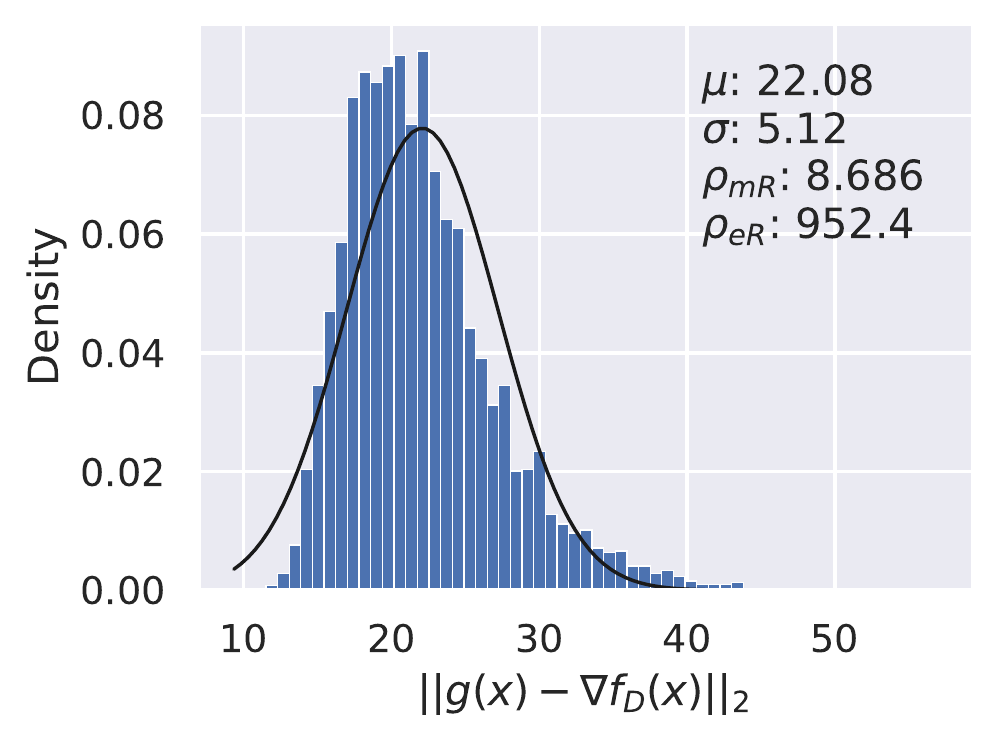} &
        \hspace{-10pt}
        \includegraphics[width=\evofigscale\textwidth]{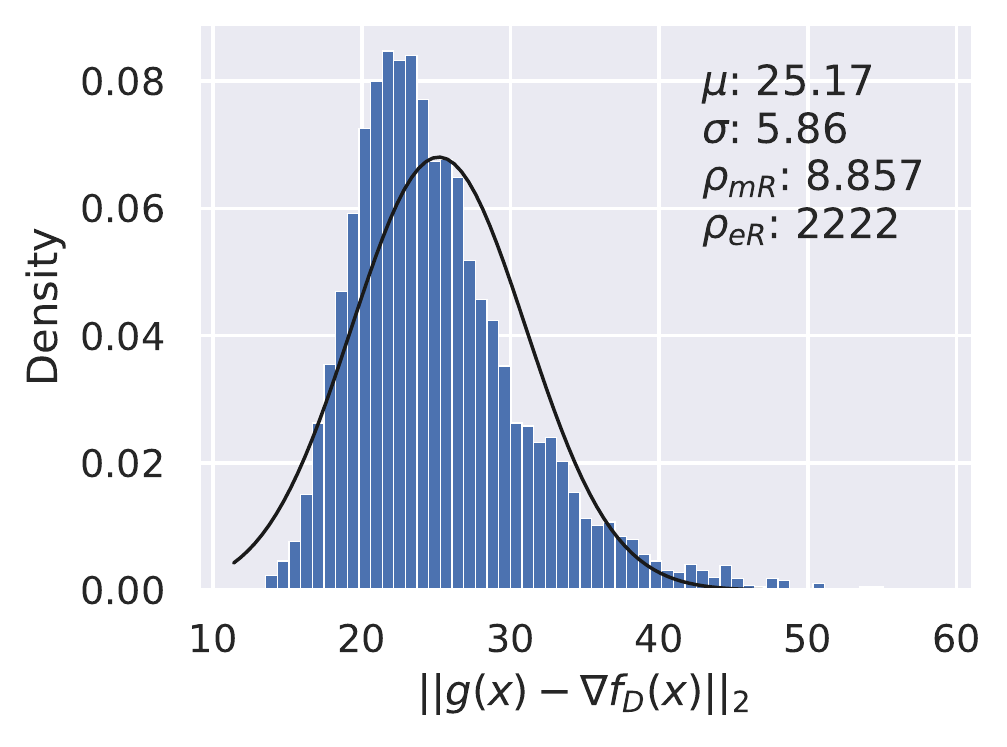} &   
        \hspace{-10pt}
        \includegraphics[width=\evofigscale\textwidth]{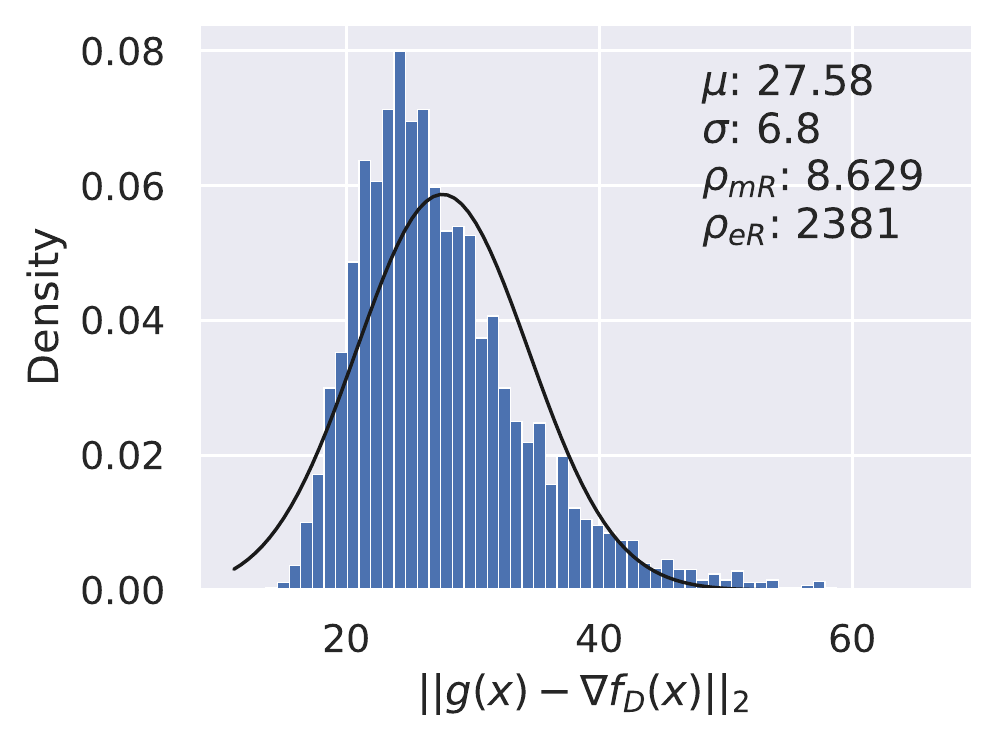} \\
    \end{tabular}
    \caption{\small Evolution of gradient noise histograms for the best WGAN-GP model trained with \algname{SGDA}.}
    \label{fig:app/evo/sgda}
\end{figure}
\begin{figure}[]
    % \hspace{-25pt}  
    % \centering
    \begin{tabular}{lcccc}
        % \cmidrule{2-5}\vspace{-3.5mm}\\
        
        \raisebox{3.5\normalbaselineskip}[0pt][0pt]{
            \rotatebox[origin=c]{90}{Generator}
        }\hspace{-7pt} & 
        \hspace{-7pt}
        \includegraphics[width=\evofigscale\textwidth]{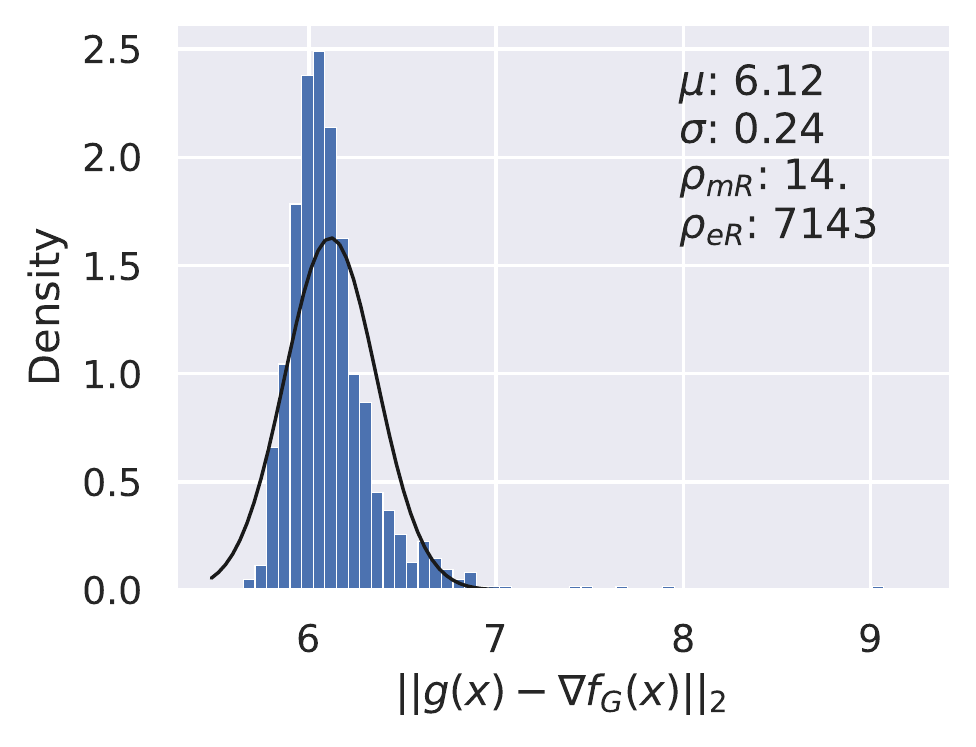} &   
        \hspace{-10pt}
        \includegraphics[width=\evofigscale\textwidth]{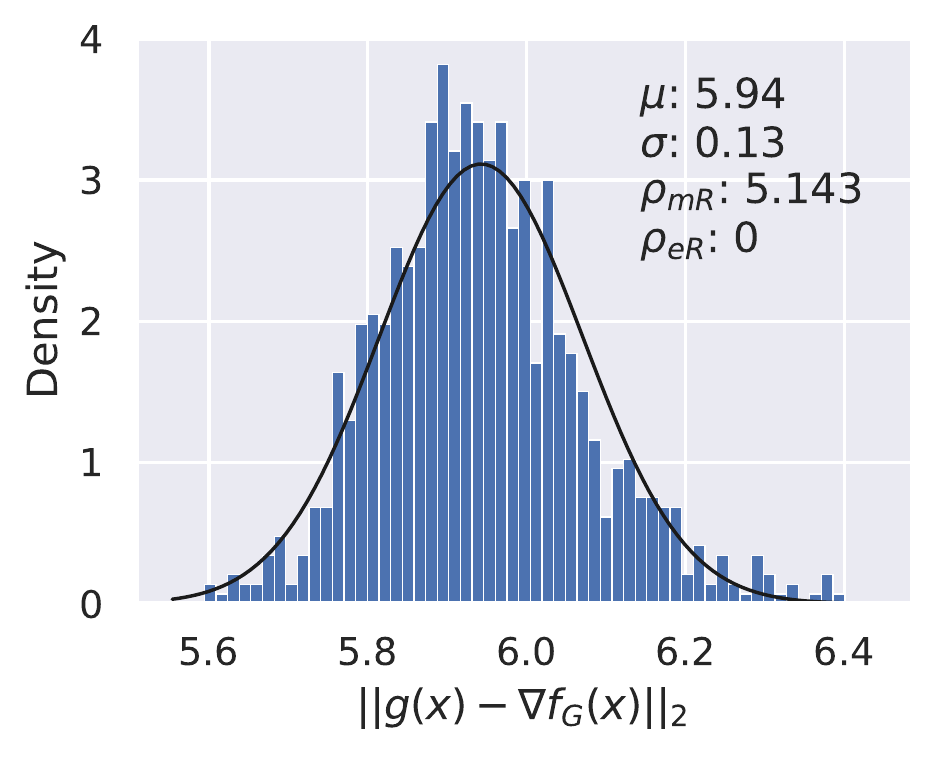} &
        \hspace{-10pt}
        \includegraphics[width=\evofigscale\textwidth]{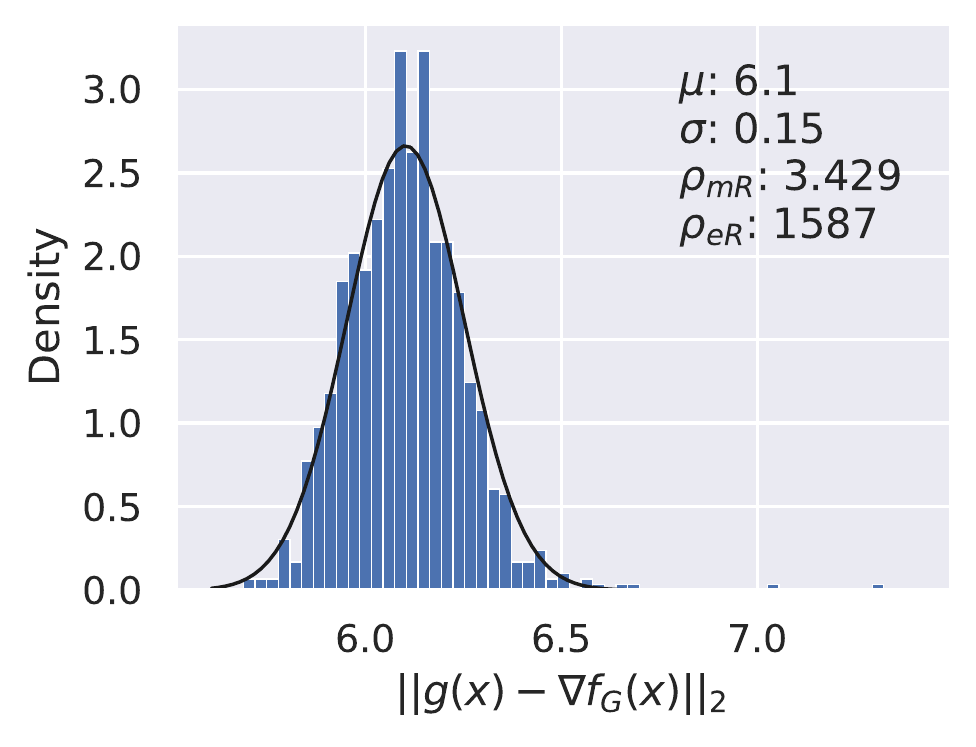} &   
        \hspace{-10pt}
        \includegraphics[width=\evofigscale\textwidth]{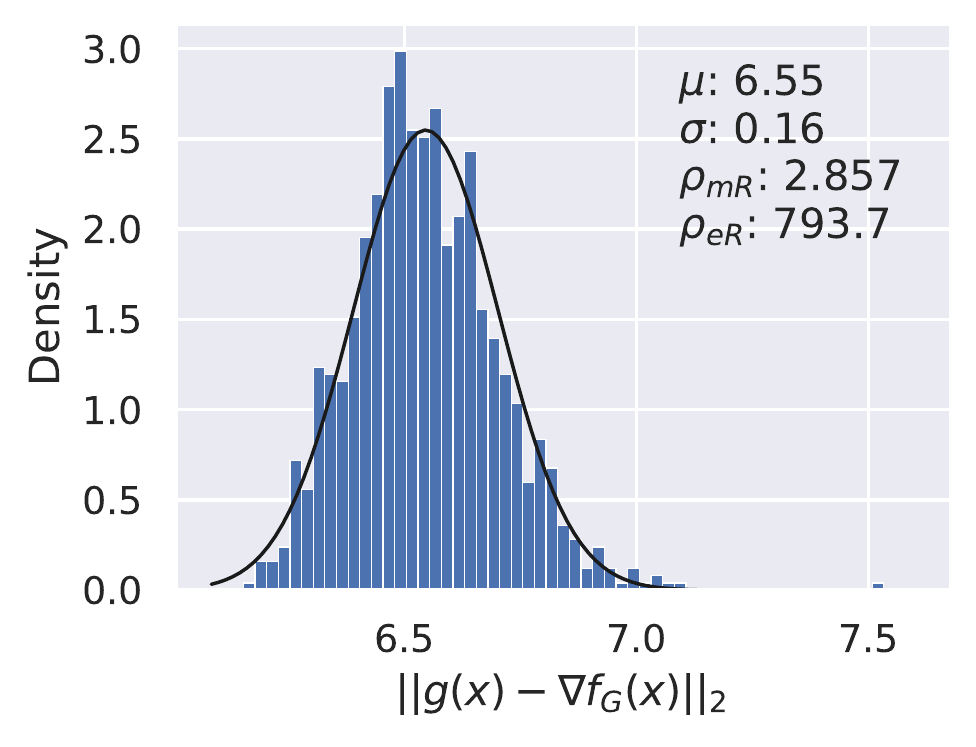} \\

        % Discriminator noise
        \raisebox{3.5\normalbaselineskip}[0pt][0pt]{
            \rotatebox[origin=c]{90}{Discriminator}
        }\hspace{-7pt} & 
        \hspace{-7pt}
        \includegraphics[width=\evofigscale\textwidth]{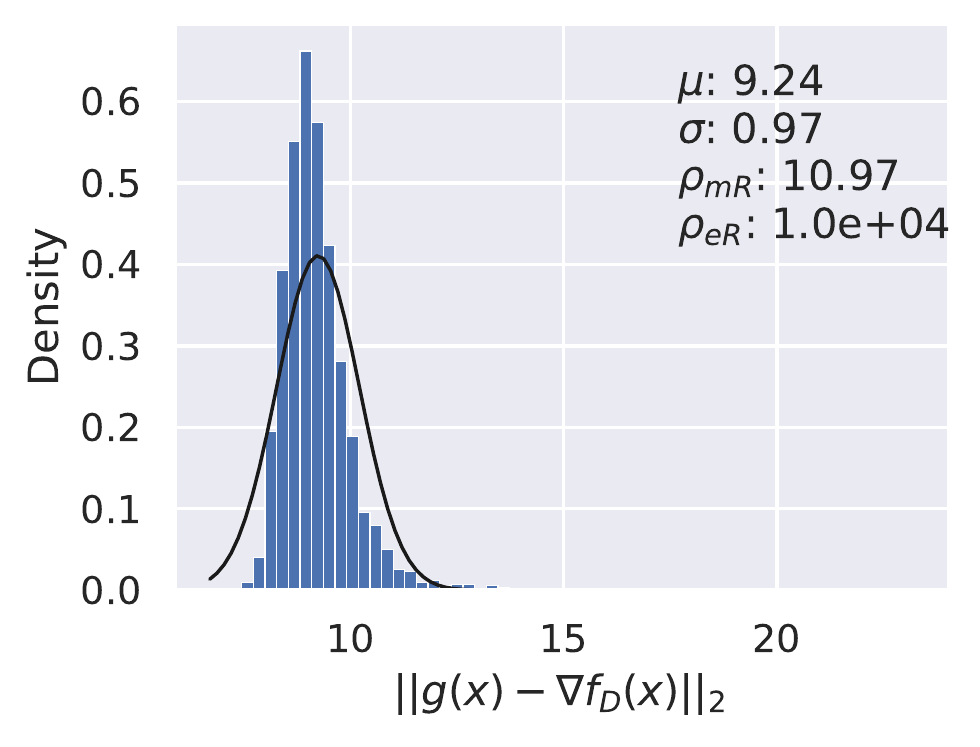} &   
        \hspace{-10pt}
        \includegraphics[width=\evofigscale\textwidth]{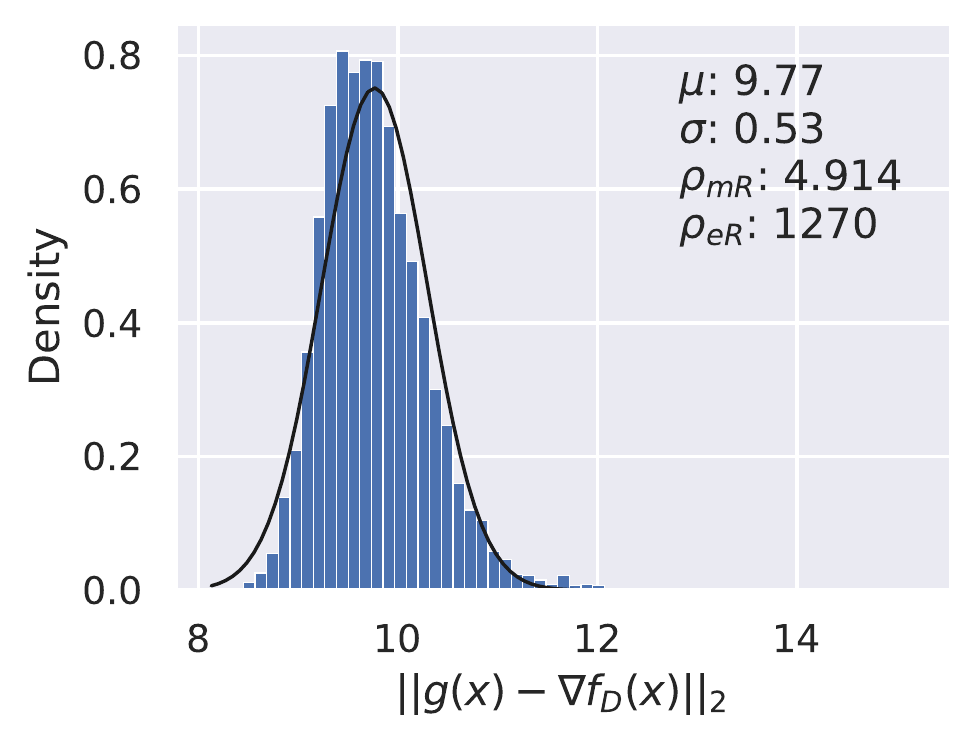} &
        \hspace{-10pt}
        \includegraphics[width=\evofigscale\textwidth]{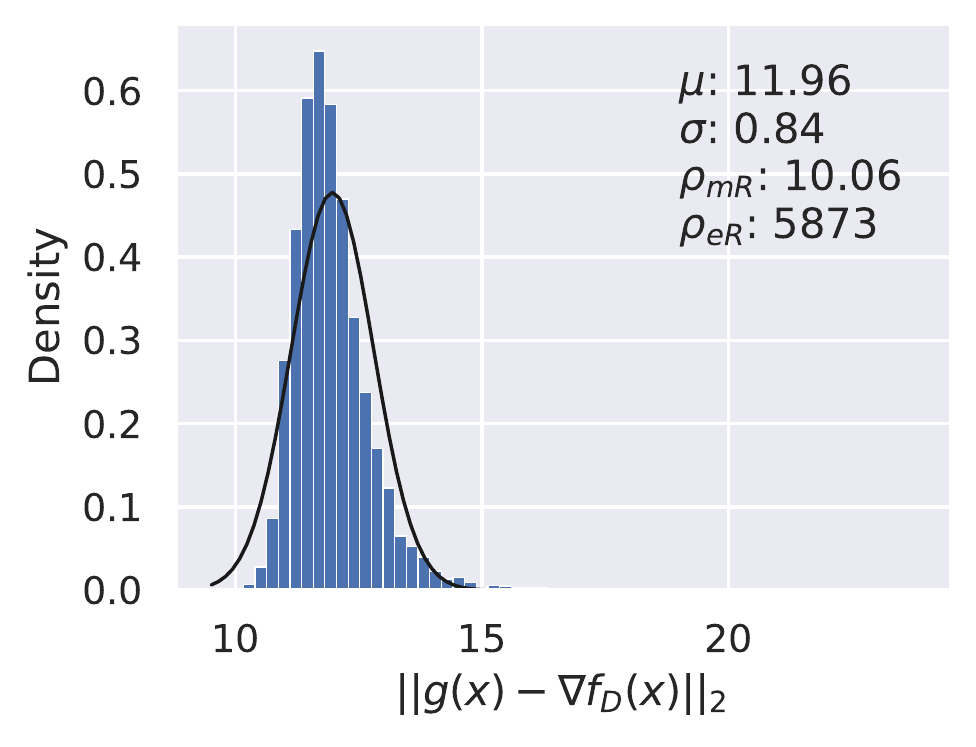} &   
        \hspace{-10pt}
        \includegraphics[width=\evofigscale\textwidth]{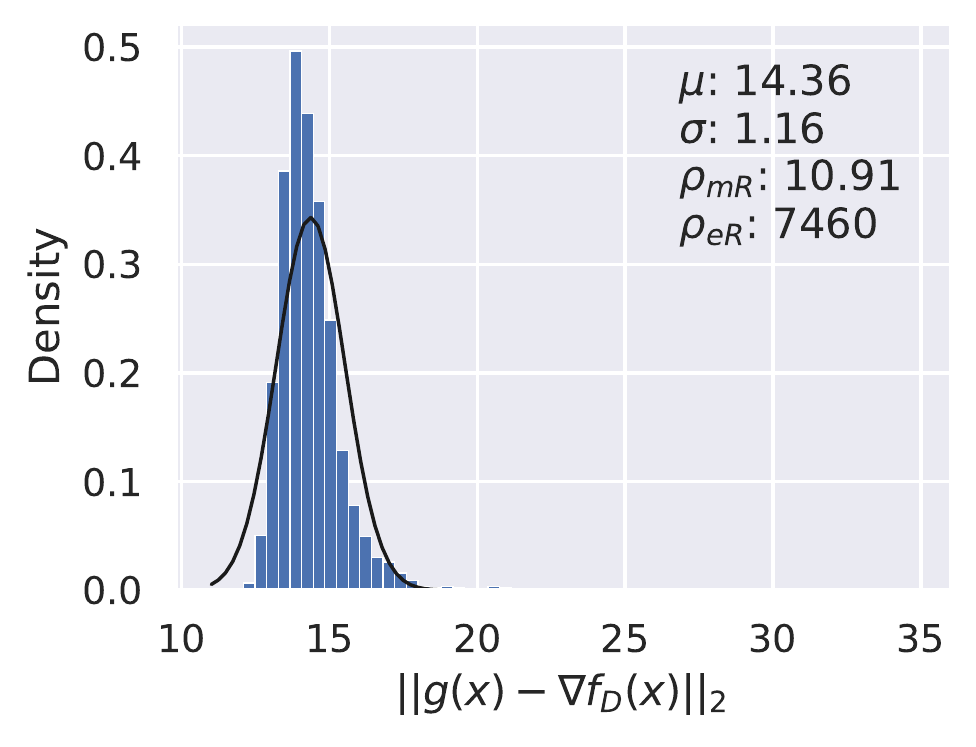} \\
        & {20000 steps} & {40000 steps} & {80000 steps} & {100000 steps} \\
    
    \end{tabular}
    \caption{\small Evolution of gradient noise histograms for the best WGAN-GP model trained with \algname{clipped-SGDA}.}
    \label{fig:app/evo/clipnorm-sgda}
\end{figure}
\begin{figure}[]
    % \hspace{-25pt}  
    % \centering
    \begin{tabular}{lcccc}
        % \cmidrule{2-5}\vspace{-3.5mm}\\
        
        \raisebox{3.5\normalbaselineskip}[0pt][0pt]{
            \rotatebox[origin=c]{90}{Generator}
        }\hspace{-7pt} & 
        \hspace{-7pt}
        \includegraphics[width=\evofigscale\textwidth]{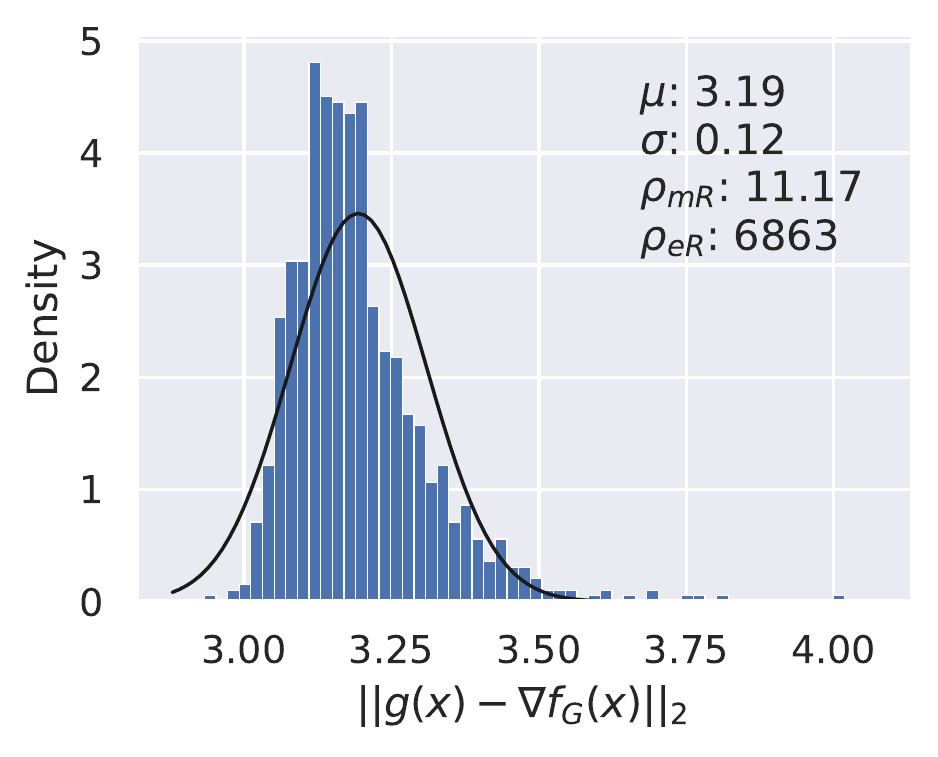} &   
        \hspace{-10pt}
        \includegraphics[width=\evofigscale\textwidth]{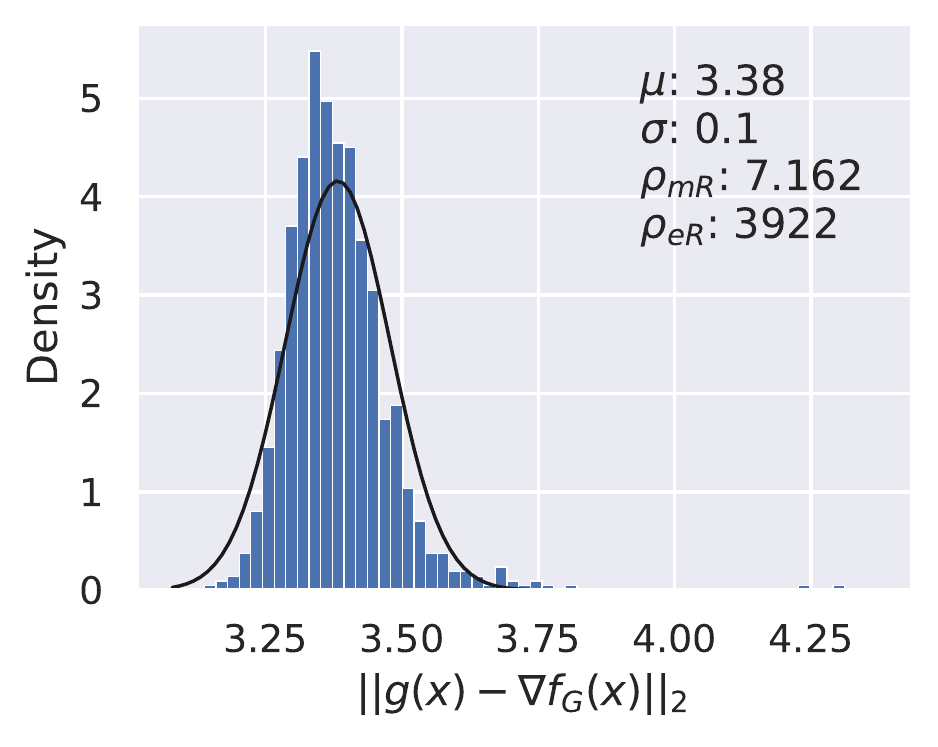} &
        \hspace{-10pt}
        \includegraphics[width=\evofigscale\textwidth]{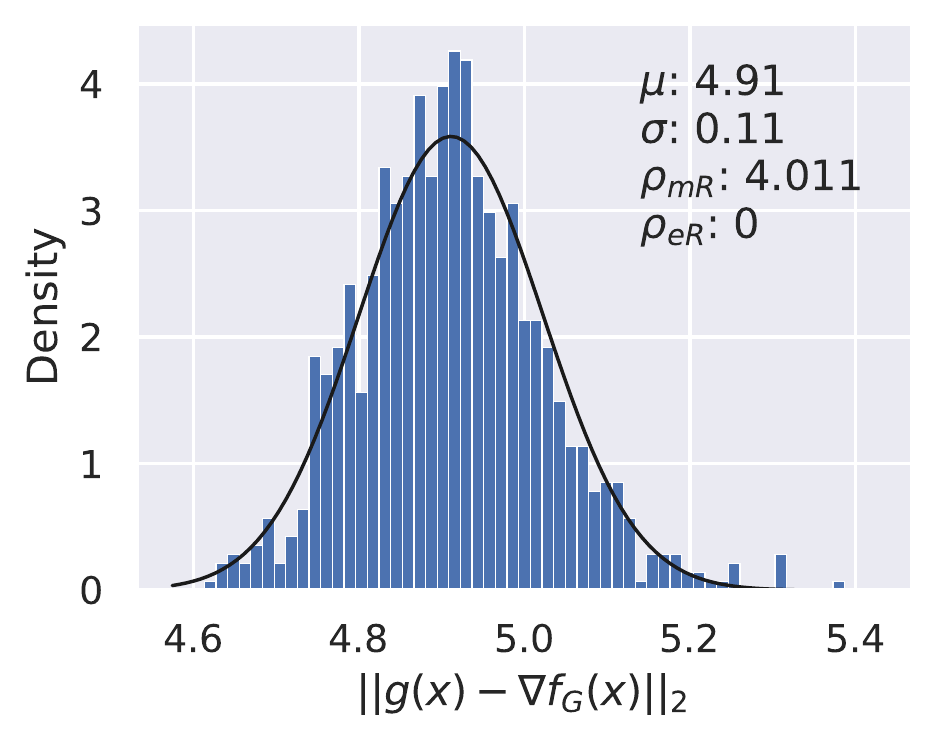} &   
        \hspace{-10pt}
        \includegraphics[width=\evofigscale\textwidth]{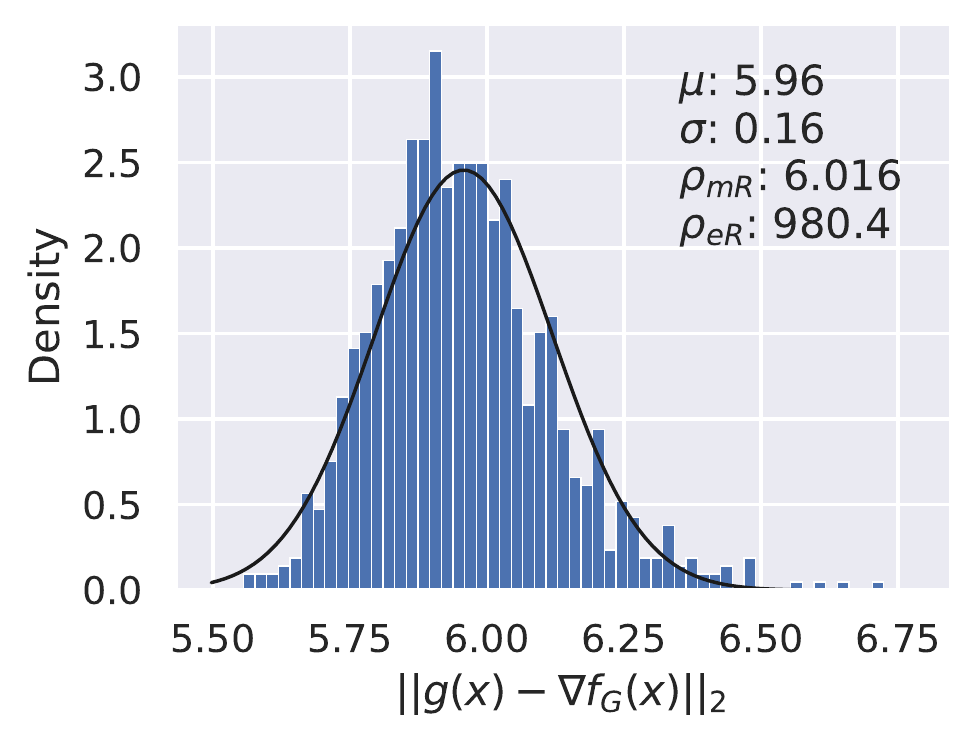} \\

        % Discriminator noise
        \raisebox{3.5\normalbaselineskip}[0pt][0pt]{
            \rotatebox[origin=c]{90}{Discriminator}
        }\hspace{-7pt} & 
        \hspace{-7pt}
        \includegraphics[width=\evofigscale\textwidth]{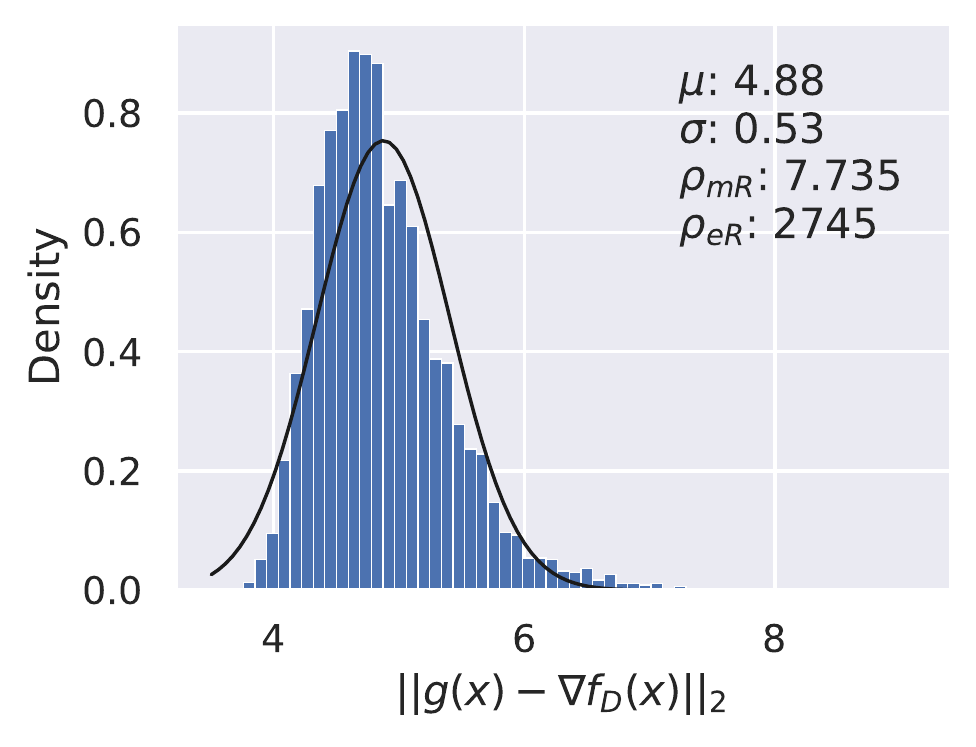} &   
        \hspace{-10pt}
        \includegraphics[width=\evofigscale\textwidth]{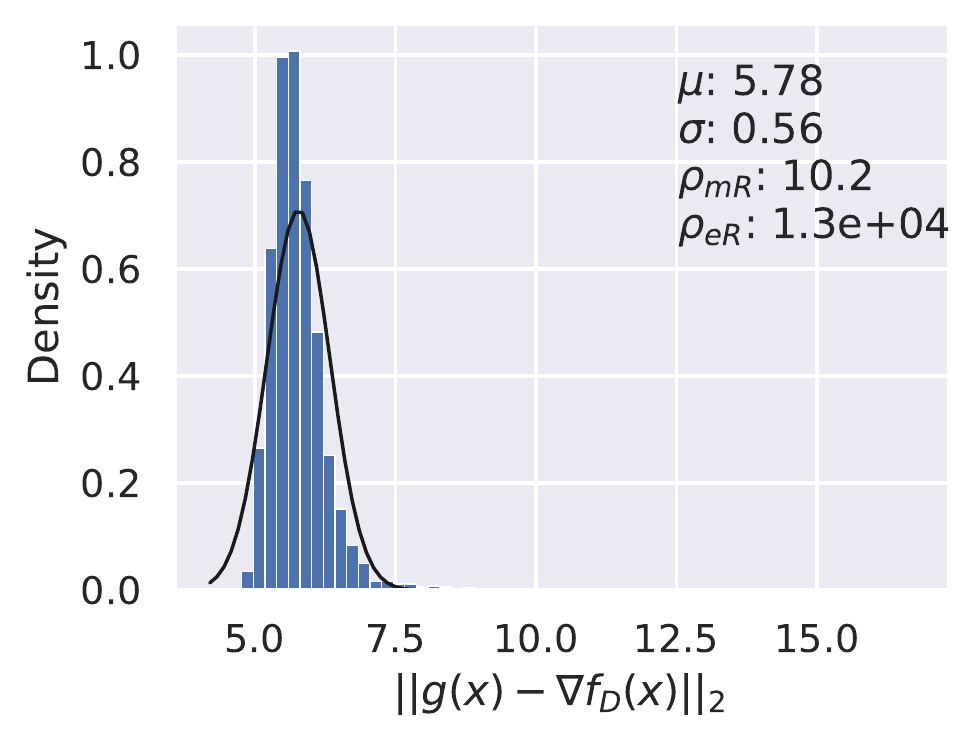} &
        \hspace{-10pt}
        \includegraphics[width=\evofigscale\textwidth]{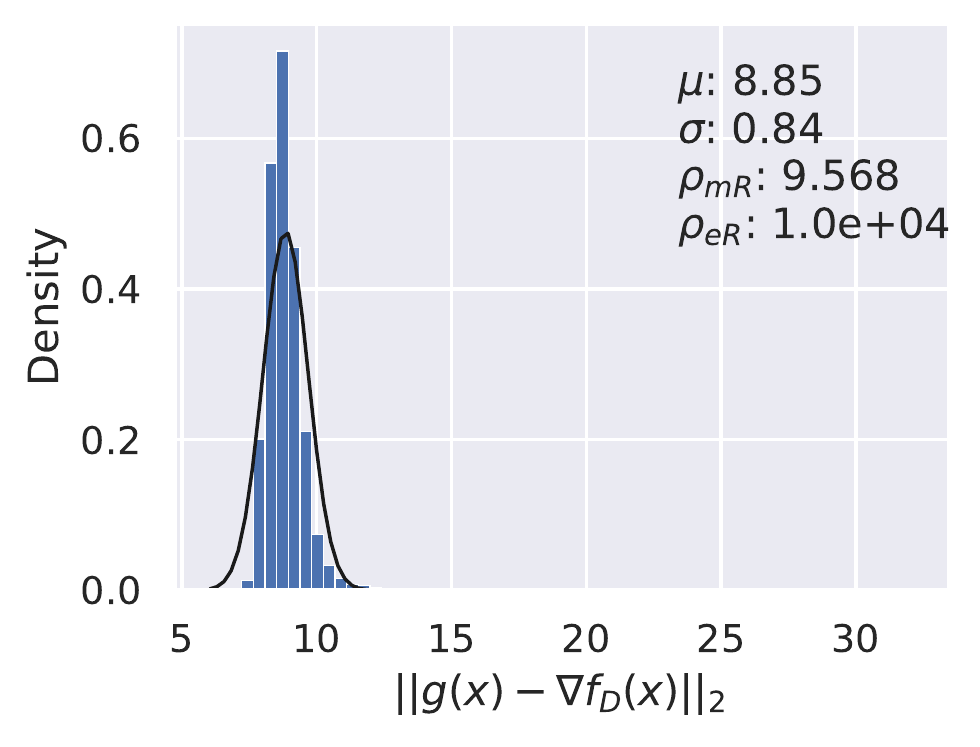} &   
        \hspace{-10pt}
        \includegraphics[width=\evofigscale\textwidth]{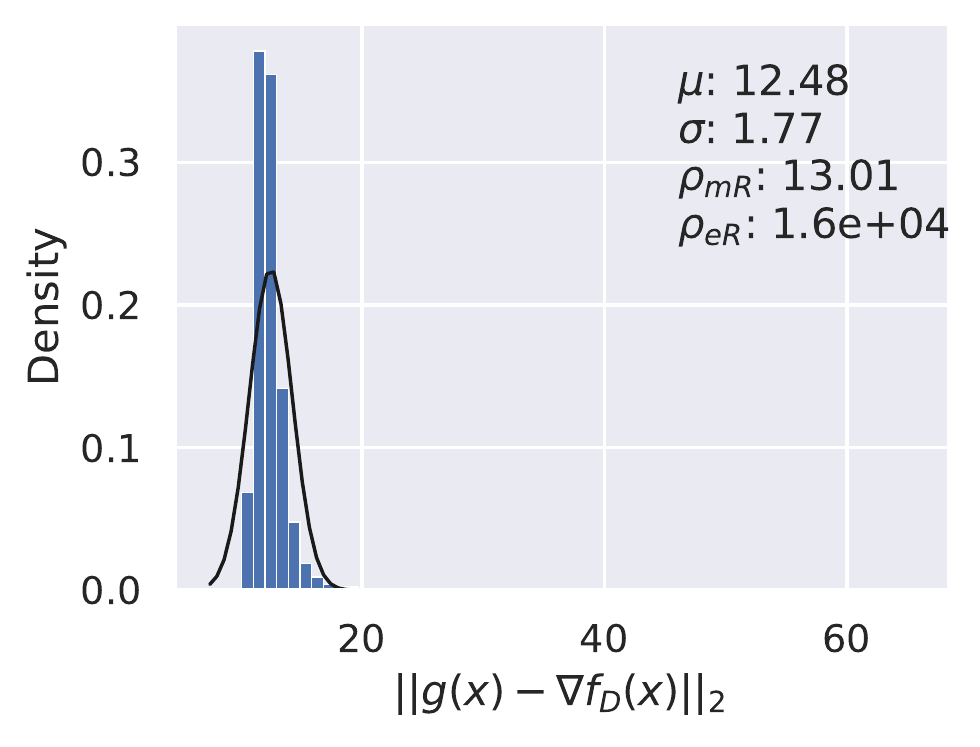} \\
        & {20000 steps} & {40000 steps} & {80000 steps} & {100000 steps} \\
    
    \end{tabular}
    \caption{\small Evolution of gradient noise histograms for the best WGAN-GP model trained with \algname{clipped-SEG}.}
    \label{fig:app/evo/clipnorm-SEG}
\end{figure}
\begin{figure}[]
    % \hspace{-25pt}  
    % \centering
    \begin{tabular}{lcccc}
        % \cmidrule{2-5}\vspace{-3.5mm}\\
        
        \raisebox{3.5\normalbaselineskip}[0pt][0pt]{
            \rotatebox[origin=c]{90}{Generator}
        }\hspace{-7pt} & 
        \hspace{-7pt}
        \includegraphics[width=\evofigscale\textwidth]{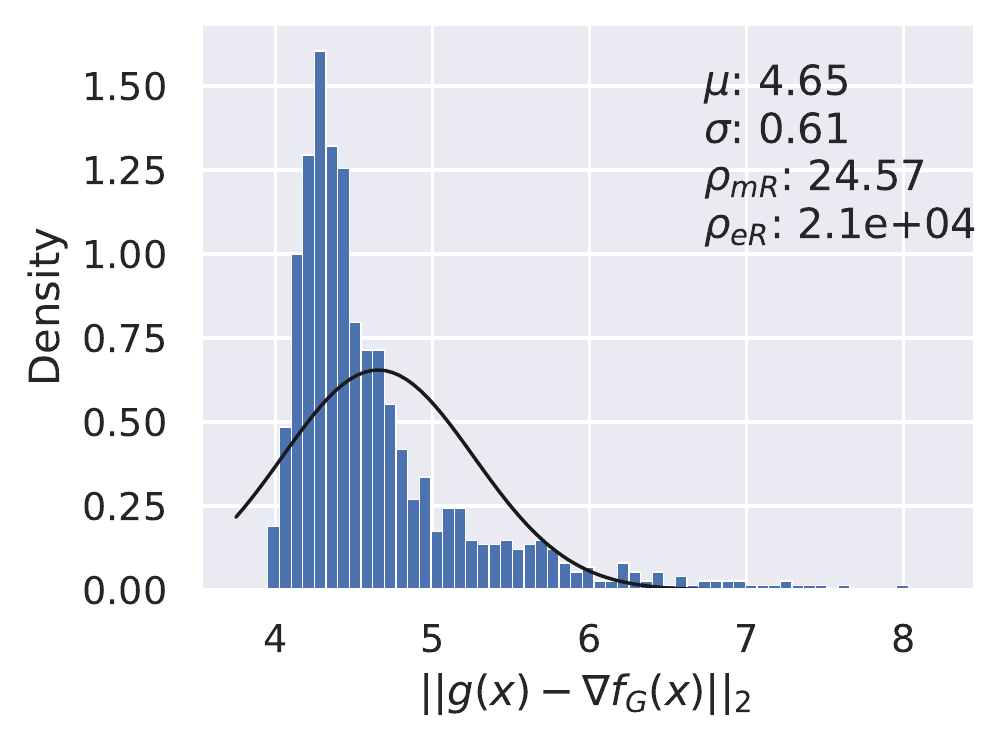} &   
        \hspace{-10pt}
        \includegraphics[width=\evofigscale\textwidth]{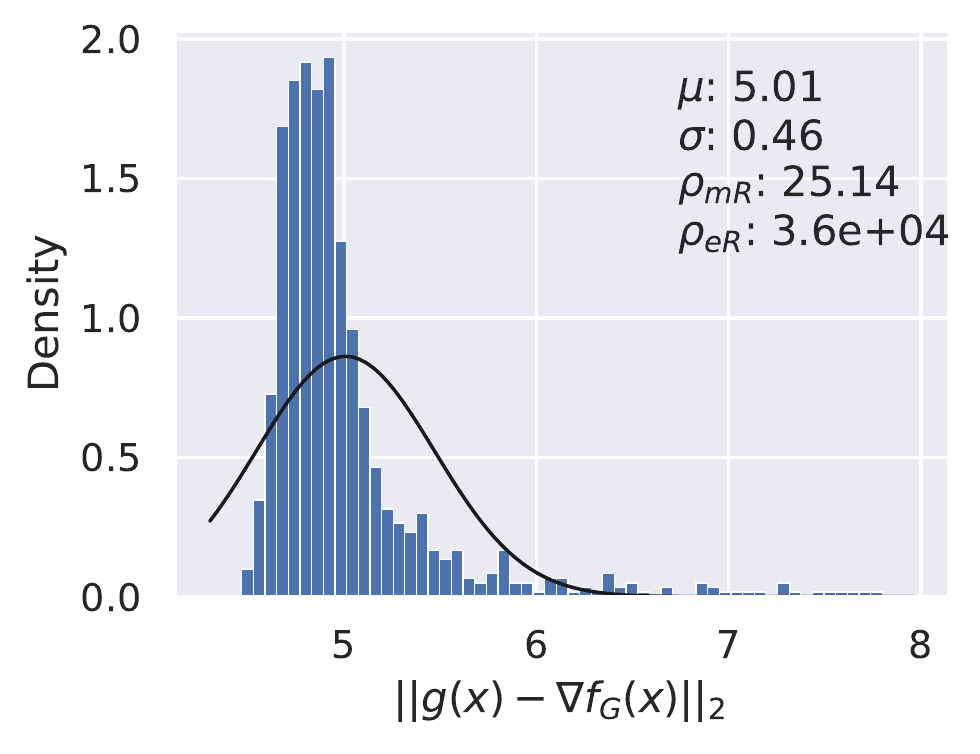} &
        \hspace{-10pt}
        \includegraphics[width=\evofigscale\textwidth]{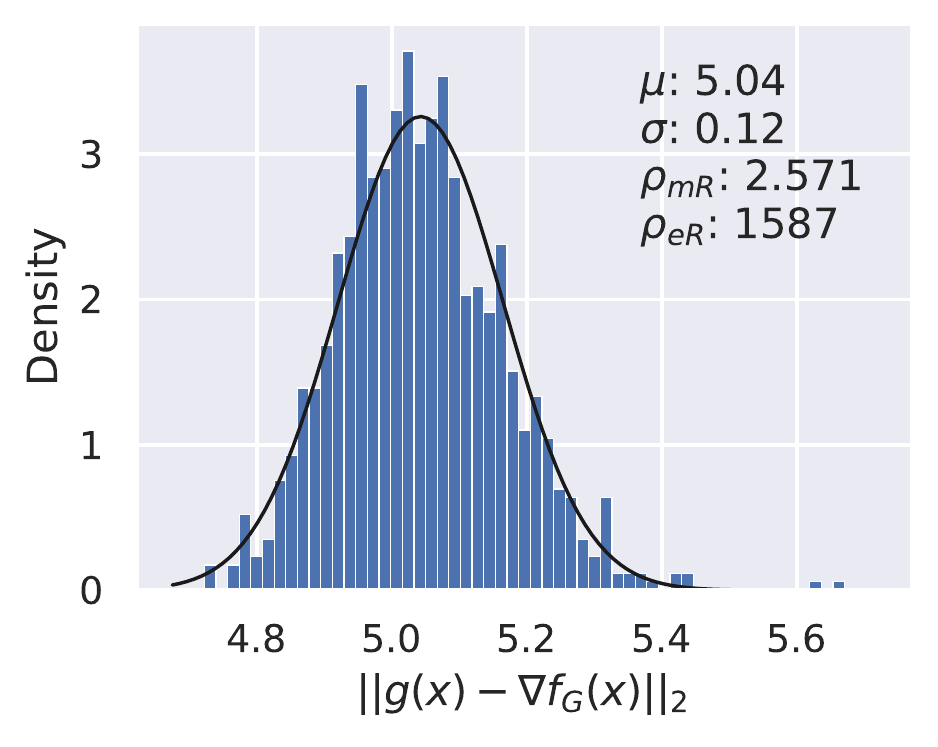} &   
        \hspace{-10pt}
        \includegraphics[width=\evofigscale\textwidth]{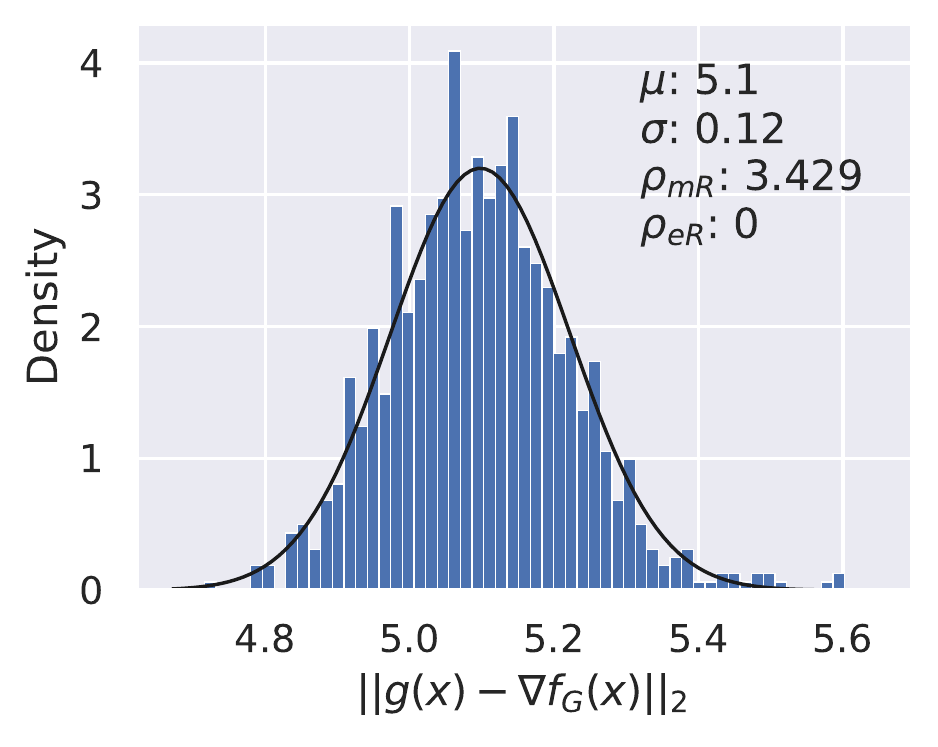} \\

        % Discriminator noise
        \raisebox{3.5\normalbaselineskip}[0pt][0pt]{
            \rotatebox[origin=c]{90}{Discriminator}
        }\hspace{-7pt} & 
        \hspace{-7pt}
        \includegraphics[width=\evofigscale\textwidth]{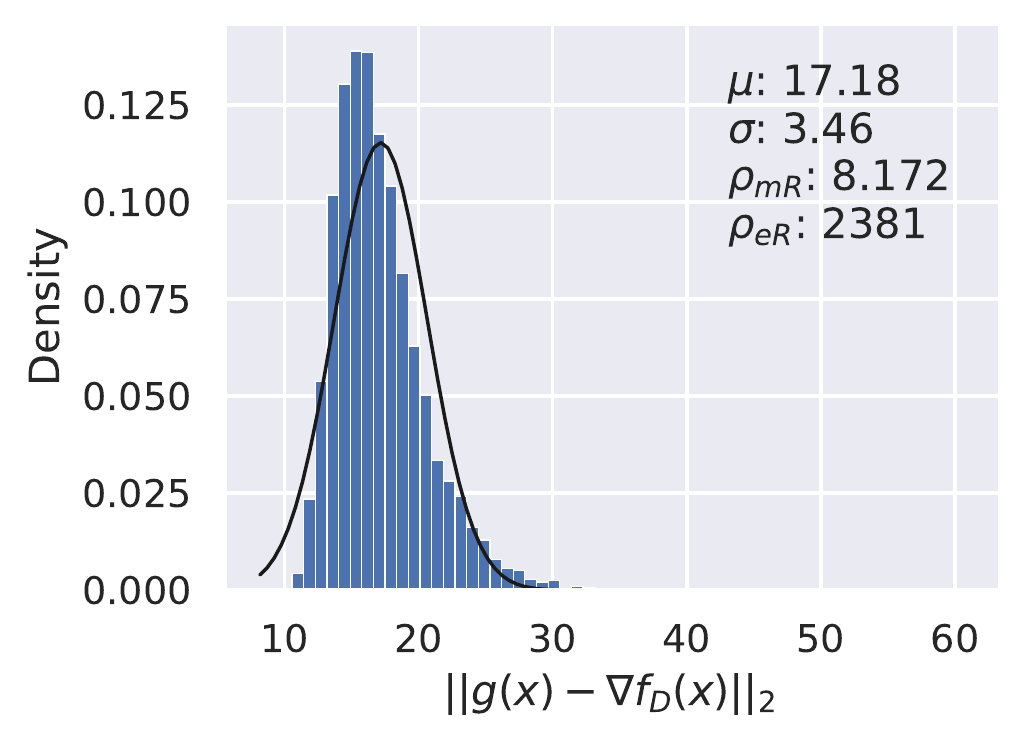} &   
        \hspace{-10pt}
        \includegraphics[width=\evofigscale\textwidth]{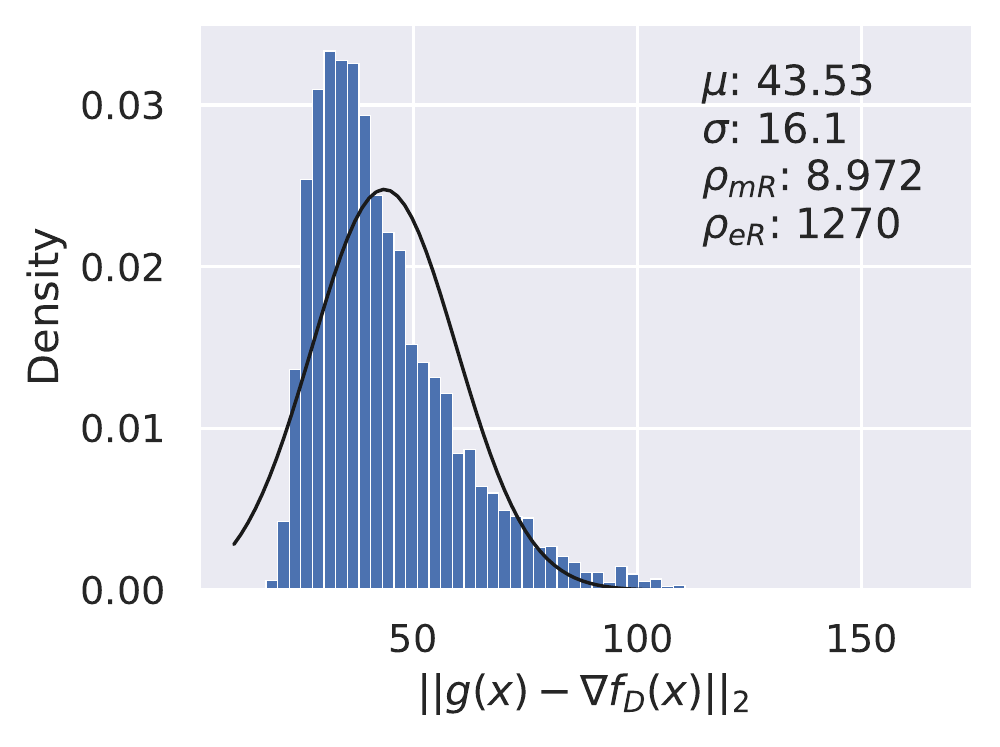} &
        \hspace{-10pt}
        \includegraphics[width=\evofigscale\textwidth]{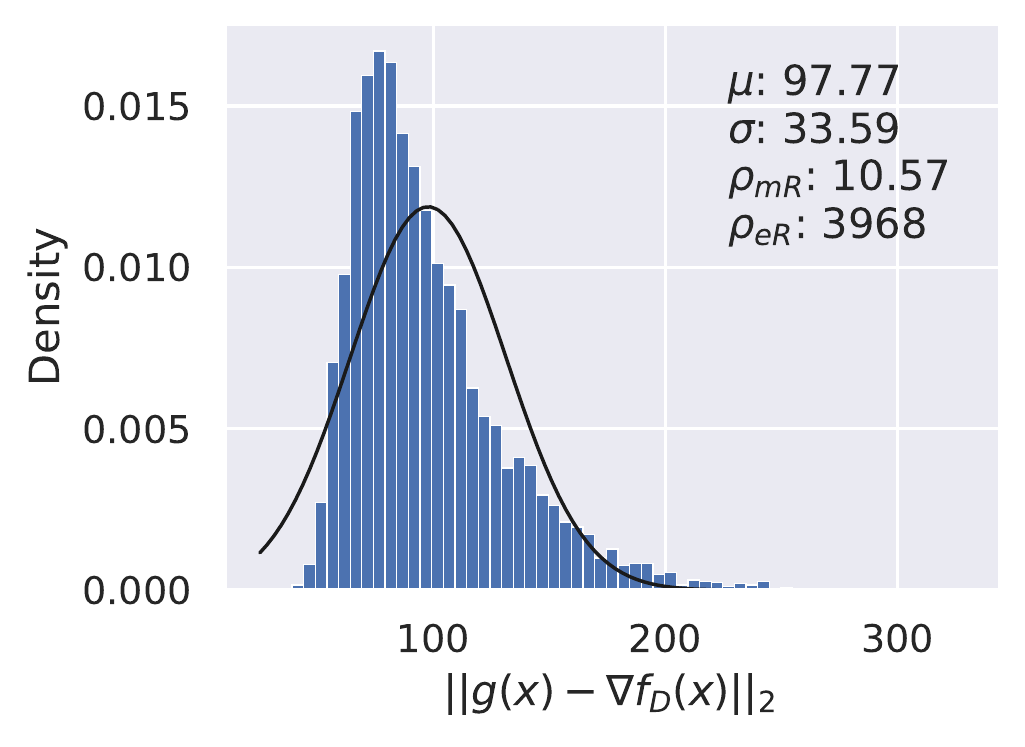} &   
        \hspace{-10pt}
        \includegraphics[width=\evofigscale\textwidth]{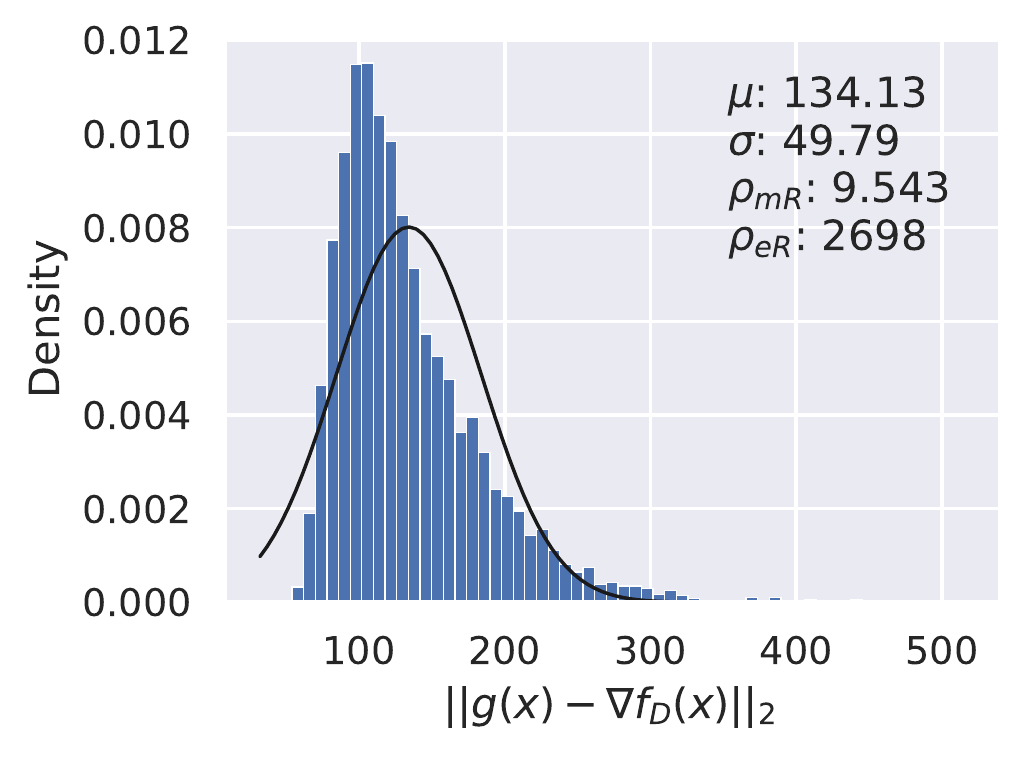} \\
        & {20000 steps} & {40000 steps} & {80000 steps} & {100000 steps} \\
    
    \end{tabular}
    \caption{\small Evolution of gradient noise histograms for the best WGAN-GP model trained with \algname{clipped-SGDA} (cordinate clipping).}
    \label{fig:app/evo/clipval-SGDA}
\end{figure}
\begin{figure}[]
    % \hspace{-25pt}  
    % \centering
    \begin{tabular}{lcccc}
        % \cmidrule{2-5}\vspace{-3.5mm}\\
        
        \raisebox{3.5\normalbaselineskip}[0pt][0pt]{
            \rotatebox[origin=c]{90}{Generator}
        }\hspace{-7pt} & 
        \hspace{-7pt}
        \includegraphics[width=\evofigscale\textwidth]{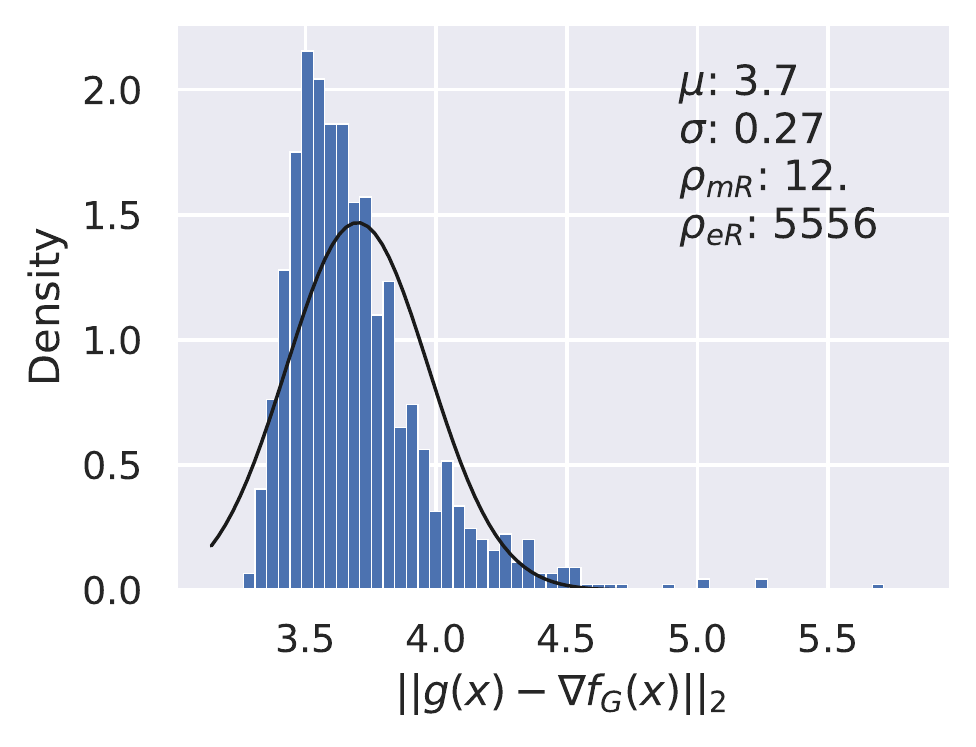} &   
        \hspace{-10pt}
        \includegraphics[width=\evofigscale\textwidth]{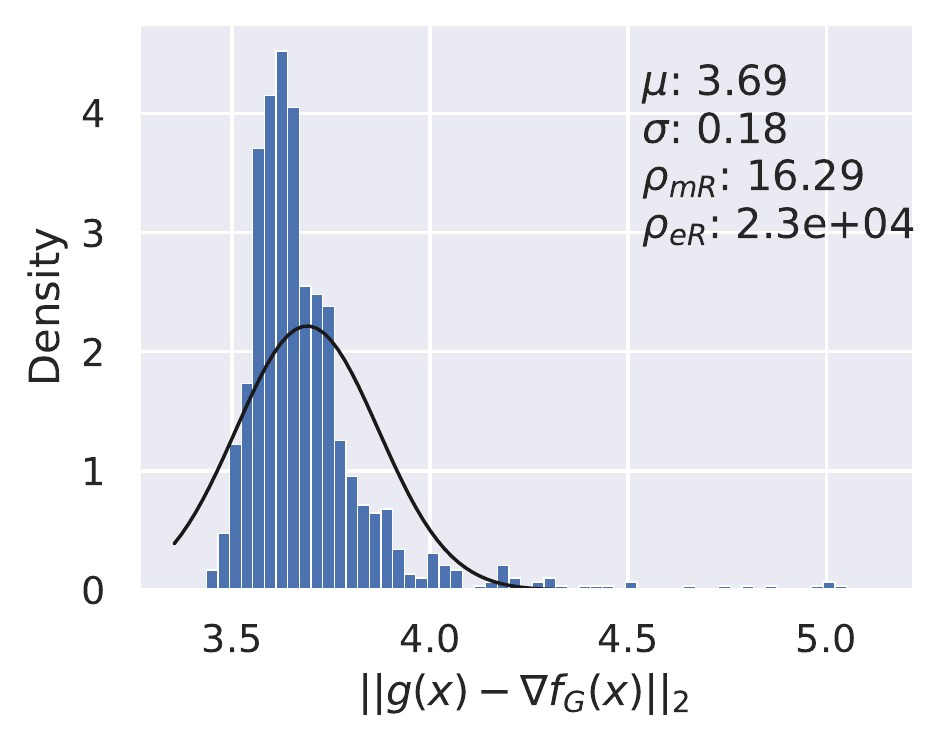} &
        \hspace{-10pt}
        \includegraphics[width=\evofigscale\textwidth]{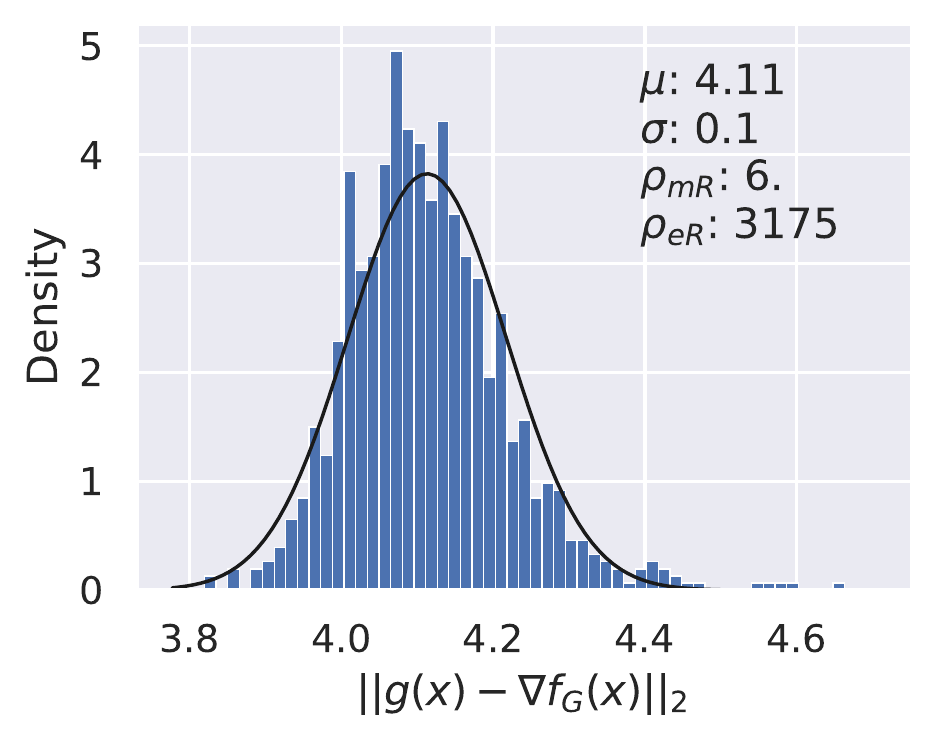} &   
        \hspace{-10pt}
        \includegraphics[width=\evofigscale\textwidth]{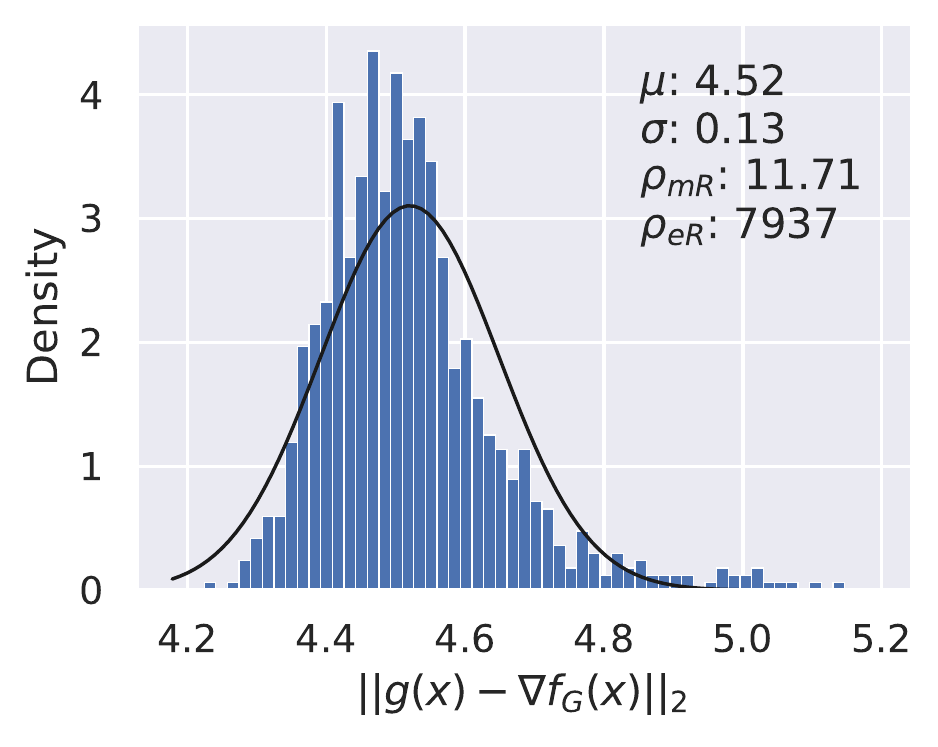} \\

        % Discriminator noise
        \raisebox{3.5\normalbaselineskip}[0pt][0pt]{
            \rotatebox[origin=c]{90}{Discriminator}
        }\hspace{-7pt} & 
        \hspace{-7pt}
        \includegraphics[width=\evofigscale\textwidth]{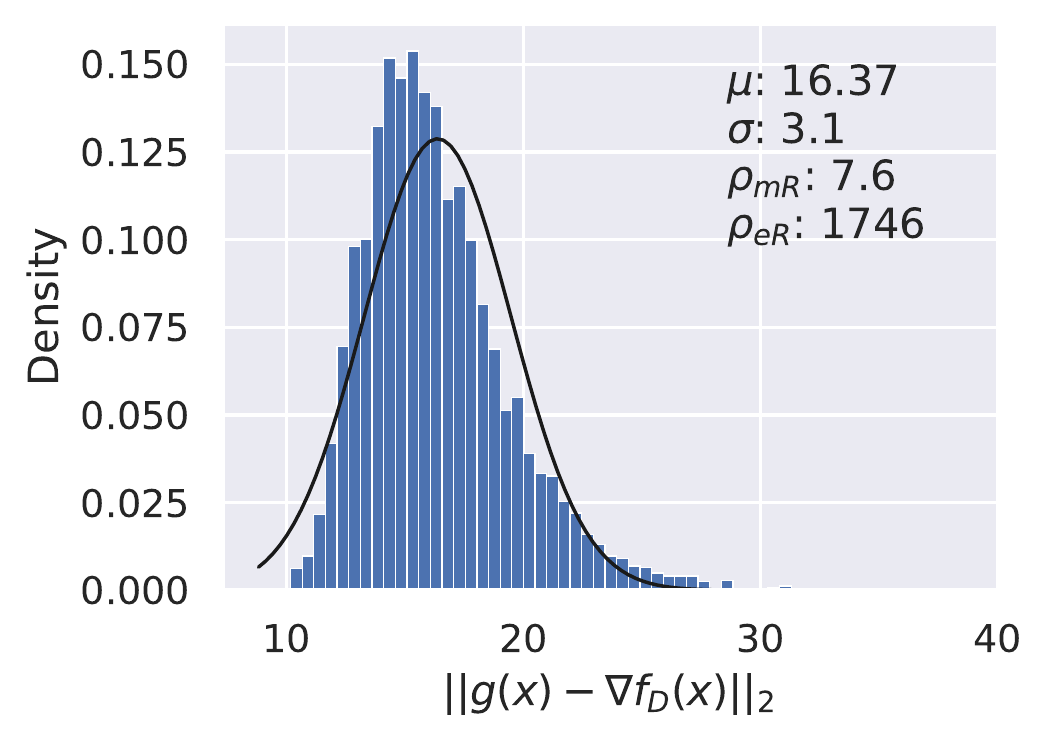} &   
        \hspace{-10pt}
        \includegraphics[width=\evofigscale\textwidth]{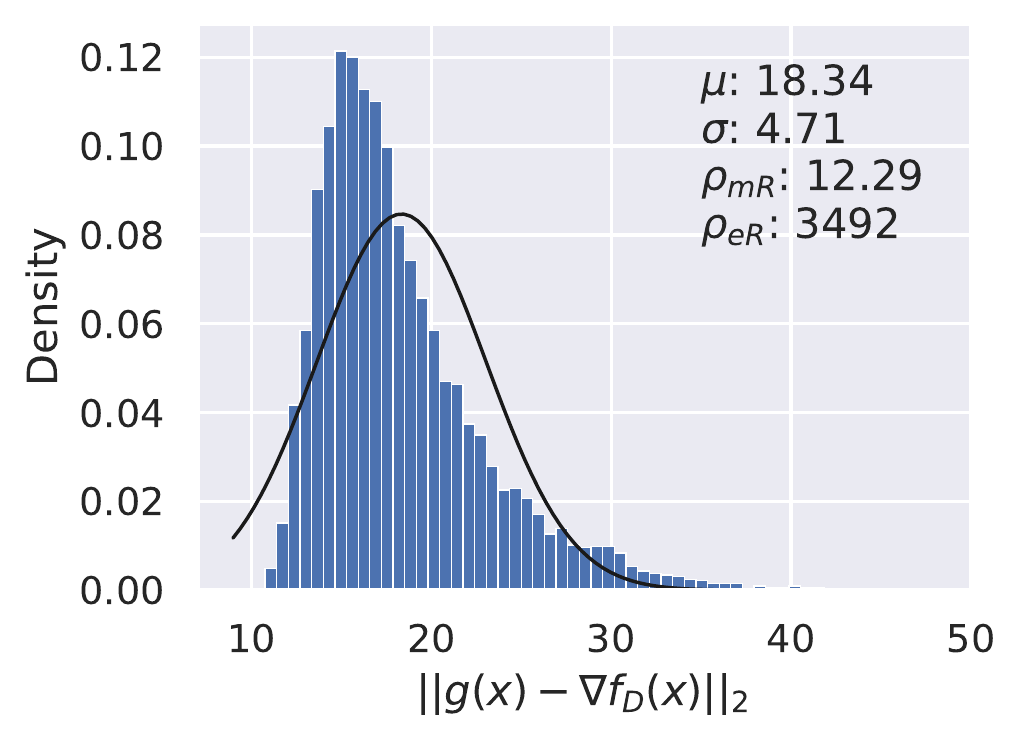} &
        \hspace{-10pt}
        \includegraphics[width=\evofigscale\textwidth]{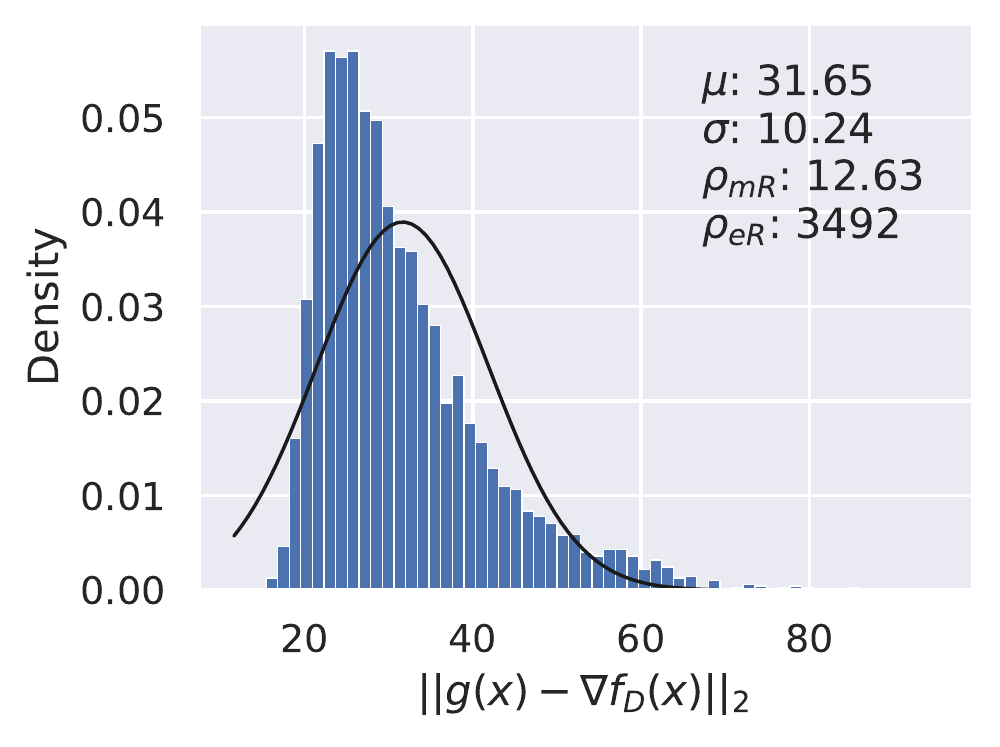} &   
        \hspace{-10pt}
        \includegraphics[width=\evofigscale\textwidth]{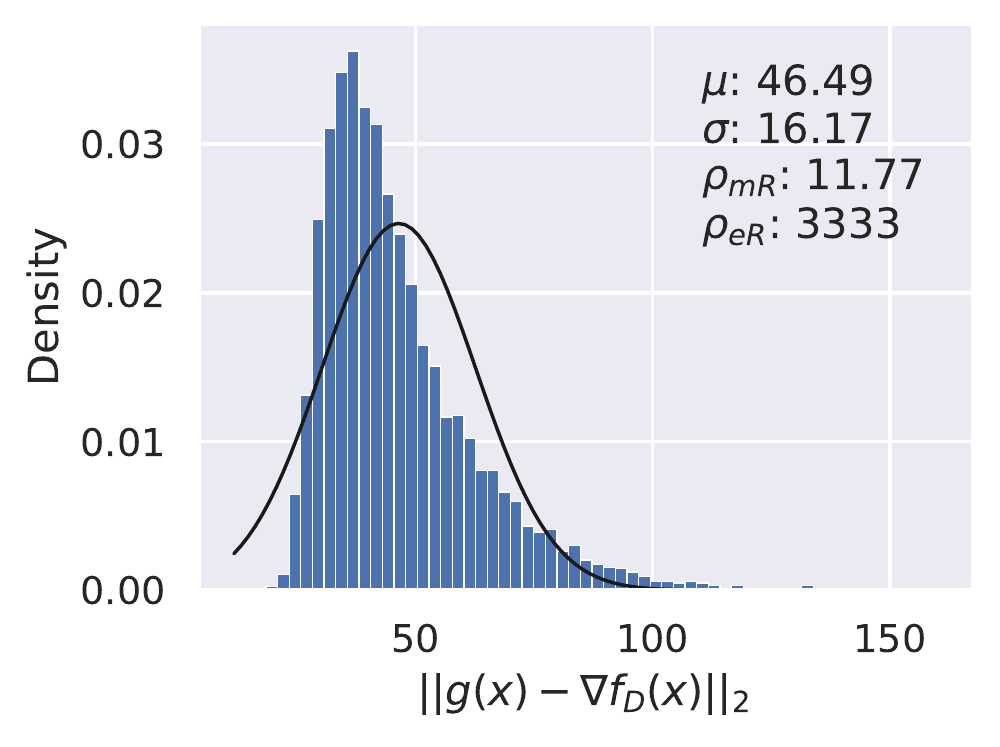} \\
        & {20000 steps} & {40000 steps} & {80000 steps} & {100000 steps} \\
    
    \end{tabular}
    \caption{\small Evolution of gradient noise histograms for the best WGAN-GP model trained with \algname{clipped-SEG} (cordinate clipping).}
    \label{fig:app/evo/clipval-SEG}
\end{figure}

\newpage

\begin{table}[]
    \begin{minipage}{.45\linewidth}
        \centering
        \caption{\small
        StyleGAN2 \algname{SGDA} hyperparameter sweep, and the best FID score obtained in 2600~kimgs.}
        \label{tab:stylegan2/hparam/sgd}
        \medskip
        %
        %       SGDA
        % 
        \begin{tabular}{ccc}
        \toprule
          G-LR &   D-LR &   FID \\
        \midrule
         0.003 &  0.003 & 319.7 \\
        0.0075 & 0.0075 & 318.5 \\
          0.01 &   0.01 & 317.7 \\
         0.035 &  0.035 & 301.9 \\
          0.05 &   0.05 & 300.3 \\
         0.075 &  0.075 & 299.6 \\
           0.1 &    0.1 & 308.5 \\
          0.35 &   0.35 & 342.6 \\
           0.5 &    0.5 & 346.6 \\
        \bottomrule
        \end{tabular}
    \end{minipage}
    \hfill
    \begin{minipage}{.45\linewidth}
        \centering
        \caption{\small
        StyleGAN2 \algname{clipped-SGDA} (coordinate) hyperparameter sweep, and the best FID score obtained in 2600~kimgs. Bold row denotes the best run which was trained to convergence.}
        \label{tab:stylegan2/hparam/clipval-sgd}
        \medskip
        %
        %       val clip SGDA
        % 
        \begin{tabular}{ccccc}
        \toprule
         G-LR &  D-LR &  G-clip &  D-clip &   FID \\
        \midrule
          0.2 &   0.2 &   0.001 &   0.001 & 243.5 \\
          0.3 &   0.3 &   0.001 &   0.001 & 169.5 \\
         0.35 &  0.35 &  0.0005 &  0.0005 & 192.9 \\
         0.35 &  0.35 &   0.001 &   0.001 & 148.6 \\
         \textbf{0.35} &  \textbf{0.35} &  \textbf{0.0025} &  \textbf{0.0025} & \textbf{104.9} \\
         0.35 &  0.35 &   0.005 &   0.005 & 149.1 \\
         0.35 &   0.5 &    0.01 &    0.01 & 170.6 \\
          0.4 &   0.4 &   0.001 &   0.001 & 155.8 \\
          0.5 &   0.5 &  0.0001 &  0.0001 & 289.8 \\
          0.5 &   0.5 &   0.001 &   0.001 & 136.1 \\
        \bottomrule
        \end{tabular}
    \end{minipage}
\end{table}

\begin{figure}[]
    \centering
    \hspace*{-1.8cm}       
    \includegraphics[width=1.25\textwidth]{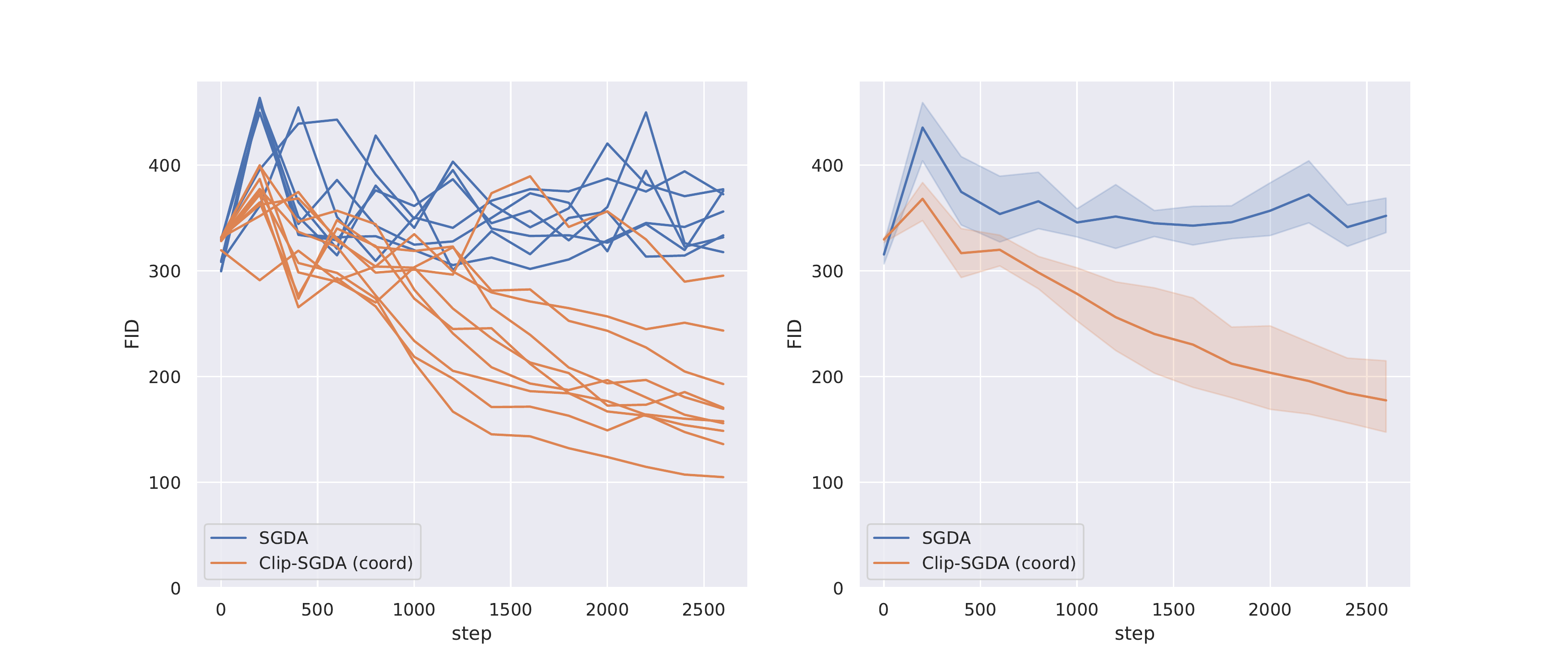}
    \caption{FID curves when training StyleGAN2 for 2600 kimgs (thousands of images seen by the discriminator) with \algname{SGDA} and \algname{clipped-SGDA} (coordinate), corresponding to the hyperparameters in Tables~\ref{tab:stylegan2/hparam/sgd}~\&~\ref{tab:stylegan2/hparam/clipval-sgd} respectively.
    The left figure is the individual runs for each choice of hyperparameters, and the right is the mean and $95\%$ confidence interval of these runs.
    Every \algname{SGDA}-trained model for the wide range of learning rates we tried failed to improve the FID, while models trained with \algname{clipped-SGDA} (with appropriately set hyperparameters) are generally able to learn some meaningful features and improve the FID.}
    \label{fig:stylegan/FID}
\end{figure}

\newpage

\begin{figure}[]
    \subfigure[Additional samples generated from our best model trained with \algname{clipped-SGDA} (lr=0.35, clip=0.0025).]{
        \includegraphics[width=\textwidth]{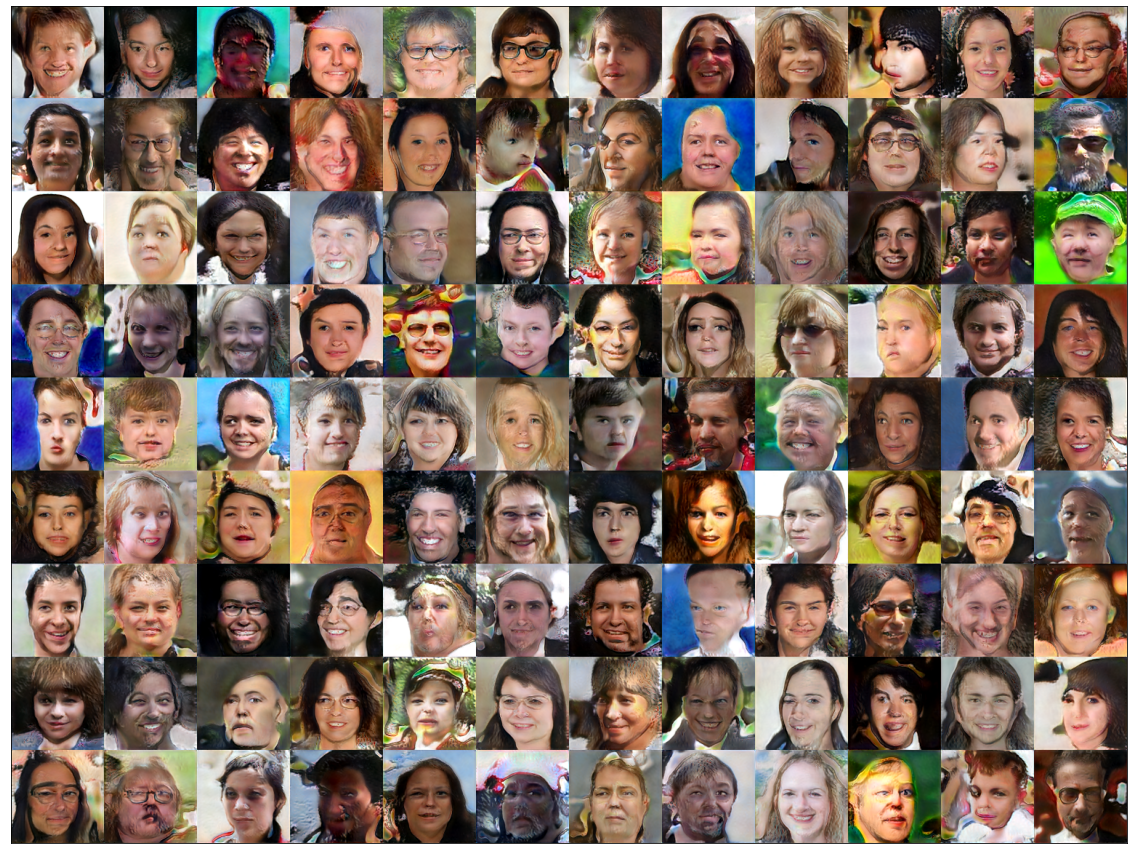}
        \label{fig:stylegan/clippedsgda/moresamples}
    } 
    \subfigure[Additional samples generated from several different \algname{SGDA} trained models, all of which failed to generate meaningful features. Each row corresponds to a model trained with different learning rates.]{
        \includegraphics[width=\textwidth]{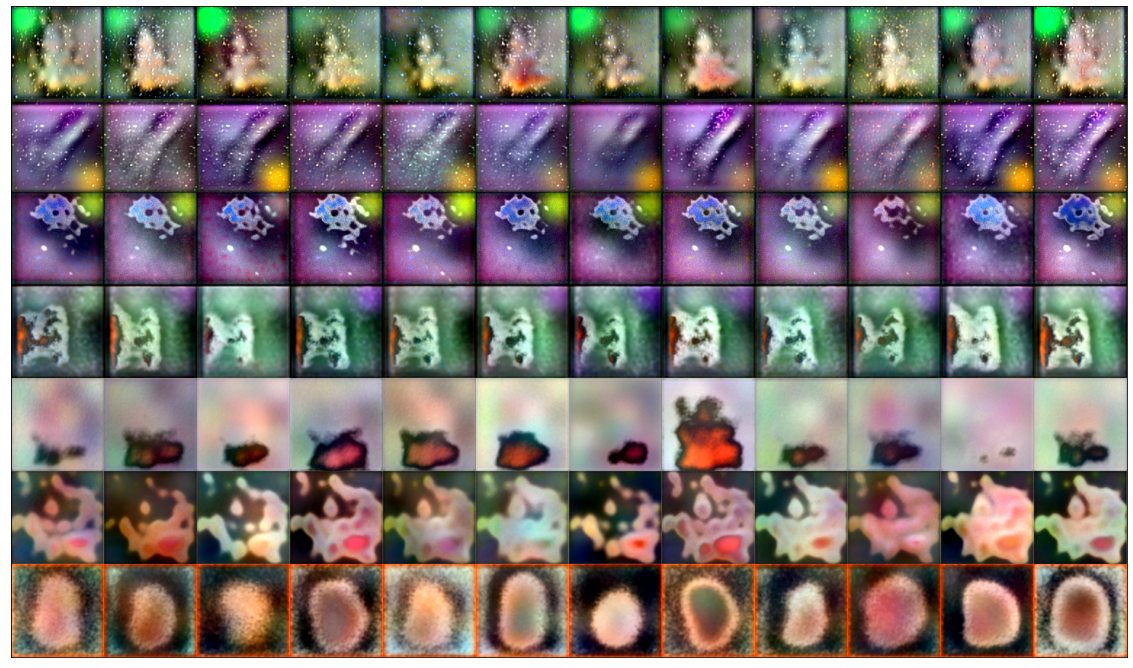}
        \label{fig:stylegan/sgda/moresamples}
    } 
    \caption{More StyleGAN2 samples.}
    \label{fig:stylegan/more}
    \vspace{-6mm}
\end{figure}

\end{document}